\documentclass[11pt]{amsart}

\newcommand{\R}{\mathbb R}
\newcommand{\C}{\mathbb C}

\newcommand{\Z}{\mathbb Z}
\newcommand{\N}{\mathbb N}

\usepackage{graphicx}
\usepackage{amssymb,amsmath,amsfonts,geometry,color,float,caption,multirow}
\usepackage{hyperref}
\usepackage{mathrsfs}

\definecolor{gray1}{rgb}{0.7,0.7,0.7}
\definecolor{green1}{rgb}{0.0,0.5,0.1}

\numberwithin{equation}{section}

\newcommand{\beq}{\begin{eqnarray*}}
\newcommand{\eeq}{\end{eqnarray*}}
\newcommand{\tr}{\mathrm{tr}}
\newcommand{\lev}{\mathscr{L}}
\newcommand{\coll}{\mathcal{C}}

\theoremstyle{definition}
\newtheorem{definition}{Definition}[section]
\newtheorem*{notation}{Notation}
\newtheorem*{notations}{Notations}
\newtheorem{example}{Example}[section]
\newtheorem{remark}{Remark}[section]

\theoremstyle{plain}
\newtheorem{theorem}{Theorem}[section]
\newtheorem{corollary}[theorem]{Corollary}
\newtheorem{proposition}[theorem]{Proposition}
\newtheorem{lemma}[theorem]{Lemma}

\begin{document}

\title[Finiteness of central configurations for the six-body problem]{
Toward finiteness of central configurations for the planar six-body problem by symbolic computations}

\author{Ke-Ming Chang}
\address{Department of Mathematics, National Tsing Hua University, Hsinchu 30013, Taiwan}
\email{sc5764801@gmail.com}

\author{Kuo-Chang Chen}
\address{Department of Mathematics, National Tsing Hua University, Hsinchu 30013, Taiwan}
\email{kchen@math.nthu.edu.tw}
\thanks{This work is supported in parts by the National Science Council in Taiwan}  

\subjclass[2010]{Primary 70F10, 70F15; Secondary 68W30, 03B35}  

\date{March 1, 2023}

\keywords{$n$-body problem, central configuration, symbolic computation}  

\begin{abstract}
In this paper we develop symbolic computation algorithms to investigate finiteness of central configurations for the planar 
$n$-body problem. Our approach  is based on Albouy-Kaloshin's work on finiteness of central configurations for the 
5-body problems. 
In their paper, bicolored graphs called $zw$-diagrams were introduced for possible scenarios when the finiteness conjecture fails, 
and proving finiteness amounts to exclusions of central configurations associated to these diagrams. 
Following their method, the amount of computations becomes enormous when there are more than five bodies.
Here we introduce matrix algebra for determination of both diagrams and asymptotic orders, 
devise several criteria to  reduce computational complexity, and verify finiteness mostly through automated deductions. 
For the planar six-body problem, our first algorithm effectively narrows the proof for finiteness  down to  117 $zw$-diagrams,  the
second algorithm eliminates 31 of them, the last algorithm eliminates 62 other diagrams 
except for masses in some  co-dimension 2 variety in the mass space, and leaving 24 cases unsolved. 
\end{abstract}

\maketitle

\vspace{-4mm}
\setcounter{tocdepth}{1}
\renewcommand*\contentsname{Table of Contents}
\tableofcontents

\section{Introduction}

The Newtonian $n$-body problem concerns $n$ mass points $m_1, \cdots, m_n$ moving in space in accordance with 
Newton's law of universal gravitation:
\begin{eqnarray}\label{eqn:NBD}
\ddot{q}_k &=& \sum_{l\neq k}\frac{m_l (q_l-q_k)}{|q_l-q_k|^3},
\quad k=1,2,\cdots,n, 
\end{eqnarray}
where $q_k\in\R^d$ ($1\leq d\leq3$)  is the position of mass $m_k$.
The position vector $q=(q_1,\cdots,q_n)\in(\R^d)^n$ is called  the {\em configuration} of the system. 
The system (\ref{eqn:NBD}) is smooth except on the {\em collision set}   $\Delta$ defined by
\begin{eqnarray*}
\Delta \;=\;\{q\in (\R^d)^n: q_i=q_j\; \text{for some}\; i\neq j\}.
\end{eqnarray*}
Configurations in $\Delta$ are called {\em collision configurations}. 

Central configurations are non-collision configurations with 
the property that gravitational acceleration $\ddot{q}_k$ of each mass point $m_k$ is a 
constant multiple of the vector from $q_k$ to the mass center of the system.
More precisely, let 
$$q_c=\frac{1}{m_1+\cdots+m_n}(m_1q_1+\cdots+m_nq_n)$$
be the mass center of the system. We say 
$q=(q_1,\cdots,q_n)\in(\R^d)^n\setminus\Delta$ is a {\em central configuration} for the system
of masses $(m_1,\cdots,m_n)$ if there exists a positive
constant $\lambda$, called the {\em multiplier}, such that
\begin{eqnarray}\label{eqn:cc}
-\lambda (q_k-q_c) = 
\sum_{l\neq k} \frac{m_l(q_l-q_k)}{|q_l-q_k|^3}, \quad k=1,2,\cdots,n.
\end{eqnarray}
The set of central configurations is clearly invariant under similarity transformations, so 
two central configurations are considered equivalent if they are in the same similarity class.

Central configurations are of special importance in celestial mechanics because of their pivotal roles
in self-similar solutions, singularity analysis, parabolic solutions, and bifurcations of integral manifolds.  
There are also practical applications such as space mission designs, asteroids' motions, modeling planetary rings, et cetera. 
From purely mathematical perspective, it is a profound and fascinating subject with many fundamental questions been left open. 
We refer readers to  \cite{A03,AK,Moeckel90,MoeckelLN} for introductions and bibliographies, 
to \cite{ACS,MoeckelLN} for many open questions.
A fundamental question concerns the finiteness of central configurations: \\

\noindent{\bf Finiteness Conjecture}: 
{\em For any choice of $n$ positive masses, there are only finitely many similarity classes of central configurations.} \\ 

The finiteness conjecture, also known as the Chazy-Wintner conjecture, 
has its prototype in Chazy's 1918 paper \cite{Chazy} which postulated non-degeneracy of central configurations, 
and phrased  in present form in Wintner's classic treatise \cite{Wintner}. 
Moulton~\cite{Moulton} proved in 1910 that the collinear case ($d=1$) has exactly $n!/2$ similarity classes of central configurations, 
so the finiteness conjecture is only for $d\ge 2$. 
The planar version ($d=2$) of the finiteness conjecture was  included in Smale's list of mathematical problems for the 
21st century~\cite{Smale}, thereby substantially increases  the attention and popularity of the topic within the mathematical community. 
See \cite[\S1]{AK} for a comprehensive  introduction to the history of the finiteness conjecture. 
In below we shall only briefly describe its current state, and progresses that are most relevant to our present work.

We consider the planar problem throughout this paper. 
It is well-known that the three-body problem has 5 similarity classes of central configurations~\cite{Euler,Lagrange}, 
but counting central configurations for 
cases with more than three bodies are generally challenging tasks which involve counting real roots of sophisticated polynomial systems.  
Such polynomial systems vary according to variables selected, and were deduced from algebraic manipulations of (\ref{eqn:cc}), 
in a manner  that similarity classes of central configurations are embedded into the zero locus of the polynomial system. 

Hampton and Moeckel~\cite{HM06} settled the finiteness conjecture in 2006  for the four-body problem,
improving an earlier work for generic mass \cite{Moeckel01}, 
by using the Bernstein-Khovanskii-Kushnirenko (BKK) theory in computational algebraic geometry.  
Their polynomial equations are primarily Albouy-Chenciner equations~\cite{AC} and Dziobek equations~\cite{Dziobek}. 
The proof involves computations of Newton polytopes for some complicated polynomials, the Minkowski sum and mixed volume, 
and eliminations of variables for reduced polynomial systems corresponding to faces of the Minkowski sum polytope.  
The number of faces is excessive, many of them correspond polynomial systems that are virtually 
impossible to handle manually, so the assistance of computer algebra becomes inevitable. 
Although computations are massive, there is no round-off error since all computations are exact and involve only integers. 
With similar settings and by adopting tools from tropical geometry, finiteness results were obtained for some classes of 
planar and spatial five-body problems by Hampton~\cite{Hampton09}, Hampton-Jensen~\cite{HJ11},  
Jensen-Leykin~\cite{JL}. 

In 2012, Albouy and Kaloshin~\cite{AK} settled the finiteness conjecture for the planar five-body problem except for masses in 
some  co-dimension 2 variety in the mass space. 
They selected another set of variables for central configurations, yielding a different polynomial system, and 
devised a graph-based apparatus  to tackle finiteness of real roots. 
When applied to the four-body problem,  their approach significantly reduces the amount of computations needed for the BKK approach, 
and can be carried out without  the assistance of computer. 
The case with five bodies requires much more computations, and 
symbolic computations were used in due course  to factorize some ``huge'' polynomials. 

The finiteness conjecture cannot be extended to arbitrary nonzero masses as a famous counterexample by Roberts~\cite{Roberts} 
shows a continuum of central configurations for $n=5$  if negative masses are allowed. 
An interesting recent work of Yu-Zhu~\cite{YuZhu} based on Albouy-Kaloshin's approach proves finiteness for $n=4$ in the presence 
of negative masses. 

Recent advancements for the finiteness conjectures were 
based on  insightful observations for central configurations, modern tools from 
computational algebraic geometry and computer algebra. 
If one follows previous settings and approaches directly, the case $n=6$ with general masses still seems  far beyond reach due to the 
enormous amount of  computations, and astounding complexity of polynomials involved.  
Finiteness for $n=6$ has been proved for only a few exceptional cases \cite{DP20,Leandro,MZ19,MZ20,Xia91},
for some of which the exact number of central configurations were obtained.

As a step toward the finiteness conjecture for $n=6$, in this paper we follow Albouy-Kaloshin's approach, 
introduce two families of matrices to facilitate algebraic manipulations, and 
design three algorithms of symbolic computations to carry out large part of computations. 
It is our goal to develop efficient tools to tackle the finiteness conjecture through automated deductions, 
save as much computing time as possible, and to leave as few unsolved cases as possible, 
so that further investigations can  confidently depart from our remaining cases. 
Another aim of our work is to provide a convenient platform for this research direction via symbolic computations -- 
by exploring more and more criteria or rules for central configurations, one can simply transform them to  
new subroutines and add them to our algorithms. For this purpose we open our source codes with Mathematica in Appendices, 
where we also include clear instructions on correspondence between subroutines and criteria been applied.

In \cite{AK}  bicolored graphs called {\em $zw$-diagrams} were introduced for possible scenarios when the finiteness conjecture fails, 
and proving finiteness amounts to exclusions of central configurations associated to these diagrams. 
In their paper, proving finiteness for the case $n=4$ was narrowed down to 5 $zw$-diagrams, the case $n=5$ to 16 $zw$-diagrams.  
Comparing with other works about finiteness, their settings appear to be most ideal for the case $n=6$, 
although we have not resolve all cases.  
Here we introduce matrix algebra for determination of both $zw$-diagrams and asymptotic orders of primary variables, 
devise several criteria to  reduce computational complexity, and use our algorithms to verify finiteness. 
Our first algorithm narrows down the case $n=6$ to 117 $zw$-diagrams, the second algorithm eliminates 31 of them, 
and the third algorithm finds mass relations for 62 diagrams, one of which is impossible for positive masses. 
This leaves  finiteness of 24  cases unsolved. 

The paper is organized as follows. 
We will begin with a brief introduction of Albouy-Kaloshin's approach in section~\ref{sec:ColoringRules}. 
In particular, we recall their concepts  of singular sequences, $zw$-diagrams,  coloring rules, and some estimates 
that are critical to our analysis and algorithms. 
In Section~\ref{sec:MatrixRules} we associate $zw$-diagrams to their adjacency matrices, and establish several matrix rules. 
Some matrix rules are direct consequence of coloring rules, some are more sophisticated. 
In section~\ref{sec:OrderingRules}, we associate $zw$-diagrams another set of matrices, 
called {\em order matrices}, which tracks possible orders of primary variables. 
Several criteria for order matrices will be proved. 
Section~\ref{sec:poly} shows the procedure of collecting polynomial equations. 
Then, in section~\ref{sec:algorithm I}, we provide an effective algorithm to generate all possible $zw$-diagrams.  
The algorithm for generating order matrices will be given in section~\ref{sec:algorithm II}. 
Based on these order matrices, we either eliminate the $zw$-diagram directly or collect polynomial equations for the diagram.  
Section~\ref{sec:algorithm III}  shows an algorithm for eliminations and generating mass relations. 
Sections~\ref{sec:4bd},~\ref{sec:5bd}  are respectively applications of our algorithms to the four- and five-body problems. 
Section~\ref{sec:6bd} contains the major progress of our work, namely the application to the six-body problem. 
Appendices contain source codes of our Mathematica programs and details of some frequently appeared mass relations.

\section{Singular sequences, $zw$-diagrams, and coloring rules} \label{sec:ColoringRules}

This section is a brief introduction of Albouy and Kaloshin's method \cite{AK} for planar central configurations. 
We introduce their formulations for normalized central configurations, the concept of singular sequences, 
$zw$-diagrams, and then summarize most of  their estimates and coloring rules. 
These concepts and rules are necessary preliminaries for our arguments and algorithms in subsequent sections. 

\subsection{Normalized central configurations}

Let $q_k=(x_k,y_k)\in\R^2$ be the position of mass $m_k>0$, $q_{lk}=q_k-q_l$, $q_{lk}=(x_{lk},y_{lk})$, and 
let $r_{lk}=\|q_l-q_k\|$ be the mutual distance between $q_l$ and $q_k$. 
Since the set of central configurations is invariant under similarity transformations, 
we may quotient out translations by setting the mass center at the origin, quotient out scalings by setting the multiplier $\lambda=1$, 
and quotient out rotations by setting $y_{12}=0$. Then (\ref{eqn:cc}) becomes 
\begin{eqnarray} 
\notag  x_k&=&\sum_{l\neq k}m_lr_{lk}^{-3}x_{lk}, \quad  k=1,\ldots, n, \\  
 y_k&=&\sum_{l\neq k}m_lr_{lk}^{-3}y_{lk}, \quad   k=1,\ldots, n,     \label{eqn:cc2}  \\
\notag y_{12}&=&0.  
\end{eqnarray}

Consider complex $x_k$, $y_k$, $m_k$, and add new variables $\delta_{kl}\in \C$, $1\leq k\neq l\leq n$, for  $r_{kl}^{-1}$. 
Then (\ref{eqn:cc2})  can be rewritten as the following polynomial system in $\C^{2n}\times \C^{\frac{n(n-1)}{2}}$. 
\begin{eqnarray}
\notag  x_k&=&\sum_{l\neq k}m_l\delta_{lk}^3x_{lk}, \quad  k=1,\ldots,n, \\
 y_k&=&\sum_{l\neq k}m_l\delta_{lk}^3y_{lk}, \quad  k=1,\ldots,n,     \label{eqn:cc3}\\
\notag 1&=&\delta_{kl}^2(x_{kl}^2+y_{kl}^2), \quad  1\leq k<l\leq n, \\
\notag y_{12}&=&0.
\end{eqnarray}

With complexified $m_k$, it is possible that the total mass of a subsystem equals zero. 
In order that  ``center of mass of a subsystem'' still makes sense, a weak hypothesis on masses were imposed:
\beq
   \sum_{k\in I} m_k \;\ne \; 0\quad \text{for any nonempty $I\subset\{1,\cdots,n\}$. }
\eeq

\begin{definition}
A vector $(x_1,\ldots,x_n,y_1,\ldots,y_n,\delta_{12},\ldots,\delta_{(n-1)n})\in \C^{2n}\times \C^{\frac{n(n-1)}{2}}$ is 
said to be a \textit{normalized central configuration} if it is a solution of the system  (\ref{eqn:cc3}). 
A  \textit{real normalized central configuration}  is a normalized central configuration with real $x_k$'s and $y_k$'s, 
and a \textit{positive normalized central configuration} is a real normalized central configuration with positive $\delta_{kl}$'s.
\end{definition}

\begin{lemma}
Let $\mathcal{A}\subset \C^{2n}\times \C^{\frac{n(n-1)}{2}}$ be the set of normalized central configurations,   and let 
$$
  U=\sum_{k<l}\frac{m_km_l}{r_{kl}}=\sum_{k<l}m_km_l\delta_{kl}
$$
be the potential. Then  $U(\mathcal{A})$  is finite. 
\end{lemma}

\subsection{Singular sequences.}

We set $z_k=x_k+iy_k$, $w_k=x_k-iy_k$, $z_{kl}=z_k-z_l$, and $w_{kl}=w_k-w_l$. Then $z_{kl}w_{kl}=r_{kl}^2$, 
$$
\begin{pmatrix}z_k\\w_k\end{pmatrix}=\sum_{l\neq k}m_lr_{lk}^{-3}\begin{pmatrix}z_{lk}\\w_{lk}\end{pmatrix}
=\sum_{l\neq k}m_l
\begin{pmatrix}{
z_{lk}^{-1/2}w_{lk}^{-3/2}}\\{z_{lk}^{-3/2}w_{lk}^{-1/2}}
\end{pmatrix}, \quad k=1, \cdots, n, 
$$ and the condition $y_{12}=0$ in system (\ref{eqn:cc3}) is equivalent to $z_{12}=w_{12}$. 
Terms  $r_{kl}$ are called ``distances'', terms  $z_{kl}$ and $w_{kl}$ are called   ``$z$-separation'' and ``$w$-separation'' respectively.

Let $Z_{kl}=z_{kl}^{-1/2}w_{kl}^{-3/2}$ and $W_{kl}=z_{kl}^{-3/2}w_{kl}^{-1/2}$. The system (\ref{eqn:cc3}) becomes
\begin{eqnarray}
\notag z_k&=&\sum_{l\neq k}m_lZ_{lk}, \quad  k=1,\ldots,n, \\
 w_k&=&\sum_{l\neq k}m_lW_{lk}, \quad  k=1,\ldots,n,    \label{eqn:cc4} \\
\notag 1&=&\delta_{kl}^2z_{kl}w_{kl}, \quad  1\leq k<l\leq n, \\
\notag z_{12}&=&w_{12}.
\end{eqnarray}

With these new variables, a normalized central configuration takes the form
$$\mathcal{Q}=(z_1,\ldots,z_n,w_1,\ldots,w_n,\delta_{12},\ldots,\delta_{(n-1)n}),$$ 
to which we associate two vectors $\mathcal{Z}$ and $\mathcal{W}$ in $\C^\mathcal{N}$, where $\mathcal{N}=\frac{n(n+1)}{2}$,
given by 
\begin{eqnarray*}
\mathcal{Z} &=& (\mathcal{Z}_1,\ldots,\mathcal{Z}_{\mathcal{N}})\;\; \;= \; (z_1,\ldots,z_n,Z_{12},\ldots,Z_{(n-1)n}), \\
\mathcal{W} &=& (\mathcal{W}_1,\ldots,\mathcal{W}_{\mathcal{N}})\; =\; (w_1,\ldots,w_n,W_{12},\ldots,W_{(n-1)n}).  
\end{eqnarray*}

Let $\|\mathcal{Z}\|=\max_{1\leq k\leq \mathcal{N}}|\mathcal{Z}_k|$ and $\|\mathcal{W}\|=\max_{1\leq k\leq \mathcal{N}}|\mathcal{W}_k|$. 
For a sequence $\mathcal{Q}^{(m)}$ of normalized central configurations, we choose a subsequence 
$\mathcal{Q}^{(m_j)}$ such that $\|\mathcal{Z}^{(m_j)}\|=|\mathcal{Z}_{p}^{(m_j)}|$ for some 
$1\leq p\leq \mathcal{N}$ and $\|W^{(m_j)}\|=|\mathcal{W}_{q}^{(m_j)}|$ for some $1\leq q\leq \mathcal{N}$. 
We extract a subsequence again such that both $\|\mathcal{Z}\|^{-1}\mathcal{Z}$ and $\|\mathcal{W}\|^{-1}\mathcal{W}$ converge.

If $\mathcal{Z}$ is bounded, since $z_{12}=w_{12}$, the term $Z_{12}$ equals $z_{12}^{-2}$, 
so it turns out that $z_1$, $z_2$, and $Z_{12}$ can not go to zero simultaneously. 
Therefore $\|\mathcal{Z}\|$ of the extracted subsequence is bounded away from zero. 
Likewise, if $\mathcal{W}$ is bounded, then $\|\mathcal{W}\|$ of the extracted subsequence is bounded away from zero. 

There are two kinds of extracted subsequences: 
\begin{enumerate}
\item[(1)] both $\mathcal{Z}$ and $\mathcal{W}$ are bounded,
\item[(2)] at least one of $\mathcal{Z}$ and $\mathcal{W}$ is unbounded.
\end{enumerate}
In either case, $\|\mathcal{Z}\|$ and $\|\mathcal{W}\|$ are bounded away from zero. 
One can show that the first kind of extracted subsequence converges to a normalized central configuration.

\begin{definition}
Consider a sequence of normalized central configurations. 
A subsequence extracted by the above process is said to be a \textit{singular sequence} if it is of the second kind.
\end{definition}

\subsection{Colored diagrams for singular sequences.} \label{subsec:ColoredDiagram}

The next step is to classify singular sequences. 
For two sequences $a,b$ of non-zero numbers, we use $a\sim b$, $a\preceq b$, $a\prec b$, and $a\approx b$ to 
represent ``$a/b\rightarrow 1$'', ``$a/b$ is bounded'',  ``$a/b\rightarrow 0$'', and ``$a\preceq b$ and $b\preceq a$'' respectively.

From now on, a new normalization will be used. 
We multiply $z_k$'s by a positive number $a$ and multiply $w_k$'s by $a^{-1}$ such that $\|\mathcal{Z}\|=\|\mathcal{W}\|$. 
In this normalization, we don't have the relation $z_{12}=w_{12}$.

Since $\mathcal{Z}$ and $\mathcal{W}$ are bounded away from zero, for a singular sequence, 
$\|\mathcal{Z}\|\,\|\mathcal{W}\|$ is always unbounded. 
We set $\|\mathcal{Z}\|=\|\mathcal{W}\|=\epsilon^{-2}$, where $\epsilon$ is a sequence converging  to 0.

For a singular sequence, we write down indices of bodies, draw a circle around index $k$ if $z_k$ has the 
maximum order (i.e. $z_k\approx \epsilon^{-2}$), and draw a stroke between $k$ and $l$ if $Z_{kl}\approx\epsilon^{-2}$. 
We call this diagram the {\em $z$-diagram}. 
For $w_k$'s and $W_{kl}$'s, we use another color to do the same things, and draw a {\em $w$-diagram}. 
Combining $z$-diagram and $w$-diagram, we obtain a bicolored diagram, 
called the \textit{$zw$-diagram}, or simply {\em diagram}, associated to the singular sequence.
Indices of bodies are  {\em vertices} of the diagram. 

The term  ``edge'' is different from ``stroke''.  
There is an edge between vertex $k$ and vertex $l$ means there is a $z$-stroke or $w$-stroke between vertex $k$ and vertex $l$. 
There are three types of edges: $z$-edge, $w$-edge, $zw$-edge. 
The first means  there is only $z$-stroke between these two vertices, the second means there is only one $w$-stroke 
between these two vertices, the third means there are both $z$-stroke and $w$-stroke between these two vertices.

\subsection{Closeness and clustering scheme.}

In the $z$-diagram,
\begin{enumerate}
\item[(1)] we say two vertices $k$ and $l$ are \textit{close in $z$-diagram} or \textit{$z$-close} if $z_{kl}\prec \epsilon^{-2}$.
\item[(2)] a nonempty subset $V$ of $\{1,2,\ldots, n\}$ is said to be an \textit{isolated component} if there is no $z$-stroke from 
one vertex in $V$ to another vertex outside $V$.
\item[(3)] a $z$-stroke between $k$ and $l$ is said to be \textit{maximal} if $z_{kl}\approx \epsilon^{-2}$.
\end{enumerate}

Concepts of \textit{$w$-close}, isolated component, maximal $w$-stroke in the $w$-diagram are defined in the same way.

As we pass to the limit of a singular sequence, the $z_k$'s and $w_k$'s may form clusters. 
For example, if $z_{12}\prec z_{13}$, we say 1 clusters with 2 in $z$-coordinate relatively to a subset of bodies $\{1,2,3\}$. 
When we consider more bodies, they may form sub-clusters. 
To represent the \textit{cluster scheme}, we write $z:24.1...3$ for the situation $z_{24}\prec z_{12}\prec z_{13}$. 
Two indices without dot means their $z$-separation is smallest (among all separations between these bodies); 
one dot means the separation is of medium order; three dots means the separation is largest. 
Three precise different orders of separations in terms of power functions of $\epsilon$ will be considered: 
$\epsilon^2$, $\epsilon$, $\epsilon^{-2}$. 

\subsection{Order estimates and coloring rules}  \label{subsec:estmates,coloringrules}
Here we summarize some key estimates and coloring rules in \cite{AK}. 
\begin{theorem}[Estimate 1]  \label{thm:est1}
For any two vertices $k$ and $l$, we have 
\begin{enumerate}
\item[(1)] $\epsilon\preceq r_{kl}\preceq \epsilon^{-2}$,
\item[(2)] $\epsilon^2\preceq z_{kl},w_{kl}\preceq \epsilon^{-2}$,
\item[(3)] there is a $zw$-edge between $k$ and $l$ if and only if $r_{kl}\approx \epsilon$,
\item[(4)] there is a maximal $z$-edge  between $k$ and $l$ if and only if $w_{kl}\approx \epsilon^2$.  
\end{enumerate}
\end{theorem}

\begin{theorem}[Estimate 2]  \label{thm:est2}
Suppose there is a $z$-stroke between $k$ and $l$, then
\begin{enumerate}
\item[(1)] $\epsilon\preceq r_{kl}\preceq 1$,
\item[(2)] $\epsilon\preceq z_{kl}\preceq \epsilon^{-2}$,
\item[(3)] $\epsilon^2\preceq w_{kl}\preceq \epsilon$,
\item[(4)] it is a $zw$-edge\;\; $\Leftrightarrow \;\; r_{kl}\approx \epsilon\;\; \Leftrightarrow\;\; z_{kl}\approx \epsilon\;\; \Leftrightarrow
\;\; w_{kl}\approx \epsilon$,
\item[(5)] it is a maximal $z$-stroke \;\;  $\Leftrightarrow \;\; r_{kl}\approx 1 \;\; \Leftrightarrow \;\; z_{kl}\approx \epsilon^{-2}
\;\; \Leftrightarrow \;\; w_{kl}\approx \epsilon^2$.
\end{enumerate}
\end{theorem}

Theorems~\ref{thm:est1},~\ref{thm:est2} are valid if  ``$z$'' and ``$w$'' were switched.

\begin{theorem}[One Color Rules]
In the $z$-diagram of a singular sequence, we have the following properties.
\begin{enumerate}
\item[]\hskip-13mm {\bf (Rule 1a)} 
These exists at least one $z$-stroke. 
At each end of a $z$-stroke, there is either another $z$-stroke emanating from it or a $z$-circle around it.
At each $z$-circle, there is a $z$-stroke emanating from it. 
\item[]\hskip-13mm {\bf (Rule 1b)} If bodies $k$ and $l$ are $z$-close, then they are both $z$-circled or both not $z$-circled.
\item[]\hskip-13mm {\bf (Rule 1c)} In an isolated component,  
the center of mass of bodies in this component is $z$-close to the origin.
\item[]\hskip-13mm {\bf (Rule 1d)} 
In an isolated component of the $z$-diagram, 
if there is a $z$-circle, there is another one. Moreover, these $z$-circles can't be all $z$-close together.
\item[]\hskip-13mm {\bf (Rule 1e)} In an isolated component, there is a maximal $z$-stroke if and only if there is a $z$-circle.
\end{enumerate}
Rules 1a-1e hold if ``$z$-'' are replaced by ``$w$-''.
\end{theorem}

\begin{theorem}[Two Colors Rules]
In the $zw$-diagram of a singular sequence, we have the following properties.
\begin{enumerate}
\item[]\hskip-13mm {\bf (Circling Method)} If there is a $w$-stroke (resp. $z$-stroke) between $k$ and $l$, 
then $k$ and $l$ are both $z$-circled (resp. $w$-circled) or both not $z$-circled (resp. $w$-circled).
\item[]\hskip-13mm {\bf (Rule 2a)} For each $zw$-edge in the diagram, there is another $z$-stroke and $w$-stroke emanating from it.
\item[]\hskip-13mm {\bf (Rule 2b)} If there are two consecutive $zw$-edges, 
there is a third $zw$-edge such that these three edges form a triangle.
\item[]\hskip-13mm {\bf (Rule 2c)} Consider two consecutive edges which are not part of a triangle. 
Suppose they are an edge between 1,2 and an edge between 2,3. 
Then the cluster schemes are $z:1.2...3$, $w:1...2.3$ or $z:1...2.3$, $w:1.2...3$. In this situation, we say there is \textit{skew clustering}.
\item[]\hskip-13mm {\bf (Cor. of Rule 2c)} If there are two consecutive $z$-edges ($w$-edges) which are 
not part of a triangle in the $zw$-diagram, they can not be both maximal.
\item[]\hskip-13mm {\bf (Rule 2d)} Consider a cycle of edges,  the list of $z$-separations in this cycle, 
and the maximal order within the list. At least two of these $z$-separations are of this order, 
and the corresponding edges are of the same type. Moreover, if there are exactly two $z$-separations of this order, 
then the two $z$-separations are equivalent (i.e. $z_{k_1l_1}\sim z_{k_2l_2}$ if these edges are between $k_1,l_1$ and 
between $k_2,l_2$). Same rule holds if ``$z$-'' are replaced by ``$w$-''.
\item[]\hskip-13mm {\bf (Rule 2e)} If  three edges form a triangle, 
then they are of the same type, the corresponding $z$-separations are of the same order, 
and the corresponding $w$-separations are of the same order.
\item[]\hskip-13mm {\bf (Cor. of Rule 2e)} Consider any three vertices, the number of strokes joining them is $0,1,2,3$, or $6$.
\item[]\hskip-13mm {\bf (Rule 2f)} Consider a triangle with vertices $k_1,k_2,k_3$ formed by edges in the diagram. 
Suppose there exists $p>3$ and other vertices $k_4,k_5,\ldots,k_p$ such that for each $4\leq t\leq p$, 
there are at least two edges connecting $k_t$ with vertices in $k_1,\ldots, k_{t-1}$. 
Then there is an edge between any two of $k_1,\ldots k_p$, these edges are of the same type, 
their corresponding $z$-separations are of the same order,  
their corresponding $w$-separations are of the same order. 
\item[]\hskip-13mm {\bf (Rule 2g)} If there are four edges forming a quadrilateral, then the opposite edges have the same type.
\item[]\hskip-13mm {\bf (Rule 2h)} 
If there are bodies $k_0,l_0$ such that 
$r_{k_0l_0}\preceq r_{kl}$ for all $k,l$ and $r_{k_0l_0}\rightarrow 0$, then there is another pair of bodies $k_1, l_1$ 
such that $r_{k_1l_1}\approx r_{k_0l_0}$.
\item[]\hskip-13mm {\bf (Cor. of Rule 2h)} If there is a $zw$-edge in the diagram, then there is another one.
\end{enumerate}
\end{theorem}

Some other propositions from \cite[Prop. 1,2,4]{AK} for singular sequences are also helpful for our purpose: 

\begin{proposition}  \label{prop:AK1}
If $z_{1n}\prec z_{kn}$ and $w_{1n}\approx w_{kn}$ for every $k$ different from $1$ and $n$, 
then $z_1$, $z_n$ and the origin form a cluster of size $\approx z_{1n}$; i.e. $z_1, z_n\preceq z_{1n}$.
\end{proposition}

\begin{proposition}  \label{prop:AK2}
If there is some $k_0$ such that $r_{k_0l}\succ 1$ for all $l\neq k_0$, 
then there are some $k_1,k_2$ such that $z_{k_0}\prec z_{k_1}$ and $w_{k_0}\prec w_{k_2}$.
\end{proposition}

\begin{proposition}   \label{prop:AK4}
Suppose the diagram has exactly one maximal $z$-edge,  and it forms an isolated component,  
 say bodies 1 and 2 are endpoints of this edge. 
 If $m_1^3+m_2^3\neq 0$, then there is no $k$ with $3\leq k\leq n$ such that 
$w_{1k}\prec w_{1l}$ for all $l$ satisfying $l\neq k$ and $3\leq l\leq n$.
\end{proposition}

These propositions hold if  ``$z$'' and ``$w$'' were switched.

\subsection{Remarks on proof for finiteness}

In the case of 4 bodies, Albouy and Kaloshin first use coloring rules to construct all possible diagrams associated to  singular sequences. 
For some of those diagrams, they find an algebraic relation for masses that a singular sequence associated to the diagram 
necessarily satisfies. There are four diagrams with simple mass relations and a diagram that does not yield any mass relation. 

In the proof for finiteness, they often use the following arguments:

\begin{itemize}
\item If there are infinitely many normalized central configurations, then there is some polynomial function $p$ 
taking infinitely many values, and there is a singular sequence over which  $p$ tends to 0 or $\infty$.  \smallskip 
\item A singular sequence over which  $p_1\rightarrow 0$ (or $\infty$) is impossible in some diagrams, 
therefore another specified polynomial $p_2$ tends to 0 (or $\infty$) over this singular sequence. \smallskip 
\item Over a singular sequence, a polynomial $p\rightarrow 0$ (or $\infty$) implies that there is another singular sequence over which $p$ tends to $\infty$ (or $0$).  \smallskip 
\item A singular sequence over which  $p\rightarrow 0$ (or $\infty$) is impossible in some diagrams, 
therefore the masses must satisfy the mass relation of one of other diagrams. \smallskip 
\end{itemize}

Roughly speaking, the idea of proof for finiteness  is to take several singular sequences from the set of normalized central configurations, 
they are associated to different diagrams, but there are no masses satisfying  all of these corresponding mass relations.

\section{Matrix rules for singular sequences} \label{sec:MatrixRules}

In this section we define $z$-matrix, $w$-matrix, and $zw$-matrix associate to a singular sequence of central configurations, and
establish algebraic rules for these matrices. 
Most of these matrix rules are deduced from estimates, coloring rules, and propositions in the previous section. 
Some of them are immediate corollaries of coloring rules, but some others are more sophisticated. 
They provide effective tools for the implementation of our algorithms.

\begin{definition}
Given a singular sequence of central configurations 
for the $n$-body problem and consider the corresponding $zw$-diagram. 
The adjacency matrix for its $z$-diagram with $z$-circles regarded as self-loops is called 
the {\em $z$-matrix} associated to the singular sequence. More precisely, the $z$-matrix is the symmetric 
$n\times n$ binary matrix $A=(a_{ij})$ given by 
\beq
    a_{ij} &=&1 \;\;\text{if}\;\; i\neq j\;\; \text{and there is a $z$-stroke between bodies $m_i$ and $m_j$}; \\
             &=& 0 \;\;\text{if}\;\; i\neq j\;\; \text{and there is no $z$-stroke between bodies $m_i$ and $m_j$}; \\
    a_{ii} &=& 1 \;\;\text{if there is a $z$-circle at $m_i$}; \\
             &=& 0 \;\;\text{if there is no $z$-circle at $m_i$}.
\eeq
Similarly, the {\em $w$-matrix} associated to the singular sequence is the adjacency matrix for the $w$-diagram
with $w$-circles regarded as self-loops. It is the $n\times n$ binary matrix $B=(b_{ij})$ given by 
\beq
    b_{ij} &=&1 \;\;\text{if}\;\; i\neq j\;\; \text{and there is a $w$-stroke between bodies $m_i$ and $m_j$}; \\
             &=& 0 \;\;\text{if}\;\; i\neq j\;\; \text{and there is no $w$-stroke between bodies $m_i$ and $m_j$}; \\
    b_{ii} &=& 1 \;\;\text{if there is a $w$-circle at $m_i$}; \\
             &=& 0 \;\;\text{if there is no $w$-circle at $m_i$}.
\eeq
The {\em $zw$-matrix} of the singular sequence is the $n\times 2n$ binary matrix $(A|B)$. 
We say $A$ and $B$ are {\em companion matrices} of each other. 
\end{definition}

\begin{definition}
A  {\em connected component} of a $z$-matrix $A$\ (resp. $w$-matrix $B$)\ is a subset $I$ of $\{1,2,\ldots,n\}$ 
such that the principal submatrix obtained by removing all $j$-th columns and all $j$-th rows with $j\notin I$ 
from $A^{n+1}-I$\ (resp. $B^{n+1}-I$)\ is a maximal principal submatrix whose entries are all positive.
\end{definition}

\begin{definition}
A $z$-matrix $A$\ (resp. $w$-matrix $B$)\ is said to be \textit{fully connected} if $\{1,2,\ldots,n\}$ is a 
connected component of $A$\ (resp. $B$).
\end{definition}

Our matrix rules are necessary conditions for possible $z$-matrices and $w$-matrices.
They can be classified according to types of diagrams involved. 
If a rule concerns either the $z$-diagram or the $w$-diagram, but not both, we say it is a {\em monocolored matrix rule}.  
If the criterion concerns the $zw$-diagram, we call it  a  {\em bicolored matrix rule}. 
It is worth noting that monocolored matrix rules are not merely matrix formulations of one color rules,
neither are bicolored matrix rules simply corollaries of two color rules. 

In below, we will classify matrix rules according basic features of matrices, and the number of vertices involved. 
Table~\ref{table:ClassifyMxRules} shows our classification and their correspondence with coloring rules. 
One of them does not come from any coloring rule. 
Many of them require estimates and propositions in the previous section but we omit these dependence in the table.

{\small
\begin{tabular} {p{22mm}| p{56mm} | p{25.5mm} |p{12mm}|p{12mm} } 
\multirow{2}{*}{\bf Classification} 
                             & \multirow{2}{*}{\bf Matrix Rule} &  \multirow{2}{*}{\bf Coloring Rule} &{\bf  Mono-colored} & {\bf Bi-colored} \\
\hline \hline 
\multirow{2}{*}{Column Sum}
                              & Rule of Column Sums & 1a & $\quad \surd$ & \\  \cline{2-5}
                              & Rule of Two Column Sums & 2a & & $\quad \surd$ \\ 
\hline
\multirow{6}{*}{Trace}
                              & First Rule of Trace & 1d & $\quad \surd$ & \\  \cline{2-5}
                              & Second Rule of Trace      &  1e &  & $\quad \surd$ \\  \cline{2-5}
                              & First Rule of Trace-2 Matrices &  1a,1d,2c,2e & $\quad \surd$ & \\   \cline{2-5}
                              & Second Rule of Trace-2 Matrices & \scriptsize{1a,1d,1e,2a,2c,2e,cm} & &  $\quad \surd$ \\   \cline{2-5}
                              & Rule of Trace-3 Matrices &  1c,~2c,~2e & $\quad \surd$& \\     \cline{2-5}
                              & Rule of Trace-0 Principal Minors &  (none) & $\quad \surd$& \\ 
\hline                                                         
\multirow{1}{*}{Edge}
                             & \multirow{1}{*}{Rule of Circling}    & cm &  &  $\quad \surd$  \\  
\hline              
\multirow{4}{*}{Triangle}
                              & First Rule of Triangles &  2c & $\quad \surd$ & \\  \cline{2-5}  
                              & Second Rule of Triangles      &  2c &   $\quad \surd$  & \\  \cline{2-5}
                              & Third Rule of Triangles &  2e &  &$\quad \surd$  \\ \cline{2-5}
                              & Fourth Rule of Triangles &  2e &  &$\quad \surd$  \\
\hline
\multirow{2}{*}{Quadrilateral}
                              & First Rule of Quadrilaterals &  2f & $\quad \surd$ & \\  \cline{2-5}
                              & Second Rule of Quadrilaterals     &  2g & & $\quad \surd$  \\  
\hline
Pentagon              & Rule of Pentagons &  2d &  & $\quad \surd$  \\        
\hline
\multirow{2}{*}{Connectedness}
                             & Rule of Fully Connected Companions &  1e,2d &  & $\quad \surd$  \\  \cline{2-5}
                              & Rule of Connected Components     &  1c &  & $\quad \surd$   \\  
\end{tabular}     \smallskip
\captionof{table}{Classifications of matrix rules. Among coloring rules, ``cm'' is abbreviation for circling method.} \label{table:ClassifyMxRules}
}

Names of our matrix rules are mostly intuitive. 
For example, rules about trace are restrictions on traces of $z$-, $w$-, or $zw$-matrices, or restrictions to cases with a fixed trace; 
rules about triangles are restrictions on three vertices in the diagram; 
rules about quadrilaterals are restrictions on four vertices, and so on. 
An exception is the rule about a single edge, we name it the Rule of Circling because it is an immediate consequence of 
the Circling Method.

\smallskip
\noindent{\bf Assumptions \& Notations:} 
In what follows we assume symmetric $n\times n$ binary matrices $A=(a_{ij})$ and $B=(b_{ij})$
are respectively the $z$-matrix and $w$-matrix associated to a singular sequence of
central configurations. The associated $zw$-matrix is $(A|B)$. Let $A+B=C=(c_{ij})$.

\begin{theorem}  
\label{thm:CS}   Let $E=A$ or $B$.
\begin{enumerate}
\item[(a)]  {\bf (Rule of Column Sums)}  
No column sum of $E$ is equal to 1, and column sums of $E$ cannot be all equal to  0. 

\item[(b)]  {\bf (Rule of Two Column Sums)} 
If $c_{ij}=2$ for some $1\leq i< j\leq n$, then
$$
\left( \sum_{k\neq i,j} (a_{ki}+a_{kj})\right)\left(\sum_{k\neq i,j} (b_{ki}+b_{kj})\right) \neq 0. 
$$
\end{enumerate}
\end{theorem}

\begin{proof}
Part (a) follows easily from Rule~1a. 

Note that $c_{ij}=2$ means that there is an $zw$-edge between  $i$ and $j$. By Rule~2a, 
there are another $z$-stroke emanating from this $zw$-edge. 
Therefore $\sum_{k\neq i,j}a_{ki}>0$ or $\sum_{k\neq i,j}a_{kj}>0$, it follows that $\sum_{k\neq i,j}(a_{ki}+a_{kj})> 0$. 
Similarly, $\sum_{k\neq i,j}(b_{ki}+b_{kj})> 0$. This proves (b).  
\end{proof}

\begin{theorem}  
{\bf (Rule of Circling)} For any $1\leq i<j\leq n$,  we have 
$$
   a_{ij}(b_{ii}+b_{jj}), \; b_{ij}(a_{ii}+a_{jj}) \neq 1. 
$$ 
\end{theorem}

\begin{proof}
It follows trivially from the Circling Method. 
\end{proof}

\begin{theorem}     
\label{thm:TR}  
\begin{enumerate}
\item[(a)]  {\bf (First Rule of Trace)}  
The trace $\tr(A)$ of $A$ cannot be 1. 

\item[(b)]  {\bf (Second Rule of Trace)} 
If $c_{ij}\neq 1$ for every $1\leq i<j\leq n$, then  $\tr(C)=0$.  

\item[(c)]  {\bf (First Rule of Trace-2 Matrices)} If $\tr(A)=2$ and $a_{ii}=a_{jj}=1$ for some $i\neq j$, then $a_{ij}=1$
and $a_{ik}=a_{jk}$ for any $k\ne i,j$. 

\item[(d)]  {\bf (Second Rule of Trace-2 Matrices)} 
Assume $\tr(A)= 2$ and $a_{ii} = a_{jj} =1$ for some $i\ne j$, $a_{kl}=0$ for any $k\ne l$, $\{k,l\}\ne \{i,j\}$. 
Then $b_{ik}=b_{jk}=0$ for any $k$. If $m_i^3 + m_j^3 \ne 0$, then $\tr(B)\ne n-3$. 

\item[(e)] {\bf (Rule of Trace-3 Matrices)} Assume $\tr(A)=3$ and $a_{ii}=a_{jj}=a_{kk}=1$ for distinct $i, j, k$.  
If $l\not\in\{i,j,k\}$ and $a_{ij} = a_{ik} = a_{il} =1$, then $a_{jl}+a_{kl}\ne 0$. 
\end{enumerate}
These rules hold if $A$ and $B$ were switched. 
\end{theorem}

\begin{proof}
Part (a) follows immediately from Rule~1d. 

If $c_{ij}\neq 1$ for any $1\leq i<j\leq n$, then all edges are $zw$-edges, none of which is maximal, according to Estimate~1. 
By Rule~1e, there is no circle, and so $c_{kk}=0$ for any $k$. This proves (b). 

Assume hypothesis of (c). Then there are $z$-circles around $i$ and $j$,  and no other $z$-circle. 
Suppose $a_{ij}\neq 1$, then there were no $z$-stroke between $i$ and $j$. 
By Rule~1a, there would be a $z$-strokes joining $i$ and $k$ for some $k\ne j$, and another $z$-stroke joining $j$ and $l$ for some $l\neq i$. 
These two $z$-strokes must be maximal since $k$ and $l$ were not $z$-circled.  
This implies $k\neq l$, for otherwise it would contradict  the Corollary of Rule~2c (if there is no edge between $i$ 
and $j$) or Rule~2e (if there is a $w$-edge between $i$ and $j$). 

Observe that the $z$-stroke joining $i$ and $k$ must be the only $z$-stroke emanating from $i$. 
To see this, suppose there were another $z$-stroke, say from $i$ to $p\; (\ne k, j)$, this $z$-stroke must be maximal, 
then by Corollary of Rule~2c and Rule~2e, $z_{pk}$ must be also maximal, which is impossible since 
$z_p, \, z_k \prec \epsilon^{-2}$  (i.e. neither $p$ nor $k$ were $z$-circled).   

Among $z$-separations, we have $z_{ip}\approx\epsilon^{-2}$ for all $p\neq i,j$. 
By Rule~1d, $i$ and $j$ belong to the same isolated component in the $z$-diagram, and $z_{ij}\approx\epsilon^{-2}$. 
Among $w$-separations, by Estimate~1, $w_{ik}\approx\epsilon^2$ and $w_{ip}\succ \epsilon^2$ for all $p\neq k$. 
According to Proposition~\ref{prop:AK1}, we obtain $w_i\preceq w_{ik}\approx \epsilon^2$. 
Similarly, $w_j\preceq w_{jl}\approx \epsilon^2$, and it follows that $w_{ij}\preceq \epsilon^2$, 
which contradicts  Estimate~1 and our assumption that there is no $z$-stroke between $i$ and $j$. 
This proves the first part of (c); i.e. $a_{ij}=1$. 

Suppose $a_{ik}=1$ for some $k\ne i,j$.  To complete the proof for (c), all we need to show is $a_{jk}=1$. 
The  $z$-stroke between $i$ and $k$ is maximal since $k$ is not $z$-circled. 
By Rule~1d, the $z$-stroke between $i$ and $j$ is also maximal. 
Then it follows from the Corollary of Rule~2c that $i$, $j$, $k$ form a triangle in the $zw$-diagram. 
By Rule~2e, the edge between $j$ and $k$ is of the same type as the other two edges, so $a_{jk}=1$.  This finishes the proof for (c).

Next, we assume the hypothesis of  (d). 
By (c)  we have  $a_{ij}=1$, so the $z$-diagram has exactly one $z$-stroke, and it is between $i$ and $j$.
Without loss of generality, we may assume $i=1$, $j=2$. 

By Rule~2a and the uniqueness of $z$-stroke, we have $b_{12}=0$. 
By the Circling Method, there is no $w$-stroke from vertices $1$, $2$ to other vertices, so $b_{1k}=b_{2k} =0$ for $k\ge 3$. 
Then by Rule~1a, there is no $w$-circle around $1$, $2$; i.e. $b_{11}=b_{22}=0$, and thus  $b_{1k}=b_{2k} =0$ for any $k$. 
Equivalently, it means that the two $z$-circles together with the $z$-stroke between them is 
an isolated component in the $zw$-diagram. By Rule~1e, the unique $z$-stroke is maximal. 

Assume $\tr(B)=n-3$; i.e. there are $n-3$ $w$-circles, from discussions above  these  
$w$-circles are disjoint from $z$-circles. 
Without loss of generality, assume they are around vertices  $3,\cdots, n-1$. Then for any $l\in \{3,\cdots, n-1\}$, 
$$
    w_{1n} \;=\; w_1-w_n \; \prec \; \epsilon^{-2}  \; \approx \; w_1-w_l \; =\;  w_{1l}
$$ 
since vertices $1$, $n$ are not $w$-circled but $l$ is. 
This contradicts Proposition~\ref{prop:AK4} provided $m_1^3+m_2^3\ne 0$. 
Therefore, we conclude that  $\tr(B)\ne n-3$ under this mass condition. The proves (d).

Finally, assume the hypothesis of (e). 
Then $z_p \prec z_i, z_j, z_k \approx \epsilon^{-2}$ for any $p \not\in\{i,j,k\}$, and by Rule~1c, 
$$
  (m_i+m_j+m_k) z_i - m_j z_{ij} - m_k z_{ik} \;=\;  m_i z_i + m_j z_j + m_k z_k \;\prec\; \epsilon^{-2}. 
$$
Therefore, at least one of the $z$-strokes $z_{ij}$, $z_{ik}$ is maximal.
Without loss of generality, assume $z_{ij}$ is maximal.   
The $z$-stroke $z_{il}$ is also maximal since $i$ is $z$-circled but $l$ is not. 
By Rule~2c, vertices $\{i,j,l\}$ form a triangle in the $zw$-diagram, 
and by Rule~2e the three edges of the triangle are of the same type. 
Then $a_{jl}+a_{kl} \ge a_{jl} =1$, thereby completing the proof of (e). 
\end{proof}

\begin{theorem}   
\label{thm:triangle} 
Consider $n\ge 3$ and given $1\leq i<j<k\leq n$. 
\begin{enumerate}
\item[(a)]  {\bf (First Rule of Triangles)}   
\beq
&&  \left(a_{ii}+a_{jj}+a_{kk}\right)  \left(a_{ii}+a_{ij}+a_{ik}\right)   \left(a_{ik}+a_{jk}+a_{kk}\right)  \left(a_{ij}+a_{jj}+a_{jk}\right) 
  \neq 16. 
\eeq

\item[(b)]  {\bf (Second Rule of Triangles)}  
\beq
&&  (a_{ii}+a_{ij}+a_{ik}+a_{jj}+a_{jk}+a_{kk})\, m_{i,j,k} \neq 9.
\eeq
where $m_{i,j,k}=\max\{a_{ii}+a_{ij}+a_{ik},\ a_{ji}+a_{jj}+a_{jk},\ a_{ki}+a_{kj}+a_{kk}\}$.

\item[(c)]  {\bf (Third Rule of Triangles)}  
\beq
 c_{ij} + c_{jk} + c_{ik}  \;\not\in\;  \{4, 5\}. 
\eeq

\item[(d)]  {\bf (Fourth Rule of Triangles)}  
\beq
  c_{ij} c_{jk} c_{ik} \max\left\{a_{ij}+a_{jk}+a_{ik}, b_{ij}+b_{jk}+b_{ik}\right\} \;\in\; \{0, 3, 24\}. 
\eeq
\end{enumerate}
Rules (a), (b) hold if $A$ is replaced by $B$. 
\end{theorem}

\begin{proof}
Suppose the product in (a) equals 16. The only possibility is that each of the four factors equals 2. 
One can easily enumerate all possibilities:
\begin{itemize}
\item[(i)] $a_{ii}=0$ implies $a_{ij}=a_{ik}=a_{jj}=a_{kk}=1$ and $a_{jk}=0$. 
\item[(ii)] $a_{jj}=0$ implies $a_{ii}=a_{ij}=a_{jk}=a_{kk}=1$ and $a_{ik}=0$. 
\item[(iii)] $a_{kk}=0$ implies $a_{ii}=a_{ik}=a_{jj}=a_{jk}=1$ and $a_{ij}=0$. 
\end{itemize}
In case (i), the $z$-diagram contains a $z$-stroke from $j$ to $i$, a $z$-stroke from $i$ to $k$, no $z$-stroke from $j$ to $k$.
Both $j$ and $k$ are $z$-circled, and $i$ is not $z$-circled. This implies $z$-strokes from $j$ to $i$ and from $i$ to $k$ 
are both maximal. This contradicts Rule~2c. 

Similarly, case (ii) results in a $z$-diagram with two consecutive $z$-strokes, one from $i$ to $j$, the other from $j$ to $k$,
and no $z$-stroke from $i$ to $k$. Moreover, $i$ and $k$ are $z$-circled and $j$ is not. This again contradicts 
Rule~2c. The discussion for case (iii) is identical. This proves (a).

The proof for (b) is similar. Suppose the equation in (b) equals 9. 
One of the three terms in $\max\{...\}$ must be 3, and the first factor $(a_{ii}+a_{ij}+\cdots)$ must be also equal to 3,
so there are precisely three 1's and three 0's.  Without loss of generality, consider the case
$$
  a_{ii}=a_{ij}=a_{ik}=1, \;\; a_{jj}=a_{jk}=a_{kk}=0. 
$$
The $z$-diagram contains a $z$-stroke from $j$ to $i$, a $z$-stroke from $i$ to $k$, no $z$-stroke from $j$ to $k$.
Neither $j$ nor $k$ are $z$-circled, so these two vertices are not maximal. Vertex $i$ is maximal since it is $z$-circled,
so $z$-strokes from $j$ to $i$ and from $i$ to $k$  were both maximal. Again, this contradicts Rule~2c. 

Part (c) follows immediately from the Corollary of Rule~2e. 

To prove (d), we only need to consider $c_{ij}c_{jk}c_{ik}\neq 0$, in which case there is some edge between each pair of vertices $i,j,k$. 
The three edges form a triangle, by Rule~2e they are of the same type. 
If they are $zw$-edges, then the discriminant 
$$c_{ij} c_{jk} c_{ik} \max\left\{a_{ij}+a_{jk}+a_{ik}, b_{ij}+b_{jk}+b_{ik}\right\}$$ 
in (d)  is clearly $24$; if they are $z$-edges or $w$-edges, then the discriminant is clearly  $3$. 
\end{proof}

\begin{theorem}    
\label{thm:quadrilateral} 
Consider $n\ge 4$ and given $1\leq i<j<k\leq \ell \le n$. 
\begin{enumerate}
\item[(a)]  {\bf (First Rule of Quadrilaterals)}   
\beq
&& a_{ij}+a_{ik}+a_{i\ell}+a_{jk}+a_{j\ell}+a_{k\ell} \neq 5.
\eeq

\item[(b)]  {\bf (Second Rule of Quadrilaterals)}  
If $c_{ij} c_{jk} c_{k\ell} c_{\ell i} \neq 0$, then
\beq
  a_{ij}+a_{k\ell}, \;   a_{jk}+a_{\ell i}, \;   b_{ij}+b_{k\ell }, \;   b_{jk}+b_{\ell i} \; \neq \; 1.  
\eeq
\end{enumerate}
Rule (a) holds if $A$ is replaced by $B$. 
\end{theorem}

\begin{proof}
Part (a) follows easily from Rule~2f. 
Part (b) follows easily from Rule~2g. 
\end{proof}

\begin{theorem}  
{\bf (Rule of Pentagons)} 
Consider $n\ge 5$ and distinct $i_1,i_2,i_3,i_4,i_5\in\{1,\ldots,n\}$.  Let $i_0=i_5$.
Suppose  
$$
   \prod_{k=0}^4c_{i_{k}i_{k+1}} \; \ne \; 0. 
$$
Let 
\beq
  N_z &=& \# \{k\; : \; (a_{i_k i_{k+1}}, b_{i_k i_{k+1}}) =(1,0), \, k=0,1,2,3,4\} \\
  N_w &=& \# \{k\; : \; (a_{i_k i_{k+1}}, b_{i_k i_{k+1}}) =(0,1), \, k=0,1,2,3,4\} \\
  N_{zw} &=& \# \{k\; : \; (a_{i_k i_{k+1}}, b_{i_k i_{k+1}}) =(1,1), \, k=0,1,2,3,4\}. 
\eeq
Then $N_z\ne 1$, $N_w\ne 1$, $(N_z,N_w,N_{zw}) \ne (4,0,1)$, $(0,4,1)$. 
\end{theorem}

\begin{proof}
The assumption of the rule says that vertices $i_1,i_2,i_3,i_4,i_5$ form a pentagon in the $zw$-diagram. 
Sides of the pentagon form a cycle of edges between $i_k$ and $i_{k+1}$, $k=0,\cdots,4$.  
Without loss of generality, assume $i_k=k$ for each $k$. 

Note that $N_z$, $N_w$, $N_{zw}$ are respectively the number of $z$-edges, $w$-edges, $zw$-edges 
among side edges of the pentagon. 
Suppose $N_z=1$, we may assume the unique $z$-edge is between vertices 1, 2. Then by Rule~2d, 
$z_{12}$ cannot have the maximal order among $\{z_{12},z_{23}, z_{34}, z_{45}, z_{51}\}$. 
Let $z_{ij}$ be an edge among them with maximal order, then $ \epsilon \prec z_{12} \prec z_{ij}$  by Estimate~2. 
There is a $w$-stroke between $i$ and $j$, by Estimate~2 again we have $\epsilon \preceq w_{ij}$, so 
$$
  \epsilon^4 \; \prec  \;z_{ij}^3 w_{ij}\;  =\;  W_{ij}^{-2}  \;   \approx \; \epsilon^4,  
$$
a contradiction. This shows $N_z\ne 1$. Similarly,  $N_w\ne1$. 

Suppose $(N_z,N_w,N_{zw}) = (4,0,1)$;  i.e. there exists unique $zw$-edge, four $z$-edges, and no $w$-edge in this pentagon. 
By Estimate~2, the order of  the $w$-separation for this $zw$-edge is strictly greater than the order of  
$w$-separations for the four $z$-edges, but that contradicts Rule~2d. 
The proof  for $(N_z,N_w,N_{zw}) \ne (0,4,1)$ is similar. 
\end{proof}

\begin{theorem}   
\label{thm:conn} 
\begin{enumerate}
\item[(a)] {\bf (Rule of Fully Connected Companions)} 
If the $z$-matrix or $w$-matrix is fully connected, then its companion matrix has trace 0. 

\item[(b)]  {\bf (Rule of Connected Components)} 
If $I$ is a connected component of $A$ (resp. $B$), then
\begin{enumerate}
\item[(i)] $\sum_{i\in I}a_{ii}\neq 1$ (resp. $\sum_{i\in I}b_{ii}\neq 1$),
\item[(ii)] If $a_{ii}=1$ for all $i\in I$ (resp. $b_{ii}=1$), then there are distinct $i,j\in I$ such that $c_{ij}\neq 2$.
\end{enumerate}

\end{enumerate}
\end{theorem}

\begin{proof}
To prove (a), it is sufficient to 
assume $A$ is fully connected,  and show that $b_{kk}=0$ for every $k$.
Suppose the rule were false; i.e.  $b_{kk}=1$ for some $k$. 
By Rule~1e, there is a maximal $w$-stroke between some vertices $k_1$ and $k_2$. 
Estimate~2 implies that the edge between $k_1$ and $k_2$ is a $w$-edge. 
Since $A$ is fully connected, there is a path consisting of $z$-strokes from $k_1$ to $k_2$.  
This path together with the $w$-edge between $k_1$ and $k_2$ form a polygon. 
But $w_{k_1k_2}$ is the unique maximal $w$-separation among  sides of this polygon, 
which contradicts Rule~2d. Therefore $b_{kk}=0$ for every $k$.  
This proves (a). 

If $I$ is a connected component of $A$, then $I$ is an isolated component in the $z$-diagram. 
From Rule~1c, the center of mass of bodies in $I$ is $z$-close to the origin; that is, 
$\sum_{i\in I}m_iz_i\prec \epsilon^{-2}$. Therefore there can't be exactly one $z$-circled body in $I$, 
and thus $\sum_{i\in I}a_{ii}\neq 1$. Similarly, If $I$ is a connected component of $B$, 
then $\sum_{i\in I}b_{ii}\neq 1$. If $a_{ii}=1$ for all $i\in I$, then all bodies in $I$ are $z$-circled, 
and hence $z_{ii}\approx \epsilon^{-2}$ for all $i\in I$. Pick an $i\in I$. Since $z_j=z_{ji}+z_{i}$ for 
$j\neq i$ we have
$$\epsilon^{-2}\succ\sum_{j\in I}m_jz_{j}=\sum_{j\in I}m_jz_{i}+\sum_{j\in I,\ j\neq i}z_{ji},$$ 
and it follows that there is some $j\in I$ with $j\neq i$ such that $z_{ji}\approx \epsilon^{-2}$, 
and then there can't be a $zw$-edge between vertices $i$ and $j$. Therefore there exist $i,j\in I$ 
with $j\neq i$ such that $c_{ij}\neq 2$. This completes the  proof for (b). 
\end{proof}

\begin{theorem}    {\bf (Rule of Trace-0 Principal Minors)} 
Let $I\subset\{1,\ldots,n\}$. If $a_{ii}=0$ for all $i\in I$, then $\sum_{i\in I, j\notin I}a_{ij}\neq 1$. This rule holds if $A$ and $B$ were switched. 
\end{theorem}

\begin{proof}
Suppose $\sum_{i\in I, j\notin I}a_{ij}=1$. Then there is a unique $(i_0,j_0)$ with $i_0\in I$ and $j_0\notin I$ such that $a_{i_0j_0}=1$. From system (\ref{eqn:cc4}), we have
$$
\sum_{i\in I}m_i z_i=\sum_{i\in I, j\notin I} m_i m_j Z_{ij}.
$$
Since $a_{ii}=0$ for all $i\in I$, $m_{i_0}m_{j_0}Z_{i_0j_0}$ is the unique term of maximal order in this equation, but this is impossible.
 Therefore $\sum_{i\in I, j\notin I}a_{ij}$ can't be 1.
\end{proof}

We remark that the Rule of Trace-0 Principal Minors can be skipped in the process of eliminating $zw$-matrices when $n\le 5$. 
For $n=6$ this rule becomes irreplaceable.

\section{Order matrices for singular sequences}   \label{sec:OrderingRules}

In \S\ref{subsec:ColoredDiagram}, singular sequences were normalized 
so that $\|\mathcal{Z}\|=\|\mathcal{W}\|=\epsilon^{-2}$, where $\epsilon$ is a sequence of positive numbers converging  to 0.
In this section we introduce order matrices to discuss orders for three sets of variables, also in terms of power functions of $\epsilon$. 
These three sets of variables are 
``{\em positions}''  $z_k$'s, $w_k$'s, ``{\em separations}'' $z_{kl}$'s, $w_{kl}$'s, and ``{\em distances}'' $r_{kl}$'s.

\subsection{Levels of orders and order matrices}

Recall that a sequence $a$ of non-zero complex numbers satisfies $a\approx \epsilon^\lambda$, $\lambda\in \R$, 
if both $a/\epsilon^\lambda$ and $\epsilon^\lambda/a$ are bounded. 
In this case we say $a$ has {\em order} $\epsilon^\lambda$. 
If $a/\epsilon^\lambda\to 0$, denoted $a \prec \epsilon^\lambda$, we say  $a$ has {\em order strictly less than} $\epsilon^\lambda$. 

Let us extend this concept a bit further.  Given a sequence $b$ of complex numbers, possibly with infinitely many zeros, we say 
$b$ has {\em order strictly less than} $\epsilon^\lambda$  if  $b/\epsilon^\lambda\to 0$. 
The same notation $b\prec \epsilon^\lambda$ will be used. 
Similarly, we say $b$ has {\em order strictly greater than} $\epsilon^\lambda$  if  $b/\epsilon^\lambda\to \infty$. 
The same notation $\epsilon^\lambda \prec b$ will be used.

While Estimate~1 provides bounds for orders of separations and distances, 
separations and distances may not have well-defined orders in terms of power functions of $\epsilon$, neither are 
position variables. 
But it makes sense to say a variable's order is {\em strictly between} some $\epsilon^{\lambda_1}$ and $\epsilon^{\lambda_2}$. 
We shall use this broader concept while referring to the {\em order} of a variable.

In order to simplify discussions and clear the way for symbolic computations, we discretize the problem by dividing 
possible orders into a few levels, and then introduce order matrices to record possible levels of these variables.

Based on Estimates~1 and~2, we divide orders of positions and separations into six levels:
\begin{enumerate}
\item[]\hskip-8mm {\sl Level 0:} the order is strictly less than $\epsilon^2$ 
\item[]\hskip-8mm {\sl Level 1:} the order is $\epsilon^2$
\item[]\hskip-8mm {\sl Level 2:} the order is strictly between $\epsilon^2$ and $\epsilon$ 
\item[]\hskip-8mm {\sl Level 3:} the order is $\epsilon$
\item[]\hskip-8mm {\sl Level 4:} the order is strictly between $\epsilon$ and $\epsilon^{-2}$ 
\item[]\hskip-8mm {\sl Level 5:} the order is $\epsilon^{-2}$
\end{enumerate}
Since  separations satisfy  $\epsilon^2\preceq z_{kl}, \, w_{kl}\preceq \epsilon^{-2}$, level 0 is 
only possible  for positions. 

Distances satisfy $\epsilon\preceq r_{kl}\preceq \epsilon^{-2}$, so we divide their orders differently:  
\begin{enumerate}
\item[]\hskip-8mm {\sl Level 1:}  the order is $\epsilon$
\item[]\hskip-8mm {\sl Level 2:}  the order is strictly between $\epsilon$ and $1$ 
\item[]\hskip-8mm {\sl Level 3:}  the order is $1$
\item[]\hskip-8mm {\sl Level 4:}  the order is strictly between $1$ and $\epsilon^{-2}$ 
\item[]\hskip-8mm {\sl Level 5:} the order is $\epsilon^{-2}$
\end{enumerate}

\begin{notation}
Let $a$ be either a position, separation, or distance variable. 
$$\lev(a)=k  \;\; \Longleftrightarrow \;\;  \text{$a$ has order in level $k$}.  $$ 
If $\lev(a) \in s$, we say the set $s$ {\em covers the order of $a$}. 
\end{notation}

A simple observation will be useful later: 

\begin{lemma} \label{lem:order6}
If there is an edge between bodies $i$ and $j$, then  $\lev(z_{ij})+\lev(w_{ij}) = 6$. 
If there is no edge between bodies $i$ and $j$, then  $\lev(z_{ij})+\lev(w_{ij}) \ge 6$. 
\end{lemma}

\begin{proof}
If there is an edge between $i$ and $j$,  then 
\beq
\text{either}\quad  z_{ij}^{-1/2}w_{ij}^{-3/2}=Z_{ij}\approx \epsilon^{-2} 
\quad\text{or}\quad w_{ij}^{-1/2}z_{ij}^{-3/2}=W_{ij}\approx \epsilon^{-2}. 
\eeq
If there is no edge between $i$ and $j$, then
\beq
  z_{ij}^{-1/2}w_{ij}^{-3/2}=Z_{ij}\prec \epsilon^{-2} 
\quad\text{and}\quad w_{ij}^{-1/2}z_{ij}^{-3/2}=W_{ij}\prec \epsilon^{-2}. 
\eeq
The lemma follows easily from these formula and Estimates~1,~2. 
\end{proof}

\begin{definition}
A  {\em $zw$-order matrix} is a pair $(S,T)$ of symmetric $n\times n$ matrices, $S=(s_{ij})$, $T=(t_{ij})$, such that
\beq
 s_{ij}, \; t_{ij} \; \subset  \; \{1,2,3,4,5\}  & \text{for $i\neq j$},  \\
 s_{ii}, \; t_{ii} \; \subset \; \{0,1,2,3,4,5\} & \text{for any $i$}. 
\eeq
We say $(S,T)$ is {\em nonempty} if every $s_{ij}$ and $t_{ij}$ is nonempty. 
An {\em $r$-order matrix} is a symmetric $n\times n$ matrix $D=(d_{ij})$ such that  
\beq
   d_{ij} \subset \{1,2,3,4,5\}  & \text{for $i\neq j$},  \\
   d_{ii}  = \emptyset & \text{for any $i$}. 
\eeq
We say $D$ is nonempty if every $d_{ij}$ with $i\ne j$ is nonempty. 
\end{definition}

Be cautious that entries of order matrices are sets, not scalars. We shall apply notations and terminologies for sets to order matrices. 
For example, given $zw$-order matrices $(S,T)$, $(S',T')$, $(S'',T'')$,  
\beq
  (S',T')\subset (S,T)  &\Longleftrightarrow&  s'_{ij} \subset s_{ij}, \; t'_{ij} \subset t_{ij} \;\; \forall i,j; \\
  (S,T)\cup (S',T') = (S'',T'') &\Longleftrightarrow&  s''_{ij} = s_{ij}\cup s'_{ij}, \; t''_{ij} =t_{ij}\cup t'_{ij}\;\; \forall i,j; \\  
  (S,T)\cap (S',T') = (S'',T'') &\Longleftrightarrow& s''_{ij} = s_{ij}\cap s'_{ij}, \; t''_{ij} =t_{ij}\cap t'_{ij} \;\; \forall i,j.
\eeq

\subsection{Covering orders by order matrices}  \label{subsec:coverorders}

Given a $zw$-diagram associated to a singular sequence. 
We begin with one $zw$-order matrix which covers all possible levels for positions and separations, 
one $r$-order matrix which covers all possible levels for distances. 
For example, we may begin with largest order matrices; i.e. choose
$$s_{ii}=t_{ii}=\{0,1,2,3,4,5\} \;\;\forall i, \quad s_{ij}=t_{ij}=r_{ij}=\{1,2,3,4,5\} \;\; \forall i\ne j. $$  
By direct estimates and applying tools in section~\ref{sec:ColoringRules}, more and more restrictions on possible levels will be imposed.  
Eventually there are two possible outcomes: 
either we arrive at a contradiction (e.g. some $s_{ij}$ becomes empty), for which case we can exclude the $zw$-diagram outright; 
or we end up with possible ranges of levels for positions, separations, and distances. 
For the later case, we shall either utilize this $zw$-order matrix directly, or 
choose a finer collection of $zw$-order matrices for further investigations. 

\begin{definition}
We say a collection $\coll$ of nonempty $zw$-order matrices {\em covers $zw$-orders} 
of a singular sequence if for any subsequence of the singular sequence such that 
orders of positions and separations are in fixed levels,
there exists  $(S,T)\in \coll$, $S=(s_{ij})$, $T=(t_{ij})$, such that 
 \begin{eqnarray*}
\begin{array}{ll} 
\lev(z_k)\in s_{kk}, \;\; \lev(w_k) \in t_{kk} & \text{for every $k$},   \\
\lev(z_{kl}) \in s_{kl}, \;\; \lev(w_{kl}) \in t_{kl}  & \text{for every $k\ne l$}.  
\end{array} \label{eqn:coverorders}
\end{eqnarray*}

We say a collection $\mathcal{R}$ of nonempty $r$-order matrices {\em covers $r$-orders} 
of a singular sequence if for any subsequence of the singular sequence such that 
orders of distances are in fixed levels, 
there exists  $D\in \mathcal{R}$, $D=(d_{ij})$, such that 
 \begin{eqnarray*}
\begin{array}{ll} 
\lev(r_{kl}) \in d_{kl}   & \text{for every $k\ne l$}.  
\end{array} \label{eqn:coverorders-r}
\end{eqnarray*}
\end{definition}

\begin{notations}
We use the shorthand notation $$\coll  \Supset  (A|B)$$ when 
the collection $\coll$ of nonempty $zw$-order matrices covers $zw$-orders of any singular sequence associated with $zw$-matrix $(A|B)$. 
For brevity, we say $\coll$ {\em covers orders of } $(A|B)$. 
If the collection $\coll$ consists of only one $zw$-order matrix $(S,T)$, 
we shall also use the notation $$(S,T)  \Supset  (A|B).$$ 

Unless specified otherwise, when an upper case letter is used for a matrix, their corresponding lower case letters
indicate entries of the matrix. For example, by writing $(S,T)  \Supset  (A|B)$, we assume 
$A=(a_{ij})$, $B=(b_{ij})$, $S=(s_{ij})$, $T=(t_{ij})$, $1\le i, j\le n$. 
\end{notations}

\begin{remark}
The property ``$\Supset (A|B)$'' is closed under unions and intersections; i.e. 
\beq
  (S_1,T_1),\; (S_2,T_2)  \Supset  (A|B) &\Longrightarrow & (S_1,T_1)\cup (S_2,T_2),\; (S_1,T_1)\cap (S_2,T_2) \Supset (A|B), \\
  \coll_1,\; \coll_2 \Supset  (A|B) &\Longrightarrow &
    \coll_1\cup \coll_2,\; \coll_1\cap \coll_2\Supset (A|B). 
\eeq
\end{remark}

As soon as ranges of levels for positions and separations were determined, 
our next goal is to find a  ``suitable'' collection of $zw$-order matrices which covers $zw$-orders of the singular sequence.
They will automatically generate a collection of $r$-order matrices which covers $r$-orders of the singular sequence. 
Loosely speaking, a ``suitable'' collection of $zw$-order matrices are expected to meet two criteria:
\begin{enumerate}
\item[(i)] 
The collection is not too big for the purpose of symbolic computations.  
\item[(ii)] Members of the collection
yield polynomial equations which are useful in generating mass relations for the singular sequence. 
\end{enumerate}

These two criteria compete with each other. 
If orders of positions and separations are more precise, we are more likely to obtain useful polynomial equations from them, 
but this comes at the price of having larger collection of $zw$-order matrices. 
We will illustrate this point by showing a simple example. 
This example also provides some incentives for our methods of choosing the collection $\coll$. 

\begin{example}  \label{exam:chooseom}
Suppose we end up with 
$$\lev(z_{12}), \, \lev(z_{13}), \,  \lev(w_{12}), \, \lev(w_{13}) \in \{2,3,4\}, \; \lev(z_{23}), \, \lev(w_{23})  \in \{3\}. $$ 
We may use one $zw$-order matrix with 
$$s_{12}=s_{13}=t_{12}=t_{13}=\{2,3,4\},\;\; s_{23}=t_{23}=\{3\}$$ 
to cover orders of $z_{12}, \, z_{13},\, z_{23}, \, w_{12}, \, w_{13},\, w_{23}$. 
For a finer characterization of their orders, from the relations $z_{12}+z_{23}=z_{13}$,  $w_{12}+w_{23}=w_{13}$
we know it is impossible to have one $z$- (or $w$-)separation with strictly greater order than the other two, so  
we may replace the $zw$-order matrix by a collection of $zw$-order matrices, 
with  $(s_{12}, s_{13}, s_{23})$ and  $(t_{12}, t_{13}, t_{23})$
equal to one of the followings, 
$$
 (\{2\}, \{3\}, \{3\}),\; (\{3\}, \{2\}, \{3\}),\; (\{3\}, \{3\}, \{3\}), \; (\{4\}, \{4\}, \{3\}), 
$$
and leaving other $s_{ij}$'s, $t_{ij}$'s unchanged. 

There are 16 ways to match $(s_{12}, s_{13}, s_{23})$ and $(t_{12}, t_{13}, t_{23})$. 
Some of them can be excluded and the collection still covers $zw$-orders.  
For example, from $z_{12}^{-1} w_{12}^{-3}=Z_{12}^2 \preceq \epsilon^{-4}$ we know that
if one of $s_{12}$, $t_{12}$ is $\{2\}$, the other must be $\{4\}$. Same for $s_{13}$, $t_{13}$. 
Therefore, we may replace the original $zw$-order matrix by 8 $zw$-order matrices, with 
\beq
& \text{(i)}  &(s_{12}, s_{13}, s_{23}, t_{12}, t_{13}, t_{23}) = (\{2\}, \{3\}, \{3\}, \{4\}, \{4\}, \{3\}),  \\
& \text{(ii)} &(s_{12}, s_{13}, s_{23}, t_{12}, t_{13}, t_{23}) = (\{3\}, \{2\}, \{3\}, \{4\}, \{4\}, \{3\}),  \\
& \text{(iii)} &(s_{12}, s_{13}, s_{23}, t_{12}, t_{13}, t_{23}) = (\{3\}, \{3\}, \{3\}, \{3\}, \{3\}, \{3\}),  \\
& \text{(iv)} &(s_{12}, s_{13}, s_{23}, t_{12}, t_{13}, t_{23}) = (\{3\}, \{3\}, \{3\}, \{4\}, \{4\}, \{3\}),  \\
& \text{(v)}  &(s_{12}, s_{13}, s_{23}, t_{12}, t_{13}, t_{23}) =  (\{4\}, \{4\}, \{3\}, \{2\}, \{3\}, \{3\}), \\
& \text{(vi)}  &(s_{12}, s_{13}, s_{23}, t_{12}, t_{13}, t_{23}) = (\{4\}, \{4\}, \{3\}, \{3\}, \{2\}, \{3\}),   \\
& \text{(vii)} &(s_{12}, s_{13}, s_{23}, t_{12}, t_{13}, t_{23}) = (\{4\}, \{4\}, \{3\}, \{3\}, \{3\}, \{3\}),  \\
& \text{(viii)} &(s_{12}, s_{13}, s_{23}, t_{12}, t_{13}, t_{23}) = (\{4\}, \{4\}, \{3\}, \{4\}, \{4\}, \{3\}),  
\eeq
and leaving other $s_{ij}$'s, $t_{ij}$'s unchanged. 

By switching vertices 2 and 3, cases (i) and (ii) can be considered the same, cases (v) and (vi) can be also considered the same. 
If we know additionally that there is no edge between vertices 1 and 2 (or no edge between vertices 1 and 3), 
then from $z_{12}^{-1} w_{12}^{-3}=Z_{12}^2 \prec \epsilon^{-4}$ the case (iii) can be also removed. 

If we perform similar kind of splittings for orders of positions and other separations, and choose $S$ and $T$ so that each entry is
a singleton, then the number of $zw$-order matrices grows factorially. 
In practice, there are cases such that the huge collection goes beyond capacity of typical mathematical softwares when $n\ge 5$.   
Fortunately, for many cases with $n=6$ we are able to conclude finiteness with much smaller collections of $zw$-order matrices.  

This example appears in the analysis of several $zw$-diagrams which contain an isolated triangle formed by $zw$-edges. 
\end{example}

\subsection{Generate $r$-order matrices from $zw$-order matrices}  \label{subsec:generateom}
From the relation $r_{ij}^2 = z_{ij} w_{ij}$, any $zw$-order matrix generates an $r$-order matrix as follows. 

\begin{lemma}  \label{lem:genD}
Let $\phi=\{\}$ and  
$$
D_1=\begin{pmatrix}
\{\} & \{\} & \{\} & \{\} & \{3\}\\
\{\} & \{\} & \{\} & \{2\} & \{\}\\
\{\} & \{\} & \{1\} & \{\} & \{\}\\
\{\} & \{2\} & \{\} & \{\} & \{\}\\
\{3\} & \{\} & \{\} & \{\} & \{\}
\end{pmatrix},
$$
$$
D_2=\begin{pmatrix}
\{\} & \{\} & \{\} & \{\} & \{\}\\
\{\} & \{\} & \{\} & \{2,3,4\} & \{4\}\\
\{\} & \{\} & \{\} & \{2,3,4\} & \{4\}\\
\{\} & \{2,3,4\} & \{2,3,4\} & \{2,3,4\} & \{4\}\\
\{\} & \{4\} & \{4\} & \{4\} & \{5\}
\end{pmatrix}.
$$
Given $i\neq j$. If $\lev(z_{ij})=k$, $\lev(w_{ij})=l$, then $\lev(r_{ij})$ is in the $(k,l)$-entry of $D_1$ if the $zw$-matrix $(A|B)$ 
satisfies $a_{ij}+b_{ij}\geq 1$, otherwise it is in the  $(k,l)$-entry of $D_2$.
\end{lemma}

\begin{proof}
When $a_{ij}+b_{ij}\geq 1$, there is an edge between $i$ and $j$, from Lemma~\ref{lem:order6} we have 
$k+l = \lev(z_{ij}) + \lev(w_{ij}) = 6$. If $(k,l)=(1,5)$ then it is a $w$-edge, 
by Estimates~1,~2 it is actually a maximal $w$-edge and $\lev(r_{ij}) =3$.

Assume $(k,l)=(2,4)$. If $a_{ij}+b_{ij}\geq 1$, then again we have a $w$-edge between $i$ and $j$.   
By Estimate~2 it is neither a maximal $w$-edge nor a $zw$-edge, so $\lev(r_{ij}) =2$.
If $a_{ij}+b_{ij} = 0$, there is no edge between $i$ and $j$, from Estimate~1 we know $\epsilon \prec r_{ij}$. Then from 
$ r_{ij}^2 = z_{ij} w_{ij} \prec \epsilon^{-1}$ we conclude  $\lev(r_{ij})\in \{2,3,4\}$. 

Discussions for other cases are similar. 
\end{proof}

\begin{corollary} \label{cor:produceD}
Suppose $(S,T)  \Supset  (A|B)$. 
Then the matrix $D=(d_{ij})$ defined below is an $r$-order matrix which covers $r$-orders of 
the singular sequence: 
\begin{align*}
 d_{ij}=\bigcup_{k\in s_{ij}, l\in t_{ij}} d^{(1)}_{kl} &\quad  \textrm{if}\ i\neq j\ \textrm{and}\ a_{ij}+b_{ij}\geq 1,&\\
 d_{ij}=\bigcup_{k\in s_{ij}, l\in t_{ij}} d^{(2)}_{kl} &\quad  \textrm{if}\ i\neq j\ \textrm{and}\ a_{ij}+b_{ij}= 0,&\\
 d_{ij}=\{\} &\quad  \textrm{if}\ i=j.
\end{align*} 
Here $d^{(1)}_{kl}$ and $d^{(2)}_{kl}$ are the $(k,l)$-entry of $D_1$ and $D_2$, respectively.
\end{corollary}

\begin{definition}  \label{lem:ordersum6}
The $r$-order matrix $D$ in Corollary~\ref{cor:produceD} is called the {\em $r$-order matrix produced by the 
$zw$-order matrix $(S,T)$}. 
\end{definition}

\subsection{Refinements for $zw$-order matrices}
Knowing how $zw$-order matrices naturally generate $r$-order matrices, we shall focus on 
refining a $zw$-order matrix $(S,T)$, 
assuming $(S,T)  \Supset  (A|B)$, to a smaller $zw$-order matrix $(S',T')$. 
This subsection contains seven principles that are to be used in our algorithm.

\begin{proposition}   \label{prop:rom1}
Suppose $(S,T)  \Supset  (A|B)$. Then the subset  $(S',T')$ of  $(S,T)$ given below also satisfies $(S',T')  \Supset  (A|B)$: 
\begin{enumerate}
\item[(a)] For any $i\in\{1,\cdots,n\}$, 
\beq
a_{ii}=1 &\Longrightarrow & s'_{ii}=s_{ii}\cap \{5\}, \\
b_{ii}=1 &\Longrightarrow & t'_{ii}=t_{ii}\cap \{5\}, \\
a_{ii}=0 &\Longrightarrow & s'_{ii}=s_{ii}\cap \{0,1,2,3,4\}, \\
b_{ii}=0 &\Longrightarrow & t'_{ii}=t_{ii}\cap \{0,1,2,3,4\}. 
\eeq
\item[(b)] For any $i, j\in\{1,\cdots,n\}$, $i\ne j$, 
\beq
 (a_{ij}, b_{ij})=(1,0) &\Longrightarrow & s'_{ij}=s_{ij}\cap \{4,5\}, t'_{ij}=t_{ij}\cap \{1,2\}, \\
 (a_{ij}, b_{ij})=(0,1) &\Longrightarrow & s'_{ij}=s_{ij}\cap \{1,2\}, t'_{ij}=t_{ij}\cap \{4,5\}, \\
 (a_{ij}, b_{ij})=(1,1) &\Longrightarrow & s'_{ij}=s_{ij}\cap \{3\}, t'_{ij}=t_{ij}\cap \{3\}, \\
 (a_{ij}, b_{ij})=(0,0) &\Longrightarrow & s'_{ij}=s_{ij}\cap \{2,3,4,5\}, t'_{ij}=t_{ij}\cap \{2,3,4,5\}.
 \eeq
\end{enumerate}
\end{proposition} 

\begin{proof}
It follows immediately from the definition of $z$-, $w$-circles  that $\lev(z_i)\in s'_{ii}$,  $\lev(w_i)\in t'_{ii}$. 
For $i\ne j$, it follows from Estimates~1,~2 that $\lev(z_{ij}) \in s'_{ij}$, $\lev(w_{ij}) \in t'_{ij}$. 
\end{proof}

\begin{remark} 
For some special cases of (b) we may choose smaller $(S',T')$ to cover $zw$-orders. 
If $(S,T)  \Supset  (A|B)$, 
then the $(S',T')$  given below is a subset of $(S,T)$ which also satisfies $(S',T')  \Supset  (A|B)$: 
for any $i, j\in\{1,\cdots,n\}$, $i\ne j$, 
 \beq
 a_{ii} + a_{jj} = 1 &\Longrightarrow & s'_{ij}= \{5\}, \\
 b_{ii} + b_{jj} = 1 &\Longrightarrow & t'_{ij}= \{5\}, 
\eeq
and let $(S', T')$ be the same as $(S,T)$ for other entries of $S'$, $T'$. 

To see this, note that the condition $a_{ii}+a_{jj}=1$ means one of $z_i$, $z_j$ has order $\approx \epsilon^{-2}$ and 
the other $\prec \epsilon^{-2}$, so $z_{ij} \approx \epsilon^{-2}$; i.e. $5=\lev(z_{ij}) \in s_{ij}$, so $s'_{ij}\subset s_{ij}$ 
and $s'_{ij}$ covers the order of $z_{ij}$. Similar for the case $b_{ii}+b_{jj}=1$. 

We do not apply this simple principle directly in our algorithm as it is covered by other principles. 
\end{remark}

\begin{proposition}   \label{prop:rom2} 
Suppose $(S,T)  \Supset  (A|B)$. Given $1\le i, j\le n$, $i\ne j$. Choose an arbitrary $k\in \{1,\cdots,n\}\setminus\{i,j\}$. Let 
\beq
  s'_{ij}\!\!&=&\!\! \left\{\sigma\in s_{ij}| \sigma\le \max (s_{ik}\!\cup s_{kj}), 
     \min s_{ik} \le \max(\{\sigma\}\! \cup s_{kj}),  \min s_{kj} \le \max( \{\sigma\}\! \cup s_{ik})\right\}\!, \\
 s'_{ji}\!\!&=&\!\! s'_{ij}, \\ 
  s'_{ii}\!\!&=&\!\! \left\{\sigma\in s_{ii}| \sigma\le \max (s'_{ij}\cup s_{jj}), 
    \min s'_{ij} \le \max(\{\sigma\}\! \cup s_{jj}),  \min s_{jj} \le \max( \{\sigma\}\! \cup s'_{ij}) \right\}\!,  \\ 
  s'_{jj}\!\!&=&\!\! \left\{\sigma\in s_{jj}| \sigma\le \max (s'_{ij}\cup s'_{ii}),\, 
    \min s'_{ij} \le \max(\{\sigma\}\! \cup s'_{ii}), \, \min s'_{ii} \le \max( \{\sigma\}\! \cup s'_{ij}) \right\}\!,  \\ 
  t'_{ij}\!\!&=&\!\! \left\{\tau\in t_{ij}| \tau\le \max (t_{ik}\cup t_{kj}),\; 
     \min t_{ik} \le \max(\{\tau\} \cup t_{kj}), \; \min t_{kj} \le \max( \{\tau\} \cup t_{ik})\right\}\!, \\
  t'_{ji}\!\!&=&\!\!t'_{ij}, \\   
  t'_{ii}\!\!&=&\!\! \left\{\tau\in t_{ii}| \tau\le \max (t'_{ij}\cup t_{jj}),\, 
    \min t'_{ij} \le \max(\{\tau\} \cup t_{jj}), \, \min t_{jj} \le \max( \{\tau\} \cup t'_{ij}) \right\}\!,  \\ 
  t'_{jj}\!\!&=&\!\! \left\{\tau\in t_{jj}| \tau\le \max (t'_{ij}\cup t'_{ii}),\, 
    \min t'_{ij} \le \max(\{\tau\} \cup t'_{ii}), \, \min t'_{ii} \le \max( \{\tau\} \cup t'_{ij}) \right\}\!, 
\eeq
and let $(S',T')$ be the same as $(S,T)$ for other entries of $S'$, $T'$. Then $(S',T')  \Supset  (A|B)$.  
\end{proposition}

\begin{proof}
From $z_{ij}=z_{ik}+z_{kj}$ we have  
\beq
&&  \lev(z_{ij}) \;\le\; \max \{\lev(z_{ik}), \lev(z_{kj})\} \;\in\; s_{ik}\cup s_{kj}, \\ 
&\Longrightarrow& \lev(z_{ij}) \;\in \;  \{\sigma\in s_{ij}| \sigma\leq \max (s_{ik}\cup s_{kj})\}; \\
&&  \min s_{ik} \;\le \; \lev(z_{ik}) \;\le\; \max \{\lev(z_{ij}), \lev(z_{kj})\}, \\ 
&\Longrightarrow& \lev(z_{ij}) \;\in \;  \{\sigma\in s_{ij}| \min s_{ik} \le \max( \{\sigma\} \cup s_{kj}); \\
&&  \min s_{kj} \;\le \; \lev(z_{kj}) \;\le\; \max \{\lev(z_{ij}), \lev(z_{ik})\}, \\ 
&\Longrightarrow& \lev(z_{ij}) \;\in \;  \{\sigma\in s_{ij}| \min s_{kj} \le \max( \{\sigma\} \cup s_{ik}). 
\eeq
Thus $\lev(z_{ij}) \in s'_{ij}$.  Similarly, from  $z_i=z_{ij}+z_j$ we find
\beq
 &&  \lev(z_i) \;\le\; \max \{\lev(z_{ij}), \lev(z_j)\} \;\in\;  s'_{ij}\cup s_{jj} \\
 &\Longrightarrow& \lev(z_i) \;\in \;  \{\sigma\in s_{ii}| \sigma\leq \max (s'_{ij}\cup s_{jj})\}; \\  
 &&  \min s'_{ij} \;\le \; \lev(z_{ij}) \;\le\; \max \{\lev(z_i), \lev(z_j)\}, \\ 
&\Longrightarrow& \lev(z_i) \;\in \;  \{\sigma\in s_{ii}| \min s'_{ij} \le \max( \{\sigma\} \cup s_{jj}); \\ 
 &&  \min s_{jj} \;\le \; \lev(z_j) \;\le\; \max \{\lev(z_i), \lev(z_{ij})\}, \\ 
&\Longrightarrow& \lev(z_i) \;\in \;  \{\sigma\in s_{ii}| \min s_{jj} \le \max( \{\sigma\} \cup s'_{ij}).
\eeq
Thus $\lev(z_i) \in s'_{ii}$. The proof for $\lev(z_j) \in s'_{jj}$ is the same.
Similarly, we also have
$$
  \lev(w_{ij}) \;\in \;  t'_{ij}, \quad \lev(w_i) \;\in \;   t'_{ii}, \quad \lev(w_j) \;\in \;  t'_{jj}.
$$
Therefore $(S',T')  \Supset  (A|B)$.
\end{proof}

\begin{proposition}    \label{prop:rom3} 
Suppose $(S,T)  \Supset  (A|B)$. Assume $1\le i< j < k \le n$ and 
$$(a_{ij}+b_{ij}) (a_{jk}+b_{jk}) (a_{ik}+b_{ik}) \ne 0.$$ Define $(S',T')$ by setting 
\begin{align*}
&&  s'_{ij} = s'_{ji} =& s'_{jk} = s'_{kj} = s'_{ik} = s'_{ki} = s_{ij} \cap s_{jk} \cap s_{ik},  \\
 && t'_{ij}  = t'_{ji} =& t'_{jk} = t'_{kj} = t'_{ik} = t'_{ki} = t_{ij} \cap t_{jk} \cap t_{ik}, 
\end{align*}
and let $(S',T')$ be the same as $(S,T)$ for other entries of $S'$, $T'$. Then $(S',T')  \Supset  (A|B)$.  
\end{proposition}

\begin{proof}
The assumption $(a_{ij}+b_{ij}) (a_{jk}+b_{jk}) (a_{ik}+b_{ik}) \ne 0$ means bodies $i$, $j$, $k$ form a triangle. 
By Rule~2e the three edges are of the same type, and $z_{ij} \approx z_{jk} \approx z_{ik}$.  
Therefore the intersection $s_{ij} \cap s_{jk} \cap s_{ik}$ covers orders of $z_{ij}$, $z_{jk}$, $z_{ik}$. Similar for $T'$. 
\end{proof}

\begin{proposition}  \label{prop:rom4}  
Suppose $(S,T)  \Supset  (A|B)$.  Then the subset $(S',T')$ of $(S,T)$ given below also satisfies $(S',T')  \Supset  (A|B)$:  
\begin{enumerate}
\item[(a)]  For any $i, j\in\{1,\cdots,n\}$, $i\ne j$, 
\beq
 a_{ij}+b_{ij}\geq 1 &\Longrightarrow& s'_{ij}=s_{ij} \cap \{6-\tau| \tau\in t_{ij}\}, \\
&&  t'_{ij}=t_{ij} \cap \{6-\sigma| \sigma\in s'_{ij}\}; \\
 a_{ij}+b_{ij}=0  &\Longrightarrow&  
 s'_{ij}=\left\{ \sigma\in s_{ij}| \sigma \ge 6-2\,\lfloor \tau/2\rfloor \;\; \textrm{for every $\tau\in t_{ij}$} \right\},  \\
&& 
  t'_{ij}=\left\{ \tau\in t_{ij}|      \tau \ge 6-2\,\lfloor \sigma/2\rfloor \;\; \textrm{for every $\sigma\in s'_{ij}$} \right\}.
\eeq
\item[(b)] For any $i\in\{1,\cdots,n\}$, 
\beq
  s'_{ii}\; = \; \left\{\sigma\in s_{ii}| \sigma\leq \max \bigcup_{j \ne i} s'_{ij}\right\},   \quad 
t'_{ii}\;=\; \left\{\sigma\in t_{ii}| \sigma\leq \max \bigcup_{j \ne i} t'_{ij}\right\}.
\eeq
\end{enumerate}
\end{proposition}

\begin{proof}
Consider $i\ne j$ and $a_{ij}+b_{ij}\geq 1$.  By Lemma~\ref{lem:ordersum6}  we have
$$
 s_{ij} \;\ni\; \lev(z_{ij}) \;=\; 6 - \lev(w_{ij}) \;\in\; \{6-\tau| \tau\in t_{ij}\}. 
$$
Therefore the first $s'_{ij}$ defined in (a) covers order of $z_{ij}$. 

When $a_{ij}+b_{ij}=0$ we have 
$$z_{ij}^{-1/2}w_{ij}^{-3/2}=Z_{ij}\prec \epsilon^{-2} \quad \text{and}\quad w_{ij}^{-1/2}z_{ij}^{-3/2}=W_{ij}\prec \epsilon^{-2}.$$
Therefore,  if $\lev(z_{ij})=k$, then $\lev(w_{ij})\le 6-k$. 
Since each of the levels 1,3,5 contains only a precise order, for $k=1,3,5$ we actually have  $\lev(w_{ij}) < 6-k$, 
thus  $\lev(w_{ij}) \le 6-2\,\lfloor k/2\rfloor$ for every $k$.  
Therefore the second $s'_{ij}$ defined in (a) covers order of $z_{ij}$.  Discussions for $t'_{ij}$ are similar. 

Since the center of mass is at the origin, we have $m_1z_1+\cdots +m_nz_n=0$, which is equivalent to
$$
(m_1+m_2+\cdots+m_n)z_i=\sum_{j\neq i}m_jz_{ij}.
$$
Recall that the total mass cannot be zero. Therefore, for any $i$ there exists some $j\neq i$ such that  $z_i\preceq z_{ij}$, where hence
$$
  \lev(z_i) \; \le \; \max_{j\ne i} \lev(z_{ij})  \;\in\;   \bigcup_{j \ne i} s'_{ij}. 
$$ 
This implies $\lev(z_i) \in s'_{ii}$. Similarly, $\lev(w_i) \in t'_{ii}$ for each $i$.  This completes the proof. 
\end{proof}

\begin{proposition}   \label{prop:rom6}  
Suppose $(S,T)  \Supset  (A|B)$. Given a nonempty $I\subset \{1,\cdots,n\}$.  
\begin{enumerate}
\item[(a)] If $I$ is a connected component of $A$,  there exists exactly one pair $i<j$ in $I$ such that $(a_{ij},b_{ij})=(1,0)$, 
and $\sum_{k\in I} a_{kk} \ne 0$, then let  $s'_{ij} = \{5\}$. 
Let $s'_{ij}=s_{ij}$ for other cases and $s'_{kl}=s_{kl}$ for other entries of $S'$. 
\item[(b)]  If $I$ is a connected component of $B$,  there exists exactly one pair $i<j$ in $I$ such that $(a_{ij},b_{ij})=(0,1)$, 
and $\sum_{k\in I} b_{kk} \ne 0$, then let  $t'_{ij} = \{5\}$. 
Let $t'_{ij}=t_{ij}$ for other cases and $t'_{kl}=t_{kl}$ for other entries of $T'$.  
\end{enumerate}
Then $(S',T')  \Supset  (A|B)$. 
\end{proposition}

\begin{proof}
An isolated component in $A$ corresponds an isolated component in the $z$-diagram. 
Assumptions in (a) means the isolated component corresponding to $I$ contains some $z$-circle and precisely one $z$-edge, 
and the edge is between $i$ and $j$. It follows from Rule~1e that $\lev(z_{ij})=5$, so $s'_{ij}$ covers order of $z_{ij}$. 
Discussions for (b) are similar.  
\end{proof}

\begin{proposition}   \label{prop:rom7} 
Suppose $(S,T)  \Supset  (A|B)$. 
Given $1\le i, j \le n$, $i\ne j$. 
The subset $(S',T')$ of $(S,T)$ given below also satisfies $(S',T')  \Supset  (A|B)$:
 \beq
&& \bigcup_{k\neq i}t_{ik}= \{\ell\}, \; \ell\in \{1,3,5\},\; \text{and}\;\; \max s_{ij}< \min \bigcup_{k\neq i,j} s_{ik}  \\
 & \Longrightarrow & 
      s'_{ii} = \{ \sigma\in s_{ii}| \sigma \le \max s_{ij} \},  \quad  s'_{jj} = \{ \sigma\in s_{jj}| \sigma \le \max s_{ij} \},  \\
&& s'_{ij} = \{ \sigma\in s_{ij}| \sigma \ge \min(s'_{ii}\cup s'_{jj}) \};  \\ 
&& \bigcup_{k\neq i}s_{ik}= \{\ell\}, \; \ell\in \{1,3,5\},\; \text{and}\;\; \max t_{ij}< \min \bigcup_{k\neq i,j} t_{ik}  \\
 & \Longrightarrow & 
      t'_{ii} = \{ \tau\in t_{ii}| \tau \le \max t_{ij} \},  \quad  t'_{jj} = \{ \tau\in t_{jj}| \tau \le \max t_{ij} \},  \\
&& t'_{ij} = \{ \tau\in t_{ij}| \tau \ge \min(t'_{ii}\cup t'_{jj}) \}.
\eeq 
Let $(S',T')$ be the same as $(S,T)$ for other entries of $S'$, $T'$. 
\end{proposition}

\begin{proof}
The condition $\bigcup_{k\neq i}t_{ik}= \{\ell\}$, $\ell\in \{1,3,5\}$, implies $w_{ij}\approx w_{ik}$ for any $k\ne i, j$. 
The next condition  $\max s_{ij}< \min \bigcup_{k\neq i,j} s_{ik}$ implies $z_{ij}\prec z_{ik}$ for any $k\ne i,j$. 
It follows from Proposition~\ref{prop:AK1} that $z_i, z_j \preceq z_{ij}$, and so $s'_{ii}$ covers the order of $z_i$, 
$s'_{jj}$ covers the order of $z_j$, and $s'_{ij}$ covers the order of $z_{ij}$. 
Discussions for $t'_{ii}$, $t'_{jj}$, $t'_{ij}$ are similar. 
\end{proof}

\begin{proposition}   \label{prop:rom8} 
Suppose $(S,T)  \Supset  (A|B)$. Let $D$ be the $r$-order matrix produced by $(S,T)$. 
Assume  $\bigcup_{j\neq i}d_{ik}\subset\{4,5\}$. 
Then the subset $(S',T')$ of $(S,T)$ given below also satisfies $(S',T')  \Supset  (A|B)$:
\beq
 \max \bigcup_{j\neq i}s_{jj} \in \{0,2,4\}  &\Rightarrow & 
 s'_{ii} = \left\{ \sigma\in s_{ii}|  \sigma \le   \max \bigcup_{j\neq i}s_{jj}  \right\}, \\
  \max \bigcup_{j\neq i}s_{jj} \in \{1,3,5\} &\Rightarrow & 
 s'_{ii} = \left\{ \sigma\in s_{ii}|  \sigma <   \max \bigcup_{j\neq i}s_{jj}  \right\}, \\
   \max \bigcup_{j\neq i}t_{jj} \in \{0,2,4\}  &\Rightarrow & 
 t'_{ii} = \left\{ \tau\in t_{ii}|  \tau \le   \max \bigcup_{j\neq i}t_{jj}  \right\}, \\
  \max \bigcup_{j\neq i}t_{jj} \in \{1,3,5\}  &\Rightarrow & 
 t'_{ii} = \left\{ \tau\in t_{ii}|  \tau <   \max \bigcup_{j\neq i}t_{jj}  \right\}, 
\eeq
Let $s'_{ij}=s_{ij}$, $t'_{ij}=t_{ij}$ for other entries of $S'$, $T'$. 
\end{proposition}

\begin{proof}
This is an immediate corollary of Proposition~\ref{prop:AK2}.
\end{proof}

\subsection{Criteria for feasible covering order matrices} \label{subsec:deletebyom}
Given a singular sequence associated to one $zw$-matrix $(A|B)$. 
After finding a sufficiently small $zw$-order matrix $(S,T)$ so that $(S,T)\Supset (A|B)$, 
we can divide $(S,T)$ into several nonempty $zw$-order matrices that are properly contained in $(S,T)$, 
such that the collection $\coll$ satisfies $\coll\Supset (A|B)$. 

As shown in Example~\ref{exam:chooseom}, we can eliminate members in $\coll$ 
which cannot possibly cover orders of positions and separations for any subsequence of the singular sequence.
Here we present five criteria for this purpose.  
If applications of these criteria lead to contradiction, then the $zw$-diagram can be eliminated. 

\begin{proposition} \label{prop:romc1}
Suppose $\coll  \Supset  (A|B)$ and $(S,T)\in\coll$. Let $D$ be the $r$-order matrix produced by $(S,T)$. 
If $d_{ij}\ne \{1\}$ for any $1\le i,j\le n$, then the number of pairs $(i,j)$ with $i<j$ such that $d_{ij}=\{2\}$ can't be 1.
\end{proposition}

\begin{proof}
This is part of Rule~2h. 
\end{proof}

In Proposition~\ref{prop:rom4}(b) we made use of the condition on center of mass for the whole system.
Using similar ideas, we obtain another criterion for feasible covering order matrices.

\begin{proposition}  \label{prop:romc4} 
Suppose $(S,T)  \Supset  (A|B)$. Given a nonempty $I\subset \{1,\cdots,n\}$.
\begin{enumerate}
\item[(a)] Suppose $I$ is a connected component of $A$ and $J=\{i\in I| a_{ii}=1\}$. Then
$$
5\in s_{ij}
$$
for some $i,j\in J$, $i\neq j$.
\item[(b)] Suppose $I$ is a connected component of $B$ and $J=\{i\in I| b_{ii}=1\}$.  Then
$$
5\in t_{ij}
$$
for some $i,j\in J$, $i\neq j$.
\end{enumerate}
\end{proposition}

\begin{proof}
Assume $I$ is a connected component of $A$, then $I$ is an isolated component in the $z$-diagram.
By Rule~1c we have $\sum_{j\in I}m_jz_j\prec \epsilon^{-2}$. 
Since $z_{i}\prec \epsilon^{-2}$ for $i\in I\setminus J$, $\sum_{j\in J}m_jz_j\prec \epsilon^{-2}$.  For each $i\in J$, from the identity
$$
\left( \sum_{j\in J} m_j\right)z_i \;=\; \sum_{j\in J \atop j\neq i}m_jz_{ij} + \sum_{j\in J} m_j z_j
$$
and the assumption that the total mass of any subsystem is nonzero, we obtain
$$
\sum_{j\in J\atop j\neq i}m_jz_{ij}\approx \epsilon^{-2}.
$$
Therefore $z_{ij}\approx \epsilon^{-2}$ for some $i,j\in J$, $i\neq j$, and the proof of (a) is completed. The proof of (b) is similar.
\end{proof}

For the next three criteria, we need to determine possible maximal order terms of $Z_{ij}$'s and $W_{ij}$'s. 
To do so, we introduce the following \textit{order-comparing process}, through which 
ineligible candidates will be ruled out. 
This process consists of nine simple or self-evident comparison rules for which we omit proofs.  

\subsubsection*{\bf Order-comparing process} 
\begin{enumerate}
\item[(1)] Comparison rules of $z$-separations:
\begin{enumerate}
\item If $\min(s_{ij})>\max(s_{kl})$, then $z_{ij}\succ z_{kl}$.
\item If $s_{ij}=s_{kl}\subset\{1,3,5\}$ and $|s_{ij}|=1$, then $z_{ij}\approx z_{kl}$.
\item If $\min (s_{ij}\cup s_{ik})\geq \max(s_{jk})$ and $\max(s_{jk})\in\{1,3,5\}$, then $z_{ij}\approx z_{ik}$.
\item If $\min (s_{ij}\cup s_{ik})\geq \max(s_{jk})$ and at least one of $\min(s_{ij})$ and $\min(s_{ik})$ is strictly greater 
than $\max(s_{jk})$, then $z_{ij}\approx z_{ik}$.
\end{enumerate}
\item[(2)] Comparison rules of $w$-separations:
\begin{enumerate}
\item If $\min(t_{ij})>\max(t_{kl})$, then $w_{ij}\succ w_{kl}$.
\item If $t_{ij}=t_{kl}\subset\{1,3,5\}$ and $|t_{ij}|=1$, then $w_{ij}\approx w_{kl}$.
\item If $\min (t_{ij}\cup t_{ik})\geq \max(t_{jk})$ and $\max(t_{jk})\in\{1,3,5\}$, then $w_{ij}\approx w_{ik}$.
\item If $\min (t_{ij}\cup t_{ik})\geq \max(t_{jk})$ and at least one of $\min(t_{ij})$ and $\min(t_{ij})$ is strictly greater 
than $\max(t_{jk})$, then $w_{ij}\approx w_{ik}$.
\end{enumerate}
\item[(3)] Comparison rules of $z_{ij}$ and $z_i$, $i\neq j$
\begin{enumerate}
\item If $\min(s_{ij})>\max(s_{ii})$, then $z_{ij}\succ z_{i}$.
\item If $s_{ij}=s_{ii}\subset\{1,3,5\}$ and $|s_{ij}|=1$, then $z_{ij}\approx z_{i}$.
\item If $\min (s_{ii}\cup s_{ij})\geq \max(s_{jj})$ and $\max(s_{jj})\in\{1,3,5\}$, then $z_{i}\approx z_{ij}$.
\item If $\min (s_{ii}\cup s_{ij})\geq \max(s_{jj})$ and at least one of $\min(s_{ii})$ and $\min(s_{ij})$ is strictly greater 
than $\max(s_{jj})$, then $z_{ii}\approx z_{ij}$.
\end{enumerate}
\item[(4)] Comparison rules of $w_{ij}$ and $w_i$, $i\neq j$
\begin{enumerate}
\item If $\min(t_{ij})>\max(t_{ii})$, then $w_{ij}\succ w_{i}$.
\item If $t_{ij}=t_{ii}\subset\{1,3,5\}$ and $|t_{ij}|=1$, then $w_{ij}\approx w_{i}$.
\item If $\min (t_{ii}\cup t_{ij})\geq \max(t_{jj})$ and $\max(t_{jj})\in\{1,3,5\}$, then $t_{i}\approx t_{ij}$.
\item If $\min (t_{ii}\cup t_{ij})\geq \max(t_{jj})$ and at least one of $\min(t_{ii})$ and $\min(t_{ij})$ is strictly greater 
than $\max(t_{jj})$, then $w_{ii}\approx w_{ij}$.
\end{enumerate}
\item[(5)] Comparison rules of distances:
\begin{enumerate}
\item If $\min(d_{ij})>\max(d_{kl})$, then $r_{ij}\succ r_{kl}$.
\item If $d_{ij}=d_{kl}\subset\{1,3,5\}$ and $|d_{ij}|=1$, then $r_{ij}\approx r_{kl}$.
\end{enumerate}
\end{enumerate}

Next, we compare orders of $Z_{ij}$'s and $W_{ij}$'s according to the relations
\begin{align*}
&& Z_{ij}&=z_{ij}^{-1/2}w_{ij}^{-3/2}=z_{ij}r_{ij}^{-3}=w_{ij}^{-1}r_{ij}^{-1}=W_{ij}^{1/3}w_{ij}^{-4/3} &\\
&& W_{ij}&=w_{ij}^{-1/2}z_{ij}^{-3/2}=w_{ij}r_{ij}^{-3}=z_{ij}^{-1}r_{ij}^{-1}=Z_{ij}^{1/3}z_{ij}^{-4/3}. &
\end{align*}
and results of comparison rules for positions, separations, and distances:
\begin{enumerate}
\item[(6)] Comparison rules of $Z_{ij}$ and $Z_{ik}$, $i,j,k$ are distinct
\begin{enumerate}
\item If $a_{ij}=a_{ik}=1$, then $Z_{ij}\approx Z_{ik}$.
\item If $a_{ij}>a_{ik}$, then $Z_{ij}\succ Z_{ik}$.
\item If $z_{ij}\succeq z_{ik}$, $w_{ij}\succeq w_{ik}$, and at least one of the inequalities is strict, then $Z_{ij}\prec Z_{ik}$.
\item If $z_{ij}\approx z_{ik}$ and $w_{ij}\approx w_{ik}$, then $Z_{ij}\prec Z_{ik}$.
\item If $z_{ij}\preceq z_{ik}$, $r_{ij}\succeq r_{ik}$, and at least one of the inequalities is strict, then $Z_{ij}\prec Z_{ik}$.
\item If $z_{ij}\approx z_{ik}$ and $r_{ij}\approx r_{ik}$, then $Z_{ij}\prec Z_{ik}$.
\item If $w_{ij}\preceq w_{ik}$, $r_{ij}\preceq r_{ik}$, and at least one of the inequalities is strict, then $Z_{ij}\succ Z_{ik}$.
\item If $w_{ij}\approx w_{ik}$ and $r_{ij}\approx r_{ik}$, then $Z_{ij}\prec Z_{ik}$.
\item If $w_{ij}\succ w_{ik}$ and $b_{ik}=1$, then $Z_{ij}\prec Z_{ik}$.
\item If $w_{ij}\approx w_{ik}$ and $b_{ij}=b_{ik}=1$, then $Z_{ij}\approx Z_{ik}$.
\end{enumerate}
\item[(7)] Comparison rules of $W_{ij}$ and $W_{ik}$, $i,j,k$ are distinct
\begin{enumerate}
\item If $b_{ij}=b_{ik}=1$, then $W_{ij}\approx W_{ik}$.
\item If $b_{ij}>b_{ik}$, then $W_{ij}\succ W_{ik}$.
\item If $w_{ij}\succeq w_{ik}$, $z_{ij}\succeq z_{ik}$, and at least one of the inequalities is strict, then $W_{ij}\prec W_{ik}$.
\item If $w_{ij}\approx w_{ik}$ and $z_{ij}\approx z_{ik}$, then $W_{ij}\prec W_{ik}$.
\item If $w_{ij}\preceq w_{ik}$, $r_{ij}\succeq r_{ik}$, and at least one of the inequalities is strict, then $W_{ij}\prec W_{ik}$.
\item If $w_{ij}\approx w_{ik}$ and $r_{ij}\approx r_{ik}$, then $W_{ij}\prec W_{ik}$.
\item If $z_{ij}\preceq z_{ik}$, $r_{ij}\preceq r_{ik}$, and at least one of the inequalities is strict, then $W_{ij}\succ W_{ik}$.
\item If $z_{ij}\approx z_{ik}$ and $r_{ij}\approx r_{ik}$, then $W_{ij}\prec W_{ik}$.
\item If $z_{ij}\succ z_{ik}$ and $a_{ik}=1$, then $W_{ij}\prec W_{ik}$.
\item If $z_{ij}\approx z_{ik}$ and $a_{ij}=a_{ik}=1$, then $W_{ij}\approx W_{ik}$.
\end{enumerate}
\item[(8)] Comparison rules of $Z_{ij}$ and $z_{i}$, $i\neq j$
\begin{enumerate}
\item If $a_{ij}=a_{ii}=1$, then $Z_{ij}\approx z_{i}$.
\item If $a_{ij}>a_{ii}$, then $Z_{ij}\succ z_{i}$.
\item If $z_{ij}\succeq z_{i}$, $r_{ij}\preceq 1$, and at least one of the inequalities is strict, then $Z_{ij}\succ z_{i}$.
\item If $z_{ij}\approx z_{i}$ and $r_{ij}\approx 1$, then $Z_{ij}\approx z_{i}$.
\end{enumerate}
\item[(9)] Comparison rules of $W_{ij}$ and $w_{i}$, $i\neq j$
\begin{enumerate}
\item If $b_{ij}=b_{ii}=1$, then $W_{ij}\approx w_{i}$.
\item If $b_{ij}>b_{ii}$, then $W_{ij}\succ w_{i}$.
\item If $w_{ij}\succeq w_{i}$, $r_{ij}\preceq 1$, and at least one of the inequalities is strict, then $W_{ij}\succ w_{i}$.
\item If $w_{ij}\approx w_{i}$ and $r_{ij}\approx 1$, then $W_{ij}\approx w_{i}$.
\end{enumerate}
\end{enumerate}

\begin{definition} \label{def:maxorderset}
Suppose $(S,T)  \Supset  (A|B)$. Let $D$ be the $r$-order matrix produced by $(S,T)$.
For convenience, let $Z_{ii}=-z_i/m_i$.
Given $i\in \{1,\cdots,n\}$.
A subset $M_{Z,i}$ of $\{1,\ldots, n\}$ is called an \textit{$i$-th maximal $Z$-order set determined by $(A|B)$ and $(S,T)$}  
if for every $j\in M_{Z,i}$ and every $k\in\{1,\ldots, n\}$, the relation $Z_{ij}\prec Z_{ik}$ 
contradicts the order-comparing process. 
We similarly define the \textit{$i$-th maximal $W$-order set determined by $(A|B)$ and $(S,T)$} and denote it $M_{W,i}$.
\end{definition}

It is possible that the $i$-th maximal $Z$-order set $M_{Z,i}$ contains some subindex $j$ such that $Z_{ij}\prec Z_{ik}$ for some $k$. 
The key point is that it can only happen when all rules of order-comparing process have been complied with. 

\begin{proposition} \label{prop:romc2}
Suppose $(S,T)  \Supset  (A|B)$. 
Then for every $i$, $|M_{Z,i}|\geq 2$ and $|M_{W,i}|\geq 2$.
\end{proposition}

\begin{proof}
Since there must be at least one maximal order term in $Z_{i1},\ldots, Z_{in}$, we have $|M_{Z,i}|\geq 1$. 
An equation in (\ref{eqn:cc4}) is
$$
\sum_{j=1}^n m_j Z_{ji}=0,
$$
there are at least two maximal order terms in $Z_{i1},\ldots, Z_{in}$. Therefore $|M_{Z,i}|\geq 2$. 
The proof for $|M_{W,i}|$ is similar.
\end{proof}  

The next proposition follows from arguments in  \cite[\S7.5.2.2]{AK}.

\begin{proposition} \label{prop:romc3}
Suppose $(S,T)  \Supset  (A|B)$. Let $D$ be the $r$-order matrix produced by $(S,T)$. 
If $3\notin d_{ij}$ for some $i\neq j$ and $m_i+m_j\neq 0$, then $M_{Z,i}$ and $M_{Z,j}$ can not be both $\{i,j\}$, 
$M_{W,i}$ and $M_{W,j}$ can not be both $\{i,j\}$.
\end{proposition}

\begin{proof}
Suppose $M_{Z,i}=M_{Z,j}=\{i,j\}$. 
Since the number of maximal order terms in this equation can't be 0 or 1, the maximal order terms in equation
$$
\sum_{k=1}^n m_kZ_{ki}=0
$$
are $m_iZ_{ii}$ and $m_jZ_{ji}$, and the maximal order terms in equation
$$
\sum_{k=1}^n m_kZ_{kj}=0
$$
are $m_jZ_{jj}$ and $m_iZ_{ij}$. Therefore we have
$$
z_i\sim m_jZ_{ji} \quad \text{and} \quad z_j\sim m_iZ_{ij}.
$$
Subtract the second equation from the first, we obtain 
$$
z_{ij}\sim (m_i+m_j)Z_{ji}=-(m_i+m_j)r_{ij}^{-3}z_{ij},
$$
and it follows that $r_{ij}\approx 1$ since $m_i+m_j\neq 0$. This contradicts to $3\notin d_{ij}$.
\end{proof}

\begin{proposition} \label{prop:romc5}
Suppose $(S,T)  \Supset  (A|B)$. 
\begin{enumerate}
\item[(a)] If $a_{ij}a_{jk}>0$ and $a_{ki}=0$ for distinct $i,j,k$, then $M_{W,j}\neq \{i,k\}$ 
and $M_{W,p}\nsubseteq \{i,j,k\}$ for some $p\in\{i,j,k\}$.
\item[(b)] If $b_{ij}b_{jk}>0$ and $b_{ki}=0$ for distinct $i,j,k$, then $M_{Z,j}\neq \{i,k\}$ 
and $M_{W,p}\nsubseteq \{i,j,k\}$ for some $p\in\{i,j,k\}$.
\end{enumerate}
\end{proposition}

\begin{proof}
To prove (a), we suppose $a_{ij}a_{jk}>0$ and $a_{ki}=0$ for distinct $i,j,k$. Then $Z_{ij}\approx Z_{jk}\approx \epsilon^{-2}$. 
Since $a_{ki}=0$, by Rule~2e there is no edge between bodies $i,k$, so the edges between $i,j$ and between $j,k$ are 
consecutive but not part of a triangle. By Rule~2c, either $z_{ij}\prec z_{jk}$ or $z_{ij}\succ z_{jk}$. 
Without loss of generality, we assume $z_{ij}\prec z_{jk}$. Then 
$$
z_{ij}\prec z_{jk}\approx z_{ik}\preceq \epsilon^{-2}.
$$
Since
$$
Z_{ij}\approx Z_{jk}\succ Z_{ik},
$$
we have 
$$
W_{ij}\succ W_{jk}, W_{ik},
$$
and it follows that $k\notin M_{W,i}$ and $k\notin M_{W,j}$.  If $M_{W,i},M_{W,j}\subset \{i,j,k\}$, 
then $M_{W,i}=M_{W,j}=\{i,j\}$. Since $z_{ij}\prec \epsilon^{-2}$, $r_{ij}$ is not of order $1$, 
this contradicts Proposition~\ref{prop:romc3}. Therefore $M_{W,p}\nsubseteq \{i,j,k\}$ for some $p\in\{i,j,k\}$. 
The conclusion $k\notin M_{W,j}$ above implies $M_{W,j}\neq \{i,k\}$. 
The other case, $z_{ij}\succ z_{jk}$, implies $i\notin M_{W,j}$, so again we have $M_{W,j}\neq \{i,k\}$.
The proof for (b) is similar.
\end{proof}

As observed in Example~\ref{exam:chooseom}, the number of possible ways of dividing $(S,T)$ is often enormous,
even more so when $n\ge 6$.  
To be more realistic, in practice we will make divisions for off-diagonal entries of $S$, $T$ first, 
eliminate as many cases as possible at this stage, then take care of their diagonal terms.  
Fortunately, in most cases this procedure yields satisfactory collection $\coll$ of $zw$-order matrices to 
cover orders of positions and separations. 
Detailed procedure will be presented in section~\ref{sec:algorithm II}.

\section{Collect polynomial equations} \label{sec:poly}

Given a singular sequence of normalized central configurations, consider a diagram with $zw$-matrix $(A|B)$. 
After finding a suitable collection $\coll$ of $zw$-order matrices such that $\coll\Supset (A|B)$, 
we can collect some polynomial equations based on (\ref{eqn:cc4}) by picking out leading order terms. 
Comparing with the BKK approach in~\cite{HM06}, this procedure is in essence the same as 
taking one face of the Minkowski sum polytope and considering its corresponding reduced system. 
This section shows some details of how such polynomial equations are to be collected. 

In the process of picking out leading order terms for $Z_{ij}$'s and $W_{ij}$'s, the maximal $Z$- and $W$-order sets
defined in Definition~\ref{def:maxorderset} are not enough. They are candidates of maximal order terms 
obtained through the order-comparing process, 
but in general we do not rule out the possibility that $Z_{ij}$ is not really maximal for some element $j$ in $M_{Z,i}$. 
We need a stronger assumption on these maximal $Z$- and $W$-order sets. 

\begin{definition}
An $i$-th maximal $Z$-order set $M_{Z,i}$ is said to be \textit{certainly maximal} if for every 
$k_1,k_2\in M_{Z,i}$ there is a finite sequence 
$\{j_0,\ldots,j_{N}\}\subset M_{Z,i}$ such that $j_0=k_1$, $j_N=k_2$, and the relation $Z_{ij_{l-1}}\approx Z_{ij_l}$ can be obtained by 
order-comparing process for every $l=1,\ldots, N$.
\end{definition}

\subsection{Five kinds of polynomial equations}

The first kind of equations are obtained from the first two equations in the system (\ref{eqn:cc4}). 
We pick up the maximal order terms of $z_k$'s, $w_k$'s, $Z_{kl}$'s, and $W_{kl}$'s, 
then we have several equations. We may decide the maximal order terms more roughly or more precisely. 
For example, if we decide them according to the $zw$-matrix, then we only collect those terms of order $\epsilon^{-2}$; 
if we decide them according to $zw$-order matrices, then there might be some equations whose maximal order terms 
are of order less than $\epsilon^{-2}$. We call these equations \textit{main equations}, 
and simply call the $i$-th $z$-equation the \textit{$z_i$-equation} and $i$-th $w$-equation the \textit{$w_i$-equation}. 

In the process of picking up maximal order terms, we make use of the maximal $Z$--order sets $M_{Z_i}$ obtained
through the order-comparing process {\em only when} it is certainly maximal. Same for $W_{kl}$'s. 

The second kind of equations are the \textit{conditions of mass center}. 
According to Rule~1c, 
$$\sum_{i=1}^n m_iz_i\prec \epsilon^{-2}, \quad \sum_{i=1}^nm_iw_i\prec \epsilon^{-2}.$$ 
By taking terms of order $\epsilon^{-2}$,  if $I$ is the set of all $z$-circled bodies and $J$ is the set of all $w$-circled bodies,  then 
$$\sum_{i\in I}m_iz_i\; =\; 0, \quad  \sum_{i\in J}m_iw_i=0. $$ 
Although conditions of mass center follow from the main equations, 
sometimes we collect them directly, instead of through main equations, to decrease computational complexity.

The third kind of equations are obtained from clustering schemes. 
Recall that two bodies $k$ and $l$ cluster together in the $z$-diagram if $z_{kl}\prec z_k$ and $z_{kl}\prec z_l$. 
In this case, we have $z_k\sim z_l$, which can be regarded as $z_k=z_l$ after omitting lower order terms. 
We call these equations \textit{cluster equations}.  

The fourth kind of equations are \textit{relations between separations}:
\begin{eqnarray*}
z_{ij}+z_{jk}+z_{kl}&=&0\\
w_{ij}+w_{jk}+w_{kl}&=&0
\end{eqnarray*}
for distinct $i,j,k\in\{1,2,\ldots,n\}$. In order to ensure that all terms in this equation are of the same order, we only pick bodies $i,j,k$ that are vertices of a triangle in the diagram.

The fifth kind of equations are obtained from wedge product and we call them \textit{wedge equations}. 
The idea has been elaborated in \cite[\S 5.0,5.1,8.6-8.8]{AK}. 
Since 
$$q_k=\binom{x_k}{y_k}=\frac{1}{2}(z_k+w_k)\binom{1}{0}+\frac{1}{2}(z_k-w_k)\binom{0}{-i}
=\frac{z_k}{2}\binom{1}{-i}+\frac{w_k}{2}\binom{1}{i},$$ 
we have
$$
q_k\wedge q_l= C(z_kw_l-z_lw_k),
$$
where $C=\frac{1}{4}\left(\binom{1}{-i}\wedge\binom{1}{i}\right)$. 
Then, letting $f_k=\sum_{j\neq k}m_jr_{ij}^{-3}(q_j-q_k)$, from (\ref{eqn:cc2}) we have
\begin{equation}\label{eqn:wedge1}
 0 \;=\;  q_k\wedge f_k \; =\; \sum_{j\neq k}m_jr_{jk}^{-3}(q_k\wedge q_j)\; =\; C\sum_{j\neq k}m_jr_{jk}^{-3}(z_kw_j-w_kz_j),
\end{equation}
On the other hand, from relations
\beq
Z_{jk}\;=\;z_{jk}r_{jk}^{-3}, &&
W_{jk}\;=\;w_{jk}r_{jk}^{-3},
\eeq
we have
$$
r_{jk}^{-3}(q_j- q_k)=\frac{Z_{jk}}{2}\binom{1}{-i}+\frac{W_{jk}}{2}\binom{1}{i},
$$
and hence
\begin{equation}\label{eqn:wedge2}
0\;=\; q_k\wedge f_k \;=\; C\sum_{j\neq k}m_j(z_kW_{jk}-w_kZ_{jk}).
\end{equation}

\subsection{Wedge equations via $zw$-orders}
Now we show how $zw$-order matrices yield useful wedge equations. 

Consider two nonempty disjoint subsets $I_1$ and $I_2$ of $\{1,2,\ldots,n\}$.    
Let $J=\{1,\ldots, n\}\setminus(I_1\cup I_2)$. From system (\ref{eqn:cc4}) we have
\begin{align}
\label{eqn:maxorderterms1}&& \sum_{k\in I_1}m_kz_k&= \sum_{k\in I_1,\ j\in J} m_km_jZ_{jk}+\sum_{k\in I_1,\ j\in I_2} m_km_jZ_{jk},&\\
\label{eqn:maxorderterms2}&& \sum_{k\in I_1}m_kw_k&= \sum_{k\in I_1,\ j\in J} m_km_jW_{jk}+\sum_{k\in I_1,\ j\in I_2} m_km_jW_{jk},&\\
\label{eqn:maxorderterms3}&& \sum_{k\in I_2}m_kz_k&= \sum_{k\in I_2,\ j\in J} m_km_jZ_{jk}+\sum_{k\in I_2,\ j\in I_1} m_km_jZ_{jk},&\\
\label{eqn:maxorderterms4}&& \sum_{k\in I_2}m_kw_k&= \sum_{k\in I_2,\ j\in J} m_km_jW_{jk}+\sum_{k\in I_2,\ j\in I_1} m_km_jW_{jk},&\\
\label{eqn:maxorderterms5}&& \sum_{j\in J}m_jz_j&= \sum_{k\in I_1,\ j\in J} m_km_jZ_{jk}+\sum_{k\in I_2,\ j\in J} m_km_jZ_{jk},&\\
\label{eqn:maxorderterms6}&& \sum_{j\in J}m_jw_j&= \sum_{k\in I_1,\ j\in J} m_km_jW_{jk}+\sum_{k\in I_2,\ j\in J} m_km_jW_{jk}.&
\end{align}
From equation (\ref{eqn:wedge2}) above, we obtain
\begin{eqnarray}\label{eqn:69}
0&=&\sum_{j\in I_1\cup J}m_j(q_j\wedge f_j)\\
\notag &=&C\sum_{j\in I_1,\ k\in I_2}m_jm_k(z_jW_{kj}-w_jZ_{kj})+C\sum_{j\in J,\ k\in I_2}m_jm_k(z_jW_{kj}-w_jZ_{kj}), \\
\notag
0&=&\sum_{j\in I_2\cup J}m_j(q_j\wedge f_j)\\
\notag &=&C\sum_{j\in I_2,\ k\in I_1}m_jm_k(z_jW_{kj}-w_jZ_{kj})+C\sum_{j\in J,\ k\in I_1}m_jm_k(z_jW_{kj}-w_jZ_{kj}).
\end{eqnarray}

Suppose, according to the $zw$-order matrix, $I_1$ and $I_2$ satisfy the following conditions:
{\begin{enumerate}
\item[(1)] for every $k \in I_1$, $z_k\approx \epsilon^{-2}$ and $w_k\prec \epsilon^{-2}$
\item[(2)] for every $k \in I_2$, $z_k\prec \epsilon^{-2}$ and $w_k\approx \epsilon^{-2}$
\item[(3)] for $k,l\in I_1$, $k\ne l$, $z_{kl}\approx \epsilon^{-2}$ and $w_{kl}\approx \epsilon^2$
\item[(4)] for $k,l\in I_2$, $k\ne l$, $z_{kl}\approx \epsilon^{2}$ and $w_{kl}\approx \epsilon^{-2}$
\item[(5)] for $k\in I_1$ and $l\in I_2$, $z_{kl}\approx w_{kl}\approx \epsilon^{-2}$.
\end{enumerate}}
Then we have 
\beq
\sum_{k\in I_2}m_kz_k\preceq  \epsilon^2, &&
\sum_{k\in I_1,\ j\in I_2} m_km_jZ_{jk}\preceq  \epsilon^4, \\
\sum_{k\in I_1}m_kw_k \preceq  \epsilon^2, &&
\sum_{k\in I_1,\ j\in I_2} m_km_jW_{jk}\preceq \epsilon^4.
\eeq
It follows from (\ref{eqn:maxorderterms2}) and (\ref{eqn:maxorderterms3}) that
\beq
\sum_{k\in I_1,\ j\in J} m_km_jW_{jk}\preceq \epsilon^2, \quad \sum_{k\in I_2,\ j\in J} m_km_jZ_{jk}\preceq \epsilon^2.
\eeq

We also assume that $J=\emptyset$ or $J$ has only one element $k_0$ and $z_{k_0},w_{k_0}\preceq \epsilon^2$. For the former, (\ref{eqn:69}) gives that
$$
\sum_{j\in I_1,\ k\in I_2}m_jm_k(z_jW_{kj}-w_jZ_{kj})=0\prec\epsilon^2.
$$
For the latter, (\ref{eqn:maxorderterms5}) and 
(\ref{eqn:maxorderterms6}) yield
\beq
\sum_{k\in I_1,\ j\in J} m_km_jZ_{jk}\preceq \epsilon^2, \quad \sum_{k\in I_2,\ j\in J} m_km_jW_{jk}\preceq \epsilon^2.
\eeq
From (\ref{eqn:69}), since $J$ has only one element, we obtain
\beq
\notag\sum_{j\in I_1,\ k\in I_2}m_jm_k(z_jW_{kj}-w_jZ_{kj})&=& - \sum_{j\in J,\ k\in I_2}m_jm_k(z_jW_{kj}-w_jZ_{kj})\\
\notag &\approx & z_{k_0}\sum_{k\in I_2,\ j\in J} m_km_jW_{jk}+w_{k_0}\sum_{k\in I_1,\ j\in J} m_km_jZ_{jk}\\
\notag &\preceq & \epsilon^4 \prec \epsilon^2.
\eeq
Since
$$
z_jW_{kj}\approx z_jz_{kj}^{-3/2}w_{kj}^{-1/2}\approx z_j^{-1/2} w_{k}^{-1/2}\approx \epsilon^{2}
$$
and
$$
w_jZ_{kj}\approx w_jz_{kj}^{-1/2}w_{kj}^{-3/2}\approx w_j^{-1/2}z_{j}^{-1/2} w_{k}^{-3/2}\approx \epsilon^{3} \prec \epsilon^2
$$
 for every $j\in I_1$ and $k\in I_2$,  we obtain
$$
\sum_{j\in I_1,\ k\in I_2}m_jm_k z_jW_{kj}\prec \epsilon^2,
$$
which gives one of the equations
\begin{equation}\label{eqn:615}
\sum_{j\in I_1,\ k\in I_2}c_{jk}m_jm_k (z_jw_k)^{-1/2}=0,\qquad\ c_{jk}\in\{1,-1\}
\end{equation}

\section{Algorithm~ I - Determine $zw$-diagrams} \label{sec:algorithm I}

Our program contains three algorithms: the first one is to find possible diagrams; 
the second one is to  search for possible orders of various variables associated to a diagram; 
the last one is to find mass relations for diagrams.
All of them use symbolic computations, thus all outputs are exact and free from  round-off errors.

This section is devoted to the first algorithm. The following is the outline of the algorithm.    
See Appendix~I for  codes with Mathematica.  

\smallskip
\begin{enumerate}
\item[(I.1)]  Initial selection of possible $z$-matrices 
\begin{enumerate}
\item  Sort symmetric binary matrices according to column sums and diagonals
\item  Eliminate matrices with trace 1  
\item  Classify them according to trace
\end{enumerate}
\item[(I.2)] Apply monocolored matrix rules
\begin{enumerate}
\item  Rule of Column Sums
\item  First Rule of Trace-2 Matrices 
\item  Second Rule of Triangles 
\item  First Rule of Quadrilaterals
\item  First Rule of Triangles
\item  Rule of Trace-3 Matrices
\item  Rule of Trace-0 Principal Minors
\end{enumerate}
\item[(I.3)]  Produce possible $zw$-matrices
\begin{enumerate}
\item  Produce possible $w$-matrices 
\item  Classify $z$-matrices according to connectedness
\item  Match $z$- and $w$-matrices based on Rule of Fully Connected Companions
\end{enumerate}
\item[(I.4)]  Apply bicolored matrix rules
\begin{enumerate}
\item  Rule of Circling 
\item  Second Rule of Trace-2 Matrices
\item  Third Rule of Triangles 
\item  Rule of Two Column Sums
\item  Fourth Rule of Triangles 
\item  Rule of Connected Components
\item  Second Rule of Quadrilaterals
\item  Rule of Pentagons
\end{enumerate}
\item[(I.5)]  Generate $zw$-diagrams associated to remaining $zw$-matrices 
\end{enumerate}
\smallskip

The whole procedure  involves large amount of computations. 
Elimination of $zw$-matrices by brute force may result in crashes of the computer program when $n\ge 6$, 
so effectively reducing the amount of computations is a critical part of our program. 
In below we discuss some details of the program, except the easy step (I.5). 

\subsubsection*{Step (I.1)}

In the initial selection of possible $z$-matrices,  
by the First Rule of Trace, we don't produce matrices with trace 1. 
To delete some equivalent cases at the beginning, without loss of generality we may assume $z$-circled bodies have larger indices; 
that is, we may assume $z$-matrices satisfy the following condition:
\begin{equation} 
a_{ii}\leq a_{jj},\qquad \forall\  i<j.  \label{eqn:41}
\end{equation} 
The number of such $z$-matrices is $n 2^{n(n-1)/2}$. 
We divide these matrices according to their traces into $n$ classes $S_1,\ldots, S_n$, 
where $S_1$ consists of such matrices with trace 0 and $S_k$ consists of such matrices with trace $k$ for $k>1$. 
Such division significantly reduces the amount of computations  involved in generating $zw$-matrices.  

Before applying matrix rules, to delete more equivalent cases, we permute vertices so that column sums are in increasing order, 
then delete duplicates, and then 
transform them to meet property (\ref{eqn:41})  again. 
This step quickly deletes more than $90\%$ of candidates for possible $z$-matrices  for the case $n=6$.

\subsubsection*{Step (I.2)}

The First Rule of Triangles fails only when $a_{ii}+a_{jj}+a_{kk}=2$ for some $1\leq i<j<k\leq n$, 
therefore the rule holds for matrices in  $S_1\cup S_n$, and so we only need to check whether members in 
$S_2,\ldots ,S_{n-1}$ satisfy this rule. Moreover, since for $A\in S_K$, $2\leq K\leq n-1$, 
we have $a_{ii}=1$ if and only if $i\geq n-K$, this rule on $S_K$ is then reduced to 
$$
(a_{ij}+a_{ik})(1+a_{ij}+a_{jk})(1+a_{ik}+a_{jk})\;\neq\; 8\qquad \forall\ 1\leq i<K\leq j<k\leq n.
$$

From the proof of the Second Rule of Triangles, we know that it fails only when 
$a_{ii}+a_{jj}+a_{kk}=1$ for some $1\leq i<j<k\leq n$, therefore we only need to check whether 
members in $S_2,\ldots ,S_{n-2}$ satisfy this rule. 
For $A\in S_K$, $2\leq K\leq n-2$, we have $a_{ii}=1$ if and only if $i\geq n-K$, this rule on $S_K$ is reduced to
$$
a_{ik}+a_{jk}\;\neq\; 2\qquad \forall\ 1\leq i<j<n-K\leq k\leq n.
$$

The First Rule of Trace-2 Matrices applies to matrices in $S_2$. 
The property (\ref{eqn:41})  implies  $a_{nn}=a_{(n-1)(n-1)}=1$ for every $A\in S_2$. 
Therefore the discriminant is reduced to $$a_{n(n-1)}\;=\;1.$$  

\subsubsection*{Step (I.3)}

After running through monocolored matrix rules for $z$-matrices, we obtain not only possible $z$-matrices but also possible $w$-matrices. 
We cannot eliminate $w$-matrices that are related by conjugation of permutation matrices, since after matching with
a given $z$-matrix they are generally non-equivalent. 
To decrease the initial number of $zw$-matrices, we match each $z$-matrix only with appropriate classes of $w$-matrices. 

We divide each $S_k$ into two parts $S_{k,c}$ and $S_{k,d}$,  where $S_{k,c}$ consists of fully connected matrices 
and $S_{k,d}=S_k\setminus S_{k,c}$.   In addition, we assume each $zw$-matrix $(A|B)$ satisfies
\begin{equation}
\tr(A) \;\leq\; \tr(B).   \label{eqn:42}
\end{equation}
Let $T_1$ be the set of all possible $w$-matrices with trace $0$ and $T_k$ the set of all possible $w$-matrices with 
trace $k$ for $2\leq k\leq n$. 
By the Rule of Fully Connected Companions, we match $z$-matrices in $S_{k,c}$ only with $w$-matrices in $T_1$. 
With assumption (\ref{eqn:42}), we match matrices in $S_{1,c}$ with  matrices in $T_1$ and match matrices in 
$S_{k,d}$ with matrices in $T_1\cup\cdots\cup T_k$ for $1\leq k\leq n$. 
The sets of those $zw$-matrices are  denoted by 
$U_c$ and $U_k$ respectively. 

\subsubsection*{Step (I.4)}

The number of remaining  $zw$-matrices is still huge, about  $2.85\times 10^6$ for the case $n=6$, 
and the order of checking  bicolored matrix rules is critical to the efficiency of the algorithm. 
Some rules involve less computations on each matrix, but exclude more of them, such as the Rule of Circling. 
We implement this kind of rules first. 

For better efficiency, we divide 
the Rule of Circling  into two parts, the first one is $b_{ij}(a_{ii}+a_{jj})\neq 1$; the other is $a_{ij}(b_{ii}+b_{jj})\neq 1$. 
Since the first part $b_{ij}(a_{ii}+a_{jj})\neq 1$ is false only when $1\leq\tr(A)\leq n-1$, 
we only need to check this part for matrices in $U_K$ with $2\leq K\leq n-1$. 
Since each $zw$-matrix $(A|B)$ satisfies  (\ref{eqn:41}), in $U_K$, $2\leq K\leq n-1$,  the discriminant $b_{ij}(a_{ii}+a_{jj})\neq 1$ 
becomes $b_{ij}\neq 1$ for $i< n-K$ and $j\geq n-K$. 
The second part $a_{ij}(b_{ii}+b_{jj})\neq 1$ is false only when $\tr(B)>0$, therefore the rule holds true for matrices in $U_c\cup U_1$.

Together with 
the Second Rule of Trace-2 Matrices, more than $88\%$ of $zw$-matrices were excluded for the case $n=6$,
about $3.35\times 10^5$ cases left. 
Then we combine $U_c$ and all $U_k$ into a collection $U$,  implement remaining bicolor matrix rules in order,  
and then remove all duplicated matrices up to permutations of bodies and the exchange of $z$ and $w$. 
This finishes the search for all possible $zw$-matrices, and that  allows us to generate all possible $zw$-diagrams.
When $n=6$, there are 117 diagrams remaining. 
The next step is to apply our second algorithm to each diagram.

\section{Algorithm~II - Determine orders of variables} \label{sec:algorithm II}

Given a singular sequence associated with the $zw$-matrix $(A|B)$. 
Our second algorithm aims to find a suitable collection of order matrices to cover orders of positions, separations, and distances. 

Before showing steps of our algorithm, we classify $zw$-orders according to sizes of off-diagonal entries and 
diagonal entries. 

\begin{definition}
Given a $zw$-matrix $(A|B)$. 
\begin{enumerate}
\item[(a)] Any $zw$-order matrix $(S,T)$ such that $(S,T)\Supset(A|B)$ is called a $zw$-order matrix of {\em Type~1
for the $zw$-matrix $(A|B)$}. 
\item[(b)] We say a nonempty $zw$-order matrix $(S,T)$ is of {\em Type 2\,}
if off-diagonal entries of $S$, $T$ are singletons; that is,  $$\prod_{i<j}|s_{ij}|\ |t_{ij}|\;=\;1.$$
\item[(c)] We say a nonempty $zw$-order matrix  $(S,T)$  is of {\em Type 3\,} if all entries of $S$, $T$
are singletons; that is, $$\prod_{i\le j}|s_{ij}|\ |t_{ij}|\;=\;1.$$
\end{enumerate}
\end{definition} 

Our strategy is to begin with a proper selection of Type~1 $zw$-order matrix $(S,T)$ for $(A|B)$, then choose
a collection $\coll_2$ of Type~2 subsets of $(S,T)$ such that $\coll_2\Supset (A|B)$, 
and finally choose another collection $\coll_3$ of Type~3 subsets of $(S,T)$ such that $\coll_3\Supset (A|B)$. 

\begin{notations}
Given a Type~1 $zw$-order matrix $(S,T)$ for $(A|B)$. Let
\beq
 \tilde{\coll}_1(S,T) &=& \left\{(S',T')\subset (S,T)|\, \text{$(S',T')$ is nonempty} \right\}, \\
 \tilde{\coll}_2(S,T) &=& \left\{(S',T')\in  \tilde{\coll}_1(S,T)|\, \text{$(S',T')$ is of Type~2} \right\}, \\
 \tilde{\coll}_3(S,T) &=& \left\{(S',T')\in  \tilde{\coll}_1(S,T)|\, \text{$(S',T')$ is of Type~3} \right\}. 
\eeq
\end{notations}
Note that $\tilde{\coll}_1(S,T) \supset \tilde{\coll}_2(S,T) \supset \tilde{\coll}_3(S,T)\Supset (A|B)$. 

For convenience, 
Propositions~\ref{prop:rom2},~\ref{prop:rom3},~\ref{prop:rom4},~\ref{prop:rom7},~\ref{prop:rom8} 
will be called {\em principles for updating orders}. 
The other two propositions in \S\ref{subsec:generateom} will be applied for the initial selection of Type~1
$zw$-order matrix. 
Propositions~\ref{prop:romc1},~\ref{prop:romc4},~\ref{prop:romc2},~\ref{prop:romc3}, \ref{prop:romc5} will be called 
{\em criteria for feasible order matrices}. 

The following shows steps of the algorithm. See Appendix~II for  codes in Mathematica.

\begin{enumerate}
\item[(II.1)]  Choose a Type 1 $zw$-order matrix $(S,T)$ for $(A|B)$ 
\begin{enumerate}
\item  Select diagonal entries based on circles (apply Proposition~\ref{prop:rom1}(a)) 
\item  Select off-diagonal entries for maximal edges (apply Proposition~\ref{prop:rom6})  
\item  Select other off-diagonal entries based on edges  (apply Proposition~\ref{prop:rom1}(b))  
\item  Repeatedly refine $(S,T)$ by principles for updating orders
\end{enumerate}
\item[(II.2)]  Choose a collection $\coll_2$ of Type 2 $zw$-order matrices such that $\coll_2\Supset (A,B)$
\begin{enumerate}
\item  Choose elements in $\tilde{\coll}_2(S,T)$ whose diagonal entries are the same as $(S,T)$
\item  Refine these elements by principles for updating orders
\item  Apply criteria for feasible order matrices, let $\coll_2$ be remaining elements
\item  Eliminate $zw$-diagram if $\coll_2=\emptyset$, otherwise proceed to next step
\end{enumerate}
\item[(II.3)]  Choose a collection $\coll_3$ of Type 3 $zw$-order matrices such that $\coll_3\Supset (A,B)$
\begin{enumerate}
\item  Choose elements in $\tilde{\coll}_3(S,T)$ that are contained in elements of $\coll_2$
\item  Refine these elements by principles for updating orders
\item  Apply criteria for feasible order matrices, let $\coll_3$ be remaining elements  
\item  Generate $r$-order matrices $\mathcal{R}$ from $\coll_3$
\item  Eliminate $zw$-diagram if $\coll_3=\emptyset$ 
\end{enumerate}
\end{enumerate}

In below we discuss some technical details. 

\subsubsection*{Step (II.1)}

This step is rather straightforward. 
Part (d) will terminate at some Type~1 $zw$-order matrix $(S,T)$ for $(A|B)$, we say this $(S,T)$ is {\em optimal}. 
Strictly speaking, this $zw$-order matrix is better called ``optimal with respect to principles of updating orders'', because
one may explore other principles to further refine it. See \S\ref{subsec:6-2etc} for such an example. 
With a bit abuse of terminology, we simply call it optimal for brevity.

\subsubsection*{Step (II.2)}

We use the optimal $(S,T)$ to find more precise orders for separations. 
The process of choosing elements in $\tilde{\coll}_2(S,T)$ cannot be implemented through brute force. 
We begin with separating a few off-diagonal entries for elements in $\tilde{\coll}_1(S,T)$, and 
progressively repeat the process to other off-diagonal entries. 

To be more precise, consider a pair $(i,j)$, $i\neq j$, with $|s_{ij}|>1$ or $| t_{ij}|>1$. 
If $|s_{ij}|>1$, we separate $s_{ij}$ into singletons $\{\sigma_1\},\{\sigma_2\},\ldots \{\sigma_k\}$,
but do not separate $t_{ij}$ at the same time to prevent from taking up too much memory space. 
Let $S_l$ be the matrix obtained by substituting $\{\sigma_l\}$ for $s_{ij}$ and $s_{ji}$ in $S$. 
For example, if
$$
S=\begin{pmatrix}
\{1,2,3\} & \{1,2,3\} & \{1\} \\
\{1,2,3\} & \{1,2,3\} & \{1,2\} \\
\{1\} & \{1,2\} & \{1,2\} 
\end{pmatrix}
$$
and $(i,j)=(1,2)$, let $\sigma_k=k$, then we obtain 
$$S_k=\begin{pmatrix}
\{1,2,3\} & \{k\} & \{1\} \\
\{k\} & \{1,2,3\} & \{1,2\} \\
\{1\} & \{1,2\} & \{1,2\} 
\end{pmatrix}\qquad\textrm{for }k=1,2,3.
$$
Whatever been covered by $(S,T)$ is now covered by $\{(S_1,T), (S_2,T), (S_3,T)\}$. 
Next, for each $(S_l,T)$ we apply principles for updating orders, then the resulting $(S',T')$ satisfies  
one of the followings: 
$$\prod_{i< j} |s'_{ij}|\  |t'_{ij}| \;=\; 0,\;1,\;\text{or greater than 1}.$$ 
If it is 0, we delete this $(S',T')$; 
if it is 1, we collect this $(S',T')$; if it is great than 1, we repeat the above process to other off-diagonal entries of $(S',T')$. 
Eventually, we have a collection $\coll$ that contains $zw$-order matrices of Type 2 such that $\coll\Supset (A|B)$. 

Let $P_{(A|B)}$ be the set of all permutations on $\{1,2,\ldots,n\}$ that doesn't change $(A|B)$; 
i.e. $q\in P_{(A|B)}$ means the permutation matrix $Q$ corresponding to $q$ satisfies $Q^{-1}AQ=A$ and $Q^{-1}BQ=B$. 
Conjugations by permutations correspond relabeling vertices, so it makes sense to apply them on order matrices. 
Two $zw$-order matrices in $\coll$ are equivalent if one can be obtained from the other by a permutation $q\in P_{(A|B)}$. 
In Algorithm~I we have eliminated duplicated $zw$-matrices up to permutations of bodies.  
Now for elements in $\coll$ we do the same thing, then apply criteria for feasible order matrices to delete more. 

A subroutine in this step calculates $r$-order matrices produced by $\coll$, 
another subroutine in this step carries out the order-comparing process.  
Eventually, we either obtain a collection $\coll_2$ of $zw$-order matrices of Type~2 such that $\coll_2\Supset (A|B)$, 
or end up with an empty collection, for which case the $zw$-diagram can be eliminated. 

\subsubsection*{Step (II.3)}

In this step we apply the same process as Step (II.2) on diagonal entries of  $(S,T)\in \mathcal{C}_2$. 
Again, we separate diagonal entries one by one.
Eventually we either obtain a collection $\mathcal{C}_3$ of $zw$-order matrices of Type~3 such that $\coll_3\Supset (A|B)$, 
or end up with an empty collection, for which case we eliminate the $zw$-diagram.   \medskip

Both collections $\coll_2$, $\coll_3$ record possible orders for positions and separations. 
We have seen in section~\ref{sec:poly} how  $\coll_2$ or $\coll_3$ can be used to generate polynomial equations. 
The next step is to apply our last algorithm to find mass relations.

\section{Algorithm~ III - Eliminations and mass relations} \label{sec:algorithm III}

For each $zw$-matrix $(A|B)$, we collect a system of polynomial equations as outlined in section~\ref{sec:poly}. 
This system of polynomial equations are necessary conditions for the existence of singular sequence associated to $(A|B)$. 
Our hope is to apply elimination theory to obtain polynomial equations with only mass variables, 
we call these equations {\em mass relations associated to $(A|B)$}. 
By doing so, we conclude that masses  outside the zero loci of these mass relations cannot possibly result in 
a singular sequence associated to $(A|B)$. 
If we are able to do the same for every $zw$-matrix, then we conclude finiteness of central configurations
for masses outside the zero loci of a family of mass relations. 
Since central configurations are invariant by scaling masses, such zero loci should be a co-dimension~2
algebraic variety in the mass space. 

Given $zw$-matrix $(A|B)$.  
After implementation of Algorithm~II, we obtain a collection $\coll$ of $zw$-order matrices. 
It can be either the $\coll_2$ or $\coll_3$ obtained in the previous section. 
Our third algorithm, see Appendix~III for codes in Mathematica, has two simple steps: 

\smallskip
\begin{enumerate}
\item[(III.1)]  Collect polynomial equations 
\item[(III.2)]  Generate mass relations by eliminations
\end{enumerate}

\subsubsection*{Step (III.1)}
As stated in section~\ref{sec:poly}, we collect five classes of polynomial equations. 
We do not demonstrate details here as this step is rather straightforward. 
One little twist in the collection of main equations is introducing variables
\beq
  u_{ij} \;=\; z_{ij}^{-1/2}, \quad v_{ij} \;=\; w_{ij}^{-1/2}. 
\eeq
Complex square roots are multi-valued, but our equations and eliminations are unambiguous because
terms of the form $u_{ij}^kv_{ij}^l$ with $k,l\in \Z$  satisfy  $k+l\in 2\Z$ throughout the process.

\subsubsection*{Step (III.2)}
There are several parallel subroutines designed for this step. 
There are similarities and differences in these subroutines, in terms of polynomial equations called on, 
and strategies of eliminations. We will demonstrate three of them, which are sufficient for our analysis for $n=4, 5, 6$. 
The following three subsections show some technical details.

\subsection{Elimination process}  \label{subsec:elimination}
Suppose we have several polynomial equations in several variables. 
Let $\mathcal{P}$ be a set of the polynomials corresponding to these equations and $V$ the family of variables. 
Our eliminating process contains the following four kinds of operations.

\subsubsection{Delete secluded members}
If there is some $v\in V$ that appears only at one polynomial $f\in\mathcal{P}$, 
then we delete $f$ from $\mathcal{P}$ and delete $v$ from $V$.

\subsubsection{Single out factors} If there is some $f\in \mathcal{P}$ such that $f=g_1^{n_1}\cdot\cdots\cdot g_k^{n_k}$ 
for some irreducible polynomials $g_1,\ldots, g_k$, then we replace $f$ by $g_1\cdot\cdots\cdot g_k$ 
since it has the same zeros as $f$.  

\subsubsection{Simple substitutions} If there is a polynomial $f\in\mathcal{P}$ and a variable $v\in V$ such that $f=v g+h$ where $g,h$ are polynomials consisting of a single term and independent of $v$, then we apply the substitution $v=-h/g$ for each member of $\mathcal{P}$. After multiplying each member in $\mathcal{P}$ by $g^k$ for some $k\in\N$, $\mathcal{P}$ will be a collection of polynomials independent of $v$.  

\subsubsection{Elimination} 
For each $f\in\mathcal{P}$ and $v\in V$, let $\deg(f,v)$ be the maximal degree of terms 
in $f$ with respect to $v$, and $c_f$ be the coefficient of the term $v^{\deg(f,v)}$ in $f$. 
For example, if $f=3x^2z+2xy+4zy^3$ and $v=x$, then $\deg(f,v)=2$ and $c_f=3z$; if $f=3xy^2+2xz-z^2y^2$ and $v=y$, 
then $\deg(f,v)=2$ and $c_f=3x-z^2$.

Let $f_0$ be a member in $\mathcal{P}$ such that $\deg(f_0,v)>0$ and $\deg(f_0,v) \leq \deg(f,v)$ 
for every $f\in\mathcal{P}$ with $\deg(f,v)>0$.
Then for each $f\in \mathcal{P}$ with $\deg(f,v)>0$, let $$g\;=\;c_{f_0}f-c_{f}f_0v^{\deg(f,v)-\deg(f_0,v)},$$ 
and substitute $g$ for $f$. Since the coefficient of the term $v^{\deg(f,v)}$ in $g$ is $0$, we have $\deg(g,v)<\deg(f,v)$.

Repeat this process until there is only one $f\in\mathcal{P}$ with $\deg(f,v)>0$.

\subsection{Generate mass relations}  \label{subsec:genmassrelation}
For different $zw$-matrices, we have different ways of finding mass relations. 

\subsubsection{} For some $zw$-matrices, we just collect the main equations determined by the $zw$-matrix $(A|B)$ and 
make substitutions according to the cluster equations determined by some $zw$-order matrices $(S,T)$. 
Then use the function ``GroebnerBasis'' in Mathematica directly to find out mass relations.

\subsubsection{}  For some $zw$-matrices, we check whether each $zw$-order matrix in $\coll_3$ satisfies the 
hypothesis that the system contains some wedge equation as (\ref{eqn:615}). Then for each $(S,T)\in \coll_3$ and each 
choice $f$ of (\ref{eqn:615}), we collect $f$ together with the conditions of mass center as a collection $\mathcal{P}$, and 
make substitutions according to the cluster equations determined by $(S,T)$. 

Since the power of some $z_k$ or $w_k$ in $f$ may not be an integer, we replace $z_k,w_k$ by $\zeta_k^2,\omega_k^2$ respectively. 
Then we apply the above eliminating process until there is a polynomial in $\mathcal{P}$ that depends only on masses, 
or there is no polynomial in $\mathcal{P}$.

\subsubsection{} For some $zw$-matrices, we collect main equations determined by $(A|B)$ or some $(S,T)$ and the relations 
between separations, and make substitutions according to the cluster equations determined by $(S,T)$.
Let $\mathcal{P}$ be the collection of polynomials corresponding to these equations. 
We replace $Z_{kl}$'s and $W_{kl}$'s by $z_{kl}^{-1/2}w_{kl}^{-3/2}$ and $w_{kl}^{-1/2}z_{kl}^{-3/2}$ respectively, 
then we substitute $u_{kl}^{-2}$ and $v_{kl}^{-2}$ for $z_{kl}$ and $w_{kl}$ respectively. 
After multiplying each member in $\mathcal{P}$ by some $u^{n_1}v^{n_2}$, $n_1+n_2\in 2\N$, 
$\mathcal{P}$ is a new collection of polynomials. 
Then we apply the above eliminating process until there is a polynomial in $\mathcal{P}$ that depends only on masses, 
or there is no polynomial in $\mathcal{P}$.  

\subsubsection{} 
In details of the process, we have some variations. 
By choosing $\coll$ such that  $\coll\Supset (A|B)$,  we may use $zw$-order matrices in $\coll$ to determine the main equations 
and the cluster equations. 
If for each $(S,T)\in\coll$ we have a mass relation, then their common multiple is a mass relation for this $zw$-matrix. 
In addition, we may add other equations that can be obtained by some discussion but are not collected automatically in the algorithm, 
and sometimes we only pick a part (not all) of main equations because too many equations would be computationally expensive.

\subsection{Summary of subroutines}\label{subsec:subroutines}

The first subroutine collects the main equations determined by the $zw$-matrix $(A|B)$ and 
make substitutions according to the cluster equations determined by some order matrices. 
Then use the function ``GroebnerBasis'' in Mathematica to generate mass relations. 

For the second subroutine, we check whether it satisfy the hypothesis that the system contains some 
wedge equation. Then for each choice $f$ of these equations, we collect $f$ together with the conditions 
of mass center as a collection $\mathcal{P}$, and make substitutions according to the cluster equations 
determined by some order matrices $(S,T)$.  Replace $z_k$, $w_k$ by $\zeta_k^2$, $\omega_k^2$ respectively.  
 Then apply the above eliminating process until there is a polynomial in  $\mathcal{P}$ that depends only on masses, 
 or there is no polynomial in  $\mathcal{P}$. 

In the third subroutine, we collect all or part of main equations determined by some $zw$-order matrix $(S,T)$ and 
some additional equations, if necessary, into a 
set $V_1$ and replace $Z_{kl}$, $W_{kl}$ by $u_{kl}v_{kl}^{3}$, $u_{kl}^{3}v_{kl}$ respectively.  
Next, we collect relations of 
separations whose involving bodies are vertices of a triangle in the diagram, and replace $z_{kl}$, $w_{kl}$ by $u_{kl}^{-2}$, $v_{kl}^{-2}$ 
respectively. Then we make substitutions according to cluster equations determined by the $zw$-order matrix $(S,T)$, and 
apply the elimination process for the set $V_1$ (instead of appealing to Gr\"obner basis).



\section{The case $n=4$} \label{sec:4bd}

For the four-body problem, we obtain five diagrams for singular sequences in Figure~\ref{diagrams of 4BD} by Algorithm~I.
These diagrams and mass relations obtained from Algorithms~II,~III  are the same as those in \cite{AK}, 
except the appearance of some additional factors for the third diagram.
This is no coincidence since our algorithms were designed to embody their ideas in a mechanical process.  
This section demonstrates applications of our algorithms to these diagrams. Similar cases appear
multiple times for $n=5,\, 6$. 

\begin{figure}[H]
\centering
\includegraphics[width=400pt]{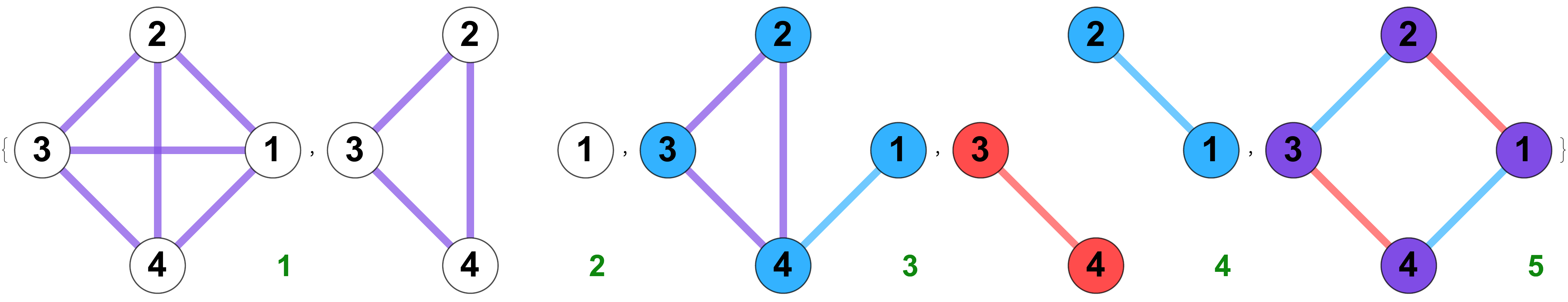}
\caption{Diagrams for singular sequences of the 4-body problem.\label{diagrams of 4BD}  }  
\end{figure}

\subsection{Diagram 1} \label{subsec:4-1}
This is the fully-edged diagram. By Algorithm~II, 
we obtain one $zw$-order matrix of Type~2 to cover orders of the $zw$-matrix.  
\begin{align*}
\small{
\left( 
\begin{pmatrix}
\{0\sim3\} & \{3\} & \{3\} & \{3\} \\
\{3\} & \!\{0\sim3\}\! & \{3\} & \{3\} \\
\{3\} & \{3\} & \!\{0\sim3\}\! & \{3\} \\
\{3\} & \{3\} & \{3\} & \{0\sim3\}
\end{pmatrix}, 
\begin{pmatrix}
\{0\sim3\} & \{3\} & \{3\} & \{3\} \\
\{3\} & \!\{0\sim3\}\! & \{3\} & \{3\} \\
\{3\} & \{3\} & \!\{0\sim3\}\! & \{3\} \\
\{3\} & \{3\} & \{3\} & \{0\sim3\} 
\end{pmatrix} \right)\!.    }
\end{align*}
This is an obvious corollary of Estimate~2 since every separation has order $\epsilon$. 
Step~(II.3) of Algorithm~II results in 25 $zw$-order matrices of Type~3 to cover orders of the $zw$-matrix. 
However, with too many leading order terms, 
in this case we are not able to find mass relation by our algorithms. 

In fact, as shown in \cite[\S5.6]{AK},  finiteness can be achieved without appealing to any mass relation for this particular diagram. 
The idea is to introduce a dominant polynomial in distances which is bounded on singular sequences
approaching this particular diagram. Using a basic property of dominant polynomials, 
it ensures the existence of a singular subsequence approaching one of the other four diagrams. Therefore, to conclude finiteness, 
it is sufficient to exclude the other four diagrams.  

\subsection{Diagram 2}  \label{subsec:4-2}
This diagram consists of a fully-edged isolated triangle. By Algorithm~II, we find 5 $zw$-order matrices of 
Type~2 to cover orders of the $zw$-matrix. They are 
\beq
\small{
\left(\left(\!\!
\begin{array}{cccc}
 \{0\sim3\} & \{2\} & \{3\} & \{3\} \\
 \{2\} & \!\{0\sim3\}\! & \{3\} & \{3\} \\
 \{3\} & \{3\} & \!\{0\sim3\}\! & \{3\} \\
 \{3\} & \{3\} & \{3\} & \{0\sim3\} \\
\end{array}
\!\!\right), 
 \left(\!\!
\begin{array}{cccc}
 \{0\sim4\} & \{4\} & \{4\} & \{4\} \\
 \{4\} & \!\{0\sim4\}\! & \{3\} & \{3\} \\
 \{4\} & \{3\} & \!\{0\sim4\}\! & \{3\} \\
 \{4\} & \{3\} & \{3\} & \{0\sim4\} \\
\end{array}
\!\!\right)\right)    }
\eeq   \vspace{-2mm}
\beq
\small{
\left( \left(\!\!
\begin{array}{cccc}
 \{0\sim4\} & \{4\} & \{4\} & \{4\} \\
 \{4\} & \!\{0\sim4\}\! & \{3\} & \{3\} \\
 \{4\} & \{3\} & \!\{0\sim4\}\! & \{3\} \\
 \{4\} & \{3\} & \{3\} & \{0\sim4\} \\
\end{array}
\!\!\right),\left( \!\!
\begin{array}{cccc}
 \{0\sim3\} & \{2\} & \{3\} & \{3\} \\
 \{2\} & \!\{0\sim3\}\! & \{3\} & \{3\} \\
 \{3\} & \{3\} & \!\{0\sim3\}\! & \{3\} \\
 \{3\} & \{3\} & \{3\} & \{0\sim3\} \\
\end{array}
\!\! \right) \right)   }
\eeq  \vspace{-2mm}
\beq
\small{
\left( \left( \!\!
\begin{array}{cccc}
 \{0\sim4\} & \{4\} & \{4\} & \{4\} \\
 \{4\} & \!\{0\sim4\}\! & \{3\} & \{3\} \\
 \{4\} & \{3\} & \!\{0\sim4\}\! & \{3\} \\
 \{4\} & \{3\} & \{3\} & \{0\sim4\} \\
\end{array}
\!\! \right),\left( \!\!
\begin{array}{cccc}
 \{0\sim4\} & \{4\} & \{4\} & \{4\} \\
 \{4\} & \!\{0\sim4\}\! & \{3\} & \{3\} \\
 \{4\} & \{3\} & \!\{0\sim4\}\! & \{3\} \\
 \{4\} & \{3\} & \{3\} & \{0\sim4\} \\
\end{array}
\!\! \right) \right)    }
\eeq \vspace{-2mm}
\beq
\small{
\left( \left( \!\!
\begin{array}{cccc}
 \{0\sim3\} & \{3\} & \{3\} & \{3\} \\
 \{3\} & \!\{0\sim3\}\! & \{3\} & \{3\} \\
 \{3\} & \{3\} & \!\{0\sim3\}\! & \{3\} \\
 \{3\} & \{3\} & \{3\} & \{0\sim3\} \\
\end{array}
\!\! \right),\left( \!\!
\begin{array}{cccc} 
 \{0\sim4\} & \{4\} & \{4\} & \{4\} \\
 \{4\} & \!\{0\sim4\}\! & \{3\} & \{3\} \\
 \{4\} & \{3\} & \!\{0\sim4\}\! & \{3\} \\
 \{4\} & \{3\} & \{3\} & \{0\sim4\} \\
\end{array}
\!\! \right) \right)   }
\eeq \vspace{-2mm}
\beq
\small{
\left( \left( \!\!
\begin{array}{cccc}
 \{0\sim4\} & \{4\} & \{4\} & \{4\} \\
 \{4\} & \!\{0\sim4\}\! & \{3\} & \{3\} \\
 \{4\} & \{3\} & \!\{0\sim4\}\! & \{3\} \\
 \{4\} & \{3\} & \{3\} & \{0\sim4\} \\
\end{array}
\!\! \right),\left( \!\!
\begin{array}{cccc}
 \{0\sim3\} & \{3\} & \{3\} & \{3\} \\
 \{3\} & \!\{0\sim3\}\! & \{3\} & \{3\} \\
 \{3\} & \{3\} & \!\{0\sim3\}\! & \{3\} \\
 \{3\} & \{3\} & \{3\} & \{0\sim3\} \\
\end{array}
\!\! \right) \right)  }
\eeq

They all produce the same $r$-order matrix
\beq
\small{
\left(\!\!
\begin{array}{cccc}
 \{\} & \{2,3,4\} & \{2,3,4\} & \{2,3,4\} \\
 \{2,3,4\} & \{\} & \{1\} & \{1\} \\
 \{2,3,4\} & \{1\} & \{\} & \{1\} \\
 \{2,3,4\} & \{1\} & \{1\} & \{\} \\
\end{array}
\!\! \right).  }
\eeq
Note that the $(1,2)$, $(1,3)$, and $(2,3)$-th entries of left and right half of these $zw$-order matrices are exactly what we 
saw in Example~\ref{exam:chooseom}.  Step~(II.3) of Algorithm~II results in 468 $zw$-order matrices of Type~3
to cover orders of the $zw$-matrix. 

One may apply the third subroutine in \S\ref{subsec:elimination} with all $z_i$-, $w_i$-equations directly, or pick a smaller set of 
polynomial equations to reduce the amount of computations. For instance,  
after determining main equations and cluster equations by the $zw$-matrix,  
we may collect the $z_2$-, $z_3$-, $w_2$-, $w_3$-equations and relations between separations, 
make substitutions according to the cluster equations determined by the $zw$-order matrices of Type~1,  
and apply eliminating process. 
This results in a mass relation $f_{4,2}=0$ where $f_{4,2}$ is a polynomial of 
$\sqrt{m_1},\sqrt{m_2},\sqrt{m_3},\sqrt{m_4}$ with 6 factors:
\begin{flalign}\label{massrel:4-2.1}
\notag & \sqrt{m_2}\sqrt{m_3}-\sqrt{m_2}\sqrt{m_4}-\sqrt{m_3}\sqrt{m_4},\\
\notag & \sqrt{m_2}\sqrt{m_3}+\sqrt{m_2}\sqrt{m_4}-\sqrt{m_3}\sqrt{m_4},\\
           & \sqrt{m_2}\sqrt{m_3}-\sqrt{m_2}\sqrt{m_4}+\sqrt{m_3}\sqrt{m_4},\\
\notag & m_2m_3+m_2m_4+m_3m_4-2m_2\sqrt{m_3}\sqrt{m_4},\\
\notag & m_2m_3+m_2m_4+m_3m_4-2\sqrt{m_2}m_3\sqrt{m_4},\\
\notag & m_2m_3+m_2m_4+m_3m_4-2\sqrt{m_2}\sqrt{m_3}m_4.
\end{flalign}
The first three factors give equations
\begin{equation}\label{massrel:4-2.2}
\frac{1}{\sqrt{m_i}}=\frac{1}{\sqrt{m_j}}+\frac{1}{\sqrt{m_k}} \qquad \{i,j,k\}=\{2,3,4\}.
\end{equation}
The last three factors can be written  
\begin{equation*}\label{massrel:4-2.3}
m_i(\sqrt{m_j}-\sqrt{m_k})^2+m_jm_k\qquad  \{i,j,k\}=\{2,3,4\},
\end{equation*}
which are never zero for positive masses.

\subsection{Diagram 3}   \label{subsec:4-3}
This is the ``{\em kite diagram}'', with a $w$-edge connected to a fully-edged triangle. 
By Algorithm~II,  we obtain two $zw$-order matrices of Type~3 to cover orders of the $zw$-matrix: 
\beq
\left( \begin{pmatrix}
\{1\} & \{3\} & \{3\} & \{1\} \\
\{3\} & \{3\} & \{3\} & \{3\} \\
\{3\} & \{3\} & \{3\} & \{3\} \\
\{1\} & \{3\} & \{3\} & \{1\} 
\end{pmatrix}, \begin{pmatrix}
\{5\} & \{5\} & \{5\} & \{5\} \\
\{5\} & \{5\} & \{3\} & \{3\} \\
\{5\} & \{3\} & \{5\} & \{3\} \\
\{5\} & \{3\} & \{3\} & \{5\} 
\end{pmatrix} \right),  \\
\left( \begin{pmatrix}
\{1\} & \{3\} & \{3\} & \{1\} \\
\{3\} & \{3\} & \{3\} & \{3\} \\
\{3\} & \{3\} & \{3\} & \{3\} \\
\{1\} & \{3\} & \{3\} & \{0\} 
\end{pmatrix}, \begin{pmatrix}
\{5\} & \{5\} & \{5\} & \{5\} \\
\{5\} & \{5\} & \{3\} & \{3\} \\
\{5\} & \{3\} & \{5\} & \{3\} \\
\{5\} & \{3\} & \{3\} & \{5\} 
\end{pmatrix} \right).
\eeq
Both of them generate the $r$-order matrix 
\beq
\left(\!\!
\begin{array}{cccc}
 \{\} & \{4\} & \{4\} & \{3\} \\
 \{4\} & \{\} & \{1\} & \{1\} \\
 \{4\} & \{1\} & \{\} & \{1\} \\
 \{3\} & \{1\} & \{1\} & \{\} \\
\end{array}
\!\! \right). 
\eeq

Again, we can either apply the third subroutine in \S\ref{subsec:elimination} directly, or 
collect only the $z_2$-, $z_3$-, $w_2$-, $w_3$-equations, cluster equations, and relations between separations. 
Then, by making substitutions according to cluster equations determined by the $zw$-order matrices and applying the eliminating process,  
Algorithm~III outputs the mass relation $f_{4,3}=0$, where the polynomial $f_{4,3}$ has 11 factors. 
Six of them are exactly the same as (\ref{massrel:4-2.1}). Two other short factors are 
\begin{flalign*}
        & -m_2^{3/4} m_3^{3/4}+m_2 \sqrt{m_4}+m_3 \sqrt{m_4}, \\
\notag & m_2^3- (m_2+ m_3) m_4^2.
\end{flalign*}
The other three factors $\mu_{15}(m_2,m_3,m_4)$, $\mu_{21}(m_2,m_3,m_4)$, $\mu_{33}(m_2,m_3,m_4)$ are longer. 
See Appendix~IV for details. The first factor above  is equivalent to 
\begin{equation}\label{massrel:4-3.2}
(m_2m_3)^3=m_4^2(m_2+m_3)^4.
\end{equation}

\subsection{Diagram 4}  \label{subsec:4-4}
The fourth diagram in Figure~\ref{diagrams of 4BD} is the disconnected diagram consisting of a $z$-edge and a $w$-edge. 
Applying Algorithm~II,  one $zw$-order matrix of Type~3 covers orders of the $zw$-matrix:
$$
\left( \begin{pmatrix}
\{1\} & \{1\} & \{5\} & \{5\} \\
\{1\} & \{1\} & \{5\} & \{5\} \\
\{5\} & \{5\} & \{5\} & \{5\} \\
\{5\} & \{5\} & \{5\} & \{5\} 
\end{pmatrix}, \begin{pmatrix}
\{5\} & \{5\} & \{5\} & \{5\} \\
\{5\} & \{5\} & \{5\} & \{5\} \\
\{5\} & \{5\} & \{1\} & \{1\} \\
\{5\} & \{5\} & \{1\} & \{1\} 
\end{pmatrix} \right).
$$

We have wedge equations
\begin{equation*}\label{massrel:4-4.1}
m_1m_3(z_1w_3)^{-1/2}\pm m_1m_4(z_1w_4)^{-1/2}\pm m_2m_3(z_2w_3)^{-1/2}\pm m_2m_4(z_2w_4)^{-1/2}.
\end{equation*}
For each wedge equation, we make substitutions according to the cluster equations and apply the elimination process on the 
collection consisting of this wedge equation and conditions of mass center. 
Two mass relations can be obtained: 
for 4 of the wedge equations, there is a mass relation
\begin{equation*}\label{massrel:4-4.2}
(m_1^2-m_1m_2+m_2^2)(m_3^2-m_3m_4+m_4^2)=0,
\end{equation*}
which doesn't hold for positive masses; for 4 other wedge equations, we find the mass relation
\begin{equation*}\label{massrel:4-4.3}
m_1^3m_3^3+m_2^3m_3^3-4m_1^{3/2}m_2^{3/2}m_3^{3/2}m_4^{3/2}+m_1^3m_4^3+m_2^3m_4^3=0.
\end{equation*}
These can be implemented automatically by the second subroutine in \S\ref{subsec:subroutines}

The last equation can be rewritten 
\begin{equation*}\label{massrel:4-4.4}
[(m_1m_3)^{3/2}-(m_2m_4)^{3/2}]^2+[(m_1m_4)^{3/2}-(m_2m_3)^{3/2}]^2=0.
\end{equation*}
For positive masses, the above equation holds if and only if
$$
m_1m_3=m_2m_4\quad\textrm{and}\quad m_1m_4=m_2m_3,
$$ which is equivalent to the condition
\begin{equation}\label{massrel:4-4.5}
m_1=m_2\quad\textrm{and}\quad m_3=m_4.
\end{equation}

\subsection{Diagram 5}   \label{subsec:4-5}
The last diagram in Figure~\ref{diagrams of 4BD} is the square diagram. 
By Algorithm~II,  again one $zw$-order matrix of Type~3 covers orders of the $zw$-matrix: 
$$
\left( \begin{pmatrix}
\{5\} & \{5\} & \{5\} & \{1\} \\
\{5\} & \{5\} & \{1\} & \{5\} \\
\{5\} & \{1\} & \{5\} & \{5\} \\
\{1\} & \{5\} & \{5\} & \{5\} 
\end{pmatrix},  
\begin{pmatrix}
\{5\} & \{1\} & \{5\} & \{5\} \\
\{1\} & \{5\} & \{5\} & \{5\} \\
\{5\} & \{5\} & \{5\} & \{1\} \\
\{5\} & \{5\} & \{1\} & \{5\} 
\end{pmatrix} \right).
$$
As in the previous diagram, we know precise orders for all separations, and they are either maximal or minimal. 

We collect main equations determined by the $zw$-matrix and make substitutions according to cluster equations 
determined by some $zw$-order matrices. 
By either applying the elimination process (e.g. third subroutine in \S\ref{subsec:elimination}), 
or by using the function ``GroebnerBasis'' in Mathematica directly (first subroutine in \S\ref{subsec:elimination}), 
the mass relation below can be easily obtained:
\begin{equation}\label{massrel:4-5.1}
m_1m_3=m_2m_4.
\end{equation}

Mass relations (\ref{massrel:4-2.2}), (\ref{massrel:4-3.2}), (\ref{massrel:4-4.5}), (\ref{massrel:4-5.1}) 
above coincide those in \cite[\S5.1-5.4]{AK}.

\section{The case $n=5$} \label{sec:5bd}

For the case $n=5$, we obtain 20 diagrams by Algorithm~I. See Figure~\ref{diagrams of 5BD}.  
Diagrams~9,11,13,17 were excluded by Algorithm~II, and the remaining 16 diagram are precisely the 
diagrams in \cite[Figure 11]{AK}.

\begin{figure}[h]
\centering
\includegraphics[width=400pt]{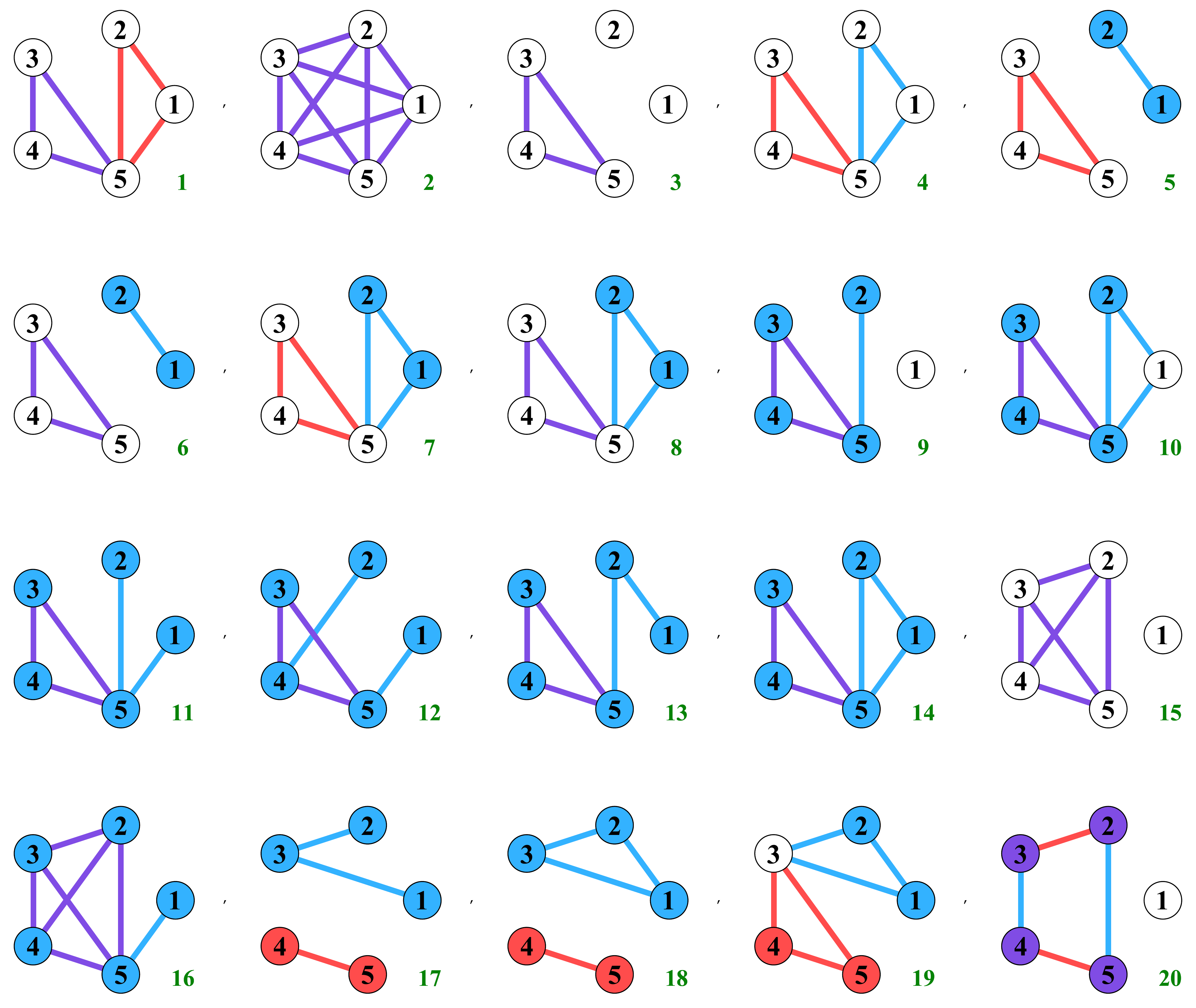}
\caption{Diagrams for singular sequences of the 5-body problem. \label{diagrams of 5BD}}
\end{figure}

\subsection{Diagrams 1, 3, 6, 8}
They are similar to the case in \S\ref{subsec:4-2}. 
For Diagram~1, one $zw$-order matrix of Type~2 is sufficient to cover orders of the $zw$-matrix:
\beq
\small{
\left(\!\! \begin{pmatrix}
\{0\sim4\} & \{4\} & \{4\} & \{4\} & \{4\} \\
\{4\} & \!\!\{0\sim4\}\!\! & \{4\} & \{4\} & \{4\} \\
\{4\} & \{4\} & \!\!\{0\sim4\}\!\! & \{3\} & \{3\} \\
\{4\} & \{4\} & \{4\} & \!\!\{0\sim4\}\!\! & \{3\} \\ 
\{4\} & \{4\} & \{3\} & \{3\} & \{0\sim4\}  
\end{pmatrix}\!,\!  
\begin{pmatrix}
\{0\sim3\} & \{2\} & \{3\} & \{3\} & \{2\} \\
\{2\} & \!\!\{0\sim3\}\!\! & \{3\} & \{3\} & \{2\} \\
\{3\} & \{3\} & \!\!\{0\sim3\}\!\! & \{3\} & \{3\} \\
\{3\} & \{3\} & \{3\} & \!\!\{0\sim3\}\!\! & \{3\} \\ 
\{2\} & \{2\} & \{3\} & \{3\} & \{0\sim3\}  
\end{pmatrix} \!\! \right)\!.
}
\eeq
Step~(II.3) of Algorithm~II results in 125 $zw$-order matrices of Type~3. 

Similarly, there is also one  $zw$-order matrix of Type~2 that covers orders of the $zw$-matrix for Diagram~8:
\beq
\small{
\left(\!\! \begin{pmatrix}
\{0\sim3\} & \{1\} & \{3\} & \{3\} & \{1\} \\
\{1\} & \!\{0\sim3\}\! & \{3\} & \{3\} & \{1\} \\
\{3\} & \{3\} & \!\{0\sim3\}\! & \{3\} & \{3\} \\
\{3\} & \{3\} & \{3\} & \!\{0\sim3\}\! & \{3\} \\ 
\{1\} & \{1\} & \{3\} & \{3\} & \{0\sim3\}  
\end{pmatrix}\!,\!  
\begin{pmatrix}
\{5\} & \{5\} & \{5\} & \{5\} & \{5\} \\
\{5\} & \{5\} & \{5\} & \{5\} & \{5\} \\
\{5\} & \{5\} & \!\{0\sim4\}\! & \{3\} & \{3\} \\
\{5\} & \{5\} & \{3\} & \!\{0\sim4\}\! & \{3\} \\ 
\{5\} & \{5\} & \{3\} & \{3\} & \{0\sim4\}  
\end{pmatrix} \!\! \right)\!.
}
\eeq
Step~(II.3) of Algorithm~II results in 24 $zw$-order matrices of Type~3. 

Diagram~6 requires three $zw$-order matrices of Type~2 to cover orders of its $zw$-matrix:
\beq
\small{
\left(\!\! \begin{pmatrix}
\{0,1\} & \{1\} & \{2\} & \{3\} & \{3\} \\
\{1\} & \{0,1\} & \{2\} & \{3\} & \{3\} \\
\{2\} & \{2\} & \{2\} & \{3\} & \{3\} \\
\{3\} & \{3\} & \{3\} & \{3\} & \{3\} \\ 
\{3\} & \{3\} & \{3\} & \{3\} & \{3\}  
\end{pmatrix}\!,\!  
\begin{pmatrix}
\{5\} & \{5\} & \{5\} & \{5\} & \{5\} \\
\{5\} & \{5\} & \{5\} & \{5\} & \{5\} \\
\{5\} & \{5\} & \{0\sim4\} & \{3\} & \{3\} \\
\{5\} & \{5\} & \{3\} & \{0\sim4\} & \{3\} \\ 
\{5\} & \{5\} & \{3\} & \{3\} & \{0\sim4\}  
\end{pmatrix} \!\! \right)\!,
}  \vspace{-3mm}
\eeq
\beq
\small{
\left(\!\! \begin{pmatrix}
\{0,1\} & \{1\} & \{4\} & \{4\} & \{4\} \\
\{1\} & \{0,1\} & \{4\} & \{4\} & \{4\} \\
\{4\} & \{4\} & \{4\} & \{3\} & \{3\} \\
\{4\} & \{4\} & \{3\} & \{4\} & \{3\} \\ 
\{4\} & \{4\} & \{3\} & \{3\} & \{4\}  
\end{pmatrix}\!,\!  
\begin{pmatrix}
\{5\} & \{5\} & \{5\} & \{5\} & \{5\} \\
\{5\} & \{5\} & \{5\} & \{5\} & \{5\} \\
\{5\} & \{5\} & \{0\sim4\} & \{3\} & \{3\} \\
\{5\} & \{5\} & \{3\} & \{0\sim4\} & \{3\} \\ 
\{5\} & \{5\} & \{3\} & \{3\} & \{0\sim4\}  
\end{pmatrix} \!\! \right)\!,
}  \vspace{-3mm}
\eeq
\beq
\small{
\left(\!\! \begin{pmatrix}
\{0,1\} & \{1\} & \{3\} & \{3\} & \{3\} \\
\{1\} & \{0,1\} & \{3\} & \{3\} & \{3\} \\
\{3\} & \{3\} & \{3\} & \{3\} & \{3\} \\
\{3\} & \{3\} & \{3\} & \{3\} & \{3\} \\ 
\{3\} & \{3\} & \{3\} & \{3\} & \{3\}  
\end{pmatrix}\!,\!  
\begin{pmatrix}
\{5\} & \{5\} & \{5\} & \{5\} & \{5\} \\
\{5\} & \{5\} & \{5\} & \{5\} & \{5\} \\
\{5\} & \{5\} & \{0\sim4\} & \{3\} & \{3\} \\
\{5\} & \{5\} & \{3\} & \{0\sim4\} & \{3\} \\ 
\{5\} & \{5\} & \{3\} & \{3\} & \{0\sim4\}  
\end{pmatrix} \!\! \right)\!.
} 
\eeq
Step~(II.3) results in 18 $zw$-order matrices of Type~3. 

Diagram~3 is more complicated. The optimal $zw$-order matrix of Type~1 which covers orders of the corresponding $zw$-matrix is:
\beq
\small{
\left(\!\! \begin{pmatrix}
\{0\sim4\}\! & \!\!\{2\sim4\}\!\! & \!\!\{2\sim4\}\!\! & \!\!\{2\sim4\}\!\! & \!\{2\sim4\} \\
\{2\sim4\}\! & \!\!\{0\sim4\}\!\! & \!\!\{2\sim4\}\!\! & \!\!\{2\sim4\}\!\! & \!\{2\sim4\} \\
\{2\sim4\}\! & \!\!\{2\sim4\}\!\! & \!\!\{0\sim4\}\!\! & \{3\} & \{3\} \\
\{2\sim4\}\! & \!\!\{2\sim4\}\!\! & \{3\} & \!\!\{0\sim4\}\!\! & \{3\} \\ 
\{2\sim4\}\! & \!\!\{2\sim4\}\!\! & \{3\} & \{3\} & \!\{0\sim4\}
\end{pmatrix}\!,\!  
\begin{pmatrix}
\{0\sim4\}\! & \!\!\{2\sim4\}\!\! & \!\!\{2\sim4\}\!\! & \!\!\{2\sim4\}\!\! & \!\{2\sim4\} \\
\{2\sim4\}\! & \!\!\{0\sim4\}\!\! & \!\!\{2\sim4\}\!\! & \!\!\{2\sim4\}\!\! & \!\{2\sim4\} \\
\{2\sim4\}\! & \!\!\{2\sim4\}\!\! & \!\!\{0\sim4\}\!\! & \{3\} & \{3\} \\
\{2\sim4\}\! & \!\!\{2\sim4\}\!\! & \{3\} & \!\!\{0\sim4\}\!\! & \{3\} \\ 
\{2\sim4\}\! & \!\!\{2\sim4\}\!\! & \{3\} & \{3\} & \!\{0\sim4\}
\end{pmatrix} \!\! \right)\!.
}
\eeq
It contains more than $4.78\times 10^6$ $zw$-order matrices of Type~2, more than 
$4.67\times 10^{13}$  $zw$-order matrices of Type~3. 
Step~(II.2) collects 32 $zw$-order matrices of Type~2 to cover orders of its $zw$-matrix.  

By applying the third subroutine in \S\ref{subsec:subroutines} with all $z_i$-, $w_i$-equations, 
diagrams~1,~3,~6,~8 all result in the same mass relation $f_5=0$, where $f_5$ is
the same as (\ref{massrel:4-2.1}) except subindices are shifted by 1. 
In particular, we also have (\ref{massrel:4-2.2}) except subindices are shifted by 1:
\begin{equation*}\label{massrel:5-1.1}
\frac{1}{\sqrt{m_i}}=\frac{1}{\sqrt{m_j}}+\frac{1}{\sqrt{m_k}} \qquad \{i,j,k\}=\{3,4,5\}.
\end{equation*}

We may also collect partial list of main equations, for instance only the $z_3$-, $z_4$-, $w_3$-, $w_4$-equations 
determined by the $zw$-matrix. 
Together with relations between separations, by making substitutions according to the cluster equations determined by 
some $zw$-order matrices, and applying eliminating process, we end up with a more restrictive mass relation: 
$$(m_3^{1/4} - m_5^{1/4})f_5 \;=\; 0. $$

\subsection{Diagrams 4, 7}\label{subsec:5-4}
In either diagram, there is a triangle formed by $z$-edges between bodies $3,4,5$, 
and for $i\in\{3,4\}$ and $j\in\{1,2\}$, there is no edge joining body $i$ and body $j$.

For diagram 4, there is one $zw$-order matrix of Type~2 that covers orders of the $zw$-matrix:
\beq
\small{
\left(\!\! \begin{pmatrix}
\{0\sim4\}\! & \!\!\{2\}\!\! & \!\!\{4\}\!\! & \!\!\{4\}\!\! & \!\{2\} \\
\{2\}\! & \!\!\{0\sim4\}\!\! & \!\!\{4\}\!\! & \!\!\{4\}\!\! & \!\{2\} \\
\{4\}\! & \!\!\{4\}\!\! & \!\!\{0\sim4\}\!\! & \{4\} & \{4\} \\
\{4\}\! & \!\!\{4\}\!\! & \{4\} & \!\!\{0\sim4\}\!\! & \{4\} \\ 
\{2\}\! & \!\!\{2\}\!\! & \{4\} & \{4\} & \!\{0\sim4\}
\end{pmatrix}\!,\!  
\begin{pmatrix}
\{0\sim4\}\! & \!\!\{4\}\!\! & \!\!\{4\}\!\! & \!\!\{4\}\!\! & \!\{4\} \\
\{4\}\! & \!\!\{0\sim4\}\!\! & \!\!\{4\}\!\! & \!\!\{4\}\!\! & \!\{4\} \\
\{4\}\! & \!\!\{4\}\!\! & \!\!\{0\sim4\}\!\! & \{2\} & \{2\} \\
\{4\}\! & \!\!\{4\}\!\! & \{2\} & \!\!\{0\sim4\}\!\! & \{2\} \\ 
\{4\}\! & \!\!\{4\}\!\! & \{2\} & \{2\} & \!\{0\sim4\}
\end{pmatrix} \!\! \right)\!.
}
\eeq
This $zw$-order matrix generates the $r$-order matrix
\beq
\left(\!\!
\begin{array}{ccccc}
 \{\} & \{2\} & \{2\sim 4\} & \{2\sim 4\} & \{2\} \\
 \{2\} & \{\} & \{2\sim 4\} & \{2\sim 4\} & \{2\} \\
 \{2\sim 4\} & \{2\sim 4\} & \{\} & \{2\} & \{2\} \\
 \{2\sim 4\} & \{2\sim 4\} & \{2\} & \{\} & \{2\} \\
 \{2\} & \{2\} & \{2\} & \{2\} & \{\} 
\end{array}
\!\! \right). 
\eeq

For diagram 7, there is one $zw$-order matrix of Type~2 that covers orders of the $zw$-matrix:
\beq
\small{
\left(\!\! \begin{pmatrix}
\{0\sim4\}\! & \!\!\{1\}\!\! & \!\!\{4\}\!\! & \!\!\{4\}\!\! & \!\{1\} \\
\{1\}\! & \!\!\{0\sim4\}\!\! & \!\!\{4\}\!\! & \!\!\{4\}\!\! & \!\{1\} \\
\{4\}\! & \!\!\{4\}\!\! & \!\!\{0\sim4\}\!\! & \{4\} & \{4\} \\
\{4\}\! & \!\!\{4\}\!\! & \{4\} & \!\!\{0\sim4\}\!\! & \{4\} \\ 
\{1\}\! & \!\!\{1\}\!\! & \{4\} & \{4\} & \!\{0\sim4\}
\end{pmatrix}\!,\!  
\begin{pmatrix}
\{5\}\! & \!\!\{5\}\!\! & \!\!\{5\}\!\! & \!\!\{5\}\!\! & \!\{5\} \\
\{5\}\! & \!\!\{5\}\!\! & \!\!\{5\}\!\! & \!\!\{5\}\!\! & \!\{5\} \\
\{5\}\! & \!\!\{5\}\!\! & \!\!\{0\sim4\}\!\! & \{2\} & \{2\} \\
\{5\}\! & \!\!\{5\}\!\! & \{2\} & \!\!\{0\sim4\}\!\! & \{2\} \\ 
\{5\}\! & \!\!\{5\}\!\! & \{2\} & \{2\} & \!\{0\sim4\}
\end{pmatrix} \!\! \right)\!.
}
\eeq
It produces the $r$-order matrix
\beq
\left(\!\!
\begin{array}{ccccc}
 \{\} & \{3\} & \{4\} & \{4\} & \{3\} \\
 \{3\} & \{\} & \{4\} & \{4\} & \{3\} \\
 \{4\} & \{4\} & \{\} & \{2\} & \{2\} \\
 \{4\} & \{4\} & \{2\} & \{\} & \{2\} \\
 \{3\} & \{3\} & \{2\} & \{2\} & \{\} 
\end{array}
\!\! \right). 
\eeq

For these two diagrams, with any $zw$-order matrix, if we only collect equations that are automatically generated in our algorithm, 
then we don't obtain any mass relation because too many maximal order sets are not certainly maximal. 
If we analyze leading order terms in some $z_i$- or $w_i$-equations, and for each possible case, add the corresponding 
$z_i$- or $w_i$-equations, then we may obtain mass relations.

A notation is introduced for convenience:
\begin{notation}
For $i\in\{1,\ldots,n\}$, let $M_{Z,i}^*$ denote the bodies of maximal order terms in equation $z_i=\sum_{j\neq i}m_jZ_{ji}$ 
and $M_{W,i}^*$ denote the bodies of maximal order terms in equation $w_i=\sum_{j\neq i}m_jW_{ji}$.
\end{notation}

It's obvious that $M_{Z,i}^*\subset M_{Z,i}$, $M_{W,i}^*\subset M_{W,i}$, $|M_{Z,i}^*|\geq 2$, $|M_{W,i}^*|\geq 2$.

For diagram 4 with the only $zw$-order matrix of Type 2, the maximal order sets and whether they are certainly maximal are 
shown in the following table:
\begin{center}
\begin{tabular}{c|c|c|c|c}
 & $M_{Z,i}$ & certainly max? & $M_{W,i}$ & certainly max?  \\
\hline
$i=1$ & $\{1,2,5\}$ & & $\{2,5\}$ & $\surd$ \\
\hline
$i=2$ & $\{1,2,5\}$ & & $\{1,5\}$ & $\surd$ \\
\hline
$i=3$ & $\{4,5\}$ & $\surd$ & $\{3,4,5\}$ &  \\
\hline
$i=4$ & $\{3,5\}$ & $\surd$ & $\{3,4,5\}$ & \\
\hline
$i=5$ & $\{3,4\}$ & $\surd$ & $\{1,2\}$ & $\surd$ \\
\end{tabular}
\end{center}

Since there is a triangle formed by $z$-edges between bodies 3,4,5, from Rule~2c  
$z_{34}\approx z_{35}\approx z_{45}$ and $w_{34}\approx w_{35}\approx w_{45}$, and then 
$Z_{34}\approx Z_{35}\approx Z_{45}$, $W_{34}\approx W_{35}\approx W_{45}$. 
Therefore there are four cases for $M_{W,3}^*$ and $M_{W,4}^*$:
\begin{enumerate}
\item[1.] $M_{W,3}^*=\{4,5\}$ and $M_{W,4}^*=\{3,5\}$,
\item[2.] $M_{W,3}^*=\{3,4,5\}$ and $M_{W,4}^*=\{3,5\}$,
\item[3.] $M_{W,3}^*=\{4,5\}$ and $M_{W,4}^*=\{3,4,5\}$,
\item[4.] $M_{W,3}^*=\{3,4,5\}$ and $M_{W,4}^*=\{3,4,5\}$.
\end{enumerate}

We first show that cases 2 and 3 are impossible. 
If $w_3\approx W_{43}=r_{43}^{-3}w_{43}$, then $w_3\succ w_{43}$ since $r_{43}\prec 1$. 
It follows that $w_4\approx w_3\approx W_{43}$ which means $M_{W,4}^*=\{3,4,5\}$, and therefore 
case 2 is impossible. Similarly, case 3 is impossible.

In case 4, since $r_{34}\prec 1$, we have $w_3\approx W_{34} =r_{34}^{-3} w_{34}\succ w_{34}$, and it follows that $w_3\sim w_4$. Therefore we also have equation $w_3=w_4$.

For case 1, we collect $z_3$- and $z_4$-equations, relations between separations, and additional equations
\beq
0&=&m_4W_{43}+m_5W_{53},\\
0&=&m_3W_{34}+m_5W_{54}.
\eeq
Then we make substitutions according to the cluster equations, and apply eliminating process. 
The result is in a mass relation $f_{5,4,1,1}=0$, where $f_{5,4,1,1}$ is the same as (\ref{massrel:4-2.1}) except subindices are shifted by 1.

For case 4, we collect $z_3$- and $z_4$-equations, relations between separations, and additional equations
\beq
w_3&=&m_4W_{43}+m_5W_{53},\\
w_4&=&m_3W_{34}+m_5W_{54},\\
w_3&=&w_4.
\eeq
Then we make substitutions according to the cluster equations, and apply eliminating process. 
The result is in a mass relation $f_{5,4,1,2}=0$, where $f_{5,4,1,2}$ is the same as $f_{4,3}$  in \S\ref{subsec:4-3}, 
except subindices are shifted by 1. 

On the other hand, by similar discussions, we also have two possibilities for $M_{Z,1}^*$ and $M_{Z,2}^*$:
\begin{enumerate}
\item[1.] $M_{Z,1}^*=\{2,5\}$ and $M_{Z,2}^*=\{1,5\}$,
\item[2.] $M_{Z,1}^*=\{1,2,5\}$ and $M_{Z,2}^*=\{1,2,5\}$.
\end{enumerate}

In case 1, we obtain a mass relation $f_{5,4,2,1}=0$ where $f_{5,4,2,1}$ is the same as (\ref{massrel:4-2.1}) 
except subindices change from $2,3,4$ to $1,2,5$ respectively. 
In case 2, we obtain a mass relation $f_{5,4,2,2}=0$ where $f_{5,4,2,2}$ is the same as $f_{4,3}$  in \S\ref{subsec:4-3}, 
except subindices change from $2,3,4$ to $1,2,5$ respectively.

Therefore for diagram 4, we have a mass relation $f_{5,4,1}=f_{5,4,2}=0$ where $f_{5,4,i}$ is the least 
common multiple of $f_{5,4,i,1}$ and $f_{5,4,i,2}$ for $i=1,2$.

For diagram 7 with the only $zw$-order matrix of Type 2, the maximal order sets and whether they are certainly 
maximal are shown in the following table:
\begin{center}
\begin{tabular}{c|c|c|c|c}
 & $M_{Z,i}$ & certainly max? & $M_{W,i}$ & certainly max?  \\
\hline
$i=1$ & $\{1,2,5\}$ & & $\{1,2,5\}$ & $\surd$ \\
\hline
$i=2$ & $\{1,2,5\}$ & & $\{1,2,5\}$ & $\surd$ \\
\hline
$i=3$ & $\{4,5\}$ & $\surd$ & $\{3,4,5\}$ &  \\
\hline
$i=4$ & $\{3,5\}$ & $\surd$ & $\{3,4,5\}$ & \\
\hline
$i=5$ & $\{3,4\}$ & $\surd$ & $\{1,2\}$ & $\surd$ \\
\end{tabular}
\end{center}

By the same discussions as in diagram 4, there are two cases for $M_{W,3}^*$ and $M_{W,4}^*$, 
and we obtain mass relations $f_{5,4,1,1}=0$ and $f_{5,4,1,2}=0$ in these cases respectively. 
Therefore, for diagram 7, we have the same mass relation as that for diagram 4.

\subsection{Diagram 5}\label{subsec:5-5}
For diagram 5, there are three $zw$-order matrices $(S_k,T)$, $k=2,3,4$ of Type~2 that cover orders of the $zw$-matrix, where
$$
S_k= \begin{pmatrix}
\{0\sim 1\}\! & \!\!\{1\}\!\! & \!\!\{k\}\!\! & \!\!\{4\}\!\! & \!\{4\} \\
\{1\}\! & \!\!\{0\sim1\}\!\! & \!\!\{k\}\!\! & \!\!\{4\}\!\! & \!\{4\} \\
\{k\}\! & \!\!\{k\}\!\! & \!\!\{k\}\!\! & \{4\} & \{4\} \\
\{4\}\! & \!\!\{4\}\!\! & \{4\} & \!\!\{4\}\!\! & \{4\} \\ 
\{4\}\! & \!\!\{4\}\!\! & \{4\} & \{4\} & \!\{4\}
\end{pmatrix},
$$
and
$$
T=  
\begin{pmatrix}
\{5\}\! & \!\!\{5\}\!\! & \!\!\{5\}\!\! & \!\!\{5\}\!\! & \!\{5\} \\
\{5\}\! & \!\!\{5\}\!\! & \!\!\{5\}\!\! & \!\!\{5\}\!\! & \!\{5\} \\
\{5\}\! & \!\!\{5\}\!\! & \!\!\{0\sim4\}\!\! & \{2\} & \{2\} \\
\{5\}\! & \!\!\{5\}\!\! & \{2\} & \!\!\{0\sim4\}\!\! & \{2\} \\ 
\{5\}\! & \!\!\{5\}\!\! & \{2\} & \{2\} & \!\{0\sim4\}
\end{pmatrix}.
$$
They produce the same $r$-order matrix
\beq
\left(\!\!
\begin{array}{ccccc}
 \{\} & \{3\} & \{4\} & \{4\} & \{3\} \\
 \{3\} & \{\} & \{4\} & \{4\} & \{3\} \\
 \{4\} & \{4\} & \{\} & \{2\} & \{2\} \\
 \{4\} & \{4\} & \{2\} & \{\} & \{2\} \\
 \{3\} & \{3\} & \{2\} & \{2\} & \{\} 
\end{array}
\!\! \right). 
\eeq

For $(S_2,T)$ and $(S_3,T)$, the maximal order sets and whether they are certainly maximal are shown in the following table:
\begin{center}
\begin{tabular}{c|c|c|c|c}
 & $M_{Z,i}$ & certainly max? & $M_{W,i}$ & certainly max?  \\
\hline
$i=1$ & $\{1,2\}$ & $\surd$ & $\{1,2\}$ & $\surd$ \\
\hline
$i=2$ & $\{1,2\}$ & $\surd$ & $\{1,2\}$ & $\surd$ \\
\hline
$i=3$ & $\{4,5\}$ & $\surd$ & $\{1,2,3,4,5\}$ &  \\
\hline
$i=4$ & $\{3,5\}$ & $\surd$ & $\{3,4,5\}$ & \\
\hline
$i=5$ & $\{3,4\}$ & $\surd$ & $\{3,4,5\}$ &  \\
\end{tabular}
\end{center}

Same as discussions for diagram 4, there are two cases for $M_{W,4}^*$ and $M_{W,5}^*$,  
we obtain mass relations $f_{5,5,1}=0$ and $f_{5,5,2}=0$ in these cases respectively, 
where $f_{5,5,i}$ is the same as $f_{5,4,1,i}$ in \S\ref{subsec:5-4}, 
except subindices were changed from $3,4,5$ to $4,5,3$ respectively for $i=1,2$. 

For $(S_4,T)$, the maximal order sets and whether they are certainly maximal are shown in the following table:
\begin{center}
\begin{tabular}{c|c|c|c|c}
 & $M_{Z,i}$ & certainly max? & $M_{W,i}$ & certainly max?  \\
\hline
$i=1$ & $\{1,2\}$ & $\surd$ & $\{1,2\}$ & $\surd$ \\
\hline
$i=2$ & $\{1,2\}$ & $\surd$ & $\{1,2\}$ & $\surd$ \\
\hline
$i=3$ & $\{4,5\}$ & $\surd$ & $\{1,2,3,4,5\}$ &  \\
\hline
$i=4$ & $\{3,5\}$ & $\surd$ & $\{1,2,3,4,5\}$ & \\
\hline
$i=5$ & $\{3,4\}$ & $\surd$ & $\{1,2,3,4,5\}$ &  \\
\end{tabular}
\end{center}

We first show that there are distinct $i,j\in\{3,4,5\}$ such that $M_{W,i}^*, M_{W,j}^* \subset\{3,4,5\}$. 
Then it's similar to diagram 4, we will obtain mass relation $f_{5,5,3}=0$ where $f_{5,5,3,k}$ is the same as 
$f_{5,4,1}$ in \S\ref{subsec:5-4}, except subindices were changed from $3,4,5$ to $i,j,k$ respectively where $\{i,j,k\}=\{3,4,5\}$.

For $i\in\{3,4,5\}$, suppose there is some $p\in\{1,2\}\cap M_{W,i}^*$. 
From the $w$-order matrix, by triangle inequality, we have $w_{ip}\approx w_{jp}$ for $j=3,4,5$. 
Since there is a triangle between bodies $3,4,5$ formed by $z$-strokes, we have $j,k\in M_{W,i}^*$ where $\{i,j,k\}=\{3,4,5\}$. 
From the relation $Z_{kl}=r_{kl}^{-3}z_{kl}$ we obtain $z_{pi}\prec z_{ij}\approx z_{34}$ for $j\in\{3,4,5\}$ different from $i$, 
and it follows that $z_{pj}\approx z_{ij}\approx z_{34}$. Therefore $Z_{pj}\prec Z_{ij}\approx Z_{34}$, 
which means $p\notin M_{W,j}^*$. On the other hand, since $z_{12}\prec z_{1k}$ for $k=3,4,5$, we have $z_{1k}\approx z_{2k}$. 
From the $w$-order matrix, we see that $w_{1k}\approx w_{2k}\approx \epsilon^{-2}$, and it follows that $Z_{1k}\approx Z_{2k}$. 
Therefore $\{1,2\}\subset M_{W,i}^*$ and $\{1,2\}\cap M_{W,j}^*=\emptyset$ for $j\in\{3,4,5\}$ different from $i$, and there will be 
distinct $j,k\in \{3,4,5\}$ such that $M_{W,j}^*,M_{W,k}^*\subset\{3,4,5\}$.

For diagram 5, we have mass relation $f_{5,5}=0$ where $f_{5,5}$ is the least common multiple of 
$f_{5,5,1}$, $f_{5,5,2}$, $f_{5,5,3,3}$, $f_{5,5,3,4}$, and $f_{5,5,3,5}$.

\subsection{Diagrams 10, 14} 
These two cases are similar to the kite diagram in \S\ref{subsec:4-3}. 
One $zw$-order matrix of Type~2  is sufficient to cover orders of the $zw$-matrix for Diagram~10:
\beq
\small{
\left(\!\! \begin{pmatrix}
\{0\sim3\} & \{1\} & \{3\} & \{3\} & \{1\} \\
\{1\} & \!\{0\sim3\}\! & \{3\} & \{3\} & \{1\} \\
\{3\} & \{3\} & \!\{0\sim3\}\! & \{3\} & \{3\} \\
\{3\} & \{3\} & \{3\} & \!\{0\sim3\}\! & \{3\} \\ 
\{1\} & \{1\} & \{3\} & \{3\} & \{0\sim3\}  
\end{pmatrix}\!,\!  
\begin{pmatrix}
\{0\sim4\} & \{5\} & \{5\} & \{5\} & \{5\} \\
\{5\} & \{5\} & \{5\} & \{5\} & \{5\} \\
\{5\} & \{5\} & \{5\} & \{3\} & \{3\} \\
\{5\} & \{5\} & \{3\} & \{5\} & \{3\} \\ 
\{5\} & \{5\} & \{3\} & \{3\} & \{5\}  
\end{pmatrix} \!\! \right)\!.
}
\eeq
Comparing with Diagram~8, the difference between their $zw$-matrices are at
the first, third, fourth, and fifth diagonal entries of the $w$-matrices. 
This fact has reflected in their differences between the covering $zw$-order matrices 
of Type~2. 

Diagram~14 also has its $zw$-diagram covered by one Type~2 $zw$-order matrix: 
\beq
\small{
\left(\!\! \begin{pmatrix}
\{0\sim3\} & \{1\} & \{3\} & \{3\} & \{1\} \\
\{1\} & \!\{0\sim3\}\! & \{3\} & \{3\} & \{1\} \\
\{3\} & \{3\} & \!\{0\sim3\}\! & \{3\} & \{3\} \\
\{3\} & \{3\} & \{3\} & \!\{0\sim3\}\! & \{3\} \\ 
\{1\} & \{1\} & \{3\} & \{3\} & \{0\sim3\}  
\end{pmatrix}\!,\!  
\begin{pmatrix}
\{5\} & \{5\} & \{5\} & \{5\} & \{5\} \\
\{5\} & \{5\} & \{5\} & \{5\} & \{5\} \\
\{5\} & \{5\} & \{5\} & \{3\} & \{3\} \\
\{5\} & \{5\} & \{3\} & \{5\} & \{3\} \\ 
\{5\} & \{5\} & \{3\} & \{3\} & \{5\}  
\end{pmatrix} \!\! \right)\!.
}
\eeq

By collecting the $z_3$-, $z_4$-, $w_3$-, $w_4$-equations determined by the $zw$-matrix, and relations between separations, 
then make substitutions according to the cluster equations determined by some $zw$-order matrices, 
the eliminating process (the third subroutine in \S\ref{subsec:subroutines}) yield mass relation
$f_5=0$, where $f_5$ is the same as the $f_{4,3}$  in \S\ref{subsec:4-3}, except subindices are shifted by 1. 


\subsection{Diagram 12}  \label{subsec:5-12}
There are 2 $zw$-order matrices of Type 2:
\beq
&\small{
\left(\!
\begin{pmatrix}
\{0,1\} & \{3\} & \{3\} & \{3\} & \{1\} \\
\{3\} & \{3\} & \{3\} & \{2\} & \{3\} \\
\{3\} & \{3\} & \{3\} & \{3\} & \{3\} \\
\{3\} & \{2\} & \{3\} & \{3\} & \{3\} \\
\{1\} & \{3\} & \{3\} & \{3\} & \{0,1\} 
\end{pmatrix}, 
\begin{pmatrix}
\{5\} & \{5\} & \{5\} & \{5\} & \{5\} \\
\{5\} & \{5\} & \{4\} & \{4\} & \{4\} \\
\{5\} & \{4\} & \{5\} & \{3\} & \{3\} \\
\{5\} & \{4\} & \{3\} & \{5\} & \{3\} \\
\{5\} & \{4\} & \{3\} & \{3\} & \{5\} 
\end{pmatrix} \! \right), } \\
&\small{
\left( \! \begin{pmatrix}
\{0\sim3\} & \{3\} & \{3\} & \{3\} & \{1\} \\
\{3\} & \!\{0\sim3\}\! & \{3\} & \{1\} & \{3\} \\
\{3\} & \{3\} & \!\{0\sim3\}\! & \{3\} & \{3\} \\
\{3\} & \{1\} & \{3\} & \!\{0\sim3\}\! & \{3\} \\
\{1\} & \{3\} & \{3\} & \{3\} & \{0\sim3\} 
\end{pmatrix}, 
\begin{pmatrix}
\{5\} & \{4\} & \{5\} & \{5\} & \{5\} \\
\{4\} & \{5\} & \{5\} & \{5\} & \{5\} \\
\{5\} & \{5\} & \{5\} & \{3\} & \{3\} \\
\{5\} & \{5\} & \{3\} & \{5\} & \{3\} \\
\{5\} & \{5\} & \{3\} & \{3\} & \{5\} 
\end{pmatrix} \! \right).
}
\eeq
For each one, we collect all main equations determined by the $zw$-order matrices and relations between separations. 
Then we make substitutions according to the cluster equations and apply the eliminating process,  
and eventually obtain two polynomial equations. 
Altogether, their factors include 
\begin{flalign*}\label{massrel:5-12}
\notag &\sqrt{m_3}\sqrt{m_4}-\sqrt{m_3}\sqrt{m_5}-\sqrt{m_4}\sqrt{m_5},\\
\notag &\sqrt{m_3}\sqrt{m_4}+\sqrt{m_3}\sqrt{m_5}-\sqrt{m_4}\sqrt{m_5},\\
\notag &\sqrt{m_3}\sqrt{m_4}-\sqrt{m_3}\sqrt{m_5}+\sqrt{m_4}\sqrt{m_5},\\
\notag &m_3m_4+m_3m_5+m_4m_5-2m_3\sqrt{m_4}\sqrt{m_5},\\
\notag &m_3m_4+m_3m_5+m_4m_5-2\sqrt{m_3}m_4\sqrt{m_5},\\
\notag &m_3m_4+m_3m_5+m_4m_5-2\sqrt{m_3}\sqrt{m_4}m_5,\\
           &m_3^3-(m_2+m_3+m_4)m_5^2, \\
\notag & m_1 m_3^3+m_2 m_3^3-m_1 m_3 m_5^2-m_1 m_4 m_5^2+m_2 m_5^3. 
\end{flalign*}
Six other longer factors are  (see Appendix~IV for details)
\begin{align*}
&\mu_{45}(m_2,m_3,m_4,m_5),\; \mu_{84}(m_1,m_2,m_3,m_4,m_5), \; \mu_{88}(m_2,m_3,m_4,m_5),\;  \\
&\mu_{111}(m_2,m_3,m_4,m_5),\; \mu_{221}(m_1,m_2,m_3,m_4,m_5), \; \mu_{374}(m_1,m_2,m_3,m_4,m_5). 
\end{align*}

\subsection{Diagram 19} 
This case is similar to the disconnected diagram in \S\ref{subsec:4-4}. 
The $zw$-diagram is covered by 9 $zw$-order matrices of Type 3: 
\beq
\small{
\left(\! \left(\!\!
\begin{array}{ccccc}
\{0\}  & \{1\} & \{1\} & \{5\} & \{5\} \\
\{1\} & \{1\} & \{1\} & \{5\} & \{5\} \\
\{1\} & \{1\} & \{1\} & \{5\} & \{5\} \\
\{5\} &\{5\} & \{5\} & \{5\} & \{5\} \\
\{5\} &\{5\} & \{5\} & \{5\} & \{5\} \\
\end{array}
\!\!\right), 
 \left(\!\!
\begin{array}{ccccc}
\{5\} &\{5\} & \{5\} & \{5\} & \{5\} \\
\{5\} &\{5\} & \{5\} & \{5\} & \{5\} \\
\{5\} &\{5\} & \{0\} & \{1\} & \{1\} \\
\{5\} &\{5\} & \{1\} & \{1\} & \{1\} \\
\{5\} &\{5\} & \{1\} & \{1\} & \{1\} 
\end{array}
\!\!\right) \! \right),  } 
\eeq \vspace{-3mm}
\beq
\small{
\left(\! \left(\!\!
\begin{array}{ccccc}
\{0\}  & \{1\} & \{1\} & \{5\} & \{5\} \\
\{1\} & \{1\} & \{1\} & \{5\} & \{5\} \\
\{1\} & \{1\} & \{1\} & \{5\} & \{5\} \\
\{5\} &\{5\} & \{5\} & \{5\} & \{5\} \\
\{5\} &\{5\} & \{5\} & \{5\} & \{5\} \\
\end{array}
\!\!\right),\left( \!\!
\begin{array}{ccccc}
\{5\} &\{5\} & \{5\} & \{5\} & \{5\} \\
\{5\} &\{5\} & \{5\} & \{5\} & \{5\} \\
\{5\} &\{5\} & \{1\} & \{1\} & \{1\} \\
\{5\} &\{5\} & \{1\} & \{0\} & \{1\} \\
\{5\} &\{5\} & \{1\} & \{1\} & \{1\} 
\end{array}
\!\! \right) \! \right), } 
\eeq  \vspace{-3mm}
\beq
\small{
\left(\! \left(\!\!
\begin{array}{ccccc}
\{0\}  & \{1\} & \{1\} & \{5\} & \{5\} \\
\{1\} & \{1\} & \{1\} & \{5\} & \{5\} \\
\{1\} & \{1\} & \{1\} & \{5\} & \{5\} \\
\{5\} &\{5\} & \{5\} & \{5\} & \{5\} \\
\{5\} &\{5\} & \{5\} & \{5\} & \{5\} \\
\end{array}
\!\!\right), 
 \left(\!\!
\begin{array}{ccccc}
\{5\} &\{5\} & \{5\} & \{5\} & \{5\} \\
\{5\} &\{5\} & \{5\} & \{5\} & \{5\} \\
\{5\} &\{5\} & \{1\} & \{1\} & \{1\} \\
\{5\} &\{5\} & \{1\} & \{1\} & \{1\} \\
\{5\} &\{5\} & \{1\} & \{1\} & \{1\} 
\end{array}
\!\!\right) \! \right),  } 
\eeq  \vspace{-3mm}
\beq
\small{
\left( \! \left(\!\!
\begin{array}{ccccc}
\{1\}  & \{1\} & \{1\} & \{5\} & \{5\} \\
\{1\} & \{1\} & \{1\} & \{5\} & \{5\} \\
\{1\} & \{1\} & \{0\} & \{5\} & \{5\} \\
\{5\} &\{5\} & \{5\} & \{5\} & \{5\} \\
\{5\} &\{5\} & \{5\} & \{5\} & \{5\} \\
\end{array}
\!\!\right), 
 \left(\!\!
\begin{array}{ccccc}
\{5\} &\{5\} & \{5\} & \{5\} & \{5\} \\
\{5\} &\{5\} & \{5\} & \{5\} & \{5\} \\
\{5\} &\{5\} & \{0\} & \{1\} & \{1\} \\
\{5\} &\{5\} & \{1\} & \{1\} & \{1\} \\
\{5\} &\{5\} & \{1\} & \{1\} & \{1\} 
\end{array}
\!\!\right) \! \right),  } 
\eeq  \vspace{-3mm}
\beq
\small{
\left( \! \left(\!\!
\begin{array}{ccccc}
\{1\}  & \{1\} & \{1\} & \{5\} & \{5\} \\
\{1\} & \{1\} & \{1\} & \{5\} & \{5\} \\
\{1\} & \{1\} & \{1\} & \{5\} & \{5\} \\
\{5\} &\{5\} & \{5\} & \{5\} & \{5\} \\
\{5\} &\{5\} & \{5\} & \{5\} & \{5\} \\
\end{array}
\!\!\right), 
 \left(\!\!
\begin{array}{ccccc}
\{5\} &\{5\} & \{5\} & \{5\} & \{5\} \\
\{5\} &\{5\} & \{5\} & \{5\} & \{5\} \\
\{5\} &\{5\} & \{0\} & \{1\} & \{1\} \\
\{5\} &\{5\} & \{1\} & \{1\} & \{1\} \\
\{5\} &\{5\} & \{1\} & \{1\} & \{1\} 
\end{array}
\!\!\right) \! \right),  }  
\eeq  \vspace{-3mm}
\beq
\small{
\left( \! \left(\!\!
\begin{array}{ccccc}
\{1\}  & \{1\} & \{1\} & \{5\} & \{5\} \\
\{1\} & \{1\} & \{1\} & \{5\} & \{5\} \\
\{1\} & \{1\} & \{0\} & \{5\} & \{5\} \\
\{5\} &\{5\} & \{5\} & \{5\} & \{5\} \\
\{5\} &\{5\} & \{5\} & \{5\} & \{5\} \\
\end{array}
\!\!\right), 
 \left(\!\!
\begin{array}{ccccc}
\{5\} &\{5\} & \{5\} & \{5\} & \{5\} \\
\{5\} &\{5\} & \{5\} & \{5\} & \{5\} \\
\{5\} &\{5\} & \{1\} & \{1\} & \{1\} \\
\{5\} &\{5\} & \{1\} & \{0\} & \{1\} \\
\{5\} &\{5\} & \{1\} & \{1\} & \{1\} 
\end{array}
\!\!\right) \! \right),   } 
\eeq  \vspace{-3mm}
\beq
\small{
\left( \! \left(\!\!
\begin{array}{ccccc}
\{1\}  & \{1\} & \{1\} & \{5\} & \{5\} \\
\{1\} & \{1\} & \{1\} & \{5\} & \{5\} \\
\{1\} & \{1\} & \{1\} & \{5\} & \{5\} \\
\{5\} &\{5\} & \{5\} & \{5\} & \{5\} \\
\{5\} &\{5\} & \{5\} & \{5\} & \{5\} \\
\end{array}
\!\!\right), 
 \left(\!\!
\begin{array}{ccccc}
\{5\} &\{5\} & \{5\} & \{5\} & \{5\} \\
\{5\} &\{5\} & \{5\} & \{5\} & \{5\} \\
\{5\} &\{5\} & \{1\} & \{1\} & \{1\} \\
\{5\} &\{5\} & \{1\} & \{0\} & \{1\} \\
\{5\} &\{5\} & \{1\} & \{1\} & \{1\} 
\end{array}
\!\!\right) \! \right),   } 
\eeq  \vspace{-3mm}
\beq
\small{
\left( \! \left(\!\!
\begin{array}{ccccc}
\{1\}  & \{1\} & \{1\} & \{5\} & \{5\} \\
\{1\} & \{1\} & \{1\} & \{5\} & \{5\} \\
\{1\} & \{1\} & \{0\} & \{5\} & \{5\} \\
\{5\} &\{5\} & \{5\} & \{5\} & \{5\} \\
\{5\} &\{5\} & \{5\} & \{5\} & \{5\} \\
\end{array}
\!\!\right), 
 \left(\!\!
\begin{array}{ccccc}
\{5\} &\{5\} & \{5\} & \{5\} & \{5\} \\
\{5\} &\{5\} & \{5\} & \{5\} & \{5\} \\
\{5\} &\{5\} & \{1\} & \{1\} & \{1\} \\
\{5\} &\{5\} & \{1\} & \{1\} & \{1\} \\
\{5\} &\{5\} & \{1\} & \{1\} & \{1\} 
\end{array}
\!\!\right) \!  \right),  } 
\eeq  \vspace{-3mm}
\beq
\small{
\left( \! \left(\!\!
\begin{array}{ccccc}
\{1\}  & \{1\} & \{1\} & \{5\} & \{5\} \\
\{1\} & \{1\} & \{1\} & \{5\} & \{5\} \\
\{1\} & \{1\} & \{1\} & \{5\} & \{5\} \\
\{5\} &\{5\} & \{5\} & \{5\} & \{5\} \\
\{5\} &\{5\} & \{5\} & \{5\} & \{5\} \\
\end{array}
\!\!\right), 
 \left(\!\!
\begin{array}{ccccc}
\{5\} &\{5\} & \{5\} & \{5\} & \{5\} \\
\{5\} &\{5\} & \{5\} & \{5\} & \{5\} \\
\{5\} &\{5\} & \{1\} & \{1\} & \{1\} \\
\{5\} &\{5\} & \{1\} & \{1\} & \{1\} \\
\{5\} &\{5\} & \{1\} & \{1\} & \{1\} 
\end{array}
\!\!\right) \! \right).  }
\eeq

For each of these $zw$-order matrices, 
we collect a wedge equation and conditions of mass center, then make substitutions according to the cluster equations.  
Similar to discussions in \S\ref{subsec:4-4},  using the second subroutine in \S\ref{subsec:subroutines},  
for part of wedge equations we obtain a mass relation which never holds for positive masses:
\begin{equation*}\label{massrel:5-19.1}
(m_1^2-m_1m_2+m_2^2)(m_4^2-m_4m_5+m_5^2)=0; 
\end{equation*}
for some other wedge equations we obtain a mass relation which is equivalent to
\begin{equation*}\label{massrel:5-19.2}
m_1=m_2\quad\textrm{and}\quad m_4=m_5
\end{equation*}
for positive masses.

\subsection{Diagram 20}
This case is similar to the square diagram analyzed in \S\ref{subsec:4-5}. 
Again, we have just one Type~2 $zw$-order matrix to cover orders of the $zw$-diagram, 
\beq
\small{
\left(\! \begin{pmatrix}
\{0\sim4\}  & \{5\} & \{5\} & \{5\} & \{5\} \\
\{5\} & \{5\} & \{5\} & \{5\} & \{1\} \\
\{5\} & \{5\} & \{5\} & \{1\} & \{5\} \\
\{5\} &\{5\} & \{1\} & \{5\} & \{5\} \\
\{5\} &\{1\} & \{5\} & \{5\} & \{5\} 
\end{pmatrix},  
\begin{pmatrix}
\{0\sim4\}  & \{5\} & \{5\} & \{5\} & \{5\} \\
\{5\} & \{5\} & \{1\} & \{5\} & \{5\} \\
\{5\} & \{1\} & \{5\} & \{5\} & \{5\} \\
\{5\} & \{5\} & \{5\} & \{5\} & \{1\} \\
\{5\} & \{5\} & \{5\} & \{1\} & \{5\} 
\end{pmatrix} \! \right).
}
\eeq
The $zw$-order matrix obtained in \S\ref{subsec:4-5} consists of $4\times4$ blocks of the matrix above. 
Following exactly the same procedure as \S\ref{subsec:4-5}, we obtain the  mass relation $$m_2m_4=m_3m_5.$$ 

\subsection{Other diagrams}  \label{subsec:5-others}
For other diagrams, number 2,15,16,18, our program does not output any mass relation. 
In \cite{AK}, some of them were treated on a case-by-case basis, and remaining diagrams were ruled out by 
arguments that are similar to the one described in \S\ref{subsec:4-1}. 
A nice common feature of their remaining diagrams is that some simple dominant polynomial in distances, which they 
call {\em 4-product}, stays bounded for singular sequences approaching them. 
This ensures the existence of singular subsequence approaching one of other diagrams. 


\section{The case $n=6$} \label{sec:6bd}

Our main result for the six-body problem is

\begin{theorem}
There are at most 85 possible $zw$-diagrams for the planar six-body problem. 
Among them, at least 61 of them are impossible except for positive masses in a co-dimensional 2 variety of the mass space. 
\end{theorem}

Indeed, we obtained 117 $zw$-diagrams by Algorithm~I. 
See Figures~\ref{fig:6i},~\ref{fig:6ii},~\ref{fig:6iii},~\ref{fig:6iv}. 
Among them, there are 27 diagrams that were eliminated by Algorithm~II directly: 
\begin{quote}
NO. 21, 23, 24, 25, 26, 33, 34, 35, 37, 38, 39, 41, 42, 53, 54, 56, 68, 70, 72, 83, 85, 87, 91, 92, 96, 97, 113  
\end{quote}
and diagrams 1, 10, 59, 60 were eliminated by Algorithm~II indirectly. Another diagram, NO. 116, 
can be eliminated by Algorithm~III. See~\S\ref{subsec:6-Grobner} below.

We will demonstrate how our algorithms eliminates diagrams 1, 10, 59, 60 and generate mass relations 
for NO. 116 and 61 other diagrams, and that leaves 24 diagrams unsolved.  
Ideally, the finiteness conjecture can be tackled by the arguments in \S\ref{subsec:4-1},~\S\ref{subsec:5-others};
namely, to find a dominant polynomial that stays bounded for singular sequences approaching 
all of the 24 remaining diagrams. 
We do not know whether such a dominant polynomial exists. 
A more practical approach is probably to rule out some of them on a case-by-case basis, then find a suitable dominant polynomial afterwards. 
We have not finish this task, thereby leaving an incomplete puzzle for future investigations. 

\subsection{Frequently appeared factors}
Many mass relations observed in the previous two sections are commonplace for the case $n=6$. 
We introduce some abbreviated notations for them:
\beq
\mu_1(m_i,m_j,m_k) &=& \sqrt{m_i} \sqrt{m_j}-\sqrt{m_j} \sqrt{m_k} +\sqrt{m_k} \sqrt{m_i}, \\
\mu_2(m_i,m_j,m_k) &=& m_i m_j +m_j m_k+m_k m_i - 2 m_i \sqrt{m_j} \sqrt{m_k}, \\
\mu_3(m_i,m_j,m_k) &=& \mu_1(m_i,m_j,m_k)\,\mu_1(m_j,m_k,m_i)\,\mu_1(m_k,m_i,m_j), \\
\mu_4(m_i,m_j,m_k) &=& \mu_2(m_i,m_j,m_k)\,\mu_2(m_j,m_k,m_i)\,\mu_2(m_k,m_i,m_j), \\
\mu_5(m_i,m_j,m_k) &=& \mu_3(m_i,m_j,m_k)\,\mu_4(m_j,m_k,m_i). 
\eeq
Although factors $\mu_3$ and $\mu_4$ appear concurrently after the elimination process, 
we separate them because $\mu_3=0$ imposes real constraint for positive masses, 
while $\mu_4>0$ for any selection of positive masses (as observed in \S\ref{subsec:4-2}). 

Another factor we saw in \S\ref{subsec:4-3} has the form
\beq
-m_i^{3/4} m_j^{3/4}+m_i \sqrt{m_k}+ m_j \sqrt{m_k}.
\eeq
It can be replaced by a new factor:
\beq
  \mu_6(m_i,m_j,m_k) &=& (m_i m_j)^3- m_k^2 (m_i+m_j)^4. 
\eeq

Several longer factors also appear frequently in mass relations. 
They are abbreviated as $\mu_{15}$, $\mu_{21}$, $\mu_{33}$, $\mu_{45}$, $\mu_{84}$, $\mu_{88}$,  
$\mu_{111}$, $\mu_{221}$, $\mu_{374}$,
where subscripts indicate the number of terms in the factor. 
See Appendix~IV for details.

\subsection{Diagrams 77, 99$\sim$112, 114$\sim$117} \label{subsec:6-Grobner}
For these 19 diagrams, we collect main equations determined by the $zw$-matrix,  
make substitutions according to the cluster equations determined by some $zw$-order matrices, and then  
and use the function ``GroebnerBasis'' in Mathematica directly. 
They all result in simple mass relations.

For example, in diagram 107, through Algorithm~II  
we obtain an optimal  $zw$-order matrix of Type~1 to cover orders of positions and separations:
$$
\small{
\left(\! \begin{pmatrix}
\{5\} & \{5\} & \{3\} & \{5\} & \{3\} & \{2,3\} \\
\{5\} & \{5\} & \{5\} & \{1\} & \{5\} & \{5\} \\
\{3\} & \{5\} & \{5\} & \{5\} & \{1\} & \{3\} \\
\{5\} & \{1\} & \{5\} & \{5\} & \{5\} & \{5\} \\
\{3\} & \{5\} & \{1\} & \{5\} & \{5\} & \{3\} \\
\{2,3\} & \{5\} & \{3\} & \{5\} & \{3\} & \{5\} 
\end{pmatrix}, 
\begin{pmatrix}
\{5\} & \{3\} & \{5\} & \{5\} & \{3\} & \{5\} \\
\{3\} & \{5\} & \{5\} & \{5\} & \{1\} & \{5\} \\
\{5\} & \{5\} & \{5\} & \{3\} & \{5\} & \{3\} \\
\{5\} & \{5\} & \{3\} & \{5\} & \{5\} & \{1\} \\
\{3\} & \{1\} & \{5\} & \{5\} & \{5\} & \{5\} \\
\{5\} & \{5\} & \{3\} & \{1\} & \{5\} & \{5\} 
\end{pmatrix} \!\right).
}
$$
The first subroutine of \S\ref{subsec:subroutines} automatically carries out the following steps. 
From those lower order separations, we determine cluster equations 
$$ z_1=z_3=z_5=z_6, \; z_2=z_4, \; w_1=w_2=w_5, \; w_3=w_4=w_6. $$
According to these cluster equations, we substitute $z_1$ for $z_3$, $z_5$, and $z_6$, substitute $z_2$ for $z_4$, 
substitute $w_1$ for $w_2$ and $w_5$, substitute $w_3$ for $w_4$ and $w_6$. 
Then the function ``GroebnerBasis'' in Mathematica was used to obtain the mass relation 
$$m_4(m_1+m_5) \;= \; m_2(m_3+m_6).$$
One can implement Algorithm~II and III in Mathematica as follows (see Appendices~II,~III).\vspace{-2mm}   \\

\noindent
{\tt
In:=  FindOrder[111]; FindMassRel1[\,]  
} \\

There are a few exceptional cases. For diagrams~77,~102,~109,~117, we may follow other procedures in 
\S\ref{subsec:genmassrelation} to generate less restrictive relations. 
For example, applying our subroutine with Gr\"obner basis to diagram 81 results in the mass relation
$$
   (m_2 m_4-m_3 m_5) (m_1 m_5-m_2 m_6) (m_1 m_4-m_3 m_6)  \;=\; 0.
$$
If one pick $z_1$-, $z_2$-, $w_1$-, $w_2$-, $w_3$-, $w_4$-, $w_5$-equations determined by the optimal Type~1 $zw$-order matrix
and apply the third subroutine in \S\ref{subsec:subroutines}, then we find the less restrictive mass relation
$$
   m_2 m_4-m_3 m_5  \;=\; 0.
$$
The collection $\coll_2$ is denoted {\tt zwotype2} in our Mathematica program. 
The command line for this elimination process is \vspace{-2mm} \\

{\tt
\noindent
In:=  FindOrder[77];  \\
\indent
\quad\ Do[FindMassRel3[zwotype2[\![k]\!], \!\{1, \!2\}, \!\{1, \!2,\! 3, \!4, \!5\}], \\
\indent
\quad\  \qquad\qquad\qquad\ \{k,1, Length@zwotype2\}]  
} \\

See table~\ref{table:simplemassrel6} for main equations selected for other exceptional cases.

Another exceptional case is diagram~116. This is a special case where the subroutine with Gr\"obner basis results in mass relation 
$$
  m_1m_4m_5+m_2m_3m_6\;=\;0.
$$
Since there is no positive mass satisfying this relation, we can eliminate this diagram. 
This can be also done by the third subroutine in  \S\ref{subsec:subroutines} by picking 
$z_1$-, $z_2$-, $w_1$-, $w_2$-, $w_3$-, $w_4$-, $w_5$, $w_6$--equations. 
The outcome would be $\{1\}$ because the subroutine eliminates a factor in masses if terms  in there have constant sign.

\subsection{Diagrams 4, 7, 9, 12, 15, 17, 19, 28, 30, 76} \label{subsec:6-4}
These 10 cases are similar to the fully-edged isolated triangle in the case $n=4$. 
We collect some $z_i$-, $w_i$-equations determined by the $zw$-order matrix and relations between separations, 
make substitutions according to the cluster equations determined by some $zw$-order matrices, and then   
apply the eliminating process through the third subroutine in \S\ref{subsec:subroutines}.

{\small
\begin{tabular}[thp]{|p{5.5mm}| p{91mm} | p{5.5mm} |p{29mm}| } 
\hline
{\bf NO}     & {\bf Mass Relations} & $|\coll_2|$ &{\bf  Method}  \\
\hline \hline 
4   &  $(m_1-m_6)\mu_5(m_1,m_2,m_6)=0$ & 1 &  $(1,2;1,2)$  \\                             
\hline
7   &  $\mu_5(m_4,m_5,m_6)=0$ &  200 &  $(5,6;5,6)$  \\      
\hline
9   &  $\mu_5(m_4,m_5,m_6)=0$ & 9 & $(4,5,6;4,5,6)$   \\       
\hline
12   & $\mu_5(m_4,m_5,m_6)=0$ & 5 &  $(5,6;5,6)$  \\        
\hline
15   & $\mu_5(m_4,m_5,m_6)=0$ & 24 &  $(5,6;5,6)$  \\     
\hline
17   & $\mu_5(m_4,m_5,m_6)=0$ & 7 & $(4,5,6;4,5,6)$   \\   
\hline
19   & $\mu_5(m_4,m_5,m_6)=0$ & 5 &  $(5,6;5,6)$   \\  
\hline
28   & $\mu_5(m_4,m_5,m_6)=0$  & 3 & $(5,6;5,6)$    \\  
\hline
30   &  $\mu_5(m_4,m_5,m_6)=0$  & 1 & $(4,5,6;4,5,6)$   \\  
\hline
76   & $\mu_5(m_1,m_2,m_3)=\mu_5(m_4,m_5,m_6)=0$ & 9  & \small{$(1,2;1,2),(5,6;5,6)$}   \\      
\hline
77   & $m_2m_4=m_3m_5$ &  1 & $(1,2;1,...,5)$  \\    
\hline
99   & $m_3m_5=m_4m_6$ &  5 & Gr\"obner basis  \\                             
\hline
100   & $m_3m_5=m_4m_6$ &  10 & Gr\"obner basis \\                             
\hline
101   & $m_3m_5=m_4m_6$ &  4 & Gr\"obner basis  \\                             
\hline
102   & $m_2m_4=m_1m_5$ &  1 &  $(1,2;1,...,5)$  \\     
\hline
103   & $m_1m_4=(m_2+m_5)(m_3+m_6)$ &  1 & Gr\"obner basis  \\                             
\hline
104   & $m_1m_4=m_2m_6$ &  1 & Gr\"obner basis  \\                             
\hline
105   & $m_1m_4=(m_2+m_5)(m_3+m_6)$ &  1 & Gr\"obner basis  \\                             
\hline
106   & $m_4(m_1+m_5)=m_2(m_3+m_6)$ &  2 & Gr\"obner basis  \\                             
\hline
107   & $m_4(m_1+m_5)=m_2(m_3+m_6)$ &  2 & Gr\"obner basis  \\                             
\hline
108   & $m_4(m_1+m_5)=m_2(m_3+m_6)$ &  1 & Gr\"obner basis  \\                             
\hline
109   & $m_1m_3=m_2m_4$ &  1 & $(1,2;1,...,5)$   \\                                     
\hline
110   & $m_1m_3=(m_2+m_5)(m_4+m_6)$  &  1 & Gr\"obner basis  \\                             
\hline
111   & $m_4(m_2+m_5)=m_1(m_3+m_6)$  &  2 & Gr\"obner basis \\                             
\hline
112   & $m_4(m_2+m_5)=m_1(m_3+m_6)$ &  1 & Gr\"obner basis  \\                             
\hline
114   & $m_1m_3=m_2(m_4+m_5+m_6)$ &  1 & Gr\"obner basis  \\                             
\hline
115   & $m_1m_3=m_2(m_4+m_5+m_6)$ &  1 & Gr\"obner basis  \\                             
\hline
116   & $m_1m_4m_5+m_2m_3m_6=0$ &  1 & Gr\"obner basis  \\                             
\hline
117   & $m_1m_3=m_2m_4$  &  1 & $(1,...,4;1,...,6)$   \\                                             
\hline
\end{tabular}    
\captionof{table}{Some diagrams with simple mass relations.    
The method is either by Gr\"obner basis or by manually choosing main 
equations. Notation $(i_1,...,i_k; j_1,...,j_l)$ means the selection $z_{i_1},...,z_{i_k}$, $w_{j_1},...,w_{j_l}$.
For instance,  $(2,3;4,5)$ means $z_2,z_3,w_4,w_5$. } \label{table:simplemassrel6}  
}


For example, in diagram 4, there is a fully-edged triangle formed by bodies $1,2,6$, 
and bodies 1,2 are only adjacent to each other and body 6. 
Based on this feature, we collect $z_1$-, $z_2$-, $w_1$-, $w_2$-equations determined by the $zw$-matrix and relations between 
separations. Then we make substitutions according to the cluster equations determined by some $zw$-order matrices, apply the 
eliminating process, and obtain the mass relation 
$$
 (m_1-m_6)\mu_5(m_1,m_2,m_6)=0
$$
By collecting more equations, say $z_i$-, $w_i$-equations for $i\in\{2,3,4,5,6\}$, 
a less restrictive mass relation without the first factor can be obtained. 

Diagram~76 is a special case which contains two isolated fully-edged triangles. We obtain a finer mass relation by applying the 
same procedure twice, once for bodies of one triangle, and once for bodies of the other triangle. 

An exceptional case worth-mentioning is diagram~7. 
Algorithm~II quickly outputs the optimal $zw$-order matrix of Type~1 for its $zw$-matrix: 
$$
\footnotesize{
\left(\!\! \begin{pmatrix}
\{0\!\sim\!4\}\! & \!\!\{2\!\sim\!4\}\!\! & \!\!\{2\!\sim\!4\}\!\! & \!\!\{2\!\sim\!4\}\!\! & \!\!\{2\!\sim\!4\}\!\! & \!\{2\!\sim\!4\} \\
\{2\!\sim\!4\}\! & \!\!\{0\!\sim\!4\}\!\! & \!\!\{2\!\sim\!4\}\!\! & \!\!\{2\!\sim\!4\}\!\! & \!\!\{2\!\sim\!4\}\!\! & \!\{2\!\sim\!4\} \\
\{2\!\sim\!4\}\! & \!\!\{2\!\sim\!4\}\!\! & \!\!\{0\!\sim\!4\}\!\! & \!\!\{2\!\sim\!4\}\!\! & \!\!\{2\!\sim\!4\}\!\! & \!\{2\!\sim\!4\} \\
\{2\!\sim\!4\}\! & \!\!\{2\!\sim\!4\}\!\! & \!\!\{2\!\sim\!4\}\!\! & \!\!\{0\!\sim\!4\}\!\! & \{3\} & \{3\} \\
\{2\!\sim\!4\}\! & \!\!\{2\!\sim\!4\}\!\! & \!\!\{2\!\sim\!4\}\!\! & \{3\} & \!\!\{0\!\sim\!4\}\!\! & \{3\} \\
\{2\!\sim\!4\}\! & \!\!\{2\!\sim\!4\}\!\! & \!\!\{2\!\sim\!4\}\!\! & \{3\} & \{3\} & \{0\!\sim\!4\} 
\end{pmatrix}\!\!,\!\! 
\begin{pmatrix}
\{0\!\sim\!4\}\! & \!\!\{2\!\sim\!4\}\!\! & \!\!\{2\!\sim\!4\}\!\! & \!\!\{2\!\sim\!4\}\!\! & \!\!\{2\!\sim\!4\}\!\! & \!\{2\!\sim\!4\} \\
\{2\!\sim\!4\}\! & \!\!\{0\!\sim\!4\}\!\! & \!\!\{2\!\sim\!4\}\!\! & \!\!\{2\!\sim\!4\}\!\! & \!\!\{2\!\sim\!4\}\!\! & \!\{2\!\sim\!4\} \\
\{2\!\sim\!4\}\! & \!\!\{2\!\sim\!4\}\!\! & \!\!\{0\!\sim\!4\}\!\! & \!\!\{2\!\sim\!4\}\!\! & \!\!\{2\!\sim\!4\}\!\! & \!\{2\!\sim\!4\} \\
\{2\!\sim\!4\}\! & \!\!\{2\!\sim\!4\}\!\! & \!\!\{2\!\sim\!4\}\!\! & \!\!\{0\!\sim\!4\}\!\! & \{3\} & \{3\} \\
\{2\!\sim\!4\}\! & \!\!\{2\!\sim\!4\}\!\! & \!\!\{2\!\sim\!4\}\!\! & \{3\} & \!\!\{0\!\sim\!4\}\!\! & \{3\} \\
\{2\!\sim\!4\}\! & \!\!\{2\!\sim\!4\}\!\! & \!\!\{2\!\sim\!4\}\!\! & \{3\} & \{3\} & \{0\!\sim\!4\} 
\end{pmatrix} \!\!\right)\!.
}
$$
Before entering step~(II.2), there are too much uncertainties in orders of positions and separations;  
the total count of its Type~2 sub-matrices is about $2.82\times 10^{11}$. 
It cost much time and memory space for the implementation of step~(II.2) to select the 
collection $\coll_2$ of Type-2 $zw$-order matrices. 

Handling this diagram by main equations from the optimal $zw$-order matrix of Type~1, 
instead of from $\coll_2$ or $\coll_3$, would be much more efficient. 
Fortunately polynomial equations collected this way is sufficient to proceed the elimination process. 
The function ``{\tt FindOrder}''  in our program has an option on time constraint. 
Step~(II.1) is very fast, so one may restrict a short time, say 2 seconds, and give up searching
the collection $\coll_2$.   
For the implementation of Algorithm~III using only the optimal $zw$-order matrix of Type~1, denoted  
{\tt \{zo1,wo1\}} in our program, one can use the third subroutine of  \S\ref{subsec:subroutines} 
and the command line below to find the mass relation. \smallskip   

\noindent
{\tt
In:=  FindOrder[7,2]; FindMassRel3[\{zo1,wo1\}, \{5, 6\}, \{5, 6\}]  
}

\subsection{Diagrams 31, 32, 43, 44, 48, 49, 51, 55, 57, 58, 78, 95} 
These 12 cases are similar to the kite diagram in the case $n=4$. 
We collect $z_i$-, $z_j$-, $w_i$-, $w_j$-equations determined by the $zw$-matrix and relations between 
separations. Then we make substitutions according to the cluster equations determined by some $zw$-order matrices, 
apply the eliminating process, and obtain mass relations similar to those in \S\ref{subsec:4-3}.  

For example, in diagram 31, the optimal $zw$-order matrix of Type~1 is
$$
\small{
\left(\!\! \begin{pmatrix}
\{2\sim4\}\! & \!\!\{2\sim4\}\!\! & \!\!\{2\sim4\}\!\! & \!\!\{2\sim4\}\!\! & \!\!\{2\sim4\}\!\! & \!\{2\sim4\} \\
\{2\sim4\}\! & \!\!\{2\sim4\}\!\! & \!\!\{2\sim4\}\!\! & \!\!\{2\sim4\}\!\! & \!\!\{2\sim4\}\!\! & \!\{2\sim4\} \\
\{2\sim4\}\! & \!\!\{2\sim4\}\!\! & \{0,1\} & \{3\} & \{3\} & \{1\} \\
\{2\sim4\}\! & \!\!\{2\sim4\}\!\! & \{3\} & \{3\} & \{3\} & \{3\} \\
\{2\sim4\}\! & \!\!\{2\sim4\}\!\! & \{3\} & \{3\} & \{3\} & \{3\} \\
\{2\sim4\}\! & \!\!\{2\sim4\}\!\! & \{1\} & \{3\} & \{3\} & \{0,1\} 
\end{pmatrix}\!\!,\!\! 
\begin{pmatrix}
\{0\sim4\}\! & \!\!\{2\sim4\}\!\! & \!\{5\}\! & \!\{5\}\! & \!\{5\}\! & \!\{5\} \\
\{2\sim4\} & \!\!\{0\sim4\}\!\!   & \!\{5\}\! & \!\{5\}\! & \!\{5\}\! & \{5\} \\
\{5\} & \!\!\{2\sim4\}\!\!           & \!\{5\}\! & \!\{5\}\! & \!\{5\}\! & \{5\} \\
\{5\} & \!\!\{2\sim4\}\!\!           & \!\{5\}\! & \!\{5\}\! & \!\{3\}\! & \{3\} \\
\{5\} & \!\!\{2\sim4\}\!\!           & \!\{5\}\! & \!\{3\}\! & \!\{5\}\! & \{5\} \\
\{5\} & \!\!\{2\sim4\}\!\!           & \!\{5\}\! & \!\{3\}\! & \!\{5\}\! & \{5\} 
\end{pmatrix} \!\!\right)\!.
}
$$
Take out lower right $4\times4$ principal minors we obtain the Type~2 $zw$-order matrix for the kite diagram in \S\ref{subsec:4-3}. 
Step~(II.2) of Algorithm~II results in 22 $zw$-order matrices of Type~2. The $r$-order matrix they produced is 
\beq
\small{
\left(\!\!
\begin{array}{cccccc}
 \{\} & \!\{2\sim4\}\! & \{4\} & \{4\} & \{4\} & \{4\}  \\
 \{2\sim4\} & \{\} & \{4\} & \{4\} & \{4\} & \{4\} \\
 \{4\} & \{4\}     & \{\}   & \{4\} & \{4\} & \{3\} \\
 \{4\} & \{4\}     & \{4\} & \{\}   & \{1\} & \{1\} \\
 \{4\} & \{4\}     & \{4\} & \{1\} & \{\}   & \{1\} \\
 \{4\} & \{4\}     & \{3\} & \{1\} & \{1\} & \{\} \\
\end{array}
\!\! \right)\!. 
}
\eeq

There is a fully-edge triangle with vertices $4,5,6$, they are all $w$-circled, 
bodies 4,5 are only adjacent to each other and body 6. 
Based on these features, we collect $z_4$-, $z_5$-, $w_4$-, $w_5$-equations 
determined by the $zw$-matrix and relations between separations. 
Using the third subroutine in \S\ref{subsec:subroutines}, we find the mass relation 
\begin{eqnarray}
 (m_4^3-m_4 m_6^2-m_5 m_6^2)\, \mu_5(m_4,m_5,m_6)\, \mu_6(m_4,m_5,m_6) & &   \label{eqn:mr6-31} \\
 \mu_{15}(m_4,m_5,m_6)\, \mu_{21}(m_4,m_5,m_6)\mu_{33}(m_4,m_5,m_6)&=& 0 \nonumber  
\end{eqnarray}

Following this procedure for all of the 12 diagrams in this subsection, except diagram~95, we find exactly the same mass relation. 
The fully-edged triangle in diagram~95 has vertices $2,3,4$, by choosing $z_1$-, $z_2$-, $w_1$-, $w_2$-equations 
we acquire the same mass relation except subindices $(4,5,6)$ are replaced by $(2,3,4)$. 

A more restrictive mass relation can be obtained for diagram~78 because there are two fully-edge triangles in there. 
One can easily obtain an additional mass relation that is the same as (\ref{eqn:mr6-31}) except subindices $(4,5,6)$ 
are replaced by $(1,2,3)$.

\subsection{Diagrams 2, 20, 22, 36, 40, 45, 46, 50, 61, 62, 84, 86}  \label{subsec:6-2etc}
For these 12 diagrams, we apply the same procedure as  diagram 12 in the case $n=5$. 
For each Type~2 $zw$-order matrix in $\coll_2$, or Type~3  $zw$-order matrix in $\coll_3$, 
we determine main equations and cluster equations by 
the $zw$-order matrix, and collect part of main equations and relations between separations. 
Then we make substitutions according to cluster equations determined by the $zw$-order matrices and apply the 
eliminating process to find mass relations. 
Different members in $\coll_2$ may result in different mass relations. The union of their zero loci is the mass relation we acquire.

Table~\ref{table:moremassrel6} contains mass relations for most diagrams within this category, excluded cases will be discussed in below. 
Several diagrams in this family result in special mass relations that do not appear in other diagrams. 
For brevity we do not list details of some long factors. 
Readers may recover them by following our instructions on the selections of $z_i$-, $w_i$-equations. 

{\small
\begin{tabular}[th]{|p{5mm}| p{102.5mm} | p{4.6mm} |p{22mm}| } 
\hline
{\bf NO}     & {\bf Mass Relations} & $|\coll_2|$ &{\bf  Method}  \\
\hline \hline 
\multirow{2}{*}{2}    &  
\scriptsize{$\mu_5(m_1,m_3,m_4) (3 m_1^2 m_3^2-5 m_1^2 m_4^2-3 m_3^2 m_4^2)(m_1^2 m_3^2+m_1^2 m_4^2-m_3^2 m_4^2) 
= \mu_5(m_2,m_5,m_6) (3 m_2^2 m_5^2-5 m_2^2 m_6^2-3 m_5^2 m_6^2)(m_2^2 m_5^2+m_2^2 m_6^2-m_5^2 m_6^2) = 0$} 
& \multirow{2}{*}{2}  &  $(1,3,4;1,3,4)$,  $(2,5,6;2,5,6)$    \\       
\hline
25   &  \scriptsize{$\mu_5(m_4,m_5,m_6)(5m_4^2 m_5^2- 3m_4^2 m_6^2+ 3m_5^2 m_6^2)=0$} & 1 & $(4,5;4,5)$   \\     
\hline
\multirow{2}{*}{36}    &  
\scriptsize{$(m_2m_4^3+m_3m_4^3-m_2m_4m_6^2-m_2m_5m_6^2+m_3m_6^3) \mu_5(m_4,m_5,m_6)  
\mu_{84}(m_2,m_3,$  $m_4,m_5,m_6)  \mu_{221}(m_2,m_3,m_4,m_5,m_6) \mu_{374}(m_2,m_3,m_4,m_5,m_6)=0$ } 
& \multirow{2}{*}{4} &  \multirow{2}{*}{$(4,5;2,...,6)$}    \\     
\hline
\multirow{3}{*}{40}    &  
\scriptsize{$(m_4^3 -(m_2+m_4+m_5) m_6^2) \mu_5(m_4,m_5,m_6) \mu_{45}(m_2,m_4,m_5,m_6)\qquad\qquad\quad$ 
$\mu_{84}(m_3,m_2,m_4,m_5,m_6) \mu_{88}(m_2,m_4,m_5,m_6) \mu_{111}(m_2,m_4,m_5,m_6)$ 
$\mu_{221}(m_3,m_2,m_4,m_5,m_6) 
\mu_{374}(m_3,m_2,m_4,m_5,m_6)  = 0$} 
& \multirow{3}{*}{2}  &  \multirow{3}{*}{$(4,5;1,...,6)$}     \\       
\hline
46   & \scriptsize{$(m_2m_4-m_3m_5) (m_1m_5-m_2m_6) =0$} & 9 &  $(4,5;1,...,6)$  \\                             
\hline
61   &  \scriptsize{$m_1(m_3+m_5)=m_2(m_4+m_6)$} &  3 &  $(4,5;1,...,6)$   \\                             
\hline
62   &  \scriptsize{$m_1(m_3+m_5)=m_2(m_4+m_6)$} &  3 & $(4,5;1,...,6)$   \\                             
\hline
84   &  \footnotesize{$m_1m_3=m_2m_4$} &  8 & $(4,5;1,...,6)$  \\                                
\hline
\multirow{2}{*}{86}    &  \scriptsize{$(m_2^3-m_2 m_4^2-m_3 m_4^2)  \mu_5(m_2,m_3,m_4) \mu_6(m_2,m_3,m_4)
\mu_{15}(m_2,m_3,m_4)\qquad\qquad$ $\mu_{21}(m_2,m_3,m_4)\mu_{33}(m_2,m_3,m_4)=0 $} 
& \multirow{2}{*}{1} &  \multirow{2}{*}{$(1,...,6;1,...,6)$}    \\     
\hline
\end{tabular}    
\captionof{table}{More diagrams and their mass relations.     } \label{table:moremassrel6}  
}
\normalsize

Consider diagram~20. We obtain one Type~2 $zw$-order matrix to cover orders of positions and separations:
\beq &  
\small{
\left(\!\! \begin{pmatrix}
\{3\} & \{3\} & \{3\} & \{3\} & \{3\} & \{1\} \\
\{3\} & \{3\} & \{3\} & \{3\} & \{1\} & \{3\} \\
\{3\} & \{3\} & \!\{0,1\}\! & \{1\} & \{3\} & \{3\} \\
\{3\} & \{3\} & \{1\} & \!\{0,1\}\! & \{3\} & \{3\} \\
\{3\} & \{1\} & \{3\} & \{3\} & \{3\} & \{3\} \\
\{1\} & \{3\} & \{3\} & \{3\} & \{3\} & \{3\} 
\end{pmatrix}\!\!,\!\! 
\begin{pmatrix}
\{5\} & \{4\} & \{5\} & \{5\} & \{5\} & \{5\} \\
\{4\} & \{5\} & \{5\} & \{5\} & \{5\} & \{5\} \\
\{5\} & \{5\} & \{5\} & \{5\} & \{5\} & \{5\} \\
\{5\} & \{5\} &\{5\} & \!\{0\!\sim\!4\}\! & \{3\} & \{3\} \\
\{5\} & \{5\} & \{5\} & \{3\} & \!\{0\!\sim\!4\}\!& \{3\} \\
\{5\} & \{5\} & \{5\} & \{3\} & \{3\} & \{0\!\sim\!4\}
\end{pmatrix} \!\!\right)\!. 
} 
\eeq
This is not the ``optimal'' Type~1 $zw$-order matrix from step~(II.1), with which 
we do not find any mass relation from our program. 
It produces the $r$-order matrix 
\beq
\small{
\left(\!\!
\begin{array}{cccccc}
 \{\} & \!\{2\sim4\}\! & \{4\} & \{4\} & \{4\} & \{3\}  \\
 \{2\sim4\}\! & \{\} & \{4\} & \{4\} & \{3\} & \{4\} \\
\{4\} & \{4\}    & \{\}   & \{3\} & \{4\} & \{4\} \\
 \{4\} & \{4\}     & \{3\} & \{\}   & \{1\} & \{1\} \\
 \{4\} & \{3\}     & \{4\} & \{1\} & \{\}   & \{1\} \\
 \{3\} & \{4\}     & \{4\} & \{1\} & \{1\} & \{\} \\
\end{array}
\!\! \right)\!. 
}
\eeq

The fully-edged triangle with vertices 4, 5, 6 suggest choosing $z_i$-, $w_i$-equations with $i\in \{4,5,6\}$
to generate mass relations. However, our algorithm does not produce any mass relation with this collection of main equations. 
By taking all $z_i$- and $w_i$-equations into consideration,  
we are able to obtain the mass relation $f_{6,20}=0$, where $f_{6,20}$ has the following factors 
\beq
&& \mu_5(m_4,m_5,m_6), \\
&& m_1 m_4^2+m_2 m_4^2-m_1 m_6^2, \\
&& m_1^3 m_2^3 m_4^6 m_5^6+3 m_1^2 m_2^4 m_4^6 m_5^6+3 m_1 m_2^5 m_4^6 m_5^6+m_2^6 m_4^6 m_5^6+ \cdots \text{(20 other terms)}, \\
&& m_1^3 m_2^3 m_4^6 m_5^6+3 m_1^2 m_2^4 m_4^6 m_5^6+3 m_1 m_2^5 m_4^6 m_5^6+m_2^6 m_4^6 m_5^6+  \cdots \text{(26 other terms)}, \\
&& 3 m_1^4 m_4^8 m_5^4+3 m_1^3 m_2 m_4^8 m_5^4-9 m_1^2 m_2^2 m_4^8 m_5^4-15 m_1 m_2^3 m_4^8 m_5^4- \cdots \text{(38 other terms)}.
\eeq

The elimination process for diagrams~40, 45 and 50 are time consuming due to the presence of too many leading order terms. 
Diagram~45 and 50 results in several special mass factors. 

In diagram 45, we have $|\coll_2|=3$, $|\coll_3|=13$. Using members in $\coll_2$
and $z_4$-, $z_5$-equations and all $w_i$-equations,  we find mass relation $f_{6,45}=0$ with 
the following factors 

\small{
\beq
&& m_4^3-(m_3+m_4+m_5)m_6^2 \\ 
&& m_4^3+m_1 m_6^2 + m_2 m_6^2 + m_6^3 ,  \\ 
&& \mu_5(m_4,m_5,m_6) \\
&&  -m_1 m_4^3 - m_3 m_4^3 - m_2 m_3 m_6^2 + m_1 m_4 m_6^2 + m_1 m_5 m_6^2 - m_3 m_6^3,  \\ 
&& \mu_{45}(m_3,m_4,m_5,m_6), \\ 
&& m_4^{11} m_5^2 - m_4^{11} m_6^2 + 3 m_1 m_4^8 m_5^2 m_6^2 + 3 m_2 m_4^8 m_5^2 m_6^2 - \cdots \text{(56 other terms)},  \\ 
&& \mu_{88}(m_3,m_4,m_5,m_6), \\ 
&& \mu_{111}(m_3,m_4,m_5,m_6), \\ 
&& -m_1^{3} m_4^{11}- m_5^{2} - 3 m_1^2 m_3 m_4^{11} m_5^2- 3 m_1 m_3^2 m_4^{11} m_5^2 - m_3^3 m_4^{11}m_5^2 + \cdots \text{(135 other terms)}, \\ 
&& 9 m_1 m_4^{11} m_5^4 + 9 m_2 m_4^{11} m_5^4 + 6 m_4^{12} m_5^4 + 9 m_4^{11} m_5^4 m_6 - \cdots \text{(158 other terms)},  \\ 
&&  m_1^3 m_4^9 m_5^6+3 m_1^2 m_2 m_4^9 m_5^6+3 m_1 m_2^2 m_4^9 m_5^6+m_2^3 m_4^9 m_5^6 + \cdots \text{(304 other terms)},  \\ 
&& 9 m_1^3 m_2 m_3  m_4^{11} m_5^4 + 27 m_1^{2} m_2 m_3^2 m_4^{11} m_5^4 + 27 m_1 m_2 m_3^3 m_4^{11} m_5^4 + 9 m_2 m_3^{4} m_4^{11} m_5^4  \\ 
&& \qquad - \cdots \text{(420 other terms)},  \\ 
&& m_1^3 m_2^3 m_3^3 m_4^9 m_5^6+3 m_1^2 m_2^3 m_3^4 m_4^9 m_5^6+3 m_1 m_2^3 m_3^5 m_4^9 m_5^6+m_2^3 m_3^6 m_4^9 m_5^6 + \cdots \text{(967 other terms)}.  \\ 
\eeq
}
\normalsize

In diagram 50, we have $|\coll_2|=4$, $|\coll_3|=27$. Using members in $\coll_2$
and $z_4$-, $z_5$-equations and all $w_i$-equations,  we find mass relation $f_{6,50}=0$ with 
the following factors 

\small{
\beq
&& \mu_5(m_4,m_5,m_6), \\
&& m_4^3 - m_2 m_6^2 - m_3 m_6^2 - m_4 m_6^2 - m_5 m_6^2 \\ 
&& m_1m_4^3+m_2m_4^3-m_1m_3m_6^2-m_1m_4m_6^2 -m_1m_5m_6^2 + m_2m_6^3, \\ 
&&  m_1 m_4^3 + m_2 m_4^3 + m_3 m_4^3 - m_1 m_4 m_6^2 - m_1 m_5 m_6^2 + m_2 m_6^3 + m_3 m_6^3,  \\ 
&& m_4^{11} m_5^2- m_4^{11} m_6^2 + 3m_1m_4^8 m_5^2 m_6^2- m_4^9 m_5^2 m_6^2 + \cdots \text{(26 other terms)}, \\ 
&& 9 m_1 m_4^{11} m_5^4 + 6 m_4^{12} m_5^4 + 9 m_4^{11} m_5^4 m_6 -  30 m_1 m_4^{11} m_5^2 m_6^2 - \cdots \text{(68 other terms)},  \\
&& m_4^{11} m_5^2 - m_4^{11} m_6^2 - 3 m_2 m_4^8 m_5^2 m_6^2 - 3 m_3 m_4^8 m_5^2 m_6^2 + \cdots \text{(79 other terms)},  \\ 
&&  \mu_{111}(m_1,m_4,m_6,m_5),  \\ 
&& m_1^3 m_4^{11} m_5^2 + 3 m_1^2 m_2 m_4^{11} m_5^2 + 3 m_1 m_2^2 m_4^{11} m_5^2 +  m_2^3 m_4^{11} m_5^2 + \cdots \text{(135 other terms)},  \\ 
&& 9 m_2 m_4^{11} m_5^{4} + 9 m_3 m_4 m_5^4 + 3 m_4^{12} m_5^4 + 9 m_4^{11} m_5^5 + \cdots \text{(183 other terms)},  \\ 
&&  m_1^3 m_4^{11} m_5^2 + 3 m_1^2 m_2 m_4^{11} m_5^2 + 3 m_1 m_2^2 m_4^{11} m_5^2 + m_2 m_4^{11} m_5^2 + \cdots \text{(191 other terms)},  \\ 
&& - m_2^3 m_4^9 m_5^{6} - 3 m_2^2 m_3 m_4^{9} m_5^6 - 3 m_2 m_3^2 m_4^{9} m_5^6 - m_3^3 m_4^{9} m_5^6 + \cdots \text{(304 other terms)},  \\ 
&& 9 m_1^4 m_3 m_4^{11} m_5^4 + 27 m_1^3 m_2 m_3 m_4^{11} m_5^4 + 27 m_1^2 m_2^2 m_3 m_4^{11} m_5^4 + 9 m_1 m_2^3 m_3 m_4^{11} m_5^4 \\
&& \qquad + \cdots \text{(420 other terms)},  \\ 
&&  3 m_1^4 m_4^{12} m_5^4 - 3 m_1^3 m_2 m_4^{12} m_5^4 - 9 m_1^2 m_2^2 m_4^{12} m_5^4 - 15 m_1 m_2^{3} m_4^{12} m_5^4 + \cdots \text{(633 other terms)},  \\ 
&& - m_1^6 m_3^3 m_4^9 m_5^6 - 3 m_1^5 m_2 m_3^3 m_4^9 m_5^6 - 3 m_1^4 m_2^2 m_3^3 m_4^9 m_5^6 - m_1^3 m_2^3 m_3^3 m_4^9 m_5^6 \\
&& \qquad + \cdots \text{(967 other terms)},  \\ 
&&  m_1^3 m_2^3 m_4^{12} m_5^6 + 3 m_1^2 m_2^4 m_4^{12} m_5^6 + 3 m_1 m_2^5 m_4^{12} m_5^6 + m_2^6 m_4^{12} m_5^6  + \cdots \text{(1492 other terms)}.  \\ 
\eeq
}
\normalsize

\subsection{Diagram 47}  \label{subsec:6-47}

In diagram 47, we have $|\coll_2|=4$. For the first three members, we apply the same procedure as diagram 12 in case $n=5$. 
We collect $z_4$-, $z_5$-equations, all $w_i$-equations and relations between separations. 
Then we make substitutions according to cluster equations determined by the $zw$-order matrices and apply the 
eliminating process to find mass relations.

The last member in $\coll_2$ is
\beq &  
\small{
\left(\!\! \begin{pmatrix}
\{3\} & \{3\} & \{3\} & \{3\} & \{3\} & \{1\} \\
\{3\} & \{3\} & \{3\} & \{3\} & \{1\} & \{3\} \\
\{3\} & \{3\} & \!\{0 \!\sim\! 1\}\! & \{1\} & \{3\} & \{3\} \\
\{3\} & \{3\} & \{1\} & \!\{0 \!\sim\! 1\}\! & \{3\} & \{3\} \\
\{3\} & \{1\} & \{3\} & \{3\} & \{3\} & \{3\} \\
\{1\} & \{3\} & \{3\} & \{3\} & \{3\} & \{3\} 
\end{pmatrix}\!\!,\!\! 
\begin{pmatrix}
\{5\} & \{4\} & \{5\} & \{5\} & \{5\} & \{5\} \\
\{4\} & \{5\} & \{5\} & \{5\} & \{5\} & \{5\} \\
\{5\} & \{5\} & \{5\} & \{5\} & \{5\} & \{5\} \\
\{5\} & \{5\} & \{5\} & \{5\} & \{3\} & \{3\} \\
\{5\} & \{5\} & \{5\} & \{3\} & \{5\} & \{3\} \\
\{5\} & \{5\} & \{5\} & \{3\} & \{3\} & \{5\} 
\end{pmatrix}\!\!\right)\!. 
} 
\eeq

We obtain $\epsilon\approx z_1\approx z_2\approx z_5 \approx z_6 \succ z_3, z_4$ from this order matrix. 
Since the center of mass is at the origin, after taking maximal order terms, we have the equation
$$
m_1z_1+m_2z_2+m_5z_5+m_6z_6=0.
$$
We collect all main equations and relations between separations together with $m_1z_1+m_2z_2+m_5z_5+m_6z_6=0$. 
Then we make substitutions according to cluster equations determined by $zw$-order matrices and apply the 
eliminating process to obtain 
$$
m_1m_5=m_2m_6.
$$ 

For diagram 47, we find mass relation $f_{6,47}=0$ with 
the following factors 

\small{
\beq
&& m_1m_5 - m_2m_6, \\ 
&& \mu_5(m_4,m_5,m_6) \\ 
&& -m_3 m_4^2-m_4^3+m_2 m_6^2+m_3 m_6^2+m_4 m_6^2+m_5 m_6^2, \\ 
&&  m_1 m_3 m_4^2+m_2 m_3 m_4^2+m_1 m_4^3+m_2 m_4^3-m_1 m_3 m_6^2-m_1 m_4 m_6^2-m_1 m_5 m_6^2+m_2 m_6^3 ,  \\ 
&& -m_1 m_4^3-m_2 m_4^3+m_3 m_4^2 m_5+m_3 m_4^2 m_6+m_1 m_4 m_6^2+m_1 m_5 m_6^2-m_2 m_6^3-m_3 m_6^3 ,  \\ 
&& m_3^3 m_4^8 m_5^2+3 m_3^2 m_4^9 m_5^2+3 m_3 m_4^{10} m_5^2+m_4^{11} m_5^2 - \cdots \text{(108 other terms)}, \\ 
&& m_1^3 m_3^3 m_4^8 m_5^2+3 m_1^2 m_2 m_3^3 m_4^8 m_5^2+3 m_1 m_2^2 m_3^3 m_4^8 m_5^2+m_2^3 m_3^3 m_4^8 m_5^2+ \cdots \text{(232 other terms)},  \\ 
&& -m_1^3 m_4^{11} m_5^2-3 m_1^2 m_2 m_4^{11} m_5^2-3 m_1 m_2^2 m_4^{11} m_5^2-m_2^3 m_4^{11} m_5^2 + \cdots \text{(244 other terms)},  \\ 
&& 9 m_2 m_3^3 m_4^8 m_5^4+3 m_3^4 m_4^8 m_5^4+27 m_2 m_3^2 m_4^9 m_5^4+12 m_3^3 m_4^9 m_5^4 + \cdots \text{(266 other terms)},  \\ 
&& -m_2^3 m_3^3 m_4^6 m_5^6-3 m_2^3 m_3^2 m_4^7 m_5^6-3 m_2^3 m_3 m_4^8 m_5^6-m_2^3 m_4^9 m_5^6 - \cdots \text{(397 other terms)},  \\
&&  3 m_1^4 m_3^4 m_4^8 m_5^4+3 m_1^3 m_2 m_3^4 m_4^8 m_5^4-9 m_1^2 m_2^2 m_3^4 m_4^8 m_5^4-15 m_1 m_2^3 m_3^4 m_4^8 m_5^4 \\
&& \qquad - \cdots \text{(763 other terms)},  \\ 
&& -3 m_1^4 m_4^{12} m_5^4 - 3 m_1^3 m_2 m_4^{12} m_5^4 + 9 m_1^2 m_2^2 m_4^{12} m_5^4 +  15 m_1 m_2^3 m_4^{12} m_5^4  + \cdots \text{(801 other terms)},  \\ 
&& m_1^3 m_2^3 m_3^6 m_4^6 m_5^6+3 m_1^2 m_2^4 m_3^6 m_4^6 m_5^6+3 m_1 m_2^5 m_3^6 m_4^6 m_5^6+m_2^6 m_3^6 m_4^6 m_5^6 \\
&& \qquad + \cdots \text{(1732 other terms)},  \\ 
&& m_1^3 m_2^3 m_4^{12} m_5^6 + 3 m_1^2 m_2^4 m_4^{12} m_5^6 + 3 m_1 m_2^5 m_4^{12} m_5^6 +  m_2^6 m_4^{12} m_5^6  + \cdots \text{(1800 other terms)}.  \\ 
\eeq
}
\normalsize

\subsection{Diagrams 1, 10, 59, 60}  \label{subsec:6-1etc}

In each of these diagrams, there is a quadrilateral without diagonal formed by $z$-edges (or $w$-edges).

For diagram 1, we obtain one $zw$-order matrix in $\coll_2$:
\beq &  
\small{
\left(\!\!  \begin{pmatrix}
\!\{0\!\sim\!4\}\!\! & \{3\} & \{4\} & \{4\} & \{4\} & \{3\} \\
\{3\} & \!\!\{0\!\sim\! 4\}\!\! & \{4\} & \{4\} & \{4\} & \{3\} \\
\{4\} & \{4\} & \!\!\{0\!\sim\! 4\}\!\! & \{4\} & \{4\} & \{4\} \\
\{4\} & \{4\} & \{4\} & \!\!\{0\!\sim\! 4\}\!\! & \{4\} & \{4\} \\
\{4\} & \{4\} & \{4\} & \{4\} & \!\!\{0\!\sim\! 4\}\!\! & \{4\} \\
\{3\} & \{3\} & \{4\} & \{4\} & \{4\} & \!\!\{0\!\sim\! 4\}\! 
\end{pmatrix}\!\!,\!\! 
\begin{pmatrix}
\!\{0\!\sim\! 3\}\!\! & \{3\} & \{3\} & \{3\} & \{3\} & \{3\} \\
\{3\} & \!\!\{0\!\sim\! 3\}\!\! & \{3\} & \{3\} & \{3\} & \{3\} \\
\{3\} & \{3\} & \!\!\{0\!\sim\! 3\}\!\! & \{2\} & \{2\} & \{2\} \\
\{3\} & \{3\} & \{2\} & \!\!\{0\!\sim\! 3\}\!\! & \{2\} & \{2\} \\
\{3\} & \{3\} & \{2\} & \{2\} & \!\!\{0\!\sim\! 3\}\!\! & \{2\} \\
\{3\} & \{3\} & \{2\} & \{2\} & \{2\} & \!\!\{0\!\sim\! 3\}\!
\end{pmatrix} \!\!\right)\!. 
} 
\eeq
and the produced $r$-order matrix
$$
\begin{pmatrix}
\{\} & \{1\} & \{2\!\sim\! 4\} & \{2\!\sim\! 4\} & \{2\!\sim\! 4\} & \{1\} \\
\{1\} & \{\} & \{2\!\sim\! 4\} & \{2\!\sim\! 4\} & \{2\!\sim\! 4\} & \{1\} \\
\{2\!\sim\! 4\} & \{2\!\sim\! 4\} & \{\} & \{2\!\sim\! 4\} & \{2\} & \{2\} \\
\{2\!\sim\! 4\} & \{2\!\sim\! 4\} & \{2\!\sim\! 4\} & \{\} & \{2\} & \{2\} \\
\{2\!\sim\! 4\} & \{2\!\sim\! 4\} & \{2\} & \{2\} & \{\} & \{2\!\sim\! 4\} \\
\{1\} & \{1\} & \{2\!\sim\! 4\} & \{2\} & \{2\!\sim\! 4\} & \{\} 
\end{pmatrix}
$$

The following table shows all maximal order sets and whether they are certainly maximal.
\begin{center}
\begin{tabular}{c|c|c|c|c}
 & $M_{Z,i}$ & certainly max? & $M_{W,i}$ & certainly max?  \\
\hline
$i=1$ & $\{2,6\}$ & $\surd$ & $\{2,6\}$ & $\surd$ \\
\hline
$i=2$ & $\{1,6\}$ & $\surd$ & $\{1,6\}$ & $\surd$ \\
\hline
$i=3$ & $\{5,6\}$ & $\surd$ & $\{3,4,5,6\}$ &  \\
\hline
$i=4$ & $\{5,6\}$ & $\surd$ & $\{3,4,5,6\}$ & \\
\hline
$i=5$ & $\{3,4\}$ & $\surd$ & $\{3,4,5,6\}$ & \\
\hline
$i=6$ & $\{1,2,3,4\}$ & $\surd$ & $\{1,2\}$ & $\surd$ 
\end{tabular}
\end{center}

Since $z$-edges between bodies $3,5$ and between bodies $3,6$ are consecutive but not a part of a triangle, by Rule~2c, 
either $z_{35}\prec z_{36}$ or $z_{35}\succ z_{36}$. 
Without loss of generality, we may assume $z_{35}\prec z_{36}$, and then $z_{36}\approx z_{56}\succ z_{35}$. 
Since bodies $3,5,4,6$ form a quadrilateral, by Rule~2g, $z_{35}\approx z_{46}$ and $z_{36}\approx z_{45}$. Therefore we have
$$
z_{36}\approx z_{45} \approx z_{34} \approx z_{56}\succ z_{35} \approx z_{46}
$$

Since $Z_{35}\approx Z_{36}\succ Z_{34}$, from $W_{ij}=Z_{ij}^{1/3}z_{ij}^{-4/3}$, we obtain
$$
W_{35}\succ W_{36}, Z_{34}.
$$

Therefore $M_{W,3}^*=\{3,5\}$, and $M_{W,5}^*=\{3,5\}$ similarly, and we have equations
\beq
w_3&=&m_5 W_{53}=m_5 r_{53}^{-3} w_{53}\\
w_5&=&m_3 W_{35}=m_3 r_{35}^{-3} w_{35},
\eeq
and it follows that $w_{35}=w_3-w_5= -(m_3+m_5)r_{35}^{-3}w_{35}$. If all masses are positive, it gives $r\approx 1$ 
which contradicts that $3\notin d_{35}$. Therefore diagram~1 is impossible.

For diagram 59, we obtain 3 $zw$-order matrix in $\coll_2$. The first two of them give the following table for maximal order sets:
\begin{center}
\begin{tabular}{c|c|c|c|c}
 & $M_{Z,i}$ & certainly max? & $M_{W,i}$ & certainly max?  \\
\hline
$i=1$ & $\{1,2\}$ & $\surd$ & $\{1,2\}$ & $\surd$ \\
\hline
$i=2$ & $\{1,2\}$ & $\surd$ & $\{1,2\}$ & $\surd$ \\
\hline
$i=3$ & $\{4,5\}$ & $\surd$ & $\{1,2,3,4,5,6\}$ &  \\
\hline
$i=4$ & $\{3,6\}$ & $\surd$ & $\{3,4,5,6\}$ & \\
\hline
$i=5$ & $\{3,6\}$ & $\surd$ & $\{3,4,5,6\}$ & \\
\hline
$i=6$ & $\{4,5\}$ & $\surd$ & $\{3,4,5,6\}$ & 
\end{tabular}
\end{center}

After the same discussion as in diagram 1, these two cases are impossible for positive masses.

The last member in $\coll_2$ is
\beq &  
\small{
\left(\!\!  \begin{pmatrix}
\!\{0\!\sim\! 1\}\! & \{1\} & \{4\} & \{4\} & \{4\} & \{4\} \\
\{1\} & \!\{0\!\sim\! 1\}\! & \{4\} & \{4\} & \{4\} & \{4\} \\
\{4\} & \{4\} & \{4\} & \{4\} & \{4\} & \{4\} \\
\{4\} & \{4\} & \{4\} & \{4\} & \{4\} & \{4\}  \\
\{4\} & \{4\} & \{4\} & \{4\} & \{4\} & \{4\}  \\
\{4\} & \{4\} & \{4\} & \{4\} & \{4\} & \{4\} 
\end{pmatrix}\!\!,\!\! 
\begin{pmatrix}
\{5\} & \{5\} & \{5\} & \{5\} & \{5\} & \{5\} \\
\{5\} & \{5\} & \{5\} & \{5\} & \{5\} & \{5\} \\
\{5\} & \{5\} & \!\{0 \!\sim\! 4\}\! & \{2\} & \{2\} & \{2\} \\
\{5\} & \{5\} & \{2\} & \!\{0 \!\sim\! 4\}\! & \{2\} & \{2\} \\
\{5\} & \{5\} & \{2\} & \{2\} & \!\{0 \!\sim\! 4\}\! & \{2\} \\
\{5\} & \{5\} & \{2\} & \{2\} & \{2\} & \!\{0 \!\sim\! 4\}\! 
\end{pmatrix}\!\!\right)\!. 
} 
\eeq
which gives the following table of maximal order sets
\begin{center}
\begin{tabular}{c|c|c|c|c}
 & $M_{Z,i}$ & certainly max? & $M_{W,i}$ & certainly max?  \\
\hline
$i=1$ & $\{1,2\}$ & $\surd$ & $\{1,2\}$ & $\surd$ \\
\hline
$i=2$ & $\{1,2\}$ & $\surd$ & $\{1,2\}$ & $\surd$ \\
\hline
$i=3$ & $\{4,5\}$ & $\surd$ & $\{1,2,3,4,5,6\}$ &  \\
\hline
$i=4$ & $\{3,6\}$ & $\surd$ & $\{1,2,3,4,5,6\}$ & \\
\hline
$i=5$ & $\{3,6\}$ & $\surd$ & $\{1,2,3,4,5,6\}$ & \\
\hline
$i=6$ & $\{4,5\}$ & $\surd$ & $\{1,2,3,4,5,6\}$ & 
\end{tabular}
\end{center}

We will show that for there is at most one $i\in\{3,4,5,6\}$ such that $\{1,2\}\cap M_{W,i}^*\neq \emptyset$. 
Then by the same discussions as other members in $\coll_2$, we also conclude that this case is impossible for positive masses.

If $p\in\{1,2\}\cap M_{W,i}^*$ for some $i\in \{3,4,5,6\}$, then for each $j\in\{3,4,5,6\}$ different from $i$, 
we have $W_{pi}\succeq W_{ij}$. It follows from $w_{ip}\succ w_{ij}$ that $z_{ip}\prec z_{ij}$, and 
therefore $z_{pj}\approx z_{ij}$ which implies that $W_{pj}\prec W_{ij}$ since $w_{pj}\succ w_{ij}$. 
It means that $p\notin M_{W,j}^*$. On the other hand, since $z_{pq}\prec z_{pi}$ where $\{p,q\}=\{1,2\}$, 
we have $z_{pi}\approx z_{qi}$. From the $w$-separations, we obtain $W_{pi}\approx W_{qi}$. 
Therefore $\{1,2\}\subset M_{W,i}^*$ and $\{1,2\}\cap M_{W,j}^*=\emptyset$ for $j\in\{3,4,5,6\}$ different from $i$.

For diagram 60, we obtain one $zw$-order matrix in $\coll_2$. By the same discussions as the first two cases in 
diagram 59, this diagram is impossible for positive masses.

\subsection{Diagrams 3, 13, 29}  \label{subsec:6-3etc}

These three diagrams are similar to diagram 4 in case $n=5$. For diagram 3, we obtain one $zw$-order matrix of Type 2 which 
gives $M_{Z,1}=\{2,6\}$, $M_{Z,2}=\{1,6\}$, and $M_{W,1}=M_{W,2}=\{1,2,6\}$. By discussions similar to \S\ref{subsec:5-4}, 
we find mass relation $f_{6,3}=0$ where $f_{6,3}$ is the same as $f_{5,4,1}$, expect subindices were changed 
from $3,4,5$ to $1,2,6$ respectively.

For diagrams 13 and 29, each member in $\coll_2$ gives $M_{Z,4}=\{5,6\}$, $M_{Z,5}=\{4,6\}$, 
and $M_{W,4}=M_{W,5}=\{4,5,6\}$. By the same discussions as in \S\ref{subsec:5-4}, we find mass relation 
$f_{6,13}=0$ where $f_{6,13}$ is the same as $f_{5,4,1}$, except subindices are shifted by 1.

\subsection{Diagram 27}\label{subsec:6-27}

We obtain three Type 2 $zw$-order matrices $(S_k, T)$, $k=2,3,4$:
$$
\left(\! \begin{pmatrix}
\!\{0\!\sim\! 4\}\!\! & \!\!\{1\}\!\! & \!\!\{1\}\!\! & \!\!\{k\}\!\! & \!\!\{2\!\sim\! 4\}\!\! & \!\!\{2\!\sim\! 4\}\!\! \\
\!\!\{1\}\!\! & \!\!\{0\!\sim\! 4\}\!\! & \!\!\{1\}\!\! & \!\!\{k\}\!\! & \!\!\{2\!\sim\! 4\}\!\! & \!\!\{2\!\sim\! 4\}\!\! \\
\!\!\{1\}\!\! & \!\!\{1\}\!\! & \!\!\{0 \!\sim\! 4\}\!\! & \!\!\{k\}\!\! & \!\!\{2\!\sim\! 4\}\!\! & \!\!\{2\!\sim\! 4\}\!\! \\
\!\!\{k\}\!\! & \!\!\{k\}\!\! & \!\!\{k\}\!\! & \!\!\{0 \!\sim\! 4\}\!\! & \!\!\{4\}\!\! & \!\!\{4\}\!\! \\
\!\!\{2\!\sim\! 4\}\!\! & \!\!\{2\!\sim\! 4\}\!\! & \!\!\{2\!\sim\! 4\}\!\! & \!\!\{4\}\!\! & \!\!\{0 \!\sim\! 4\}\!\! & \!\!\{4\}\!\! \\
\!\!\{2\!\sim\! 4\}\!\! & \!\!\{2\!\sim\! 4\}\!\! & \!\!\{2\!\sim\! 4\}\!\! & \!\!\{4\}\!\! & \!\!\{4\}\!\! & \!\!\{0 \!\sim\! 4\}\! 
\end{pmatrix}\!\!,\!\!
\begin{pmatrix}
\!\{5\}\!\! & \!\!\{5\}\!\! & \!\!\{5\}\!\! & \!\!\{5\}\!\! & \!\!\{5\}\!\! & \!\!\{5\}\!\! \\
\!\!\{5\}\!\! & \!\!\{5\}\!\! & \!\!\{5\}\!\! & \!\!\{5\}\!\! & \!\!\{5\}\!\! & \!\!\{5\}\!\! \\
\!\!\{5\}\!\! & \!\!\{5\}\!\! & \!\!\{5\}\!\! & \!\!\{5\}\!\! & \!\!\{5\}\!\! & \!\!\{5\}\!\! \\
\!\!\{5\}\!\! & \!\!\{5\}\!\! & \!\!\{5\}\!\! & \!\!\{0 \!\sim\! 4\}\!\! & \!\!\{2\}\!\! & \!\!\{2\}\!\! \\
\!\!\{5\}\!\! & \!\!\{5\}\!\! & \!\!\{5\}\!\! & \!\!\{2\}\!\! & \!\!\{0 \!\sim\! 4\}\!\! & \!\!\{2\}\!\! \\
\!\!\{5\}\!\! & \!\!\{5\}\!\! & \!\!\{5\}\!\! & \!\!\{2\}\!\! & \!\!\{2\}\!\! & \!\!\{0 \!\sim\! 4\}\!
\end{pmatrix} \!\right),
$$

For $(S_2,T)$ and $(S_3,T)$, $M_{Z,5}=\{4,6\}$, $M_{Z,6}=\{5,6\}$, $M_{W,5}^*, M_{W,6}^*\subset\{4,5,6\}$. 
It's similar to diagram 4 in case $n=5$, we find mass relation $f_{6,27,1}=0$ where $f_{6,27,1}$ is the same as $f_{5,4,1}$, 
except subindices were changed from $3,4,5$ to $5,6,4$ respectively.

For $(S_4,T)$, we have $M_{Z,4}=\{5,6\}$, $M_{Z,5}=\{4,6\}$, $M_{Z,6}=\{4,5\}$, $M_{W,j}=\{1,2,3,4,5,6\}$ for $j= 4,5,6$. 
By the same discussions as in \S\ref{subsec:5-5}, there are at most one $i\in\{4,5,6\}$ such that $M_{W,i}^*\cap\{1,2,3\}\neq \emptyset$, 
and hence $M_{W,j}^*\subset\{4,5,6\}$ for other $j$ in $\{4,5,6\}$. In this case, we obtain mass relation $f_{6,27,2,i}=0$ 
where $f_{6,27,2,i}$ is the same as $f_{5,4,1}$ in \S\ref{subsec:5-4} except subindices were changed from 
$3,4,5$ to $j,k,i$ with $\{i,j,k\}=\{4,5,6\}$.

Therefore we find mass relation $f_{6,27}=0$ where $f_{6,27}$ is the same as $f_{5,5}$ in \S\ref{subsec:5-5}, 
except subindices are shifted by 1.

\subsection{Diagram 18}\label{subsec:6-18}
The optimal Type 1 $zw$-order matrix is
\beq
\left(\! \begin{pmatrix}
\!\!\{0\!\sim\! 4\}\!\! & \!\!\{1\}\!\! & \!\!\{1\}\!\! & \!\!\{2\!\sim\! 4\}\!\! & \!\!\{2\!\sim\! 4\}\!\! & \!\!\{2\!\sim\! 4\}\!\! \\
\!\!\{1\}\!\! & \!\!\{0\!\sim\! 4\}\!\! & \!\!\{1\}\!\! & \!\!\{2\!\sim\! 4\}\!\! & \!\!\{2\!\sim\! 4\}\!\! & \!\!\{2\!\sim\! 4\}\!\! \\
\!\!\{1\}\!\! & \!\!\{1\}\!\! & \!\!\{0 \!\sim\! 4\}\!\! & \!\!\{2\!\sim\! 4\}\!\! & \!\!\{2\!\sim\! 4\}\!\! & \!\!\{2\!\sim\! 4\}\!\! \\
\!\!\{2\!\sim\! 4\}\!\! & \!\!\{2\!\sim\! 4\}\!\! & \!\!\{2\!\sim\! 4\}\!\! & \!\!\{0 \!\sim\! 4\}\!\! & \!\!\{4\}\!\! & \!\!\{4\}\!\! \\
\!\!\{2\!\sim\! 4\}\!\! & \!\!\{2\!\sim\! 4\}\!\! & \!\!\{2\!\sim\! 4\}\!\! & \!\!\{4\}\!\! & \!\!\{0 \!\sim\! 4\}\!\! & \!\!\{4\}\!\! \\
\!\!\{2\!\sim\! 4\}\!\! & \!\!\{2\!\sim\! 4\}\!\! & \!\!\{2\!\sim\! 4\}\!\! & \!\!\{4\}\!\! & \!\!\{4\}\!\! & \!\!\{0 \!\sim\! 4\}\!\! 
\end{pmatrix}\!\!, \!\!
\begin{pmatrix}
\!\!\{0\!\sim\! 4\}\!\! & \!\!\{5\}\!\! & \!\!\{5\}\!\! & \!\!\{2\!\sim\! 4\}\!\! & \!\!\{2\!\sim\! 4\}\!\! & \!\!\{2\!\sim\! 4\}\!\! \\
\!\!\{5\}\!\! & \!\!\{5\}\!\! & \!\!\{5\}\!\! & \!\!\{5\}\!\! & \!\!\{5\}\!\! & \!\!\{5\}\!\! \\
\!\!\{5\}\!\! & \!\!\{5\}\!\! & \!\!\{5\}\!\! & \!\!\{5\}\!\! & \!\!\{5\}\!\! & \!\!\{5\}\!\! \\
\!\!\{2\!\sim\! 4\}\!\! & \!\!\{5\}\!\! & \!\!\{5\}\!\! & \!\!\{0 \!\sim\! 4\}\!\! & \!\!\{2\}\!\! & \!\!\{2\}\!\! \\
\!\!\{2\!\sim\! 4\}\!\! & \!\!\{5\}\!\! & \!\!\{5\}\!\! & \!\!\{2\}\!\! & \!\!\{0 \!\sim\! 4\}\!\! & \!\!\{2\}\!\! \\
\!\!\{2\!\sim\! 4\}\!\! & \!\!\{5\}\!\! & \!\!\{5\}\!\! & \!\!\{2\}\!\! & \!\!\{2\}\!\! & \!\!\{0 \!\sim\! 4\}\!\! 
\end{pmatrix} \!\right)
\eeq
which give the following table of maximal order sets:
\begin{center}
\begin{tabular}{c|c|c|c|c}
 & $M_{Z,i}$ & certainly max? & $M_{W,i}$ & certainly max?  \\
\hline
$i=1$ & $\{1,2,3,4,5,6\}$ &  & $\{2,3\}$ & $\surd$ \\
\hline
$i=2$ & $\{1,2,3\}$ &  & $\{1,2,3\}$ & $\surd$ \\
\hline
$i=3$ & $\{1,2,3\}$ &  & $\{1,2,3\}$ & $\surd$ \\
\hline
$i=4$ & $\{5,6\}$ & $\surd$ & $\{1,4,5,6\}$ & \\
\hline
$i=5$ & $\{4,6\}$ & $\surd$ & $\{1,4,5,6\}$ & \\
\hline
$i=6$ & $\{4,5\}$ & $\surd$ & $\{1,4,5,6\}$ & 
\end{tabular}
\end{center}

By the same discussions as in \S\ref{subsec:5-5}, there are at most one $i\in\{4,5,6\}$ such that $1\in M_{W,i}^*$, 
and hence $M_{W,j}^*\subset\{4,5,6\}$ for other $j$ in $\{4,5,6\}$. In this case, we obtain mass relation $f_{6,18,i}=0$ 
where $f_{6,18,i}$ is the same as $f_{5,4,1}$ in \S\ref{subsec:5-4} except subindices were changed from $3,4,5$ to $j,k,i$ 
with $\{i,j,k\}=\{4,5,6\}$.  Therefore we find mass relation $f_{6,18}=0$ where $f_{6,18}$ is the same as $f_{5,5}$ 
in \S\ref{subsec:5-5}, except subindices are shifted by 1.

\subsection{Diagram 11}\label{subsec:6-11}
We obtain five members $(S_2, T_4), (S_3, T_4), (S_4, T_4), (S_4, T_3), (S_4, T_2)$ in $\coll_2$ where 
$$
S_k=\begin{pmatrix}
\!\{0\!\sim\! 4\}\! & \{2\} & \{2\} & \{k\} & \{4\} & \{4\} \\
\{2\} & \!\{0\!\sim\! 4\}\! & \{2\} & \{k\} & \{4\} & \{4\} \\
\{2\} & \{2\} & \!\{0\!\sim\! 4\}\! & \{k\} & \{4\} & \{4\} \\
\{k\} & \{k\} & \{k\} & \!\{0\!\sim\! 4\}\! & \{4\} & \{4\} \\
\{4\} & \{4\} & \{4\} & \{4\} & \!\{0\!\sim\! 4\}\! & \{4\} \\
\{4\} & \{4\} & \{4\} & \{4\} & \{4\} & \!\{0\!\sim\! 4\}\! 
\end{pmatrix}$$
and
$$ 
T_k=\begin{pmatrix}
\!\{0\!\sim\! 4\}\! & \{4\} & \{4\} & \{k\} & \{k\} & \{k\} \\
\{4\} & \!\{0\!\sim\! 4\}\! & \{4\} & \{4\} & \{4\} & \{4\} \\
\{4\} & \{4\} & \!\{0\!\sim\! 4\}\! & \{4\} & \{4\} & \{4\} \\
\{k\} & \{4\} & \{4\} & \!\{0\!\sim\! 4\}\! & \{2\} & \{2\} \\
\{k\} & \{4\} & \{4\} & \{2\} & \!\{0\!\sim\! 4\}\! & \{2\} \\
\{k\} & \{4\} & \{4\} & \{2\} & \{2\} & \!\{0\!\sim\! 4\}\! 
\end{pmatrix}
$$
for $k=2,3,4$.

$(S_2, T_4)$ and $(S_3, T_4)$ give $M_{Z,5}=\{4,6\}$, $M_{Z,6}=\{4,5\}$, and $M_{W,5}=M_{W,6}=\{4,5,6\}$. 
It's similar to diagram 4 in case $n=5$, we have mass relations $f_{6,11,1}=0$ where $f_{6,11,1}$ is the same as $f_{5,4,1}$, 
except subindices were changed from $3,4,5$ to $5,6,4$ respectively.

On the other hand, $(S_2, T_4)$ and $(S_3, T_4)$ also give $M_{Z,1}=M_{Z,2}=M_{Z,3}=\{1,2,3,4\}$ 
and $M_{W,i}=\{j,k\}$ for $\{i,j,k\}=\{1,2,3\}$. It's similar to diagram 18 in case $n=6$. 
By similar discussions, we have mass relation $f_{6,11,2}=0$ where $f_{6,11,2}$ is the same as $f_{6,18}$, 
expect subindices were changed from $4,5,6$ to $1,2,3$ respectively.

The cases $(S_4, T_2)$ and $(S_4, T_3)$ are similar, we have the same mass relation.

For $(S_4,T_4)$, the maximal order sets are shown as the following table
\begin{center}
\begin{tabular}{c|c|c|c|c}
 & $M_{Z,i}$ & certainly max? & $M_{W,i}$ & certainly max?  \\
\hline
$i=1$ & $\{1,2,3,4,5,6\}$ &  & $\{2,3\}$ & $\surd$ \\
\hline
$i=2$ & $\{1,2,3,4,5,6\}$ &  & $\{1,3\}$ & $\surd$ \\
\hline
$i=3$ & $\{1,2,3,4,5,6\}$ &  & $\{1,2\}$ & $\surd$ \\
\hline
$i=4$ & $\{5,6\}$ & $\surd$ & $\{1,2,3,4,5,6\}$ & \\
\hline
$i=5$ & $\{4,6\}$ & $\surd$ & $\{1,2,3,4,5,6\}$ & \\
\hline
$i=6$ & $\{4,5\}$ & $\surd$ & $\{1,2,3,4,5,6\}$ & 
\end{tabular}
\end{center}

Suppose $p\in \{1,2,3\}\cap M_{W,i}^*$ for some $i\in\{4,5,6\}$. Let $j\in\{4,5,6\}$ different from $i$. Then $z_{pi}\prec z_{ij}\approx z_{45}$ 
and $w_{pi}\succ w_{ij}\approx w_{45}$ since $Z_{pi}\prec Z_{ij}$ and $W_{pi}\approx W_{ij}$, 
and it follows that $z_{pj}\approx z_{ij}\approx z_{45}$ and $w_{pj}\approx w_{pi} \succ w_{ij}\approx w_{45}$. 
Therefore $W_{pj}\prec W_{pi}\approx W_{ij}$ which means $p\notin M_{W,j}^*$. 

Let $q\in\{1,2,3\}$ different from $p$. Since $z_{pi}\succ z_{pq}$, we have $z_{qi}\approx z_{pi}\prec z_{ij}$. 
By the same discussions as in the last paragraph, we obtain $z_{qj}\approx z_{ij}$, and it follows that $W_{qj}\prec W_{ij}$ 
which means $q\notin M_{W,j}^*$. Therefore there is at most one $i\in\{4,5,6\}$ such that $M_{W,i}^*\cap\{1,2,3\}\neq \emptyset$.

It's similar to diagram 4 in case $n=5$, we will obtain mass relation $f_{6,11,3,i}=0$ where $f_{6,11,3,i}$ is the same 
as $f_{5,4,1}$ in \S\ref{subsec:5-4}, except subindices were changed from $3,4,5$ to $j,k,i$ respectively where $\{i,j,k\}=\{4,5,6\}$.

On the other hand, since $(S_4,T_4)$ is invariant under exchanging $1,2,3,z$ and $4,5,6,w$ respectively, 
we also have mass relation $f_{6,11,4,i}=0$ where $f_{6,11,3,i}$ is the same as $f_{5,4,1}$ in \S\ref{subsec:5-4}, 
except subindices were changed from $3,4,5$ to $j,k,i$ respectively where $\{i,j,k\}=\{1,2,3\}$.

Therefore we find mass relation $f_{6,11}=0$ where $f_{6,11}$ is the least common multiple of 
$f_{6,11,1}$, $f_{6,11,2}$, $f_{6,11,3,i}$, $f_{6,11,4,j}$, $i\in\{4,5,6\}$ and $j\in\{1,2,3\}$.

\subsection{Diagram 16}\label{subsec:6-16}
For diagram 16, we obtain the optimal Type 1 $zw$-order matrix:
\beq
\left(\! \begin{pmatrix}
\!\!\{0 \!\sim\! 4\}\!\! & \!\!\{2 \!\sim\! 4\}\!\! & \!\!\{2 \!\sim\! 4\}\!\! & \!\!\{2 \!\sim\! 4\}\!\! & \!\!\{2 \!\sim\! 4\}\!\! & \!\!\{2 \!\sim\! 4\}\!\! \\
\!\!\{2 \!\sim\! 4\}\!\! & \!\!\{0 \!\sim\! 4\}\!\! & \!\{1\}\! & \!\{4\}\! & \!\{4\}\! & \!\{1\}\! \\
\!\!\{2 \!\sim\! 4\}\!\! & \!\{1\}\! & \!\!\{0 \!\sim\! 4\}\!\! & \!\{4\}\! & \!\{4\}\! & \!\{4\}\! \\
\!\!\{2 \!\sim\! 4\}\!\! & \!\{4\}\! & \!\{4\}\! & \!\!\{0 \!\sim\! 4\}\!\! & \!\{4\}\! & \!\{4\}\! \\
\!\!\{2 \!\sim\! 4\}\!\! & \!\{4\}\! & \!\{4\}\! & \!\{4\}\! & \!\!\{0 \!\sim\! 4\}\!\! & \!\{4\}\! \\
\!\!\{2 \!\sim\! 4\}\!\! & \!\{1\}\! & \!\{4\}\! & \!\{4\}\! & \!\{4\}\! & \!\!\{0 \!\sim\! 4\}\!\! 
\end{pmatrix}\!\!, \!\!
\begin{pmatrix}
\!\!\{0 \!\sim\! 4\}\!\! & \!\{5\}\! & \!\{5\}\! & \!\!\{2 \!\sim\! 4\}\!\! & \!\!\{2 \!\sim\! 4\}\!\! & \!\!\{2 \!\sim\! 4\}\!\! \\
\!\{5\}\! & \!\{5\}\! & \!\{5\}\! & \!\{5\}\! & \!\{5\}\! & \!\{5\}\! \\
\!\{5\}\! & \!\{5\}\! & \!\{5\}\! & \!\{2\}\! & \!\{5\}\! & \!\{5\}\! \\
\!\!\{2 \!\sim\! 4\}\!\! & \!\{5\}\! & \!\{2\}\! & \!\!\{0 \!\sim\! 4\}\!\! & \!\{2\}\! & \!\{2\}\! \\
\!\!\{2 \!\sim\! 4\}\!\! & \!\{5\}\! & \!\{5\}\! & \!\{2\}\! & \!\!\{0 \!\sim\! 4\}\!\! & \!\{2\}\! \\
\!\!\{2 \!\sim\! 4\}\!\! & \!\{5\}\! & \!\{5\}\! & \!\{2\}\! & \!\{2\}\! & \!\!\{0 \!\sim\! 4\}\!\! 
\end{pmatrix} \!\right)
\eeq
and the produced $r$-order matrix
$$
\begin{pmatrix}
\{\} & \{4\} & \{4\} & \!\{2 \!\sim\! 4\}\! & \!\{2 \!\sim\! 4\}\! & \!\{2 \!\sim\! 4\}\! \\
\{4\} & \{\} & \{3\} & \{4\} & \{4\} & \{3\} \\
\{4\} & \{3\} & \{\} & \{4\} & \{4\} & \{3\} \\
\!\{2 \!\sim\! 4\}\! & \{4\} & \{4\} & \{\} & \{2\} & \{2\} \\
\!\{2 \!\sim\! 4\}\! & \{4\} & \{4\} & \{2\} & \{\} & \{2\} \\
\!\{2 \!\sim\! 4\}\! & \{3\} & \{3\} & \{2\} & \{2\} & \{\} 
\end{pmatrix}
$$

The maximal order sets are shown as the following table:
\begin{center}
\begin{tabular}{c|c|c|c|c}
 & $M_{Z,i}$ & certainly max? & $M_{W,i}$ & certainly max?  \\
\hline
$i=1$ & $\{1,4,5,6\}$ &  & $\{1,4,5,6\}$ & \\
\hline
$i=2$ & $\{2,3,6\}$ &  & $\{2,3,6\}$ & $\surd$ \\
\hline
$i=3$ & $\{2,3,6\}$ &  & $\{2,3,6\}$ & $\surd$ \\
\hline
$i=4$ & $\{5,6\}$ & $\surd$ & $\{1,4,5,6\}$ & \\
\hline
$i=5$ & $\{4,6\}$ & $\surd$ & $\{1,4,5,6\}$ & \\
\hline
$i=6$ & $\{4,5\}$ & $\surd$ & $\{2,3\}$ & 
\end{tabular}
\end{center}

From the same reason as in \S\ref{subsec:5-5}, $1\notin M_{W,4}^*$ and $1\notin M_{W,5}^*$. 
In addition, by the same discussions as in \S\ref{subsec:5-4}, we also have that $4\in M_{W,4}^*$ if and only if $5\in M_{W,5}^*$. Therefore there are 5 cases for $M_{W,4}^*$ and $M_{W,5}^*$:
\begin{enumerate}
\item[1.] $M_{W,4}^*\subset\{4,5,6\}$ and $M_{W,5}^*\subset \{4,5,6\}$,
\item[2.] $M_{W,4}=\{1,5,6\}$ and $M_{W,5}^*\subset \{4,6\}$,
\item[3.] $M_{W,4}=\{5,6\}$ and $M_{W,5}^*\subset \{1,4,6\}$,
\item[4.] $M_{W,4}=\{1,4,5,6\}$ and $M_{W,5}^*\subset \{4,5,6\}$,
\item[5.] $M_{W,4}=\{4,5,6\}$ and $M_{W,5}^*\subset \{1,4,5,6\}$.
\end{enumerate}

The case 1 is similar to diagram 4 in case $n=5$. We have mass relation $f_{6,16,1}=0$ where $f_{6,16,1}$ is the same as $f_{5,4,1}$ in \S\ref{subsec:5-4}, except subindices are shifted by 1.

For case 2, we collect $z_4$- and $z_5$- equations, relations between separations, and the equation $m_4 W_{45}+ m_6 W_{65}=0$. Then we make substitution according to cluster equations and apply elimination process. We obtain mass relation $f_{6,16,2}=0$ with the following factors 
\small{
\beq
&& m_4-m_6, \\ %
&& \mu_5(m_4,m_5,m_6), \\ 
&& m_4^8 m_5^2-m_4^8 m_6^2-4 m_4^6 m_5^2 m_6^2-3 m_4^6 m_6^4+6 m_4^4 m_5^2 m_6^4-4 m_4^2 m_5^2 m_6^6+m_5^2 m_6^8, \\ %
&&  4 m_4^8 m_5^2-4 m_4^8 m_6^2-m_4^6 m_5^2 m_6^2-12 m_4^6 m_6^4-9 m_4^4 m_5^2 m_6^4+5 m_4^2 m_5^2 m_6^6+m_5^2 m_6^8 ,  \\ %
&&m_4^8 m_5^2-m_4^8 m_6^2-4 m_4^6 m_5^2 m_6^2-31 m_4^6 m_6^4+6 m_4^4 m_5^2 m_6^4-31 m_4^4 m_6^6-4 m_4^2 m_5^2 m_6^6-m_4^2 m_6^8+m_5^2 m_6^8 ,  \\ %
&& 3 m_4^8 m_5^4-43 m_4^8 m_5^2 m_6^2-12 m_4^6 m_5^4 m_6^2+40 m_4^8 m_6^4 - \cdots \text{(8 other terms)}, \\ %
&& 45 m_4^8 m_5^4+955 m_4^8 m_5^2 m_6^2-180 m_4^6 m_5^4 m_6^2-1000 m_4^8 m_6^4+ \cdots \text{(8 other terms)},  \\ %
&& 3 m_4^8 m_5^4-24 m_4^8 m_5^2 m_6^2-12 m_4^6 m_5^4 m_6^2+21 m_4^8 m_6^4 + \cdots \text{(8 other terms)},  \\ %
&& 216 m_4^8 m_5^6-861 m_4^8 m_5^4 m_6^2-864 m_4^6 m_5^6 m_6^2+2145 m_4^8 m_5^2 m_6^4 + \cdots \text{(12 other terms)}.  \\ %
\eeq
}
\normalsize

For case 3, we collect $z_4$- and $z_5$-equations, relations between separations, and the equation $m_5 W_{54}+ m_6 W_{64}=0$. 
Then we make substitution according to cluster equations and apply elimination process. 
We obtain mass relation $f_{6,16,2}=0$ where $f_{6,16,3}=\mu_5(m_4,m_5,m_6)\,(5 m_4^2 m_5^2-3 m_4^2 m_6^2+3 m_5^2 m_6^2)$.

In case 4, we have $z_{14}\prec z_{45}\approx z_{46}$, and therefore $z_{15}=z_{14}+z_{45}\approx z_{45}$ 
and $z_{16}=z_{14}+z_{46}\approx z_{46}$. 
From the relation $W_{ij}=r_{ij}^{-1}z_{ij}^{-1}$, $z_{14}\prec z_{45}$, and $W_{14}\prec W_{45}$, we obtain $r_{14}\succ r_{45}$. 
Since $w_{ij}=W_{ij}r_{ij}^3$, we have $w_{14}\succ w_{45}$, and it follows that $w_{14}\approx w_{15}\approx w_{16}$. 
Therefore 
$$
Z_{14}\succ Z_{15}\approx Z_{16}, \quad
W_{14}\succ W_{15}\approx W_{16},
$$
which yields $M_{Z,1}^*=\{1,4\}=M_{W,1}^*$.

On the other hand, since $r_{23}\approx r_{26}\approx r_{36} \approx 1$ and $z_{23}\approx z_{26}\approx z_{36}\approx \epsilon^2$, 
we have $Z_{23}\approx Z_{26}\approx Z_{36}$. From above table, we know that $M_{Z,2}^*,M_{Z,3}^*\subset \{2,3,6\}$. 
Therefore $z_2,z_3\preceq Z_{23}\approx z_{23}\epsilon^2$, and it follows that $z_6=z_2-z_{26}\preceq \epsilon^2$.

Since $z_{46}\succ \epsilon^2\approx z_6$, we have $z_4\approx z_{46}\succ z_{14}$, and it follows that $z_1\sim z_4$. 
From the relation $z_1\sim Z_{14}=z_{14}r_{14}^{-3}$, we obtain $r_{14}\prec 1$. 
It follows that $w_1\approx W_{14}=w_{14}r_{14}^{-3}\succ w_{14}$, and hence $w_1\sim w_4$.

We collect $z_4$-, $z_5$-equations, relations between separations, and the following additional equations
\beq
w_4&=&m_1W_{14}+m_5W_{54}+m_6W_{64},\\
w_5&=&m_4W_{45}+m_6W_{65},\\
w_1&=&m_4 W_{41},\\
w_1&=&w_4,\\
w_4&=&w_5.
\eeq
Then we make substitution according to cluster equations, apply elimination process, and obtain mass relation $f_{6,16,4}=0$ 
with the following factors 
\small{
\beq
&& -m_1 m_4^2-m_4^3+m_1 m_6^2+m_4 m_6^2+m_5 m_6^2, \\ %
&& \mu_5(m_4,m_5,m_6), \\ 
&& m_1^3 m_4^8 m_5^2+3 m_1^2 m_4^9 m_5^2+3 m_1 m_4^10 m_5^2+m_4^11 m_5^2 - \cdots \text{(60 other terms)}, \\ %
&& \mu_{111}(m_1,m_4,m_5,m_6),  \\ %
&& -3 m_1^4 m_4^8 m_5^4 - 12 m_1^3 m_4^9 m_5^4 - 18 m_1^2 m_4^{10} m_5^4 - 
 12 m_1 m_4^{11} m_5^4- \cdots \text{(124 other terms)}.   %
\eeq
}
\normalsize

The case 5 is similar to case 4. We We collect $z_4$-, $z_5$-equations, relations between separations, 
and the following additional equations
\beq
w_4&=&m_5W_{54}+m_6W_{64},\\
w_5&=&m_1 W_{15}+ m_4W_{45}+m_6W_{65},\\
w_1&=&m_5 W_{51},\\
w_1&=&w_5,\\
w_4&=&w_5.
\eeq
Then we make substitution according to cluster equations, apply elimination process, and obtain mass relation $f_{6,16,5}=0$ where 
\beq
f_{6,16,5}&=&(-m_4^3+m_1 m_6^2+m_4 m_6^2+m_5 m_6^2)\mu_5(m_4,m_5,m_6)\mu_{45}(m_1,m_4,m_5,m_6)\\
& &\mu_{88}(m_1,m_4,m_5,m_6)\mu_{111}(m_1,m_4,m_5,m_6).
\eeq

Notice that the collected equations in case 2 and case 3 are the same, expect the subindices 4 and 5 are exchanged. 
Therefore in case 2, we also have the mass relation that obtained by exchanging $m_4$ and $m_5$ from $f_{6,16,3}=0$, 
and hence we have $\mu_5(m_4,m_5,m_6)=0$ in cases 2,3. Similarly for cases 4 and 5, we have mass relation 
$$\mu_5(m_4,m_5,m_6)\mu_{111}(m_1,m_4,m_5,m_6)=0.$$ 
Therefore we obtain mass relation $f_{6,16}=0$ for diagram 16 where 
$$f_{6,16}=f_{6,16,1}\,\mu_{111}(m_1,m_4,m_5,m_6).$$

\subsection{Diagram 14}\label{subsec:6-14}
For diagram 14, we obtain the optimal Type 1 $zw$-order matrix:
\beq &  
\small{
\left(\!\!  \begin{pmatrix}
\!\!\{2\!\sim\! 4\}\!\! & \!\!\{2\!\sim\! 4\}\!\! & \!\!\{2\!\sim\! 4\}\!\! & \!\!\{2\!\sim\! 4\}\!\! & \!\!\{2\!\sim\! 4\}\!\! & \!\!\{2\!\sim\! 4\}\!\! \\
\!\!\{2\!\sim\! 4\}\!\! & \!\!\{0,1\}\!\! & \!\{1\}\! & \!\!\{2\!\sim\! 4\}\!\! & \!\!\{2\!\sim\! 4\}\!\! & \!\!\{2\!\sim\! 4\}\!\! \\
\!\!\{2\!\sim\! 4\}\!\! & \!\{1\}\! & \!\{0,1\}\! & \!\!\{2\!\sim\! 4\}\!\! & \!\!\{2\!\sim\! 4\}\!\! & \!\!\{2\!\sim\! 4\}\!\! \\
\!\!\{2\!\sim\! 4\}\!\! & \!\!\{2\!\sim\! 4\}\!\! & \!\{4\}\! & \!\!\{2\!\sim\! 4\}\!\! & \!\{4\}\! & \!\{4\}\!  \\
\!\!\{2\!\sim\! 4\}\!\! & \!\!\{2\!\sim\! 4\}\!\! & \!\{4\}\! & \!\{4\}\! & \!\!\{2\!\sim\! 4\}\!\! & \!\{4\}\!  \\
\!\!\{2\!\sim\! 4\}\!\! & \!\!\{2\!\sim\! 4\}\!\! & \!\{4\}\! & \!\{4\}\! & \!\{4\}\! & \!\!\{2\!\sim\! 4\}\!\! 
\end{pmatrix}\!\!,\!\! 
\begin{pmatrix}
\!\!\{0\!\sim\! 4\}\!\! & \!\{5\}\! & \!\{5\}\! & \!\!\{2\!\sim\! 4\}\!\! & \!\!\{2\!\sim\! 4\}\!\! & \!\!\{2\!\sim\! 4\}\!\! \\
\!\{5\}\! & \!\{5\}\! & \!\{5\}\! & \!\{5\}\! & \!\{5\}\! & \!\{5\}\! \\
\!\{5\}\! & \!\{5\}\! & \!\!\{0 \!\sim\! 4\}\!\! & \!\{5\}\! & \!\{5\}\! & \!\{5\}\! \\
\!\!\{2\!\sim\! 4\}\!\! & \!\{5\}\! & \!\{5\}\! & \!\!\{0 \!\sim\! 4\}\!\! & \!\{2\}\! & \!\{2\}\! \\
\!\!\{2\!\sim\! 4\}\!\! & \!\{5\}\! & \!\{5\}\! & \!\{5\}\! & \!\!\{0 \!\sim\! 4\}\!\! & \!\{2\}\! \\
\!\!\{2\!\sim\! 4\}\!\! & \!\{5\}\! & \!\{5\}\! & \!\{5\}\! & \!\{2\}\! & \!\!\{0 \!\sim\! 4\}\!\!
\end{pmatrix}\!\!\right)\!. 
} 
\eeq
and the produced $r$-order matrix
$$
\begin{pmatrix}
\{\} & \{4\} & \{4\} & \!\{2 \!\sim\! 4\}\! & \!\{2 \!\sim\! 4\}\! & \!\{2 \!\sim\! 4\}\! \\
\{4\} & \{\} & \{3\} & \{4\} & \{4\} & \{3\} \\
\{4\} & \{3\} & \{\} & \{4\} & \{4\} & \{3\} \\
\!\{2 \!\sim\! 4\}\! & \{4\} & \{4\} & \{\} & \{2\} & \{2\} \\
\!\{2 \!\sim\! 4\}\! & \{4\} & \{4\} & \{2\} & \{\} & \{2\} \\
\!\{2 \!\sim\! 4\}\! & \{3\} & \{3\} & \{2\} & \{2\} & \{\} 
\end{pmatrix}
$$

The maximal order sets are shown as the following table:
\begin{center}
\begin{tabular}{c|c|c|c|c}
 & $M_{Z,i}$ & certainly max? & $M_{W,i}$ & certainly max?  \\
\hline
$i=1$ & $\{1,4,5,6\}$ &  & $\{1,2,3,4,5,6\}$ & \\
\hline
$i=2$ & $\{2,3\}$ & $\surd$  & $\{2,3\}$ & $\surd$ \\
\hline
$i=3$ & $\{2,3\}$ & $\surd$  & $\{2,3\}$ & $\surd$ \\
\hline
$i=4$ & $\{5,6\}$ & $\surd$ & $\{1,2,3,4,5,6\}$ & \\
\hline
$i=5$ & $\{4,6\}$ & $\surd$ & $\{1,2,3,4,5,6\}$ & \\
\hline
$i=6$ & $\{4,5\}$ & $\surd$ & $\{1,2,3,4,5,6\}$ & 
\end{tabular}
\end{center}

By the same discussions for the last case in \S\ref{subsec:5-5}, there is at most one $i\in\{4,5,6\}$ 
such that $1\in M_{W,i}^*$ and at most one $j\in\{4,5,6\}$ such that $\{2,3\}\cap M_{W,j}^*\neq \emptyset$.

Therefore $M_{W,i}^*\subset \{1,4,5,6\}$ and $M_{W,j}^*\subset \{4,5,6\}$ for some distinct $i,j\in\{4,5,6\}$. Let $\{i,j,k\}=\{4,5,6\}$

If $M_{W,p}^*\subset \{4,5,6\}$ for two of $i,j,k$, then it's the same as \S\ref{subsec:5-4}, we have mass relation $f_{6,14,1}=0$ where $f_{6,14,1}$ is the same as $f_{5,4,1}$ in \S\ref{subsec:5-4}, expect subindices are shifted by 1.

If $M_{W,i}^*,M_{W,k}^*\nsubseteq\{4,5,6\}$. Then $2,3\in M_{W,k}^*$ and there are two subcases:
\begin{enumerate} 
\item[1.] $M_{W,i}^*= \{1,5,6\}$ and $M_{W,j}^*\subset \{4,6\}$
\item[2.] $M_{W,i}^*= \{1,4,5,6\}$ and $M_{W,j}^*= \{4,5,6\}$
\end{enumerate}

Case 1 is similar to the case 2 in \S\ref{subsec:6-16}, and we obtain mass relation $f_{6,14,2}=0$ where $f_{6,14,2}=f_{6,16,2}$.
In case 2, by the same discussions for the last case in \S\ref{subsec:5-5}, we have
\beq
z_{1i}&\prec &z_{1j}\approx z_{1k}\approx z_{ij}\\
w_{1i}&\approx &w_{1j} \approx w_{ik} \\
z_{pk}&\prec z_{pi}&\approx z_{pj}\approx z_{ij}
\eeq
for $p=2,3$.
It follows that $z_{12}\approx z_{13}\approx z_{1j}\approx z_{ij}$. Since $z_1\succ z_2,z_3$, we obtain $z_1\approx z_{12}\approx z_{13}\succ z_{1i}$. It follows that
\beq
z_{12}\approx z_{13}&\approx &z_{1j} \approx z_{1k}\succ w_{1i}, \\
w_{12}\approx w_{13}&\succ &w_{1i} \approx w_{1j}\approx w_{1k}, 
\eeq
we have $Z_{1i}\succ Z_{1p}$ and $W_{1i}\succ W_{1p}$ for $p=2,3,j,k$, which means that $$M_{Z,1}^*=M_{W,1}^*=\{1,i\}.$$

On the other hand, since $z_1\succ z_{1i}$, we have $z_1\sim z_i$. 
Since $z_1\approx Z_{1i}$, we have $r_{1i}\prec 1$, and it turns out that $w_{1i}\prec w_{1}$. Therefore $w_{1}\sim w_{i}$ and we have the following equations
\beq
w_j&=&m_iW_{ij}+m_kW_{kj},\\
w_i&=&m_1 W_{1i}+ m_jW_{ji}+m_kW_{ki},\\
w_1&=&m_i W_{i1},\\
w_1&=&w_i,\\
w_i&=&w_j,
\eeq
which is similar to case 4 in \S\ref{subsec:6-16}, we obtain mass relation $f_{6,14,3}=0$ where $f_{6,14,3}$ is the same as $f_{6,16,3}$.

\newpage
\begin{figure}[H]
\centering
\includegraphics[scale=0.33]{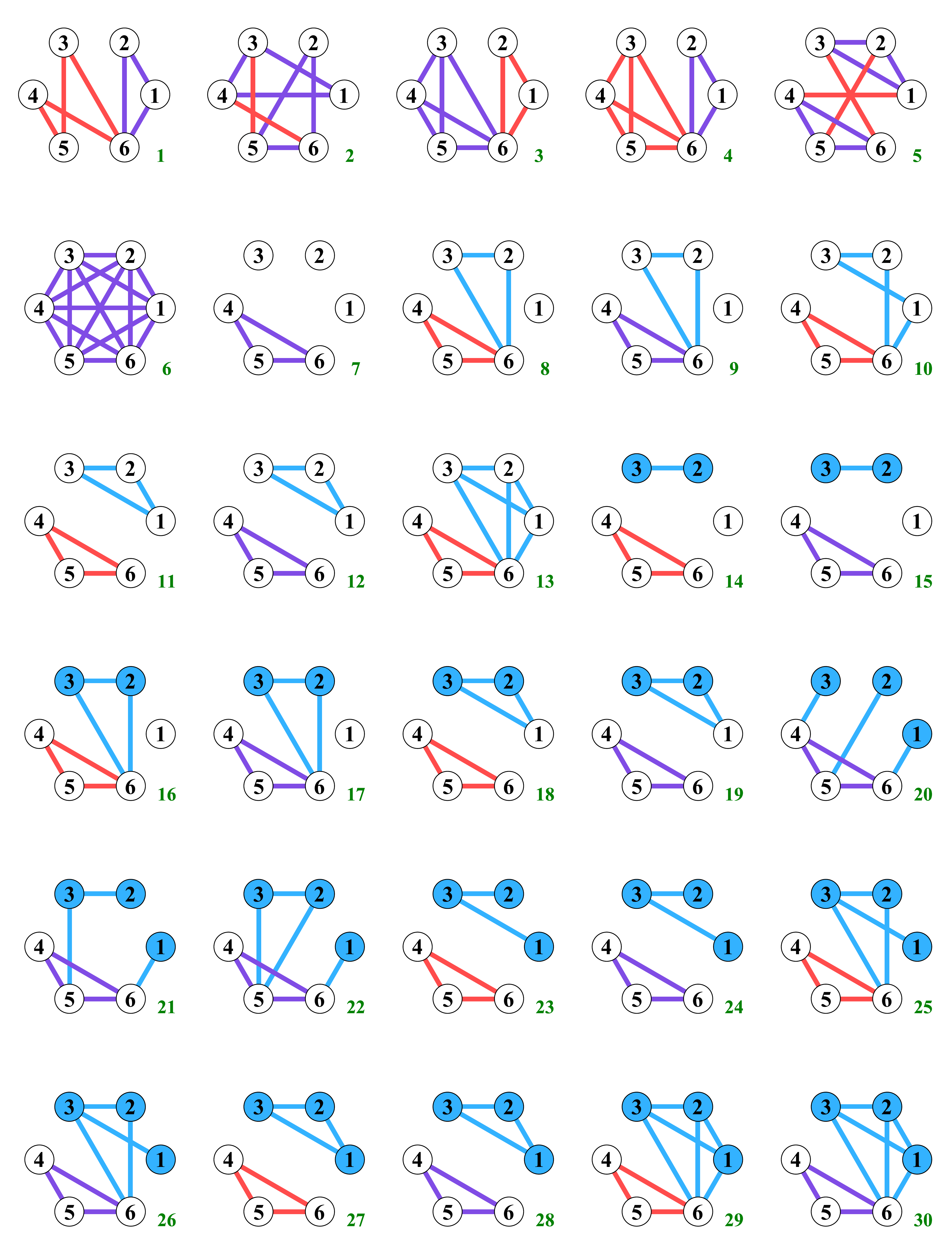}
\caption{Diagrams NO. 1--30 for $n=6$. }   \label{fig:6i}
\end{figure}

\newpage
\begin{figure}[H]
\centering
\includegraphics[scale=0.33]{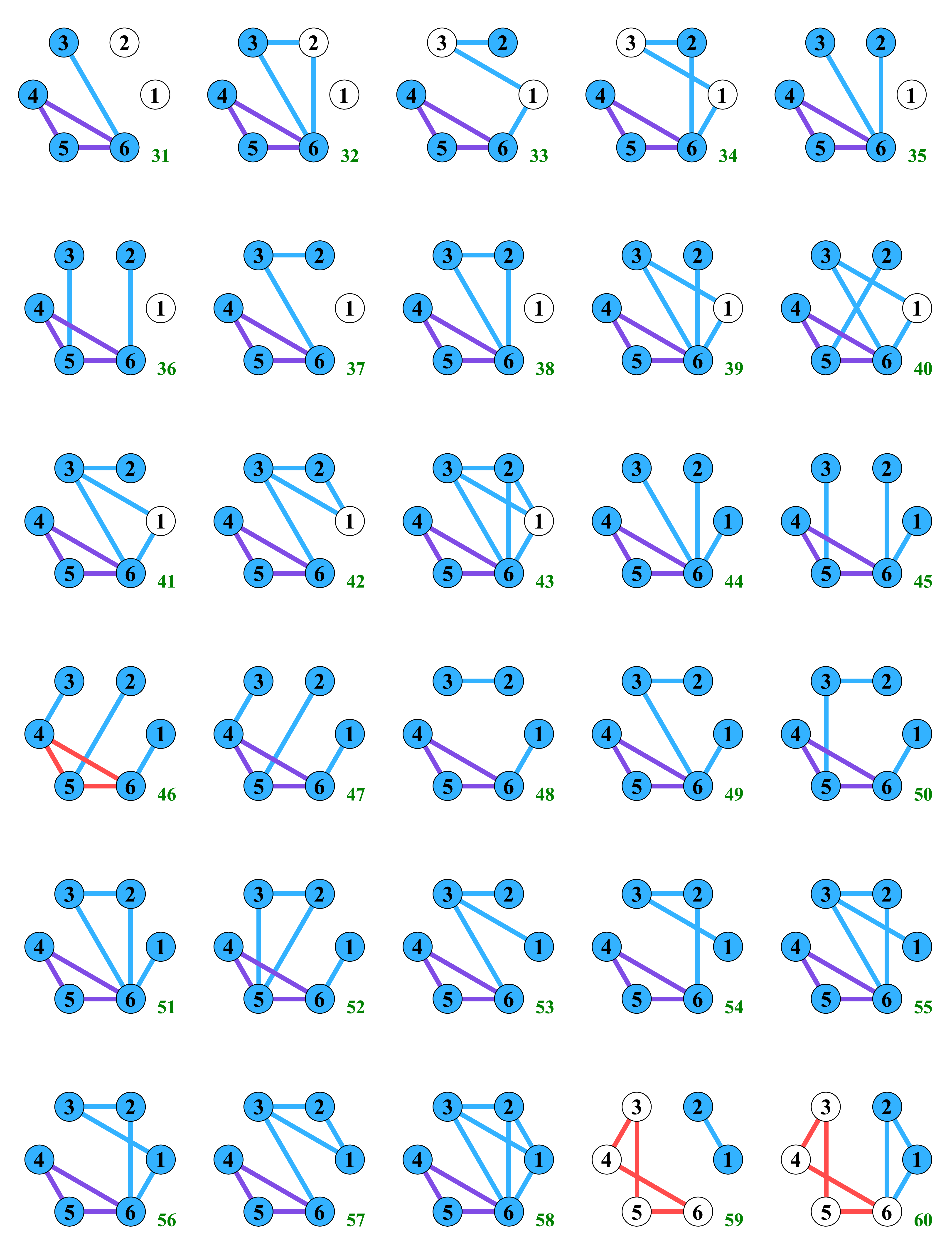}
\caption{Diagrams NO. 31--60 for $n=6$. }   \label{fig:6ii}
\end{figure}

\newpage
\begin{figure}[H]
\centering
\includegraphics[scale=0.33]{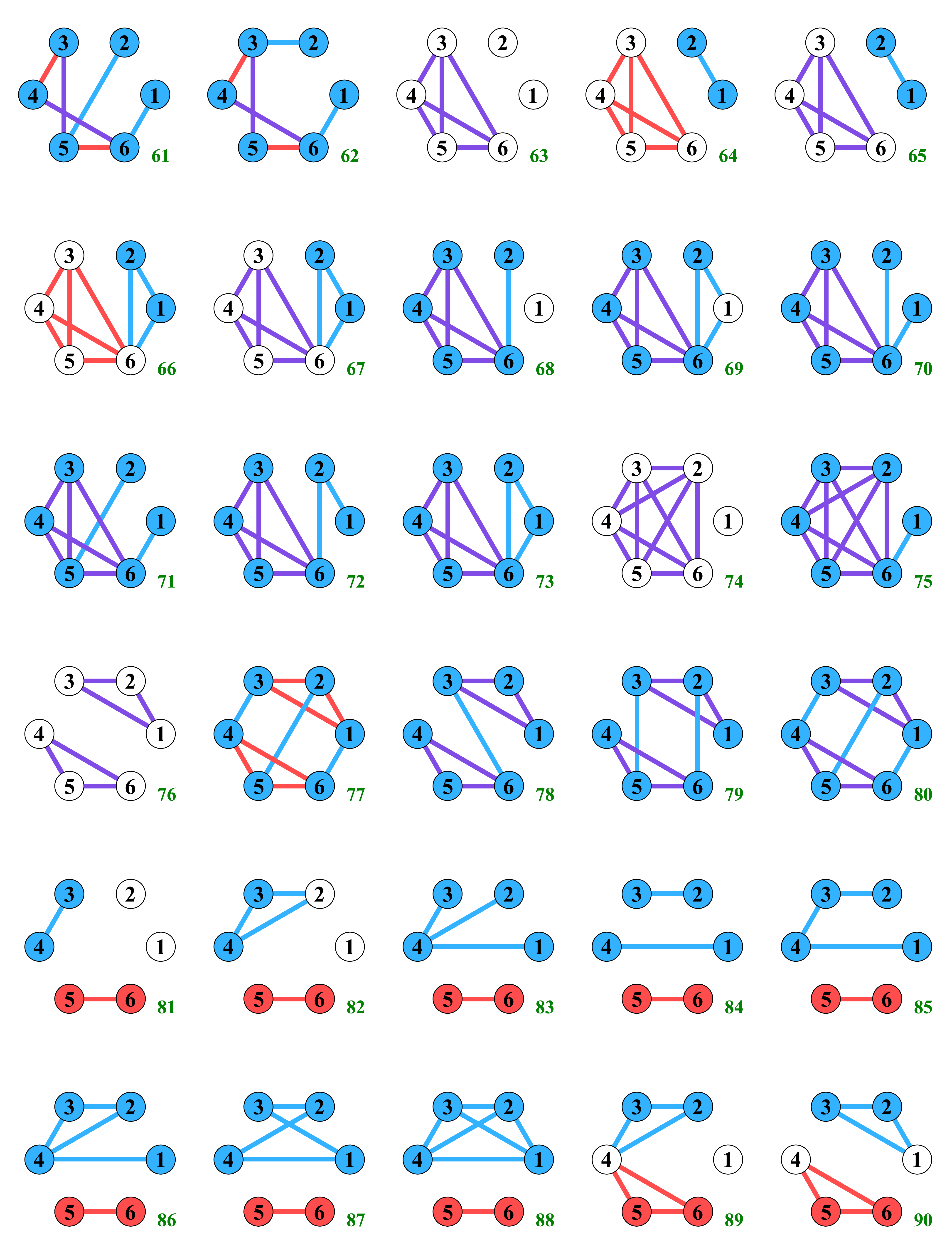}
\caption{Diagrams NO. 61--90 for $n=6$. }   \label{fig:6iii}
\end{figure}

\newpage
\begin{figure}[H]
\centering
\includegraphics[scale=0.33]{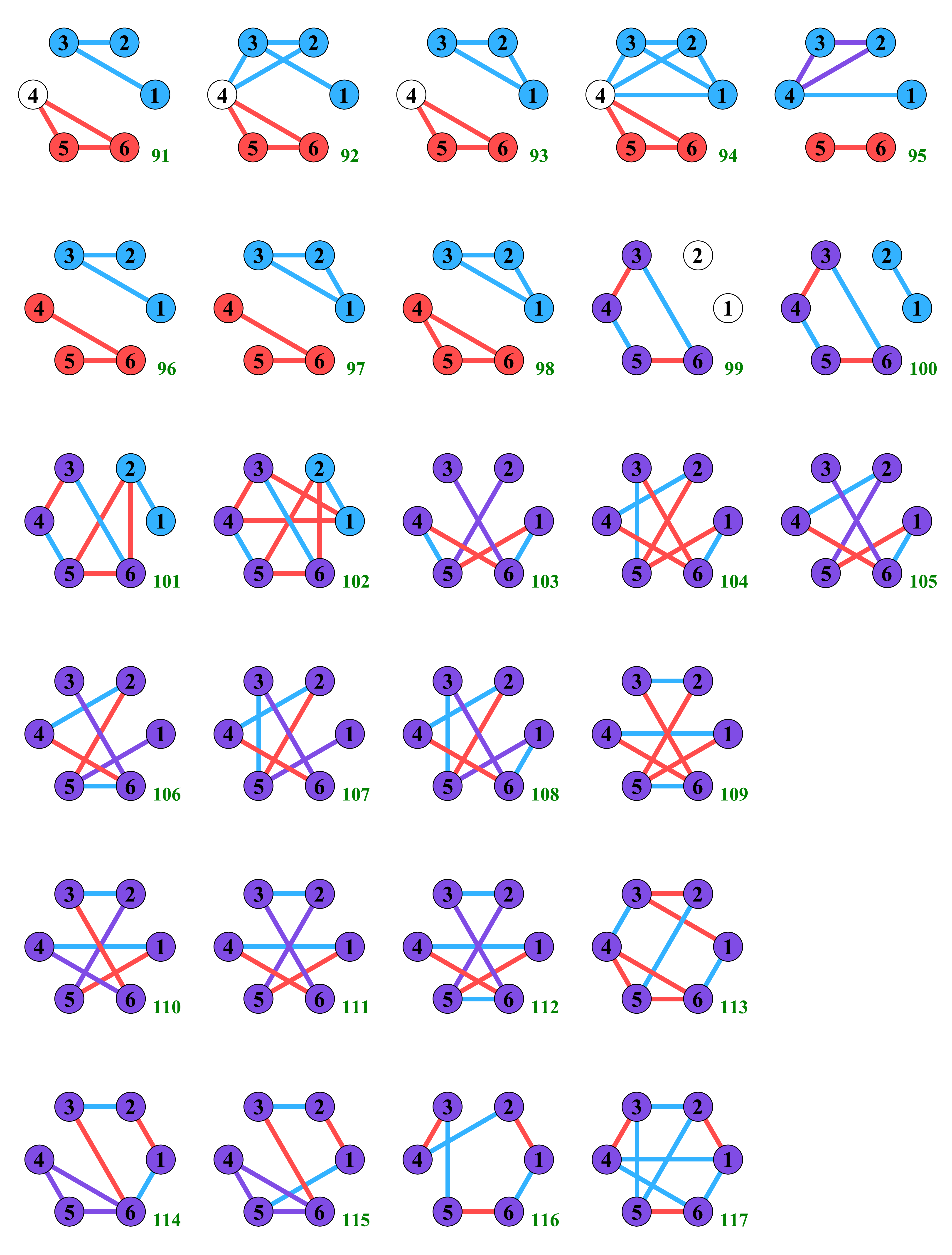}
\caption{Diagrams NO. 91--117 for $n=6$. }   \label{fig:6iv}
\end{figure}

\section*{Appendix I: Program for Algorithm I by Mathematica} \label{sec:program1}

After compiling the program, use function ``{\tt zwmatrices}'' and input the number of bodies to find out $zw$-matrices.  
To monitor performance, one can easily output the number of $zw$-matrices remaining and cumulative time used 
after each subroutine. 
We remove these parts for brevity.  Some remarks were added using ``(* ... *)'' for clarifications. \\

\small
{\tt

\noindent\(
\text{zwmatrices}[\text{nn$\_$}]\text{:=} ( \\
n=\text{nn};  
\text{pm}=\text{Permutations}@\text{Table}[k,\{k,n\}]; \\ 
\text{ConnectedMatrixQ}[\text{x$\_$}]\text{:=} \\
{} \quad \text{Count}[\text{Flatten}[\text{IdentityMatrix}[n]
+\text{Sum}[\text{MatrixPower}[x,k],\{k,1,n-1\}]],0]==0; \\
\text{zpart}[\{\text{z$\_$},\text{w$\_$}\}]\text{:=}z; \text{wpart}[\{\text{z$\_$},\text{w$\_$}\}]\text{:=}w; \\
\text{trianglepart}[\text{x$\_$}]\text{:=}\text{Table}[\text{If}[i\leq j,x[[i,j]],0],\{i,1,n\},\{j,1,n\}]; \\
{}   \\
\text{(* BEGIN SELECTION OF $z$-MATRICES *)} \\
\text{(* Collect all symmetric binary matrices, sort them according to column } \\
{}\quad\;\; \text{sums and diagonals, then delete duplicates  *) }  \\
\text{mxsort}[\text{x$\_$}?\text{ArrayQ}]\text{:=}x[[\#\#]] \& \text{@@} (\text{Ordering}[x\sim\text{Total}\sim\{\#\}] \&\text{/@} \{1,1\}); \\
\text{mxreorder}[\text{x$\_$}?\text{ArrayQ}]\text{:=}x[[\#\#]] \& \text{@@} (\text{Ordering}[\text{Diagonal}[x]] \&\text{/@} \{1,1\});  \\
A=\text{Tuples}[\{0,1\},n(n-1)/2];a=\text{Table}[k(k-1)/2,\{k,1,n+1\}]; \\
S=\{0\}\sim \!\text{Join}\!\sim \text{Table}[k,\{k,2,n\}]; \\
V=\text{Array}[\text{If}[\text{$\#$1}>\text{$\#$2},0,1]\&,\{n+1,n\}]; \\
F[\text{x$\_$}]\text{:=}\text{Array}[\text{If}[\text{$\#$1}\neq \text{$\#$2},x[[\text{    }a[[n-\text{Abs}[\text{$\#$1}-\text{$\#$2}]\text{
  }]]+\text{Min}[\text{$\#$1},\text{$\#$2}]]],0]\&,\{n,n\}]; \\
\text{fmx0}=\text{Table}[F[A[[k]]],\{k,1,\text{Length}[A]\}]; \\
\text{fmx}=\text{Table}[\text{Table}[\text{DiagonalMatrix}[V[[n-j+1]]]+\text{fmx0}[[k]],\{k,1,\text{Length}[\text{fmx0}]\}],\{j,S\}]; \\
{}   \\
\text{Do}[\text{fmxs}=\text{Table}[\text{mxsort}@\text{fmx}[[j]][[k]],\{k,1,\text{Length}[\text{fmx}[[j]]]\}]; \\
\text{fmx}[[j]]=\text{DeleteDuplicates}[\text{fmxs}],\{j,1,n\}]; \\
\text{Do}[\text{fmx}[[j]]=\text{Table}[\text{mxreorder}@\text{fmx}[[j]][[k]],\{k,1,\text{Length}[\text{fmx}[[j]]]\}],\{j,1,n\}]; \\
\text{Do}[\text{fmx}[[j]]=\text{DeleteDuplicates}[\text{fmx}[[j]]],\{j,1,n\}]; \\
{}   \\
\text{(* Rule of Column Sums *)} \\
\text{RuleCSQ}[\text{x$\_$}]\text{:=}\text{Count}[\text{Total}@x,1]==0 \&\&\text{Count}[\text{Total}@x,0]<n;  \\
\text{Do}[\text{fmx}[[k]]=\text{Select}[\text{fmx}[[k]],\text{RuleCSQ}],\{k,1,n\}]; \\
{}   \\
\text{(* First Rule of Trace-2 Matrices *)} \\ 
\text{RuleTr2Part1Q}[\text{x$\_$}]\text{:=}x[[n,n-1]]==1; \\
\text{fmx}[[2]]=\text{Select}[\text{fmx}[[2]],\text{RuleTr2Part1Q}]; \\
\text{twolinesSP}[\text{x$\_$}]\text{:=}\text{Flatten}[\text{Table}[(x[[i,n]]+x[[i,n-1]]+x[[n-1,n]]),\{i,1,n-2\}]]; \\
\text{RuleTr2MxPart2Q}[\text{x$\_$}]\text{:=} ( \text{Count}[\text{twolinesSP}[x],2]==0);  \\
\text{fmx}[[2]]=\text{Select}[\text{fmx}[[2]],\text{RuleTr2MxPart2Q}];   \\
{}   \\
\text{(* Second Rule of Triangles *)}  \\
\text{prtriangle3}[\text{x$\_$}]\text{:=}\text{Flatten}[\text{Table}[ \\ 
{}\quad x[[i,k]]+x[[j,k]]+x[[k,k]],\{k,n-K+1,n\},\{j,1,n-K\},\{i,j+1,n-K\}],2];  \\
\text{RuleTriangle2Q}[\text{x$\_$}]\text{:=}\text{Count}[\text{prtriangle3}@x,3]==0;  \\
\text{Do}[K=k;\text{fmx}[[k]]=\text{Select}[\text{fmx}[[k]],\text{RuleTriangle2Q}],\{k,2,n-2\}];  \\
{}   \\
\text{(* First Rule of Quadrilaterals *)}  \\
\text{foursum}[\text{x$\_$}]\text{:=}\text{Flatten}[\text{Table}[ 
x[[i,j]]+x[[i,k]]+x[[i,l]]+x[[j,k]]+x[[j,l]]+x[[k,l]], \\ 
{} \quad \{l,4,n\},\{k,3,l-1\},\{j,2,k-1\},\{i,1,j-1\}]];  \\
\text{Rule2fCorQ}[\text{x$\_$}]\text{:= Count}[\text{foursum}[x],5]==0; \\
\text{Do}[\text{fmx}[[k]]=\text{Select}[\text{fmx}[[k]],\text{Rule2fCorQ}],\{k,1,n\}];  \\
{}   \\
\text{(* First Rule of Triangles *)}  \\
\text{threelinesSP}[\text{x$\_$}]\text{:=}\text{Flatten}[\text{Table}[ (x[[i,i]]+x[[i,j]]+x[[i,k]]) \\
{}\quad (x[[j,i]]+x[[j,j]]+x[[j,k]]) (x[[k,i]]+x[[k,j]]+x[[k,k]])(x[[i,i]]+x[[j,j]]+x[[k,k]]),  \\
{}\quad \{k,n-K+2,n\},\{j,n-K+1,k-1\},\{i,1,n-K\}]];  \\
\text{RuleTriangle1Q}[\text{x$\_$}]\text{:= Count}[\text{threelinesSP}[x],16]==0;  \\
\text{Do}[K=k;  
\text{fmx}[[k]]=\text{Select}[\text{fmx}[[k]],\text{RuleTriangle1Q}],\{k,2,n-1\}];   \\
{}   \\
\text{(* Rule of Trace-3 Matrices, only needed for  $n\geq 6$. *)} \\
\text{If}[n>5,( \\
\text{RuleTr3MxQ}[\text{x$\_$}]\text{:=} (Z=x; 
i=n-3; \\
\text{While} [i>0,\text{If}[(Z[[i,n-2]]+Z[[i,n-1]]+Z[[i,n]])*\text{Max}[ \\
{}\quad Z[[i,n]]+Z[[n-2,n]]+Z[[n-1,n]], Z[[i,n-1]]+Z[[n-2,n-1]]+Z[[n-1,n]], \\
{}\quad Z[[i,n-2]]+Z[[n-2,n-1]]+Z[[n-2,n]]] \\
{}\quad ==3, 
\text{Return}@\text{False}]; 
i--]; 
\text{Return}@\text{True}); \\
\text{fmx}[[3]]=\text{Select}[\text{fmx}[[3]],\text{RuleTr3MxQ}] ; \\ \\
\text{(* Rule of Trace-0 Principal Minors *)} \\
\text{Rule0DiagSubmxQ}[ \text{x2}\_]:=( \\
 {}\quad S2= \text{Table}[ j2,\{ j2,1, n- K1\}]; \\
 {}\quad \text{If}[ K1==1, S2= S2\sim \text{Join}\sim \{ n\}]; \\
 {}\quad P2= \text{Subsets}[ S2]; P2= \text{Select}[ P2, \text{Length}@\#>2\&];  kk2=1; \\
 {}\quad \text{While}[ kk2 \leq \text{Length}@ P2, \\
 {}\quad \quad \text{s2a}= P2[[ kk2]];  \text{s2b}= \text{Complement}[ \text{allind}, \text{s2a}]; \\
 {}\quad \quad \text{If}[ \text{Count}[ \text{Flatten}@( \text{x2}[[ 
\text{s2a}, \text{s2b}]]),1]==1, \text{Return}[ \text{False}]];  kk2++]; \text{Return}[ \text{True}]) \\
{} \text{Do}[ K1= k1;  \text{fmx}[[ k1]]= \text{Select}[ \text{fmx}[[ k1]], 
\text{Rule0DiagSubmxQ}], \{ k1,1, n-1\}];\\
{}   \\
\text{(* Delete duplicates *)} \\
\text{Do}[ 
\text{dupind}=\{\}; \\
\text{Do}[\text{If}[\text{Length}@\text{Intersection}[\text{Table}[ 
\text{fmx}[[i]][[k]] [[\text{pm}[[j]],\text{pm}[[j]]]],\{j,1,\text{Length}[\text{pm}]\}], \\
{}\quad \text{fmx}[[i]][[1\text{;;}k-1]]]>0,
\text{dupind}=\text{Append}[\text{dupind},\{k\}]],\{k,2,\text{Length}[\text{fmx}[[i]]]\}]; \\
\text{fmx}[[i]]=\text{Delete}[\text{fmx}[[i]],\text{dupind}],\{i,1,n\}];  \\
\text{(* END SELECTION OF $z$-MATRICES *)} \\
{}   \\
\text{(* BEGIN SELECTION OF $zw$-MATRICES *)} \\
\text{(* Produce all possible $w$-matrices, classify them according to traces.*)} \\
\text{smx}=\text{Table}[\{\},\{k,1,n\}]; \\
\text{Do}[\text{Do}[\text{smx}[[k]]=\text{Union}[\text{smx}[[k]],\text{Table} [
   \text{fmx}[[k]][[i]][[\text{pm}[[j]],\text{pm}[[j]]]],\{j,1,\text{Length}[\text{pm}]\}]], \\ 
{}\quad   \{i,1,\text{Length}[\text{fmx}[[k]]]\}],\{k,1,n\}]; \\
\text{smxcnn}=\text{smx}[[1]]; \text{(* collect those with $w$-trace 0 *)} \\
\text{Do}[\text{smx}[[n-k]]=\text{smx}[[n-k]]\sim \!\text{Join}\!\sim \text{smx}[[n-k+1]],\{k,1,n-1\}]; \\
L=\text{Table}[\text{Length}[\text{smx}[[k]]]-\text{Length}[\text{smx}[[k+1]]],\{k,1,n-1\}]; \\
L=\text{Append}[L,\text{Length}@\text{smx}[[n]]]; \\
{}   \\
\text{fmxcnn}=\text{Table}[\{\},\{k,1,n\}];\text{fmxdnn}=\text{fmxcnn}; \\
\text{Do}[\text{fmxcnn}[[k]]=\text{Select}[\text{fmx}[[k]],\text{ConnectedMatrixQ}],\{k,1,n\}]; \\
\text{(* Split fmx into different categories. } \\ {}\quad \text{ Begin with those having connected graphs.*)} \\
\text{Do}[\text{fmxdnn}[[k]]=\text{Complement}[\text{fmx}[[k]],\text{fmxcnn}[[k]]],\{k,1,n\}]; \\
{}   \\
\text{(* Produce all possible $zw$-matrices whose $z$-trace is less than or equal }\\ {}\quad \text{ to $w$-trace. *)} \\
\text{FMXCNN}=\text{Flatten}[\text{Table}[ \\
{}\quad \{\text{fmxcnn}[[1]][[i]],\text{smxcnn}[[j]]\},
\{i,\text{Length}[\text{fmxcnn}[[1]]]\},\{j,\text{Length}[\text{smxcnn}]\}],1]; \\
\text{(* Fully connected $z$-matrix are to be matched with trace-0 $w$-matrices} \\ 
{}\quad \text{ according to the Rule of Fully Connected Companions  *)} \\ 
\text{FMXDNN}=\text{Table}[\{\},\{k,1,n\}]; \\ 
\text{Do}[ 
\text{FMXDNN}[[k]]=\text{Flatten}[\text{Table}[\{\text{fmxdnn}[[k]][[i]],\text{smx}[[k]][[j]]\},
\{i,\text{Length}[\text{fmxdnn}[[k]]]\},\\ 
{}\quad \{j,\text{Length}[\text{smx}[[k]]]\}],1],\{k,1,n\}]; \\ 
{}   \\
\text{(* Rule of Circling *)} \\ 
K=1; \\ 
\text{RuleCirclingWQ}[\text{x$\_$}]\text{:=}(W=\text{wpart}[x]; \\
{}\quad s=\text{Total}[\text{Flatten}@W[[1\text{;;}n-K,n-K+1\text{;;}n]]]; \\
{}\quad \text{If}[s\neq 0,\text{Return}@\text{False},\text{Return}@\text{True}]); \\ 
\text{Do}[K=k;\text{FMXDNN}[[k]]=\text{Select}[\text{FMXDNN}[[k]],\text{RuleCirclingWQ}],\{k,2,n-1\}];  \\ 
\text{RuleCirclingZQ}[\text{x$\_$}]\text{:=}(Z=\text{zpart}[x];  W=\text{wpart}[x]; \\ 
\text{dW}=\text{Diagonal}[W]; \\ 
\text{p0}=\text{Flatten}@\text{Position}[\text{dW},0]; 
\text{p1}=\text{Flatten}@\text{Position}[\text{dW},1]; \\ 
s=\text{Total}[\text{Flatten}@Z[[\text{p0},\text{p1}]]]; \\ 
\text{If}[s\neq 0,\text{Return}@\text{False},\text{Return}@\text{True}]); \\ 
\text{Do}[\text{FMXDNN}[[k]]=\text{Select}[\text{FMXDNN}[[k]],\text{RuleCirclingZQ}],\{k,1,n-1\}];  \\ 
{}   \\
\text{(* Second Rule of Trace-2 Matrices *)} \\ 
\text{Rule2ndTr2MxQ}[\text{x$\_$}]\text{:=}(Z=\text{zpart}[x]; \\ 
\text{If}[\text{Total}[Z[[n]]]==2\&\& \text{Tr}[\text{wpart}[x]]==n-3,\text{Return}@\text{False}]; \text{If}[n==5, W=\text{wpart}[x]; \\ 
\text{If}[\text{Tr}[W]==2,W=\text{mxreorder}[W]; \\ 
\text{If}[W[[5,4]]==1\&\&\text{Total}[W[[5]]]==2\&\& \text{Tr}[Z]==2,\text{Return}@\text{False}]]]; \text{Return}@\text{True}); \\ 
\text{FMXDNN}[[2]]=\text{Select}[\text{FMXDNN}[[2]],\text{Rule2ndTr2MxQ}]; \\ 
{}   \\
\text{(* Combine all classes *)} \\ 
\text{FMX}=\text{FMXCNN};\text{Do}[\text{FMX}=\text{FMX}\sim \!\text{Join}\!\sim \text{FMXDNN}[[k]],\{k,1,n\}]; \\ 
{}   \\
\text{(* Third Rule of Triangles *)} \\ 
\text{RuleTriangle3Q}[\text{x$\_$}]\text{:=}(B=\text{zpart}[x]+\text{wpart}[x];k=n;\\
{}\quad \text{While}[k\geq 3, j=2; \\
{}\qquad \text{While}[j<k, i=1; \\ 
{}\qquad\quad \text{While}[i<j,b=B[[i,j]]+B[[j,k]]+B[[k,i]];\text{If}[b==4\|b==5,\text{Return}@\text{False}]; \\
{}\qquad\qquad i\text{++}];j\text{++}];k\text{--}]; \text{Return}@\text{True}); \\
\text{FMX}=\text{Select}[\text{FMX},\text{RuleTriangle3Q}]; \\
{}   \\
\text{(* Rule of Two Column Sums *)} \\
\text{Rule2CSQ}[\text{x$\_$}]\text{:=}(Z=\text{zpart}[x]; 
Z=Z-\text{DiagonalMatrix}[\text{Diagonal}@Z]; \\
W=\text{wpart}[x]; 
W=W-\text{DiagonalMatrix}[\text{Diagonal}@W]; 
X=Z+W; \\
P=\text{Position}[X,2]; 
\text{If}[\text{Length}[P]==2,\text{Return}@\text{False}]; \\
k=1; \\
\text{While}[k\leq \text{Length}[P],\text{If}[ 
(\text{Total}[Z[[P[[k,1]] ]]]+\text{Total}[Z[[P[[k,2]] ]]]-2)* \\
{}\quad \!\!(\text{Total}[W[[P[[k,1]] ]]]+\text{Total}[W[[P[[k,2]] ]]]-2)\!==\!0,
\text{Return}@\text{False}]; 
k\text{++}]; 
\text{Return}@\text{True}); \\
\text{FMX}=\text{Select}[\text{FMX},\text{Rule2CSQ}]; \\
{}   \\
\text{(* Fourth Rule of Triangles *)} \\
\text{RuleTriangle4Q}[\text{x$\_$}]\text{:=}(Z=\text{zpart}[x];W=\text{wpart}[x];B=Z+W;k=n; \\
\text{While}[k\geq 3, j=2; \\
{}\quad \text{While}[j<k, i=1; \\
{}\qquad \text{While}[i<j,s=B[[i,j]]*B[[j,k]]*B[[k,i]]* \\
{}\qquad\quad \text{Max}[(Z[[i,j]]+Z[[k,j]]+Z[[i,k]]),(W[[i,j]]+W[[k,j]]+W[[i,k]])]; \\
{}\qquad\quad \text{If}[ s\text{!=}0 \&\&s\text{!=}3 \&\& s\text{!=}24 ,\text{Return}@\text{False}]; 
i\text{++}]; 
j\text{++}]; 
k\text{--}]; 
\text{Return}@\text{True}); \\
\text{FMX}=\text{Select}[\text{FMX},\text{RuleTriangle4Q}]; \\
{}   \\
\text{(* Rule of Connected Components *)} \\
\text{componentsZ}[\text{x$\_$}]\text{:=} (\text{ZZ}=\text{zpart}[x]; 
\text{ZC}=\text{MatrixPower}[\text{ZZ},n-1]; \\
\text{zc}=\text{Flatten}[\text{Tuples}[\{\{0\}\},n-1],1]; 
\text{Pz}=\text{Table}[k,\{k,1,n\}]; \\
k=1; \\
\text{While}[\text{Length}[\text{Pz}]>0, \text{zc}[[k]]=\text{Union}[\{\text{Pz}[[1]]\},\text{Select}[\text{Pz},\text{ZC}[[\text{Pz}[[1]],\#]]>0\&]]; \\
\text{Pz}=\text{Complement}[\text{Pz},\text{zc}[[k]]]; 
k\text{++}];  \\
\text{zc}=\text{Select}[\text{zc},\text{Length}[\#]>1\&]; 
\text{Return}@\text{zc}); \\
{}   \\
\text{componentsW}[\text{x$\_$}]\text{:=} (\text{WW}=\text{wpart}[x]; 
\text{WC}=\text{MatrixPower}[\text{WW},n-1]; \\
\text{wc}=\text{Flatten}[\text{Tuples}[\{\{0\}\},n-1],1]; 
\text{Pw}=\text{Table}[k,\{k,1,n\}]; \\ 
k=1; \\
\text{While}[\text{Length}[\text{Pw}]>0, \text{wc}[[k]]=\text{Union}[\{\text{Pw}[[1]]\},\text{Select}[\text{Pw},\text{WC}[[\text{Pw}[[1]],\#]]>0\&]]; \\
\text{Pw}=\text{Complement}[\text{Pw},\text{wc}[[k]]] ; 
k\text{++}]; \\
\text{wc}=\text{Select}[\text{wc},\text{Length}[\#]>1\&]; 
\text{Return}@\text{wc}); \\
{}   \\
\text{RuleConnCompZQ}[\text{x$\_$}]\text{:=}(Z=\text{zpart}[x]; 
W=\text{wpart}[x]; 
X=Z+W; \\
X=X-\text{DiagonalMatrix}[\text{Diagonal}[X]]; \\
\text{If}[\text{Tr}[Z]<3,\text{Return}@\text{True}];  \\
\text{zc}=\text{Select}[\text{componentsZ}[x],\text{Length}[\#]\geq 3\&]; \\
k=1; \\
\text{While}[k\leq \text{Length}[\text{zc}],l=\text{Length}[\text{zc}[[k]]]; \\
\text{sX}=X[[\text{zc}[[k]],\text{zc}[[k]]]]; 
\text{sZ}=Z[[\text{zc}[[k]],\text{zc}[[k]]]]; \\
\text{If}[\text{Tr}[\text{sZ}]==1,\text{Return}@\text{False}]; \\
\text{If}[\text{Tr}[\text{sZ}]==l\&\&\text{Count}[\text{Flatten}[\text{sX}],2]==l(l-1),\text{Return}@\text{False}]; 
k\text{++}]; \\
\text{Return}@\text{True}); \\
{}   \\
\text{RuleConnCompWQ}[\text{x$\_$}]\text{:=}(Z=\text{zpart}[x]; 
W=\text{wpart}[x]; \\
X=Z+W; 
X=X-\text{DiagonalMatrix}[\text{Diagonal}@X]; \\
\text{If}[\text{Tr}[W]<3,\text{Return}@\text{True}]; \\
\text{wc}=\text{Select}[\text{componentsW}[x],\text{Length}[\#]\geq 3\&]; \\
k=1; \\
\text{While}[k\leq \text{Length}[\text{wc}],l=\text{Length}[\text{wc}[[k]]]; \\
{}\quad \text{sX}=X[[\text{wc}[[k]],\text{wc}[[k]]]]; 
\text{sW}=W[[\text{wc}[[k]],\text{wc}[[k]]]]; \\
{}\quad \text{If}[\text{Tr}[\text{sW}]==1,\text{Return}@\text{False}]; \\
{}\quad \text{If}[\text{Tr}[\text{sW}]==l\&\&\text{Count}[\text{Flatten}[\text{sX}],2]==l(l-1),\text{Return}@\text{False}]; 
k\text{++}];  \\
\text{Return}@\text{True}); \\
\text{FMX}=\text{Select}[\text{FMX},\text{RuleConnCompZQ}]; \\
\text{FMX}=\text{Select}[\text{FMX},\text{RuleConnCompWQ}]; \\
{}   \\
\text{(* Second Rule of Quadrilaterals *)} \\
\text{odminortest}[\text{x$\_$}]\text{:=} 
\text{Flatten}@ 
\text{Table}[ \\
{}\quad \{\text{If}[l\text{==}j\|l\text{==}k\|(\text{zpart}[x][[i,j]]+\text{wpart}[x][[i,j]])
(\text{zpart}[x][[i,k]]+\text{wpart}[x][[i,k]]) \\
{}\quad (\text{zpart}[x][[l,j]]+\text{wpart}[x][[l,j]]) (\text{zpart}[x][[l,k]]+\text{wpart}[x][[l,k]])\text{==}0, \{\}, \\
{}\quad \{\text{zpart}[x][[i,j]]+\text{zpart}[x][[l,k]],\text{zpart}[x][[i,k]]+\text{zpart}[x][[l,j]], \\
{}\quad\;\, \text{wpart}[x][[i,j]]+\text{wpart}[x][[l,k]],\text{wpart}[x][[i,k]]+\text{wpart}[x][[l,j]]\}]\}, \\
{}\quad \{k,3,n\},\{i,1,k-1\},\{l,i+1,n\},\{j,i+1,k-1\}]; \\
\text{RuleQuad2Q}[\text{x$\_$}]\text{:=}\text{Count}[\text{odminortest}[x],1]\text{==}0; \\
\text{FMX}=\text{Select}[\text{FMX},\text{RuleQuad2Q}]; \\
{}   \\
\text{(* Rule of Pentagons *)} \\
\text{RulePentagonQ}[\text{x$\_$}]\text{:=}(Z=\text{zpart}[x];W=\text{wpart}[x];X=Z+2W;i=1; \\
\text{While}[i<n-3,j=i+1; \\
{}\quad \text{While}[j\leq n,k=i+1; \\
{}\qquad \text{While}[k\leq n,l=i+1; \\
{}\qquad\quad \text{While}[l\leq n,m=j+1; \\
{}\qquad\qquad \text{While}[m\leq n, \text{If}[(j-k)(k-l)(l-m)(j-l)(k-m)(j-m)\neq 0, \\
{}\qquad\qquad \text{side}=\{X[[i,j]],X[[j,k]],X[[k,l]],X[[l,m]],X[[m,i]]\}; \\
{}\qquad\qquad v=\{\text{Count}[\text{side},0],\text{Count}[\text{side},1],\text{Count}[\text{side},2],\text{Count}[\text{side},3]\}; \\
{}\qquad\qquad \text{If}[v[[1]]\text{==}0\&\&(v[[2]]==1\|v[[3]]==1\|v==\{0,4,0,1\}\|v==\{0,0,4,1\}), \\
{}\qquad\qquad \text{Return}@\text{False}]]; 
m\text{++};]; 
l\text{++};]; 
k\text{++};]; 
j\text{++};]; 
i\text{++};];\text{Return}@\text{True}); \\
\text{FMX}=\text{Select}[\text{FMX},\text{RulePentagonQ}]; \\
\text{(* } v[[1]]\text{==}0 \;\text{means these edges form a pentagon, } \\
{}\quad\; X[[i,j]]=1\Leftrightarrow \text{$z$-edge}, 
2\Leftrightarrow \text{$w$-edge}, 3\Leftrightarrow  \text{$zw$-edge}, \text{similar for } X[[j,k]], \text{etc}.  \\
{}\quad\; \text{Eliminate pentagons with exactly one $z$-edge } (v[[2]]\text{==}1),\; \text{exactly one} \\
{}\quad\; \text{$w$-edge } (v[[3]]\text{==}1), 4\; \text{$z$-edges plus  1 $zw$-edge }  (v=\{0,4,0,1\}), \; \text{and 4} \\
{}\quad\; \text{$w$-edges plus  1 $zw$-edge }  (v=\{0,0,4,1\}). \text{ *)}  \\
{}   \\
\text{(* Delete duplicates *)} \\
\text{duplicateindices}=\{\}; \\
\text{Do}[A=\text{FMX}[[k]]; \\
\text{AA}=\text{Flatten}[\text{Table}[\{\{A[[1]][[\text{pm}[[j]],\text{pm}[[j]]]],A[[2]][[\text{pm}[[j]],\text{pm}[[j]]]]\}, \\
{}\qquad\qquad\qquad \{A[[2]][[\text{pm}[[j]],\text{pm}[[j]]]],A[[1]][[\text{pm}[[j]],\text{pm}[[j]]]]\}\},\{j,1,\text{Length}[\text{pm}]\}],1]; \\
\text{If}[\text{Length}@\text{Intersection}[\text{AA},\text{FMX}[[1\text{;;}k-1]]]>0, \\
{}\quad \text{duplicateindices}=\text{Append}[\text{duplicateindices},\{k\}]],\{k,2,\text{Length}@\text{FMX}\}]; \\
\text{FMX}=\text{Delete}[\text{FMX},\text{duplicateindices}]; \\
\text{(* END SELECTION OF zw}-\text{MATRICES *)} \\
{}   \\
\text{(* BEGIN DRAWING DIAGRAMS *)} \\
{}   \\
\text{FMXR}=\text{Table}[\{\text{trianglepart}[\text{FMX}[[k,1]]],\text{trianglepart}[\text{FMX}[[k,2]]]\},\{k,\text{Length}[\text{FMX}]\}];
 \\
\text{(* Remove the lower left triangular parts for square matrices *)} \\
{}   \\
\text{rededges}[\text{k$\_$}]\text{:=}\text{Rule}\text{@@@}\text{Position}[\text{zpart}@\text{FMXR}[[k]],1]; 
\text{(* Collect $z$-strokes  *)} \\ 
\text{bluedges}[\text{k$\_$}]\text{:=}\text{Rule}\text{@@@}\text{Position}[\text{wpart}@\text{FMXR}[[k]],1];
\text{(* Collect $w$-strokes *)} \\ 
\text{alledges}[\text{k$\_$}]\text{:=}\text{Flatten}[\text{Append}[\text{rededges}[k],\text{bluedges}[k]]]; \\
\text{(* Delete duplicate edges, then determine color of edges.*)} \\
\text{Print}@\text{Table}[\text{GraphPlot}[ \\
{}\quad \text{DeleteDuplicates}[\text{Table}[ 1\to j,\{j,2,n\}]\sim \!\text{Join}\!\sim \text{alledges}[k]], \\
{}\quad \text{VertexLabeling}\to \text{False},  \text{EdgeRenderingFunction}\to  (\text{Which}[  \\
{}\qquad \text{MemberQ}[\text{rededges}[k],\text{First}[\text{$\#$2}]\to \text{Last}[\text{$\#$2}]]== \text{True}\&\& \\
{} \qquad \text{MemberQ}[\text{bluedges}[k],\text{First}[\text{$\#$2}]\to \text{Last}[\text{$\#$2}]]==\text{False},\\
{} \qquad \{\text{RGBColor}[1,.3,.3],\text{Thickness}[.025],\text{Line}[\text{$\#$1}]\}, \\
{} \qquad \text{MemberQ}[\text{rededges}[k],\text{First}[\text{$\#$2}]\to \text{Last}[\text{$\#$2}]]==\text{False}\&\& \\
{} \qquad \text{MemberQ}[\text{bluedges}[k],\text{First}[\text{$\#$2}]\to \text{Last}[\text{$\#$2}]]==\text{True},\\
{} \qquad \{\text{RGBColor}[.2,.7,1],\text{Thickness}[.025],\text{Line}[\text{$\#$1}]\}, \\
{} \qquad \text{MemberQ}[\text{rededges}[k],\text{First}[\text{$\#$2}]\to \text{Last}[\text{$\#$2}]]==\text{True}\&\&\\
{} \qquad \text{MemberQ}[\text{bluedges}[k],\text{First}[\text{$\#$2}]\to \text{Last}[\text{$\#$2}]]==\text{True}, \\
{} \qquad \{\text{Thickness}[.025],\text{RGBColor}[.5,.3,.9],\text{Line}[\text{$\#$1}]\}, \\
{} \qquad \text{MemberQ}[\text{rededges}[k],\text{First}[\text{$\#$2}]\to \text{Last}[\text{$\#$2}]]==\text{False}\&\& \\
{} \qquad \text{MemberQ}[\text{bluedges}[k],\text{First}[\text{$\#$2}]\to \text{Last}[\text{$\#$2}]]==\text{False}, 
\{\text{Thickness}[.025]\}]\&), \\
{}\quad \text{VertexRenderingFunction}\to  (\{\text{$\#$2}\text{/.}\text{Table}[ \\
{}\qquad \text{Which}[\text{FMXR}[[k,1]][[i]][[i]]\text{==}1\&\&\text{FMXR}[[k,2]][[i]][[i]]\text{==}0,i\to \text{RGBColor}[1,.3,.3],  \\
{}\qquad \text{FMXR}[[k,1]][[i]][[i]]\text{==}0\&\&\text{FMXR}[[k,2]][[i]][[i]]\text{==}1, i\to \text{RGBColor}[.2,.7,1], \\
{}\qquad \text{FMXR}[[k,1]][[i]][[i]]\text{==}1\&\&\text{FMXR}[[k,2]][[i]][[i]]\text{==}1,i\to \text{RGBColor}[.5,.3,.9],\\
{}\qquad \text{FMXR}[[k,1]][[i]][[i]]\text{==}0\&\&\text{FMXR}[[k,2]][[i]][[i]]\text{==}0,i\to \text{White}],  \{i,n\}], \\
{}\qquad \text{EdgeForm}[\text{Black}], \text{Disk}[\text{$\#$1},.24],\text{Black},\text{FontSize}\to 28,\text{FontWeight}\to \text{Bold}, \\
{}\qquad \text{Text}[\text{$\#$2},\text{$\#$1}],\text{Opacity}[.5],\text{RGBColor}[0,.5,0],\text{FontSize}\to 24, \text{Text}[k,\{1.1,-1\}]\}\&),\\
{}\quad \text{SelfLoopStyle}\to \text{None},\text{ImageSize}\to \{240,240\}],\{k,1,\text{Length}[\text{FMXR}]\}]; \\
\text{(* Color vertices:\;$z$-circles red, $w$-circles blue, $zw$-circles purple  *)} \\
{}   \\
\text{(* END DRAWING DIAGRAMS *)} \\
) 
\)

}

\normalsize


\section*{Appendix II: Program for Algorithm II by Mathematica} \label{sec:program2}


This algorithm must be compiled after Algorithm~I. 
One may use the command 
{\tt FindOrder[$n$]} or {\tt FindOrder[$n,t$]} to 
generate the optimal $zw$-order matrices of Type~1 and the collections $\coll_2$, $\coll_3$ for diagram $n$.  
A default time constraint is set to prevent from enumerating all Type~2 or Type~3 $zw$-order matrices when there are too many possibilities. 
The optional argument $t$ is to change this time constraint. 

The optimal $zw$-order matrices of Type~1 is denoted {\tt \{zo,wo\}}, and the collection $\coll_2$  is denoted  
{\tt zwotype3} in the program. Algorithm~III may generate mass relations by calling either of them. \\

\small
{\tt
\noindent\(
 (* \text{ Generate} $ r$- \text{ order} \text{ matrices} \text{ from} $ 
\text{ zw}$- \text{ order} \text{ matrices}\ *) \\
\text{Dis}=\{ \text{  }  \{ \{ \} ,  \{ \} ,  \{ \} , \{\} {  },\{\}  {   } \}, \
\{  \{\} , \{\} , \{ \} ,\{2,3,4\},\{4\}  \}, \
\{  \{\} , \{\} , \{\} ,\{2,3,4\} ,\{4\} \},  \\
{}\qquad\quad\;\  
\{  \{\}  ,\{2,3,4\},\{2,3,4\}  ,\{2,3,4\}  ,\{4\}  \}, \
\{  \{\} ,  \{4\} , \{4\} , \{4\} ,\{5\}  \} \text{  } \}; \\
\text{DisEdge}=  \{\text{   }\{ \{\} , \{\} , \{\} , \{\} ,\{3\}\}, \ 
\{  \{\} , \{\} , \{\} , \{2\} ,\{\}  \}, \ 
\{  \{\}  , \{\} , \{1\} , \{\} ,\{\}  \}, \\
{}\qquad\;\ \qquad\quad \{  \{\}  , \{2\} , \{\} , \{\},\{\}  \}, \
\{    \{3\} , \{\} , \{\} , \{\} ,\{\}  \} \text{ } \}; \\
 {} \text{Dis2}[ \text{zorderset6}$\_$, \text{worderset6}$\_$]:= \\
{} \quad \text{DeleteDuplicates}[ \text{Flatten}[ \text{Table}[ \text{Dis}[[ 
\text{zorderset6}[[ i6]], \text{worderset6}[[ j6]]]],\\
{} \quad  \{ i6,1, \text{Length}[ \text{zorderset6}]\},\{ j6,1, \text{Length}[ \text{worderset6}
]\}]]]; \\
 {} \text{Dis2Edge}[ \text{zorderset6}$\_$, \text{worderset6}$\_$]:= 
\text{DeleteDuplicates}[\\
{}\quad \text{Flatten}[ \text{Table}[ \text{DisEdge}[[ \text{zorderset6}[[ i6]], \text{worderset6}[[ j6]]]],\{ i6,1, 
\text{Length}[ \text{zorderset6}]\},\\ {}\quad \{ j6,1, \text{Length}
 \text{ngth}[ \text{worderset6}]\}]]]; \\
 {} \text{Dis3}[\{ \text{zomatrix5}$\_$, \text{womatrix5}$\_$\}]:=( 
\text{dx5}= \text{Table}[\{\},\{ i5,1, n\},\{ j
5,1, n\}]; \\
 {}\quad \text{Do}[ \text{Do}[ \text{If}[ XX1[[ i5, j5]]==0, \\
 {}\quad \quad \text{dx5}[[ i5, j5]]= \text{Dis2}[ \text{zomatrix5}[[ 
i5, j5]], \text{womatrix5}[[ i5, j5]]], \\
 {}\quad \quad \text{dx5}[[ i5, j5]]= \text{Dis2Edge}[ 
\text{zomatrix5}[[ i5, j5]], \text{womatrix5}[[ i5, j5]]]]; \\
 {}\quad \text{dx5}[[ j5, i5]]= \text{dx5}[[ i5, j5]],\{ i5, j5+1, n\}],\{ 
j5,1, n-1\}]; \\
 {}\quad \text{Return}[ \text{dx5}]) \\ \\
 {}(* \text{ PRINCIPLES} \text{ FOR} \text{ UPDATING} \text{ ORDERS}\ *) 
\\ \\
 (* \text{ Proposition}\ 4.5: \text{ triangle} \text{ inequalities}\ *) \\
\text{TriIneqQ}[ \text{order5a}$\_$, \text{order5b}$\_$, 
\text{order5c}$\_$]:= (\\
{} \quad \text{If}[ \text{Count}[\{ \text{order5a}, \text{order5b}, \text{order5c}\}, \text{Max}[ \text{order5a}, 
\text{order5b}, \text{order5c}]]==1, \\
 {}\quad \quad \text{Return}[ \text{False}], \text{Return}[ \text{True}]]) \\
 {} \text{TriIneqQ2}[ \text{order4}$\_$, \text{orderset4A}$\_$, 
\text{orderset4B}$\_$]:=( ii4=1; \\
 {}\quad \text{While}[ ii4<= \text{Length}[ \text{orderset4A}],  jj4=1; \\
 {}\quad \quad \text{While}[ jj4<= \text{Length}[ \text{orderset4B}], \\
 {}\quad \quad \quad \text{If}[ \text{TriIneqQ}[ \text{order4}, 
\text{orderset4A}[[ ii4]], \text{orderset4B}[[ jj4]]]== \text{True}, \\
 {}\quad \quad \quad \quad \text{Return}[ \text{True}]]; jj4++]; ii4++];\text{Return}[ \text{False}]) \\
 {} \text{TriIneq}[\{ \text{zomatrix3}$\_$, 
\text{womatrix3}$\_$\}]:=( \text{Clear}[ \text{zo3}, \text{wo3}]; \\
 {}\quad \text{zo3}= \text{zomatrix3}; \\
 {}\quad \text{wo3}= \text{womatrix3}; \\
 {}\quad \text{Do}[ \text{Do}[ \text{zo3}[[ i3, i3]]= \text{Select}[ 
\text{zo3}[[ i3, i3]], \text{TriIneqQ2}[\#, \text{zo3}[[ j3, j3]], 
\text{zo3}[[ i3, j3]]]\&]; \\
 {}\quad\quad \quad  \text{zo3}[[ i3, j3]]= \text{Select}[ \text{zo3}[[ i3, j3]], 
\text{TriIneqQ2}[\#, \text{zo3}[[ j3, j3]], \text{zo3}[[ i3, i3]]]\&]; \\
 {}\quad\quad \quad  \text{zo3}[[ j3, j3]]= \text{Select}[ \text{zo3}[[ j3, j3]], 
\text{TriIneqQ2}[\#, \text{zo3}[[ i3, i3]], \text{zo3}[[ i3, j3]]]\&]; \\
 {}\quad\quad \quad  \text{zo3}[[ j3, i3]]= \text{zo3}[[ i3, j3]]; \\
 {}\quad\quad \quad  \text{wo3}[[ i3, i3]]= \text{Select}[ \text{wo3}[[ i3, i3]], 
\text{TriIneqQ2}[\#, \text{wo3}[[ j3, j3]], \text{wo3}[[ i3, j3]]]\&]; \\
 {}\quad\quad \quad  \text{wo3}[[ i3, j3]]= \text{Select}[ \text{wo3}[[ i3, j3]], 
\text{TriIneqQ2}[\#, \text{wo3}[[ j3, j3]], \text{wo3}[[ i3, i3]]]\&]; \\
 {}\quad\quad \quad  \text{wo3}[[ j3, j3]]= \text{Select}[ \text{wo3}[[ j3, j3]], 
\text{TriIneqQ2}[\#, \text{wo3}[[ i3, i3]], \text{wo3}[[ i3, j3]]]\&]; \\
 {}\quad\quad \quad  \text{wo3}[[ j3, i3]]= \text{wo3}[[ i3, j3]]; \\
 {}\quad\quad \quad  ,\{ i3, j3+1, n\}],\{ j3,1, n-1\}]; \\
 {}\quad \text{Do}[ \text{Do}[ \text{Do}[\\
 {}\quad\quad \quad\quad   \text{zo3}[[ i3, j3]]= 
\text{Select}[ \text{zo3}[[ i3, j3]], \text{TriIneqQ2}[\#, \text{zo3}[[ j3, 
k3]], \text{zo3}[[ i3, k3]]]\&]; \\
 {}\quad\quad \quad\quad  \text{zo3}[[ j3, k3]]= \text{Select}[ \text{zo3}[[ j3, k3]], 
\text{TriIneqQ2}[\#, \text{zo3}[[ i3, j3]], \text{zo3}[[ i3, k3]]]\&]; \\
 {}\quad\quad \quad\quad  \text{zo3}[[ i3, k3]]= \text{Select}[ \text{zo3}[[ i3, k3]], 
\text{TriIneqQ2}[\#, \text{zo3}[[ i3, j3]], \text{zo3}[[ j3, k3]]]\&]; \\
 {}\quad\quad \quad\quad  \text{zo3}[[ j3, i3]]= \text{zo3}[[ i3, j3]]; \\
 {}\quad\quad \quad\quad  \text{zo3}[[ k3, i3]]= \text{zo3}[[ i3, k3]]; \\
 {}\quad\quad \quad\quad  \text{zo3}[[ k3, j3]]= \text{zo3}[[ j3, k3]]; \\
 {}\quad\quad \quad\quad  \text{wo3}[[ i3, j3]]= \text{Select}[ \text{wo3}[[ i3, j3]], 
\text{TriIneqQ2}[\#, \text{wo3}[[ j3, k3]], \text{wo3}[[ i3, k3]]]\&]; \\
 {}\quad\quad \quad\quad  \text{wo3}[[ j3, k3]]= \text{Select}[ \text{wo3}[[ j3, k3]], 
\text{TriIneqQ2}[\#, \text{wo3}[[ i3, j3]], \text{wo3}[[ i3, k3]]]\&]; \\
 {}\quad\quad \quad\quad  \text{wo3}[[ i3, k3]]= \text{Select}[ \text{wo3}[[ i3, k3]], 
\text{TriIneqQ2}[\#, \text{wo3}[[ i3, j3]], \text{wo3}[[ j3, k3]]]\&]; \\
 {}\quad\quad \quad\quad  \text{wo3}[[ j3, i3]]= \text{wo3}[[ i3, j3]]; \\
 {}\quad\quad \quad\quad  \text{wo3}[[ k3, i3]]= \text{wo3}[[ i3, k3]]; \\
 {}\quad\quad \quad\quad  \text{wo3}[[ k3, j3]]= \text{wo3}[[ j3, k3]]; \\
 {}\quad ,\{ i3, j3+1, n\}],\{ j3, k3+1, n-1\}],\{ k3,1, n-2\}]; \\
 {}\quad \text{Return}[\{ \text{zo3}, \text{wo3}\}]; ) \\ \\
 (* \text{ Proposition}\ 4.6\ *) \\
 \text{Rule2e}[\{ \text{zomatrix3}\_, \text{womatrix3}\_\}] := ( 
\text{Clear}[ \text{zo3}, \text{wo3}]; \\
 {}\quad \text{zo3} = \text{zomatrix3}; \\
 {}\quad \text{wo3} = \text{womatrix3}; \\
 {}\quad \text{Do}[ \text{If}[ i3 != j3 \&\& i3 != k3 \&\& j3 
!= k3 \&\& ( \text{XX1}[[ i3, j3]]* \text{XX1}[[ j3, k3]]* 
\text{XX1}[[ k3, i3]] != 0), \\
 {}\quad \quad \text{zo3}[[ i3, j3]] = \text{Intersection}[ 
\text{zo3}[[ i3, j3]], \text{zo3}[[ j3, k3]], 
\text{zo3}[[ k3, i3]]]; \\
 {}\quad \quad \text{zo3}[[ j3, k3]] = \text{zo3}[[ i3, j3]]; \\
 {}\quad \quad \text{zo3}[[ k3, i3]] = \text{zo3}[[ i3, j3]]; \\
 {}\quad \quad \text{wo3}[[ i3, j3]] = \text{Intersection}[ 
\text{wo3}[[ i3, j3]], \text{wo3}[[ j3, k3]], 
\text{wo3}[[ k3, i3]]]; \\
 {}\quad \quad \text{wo3}[[ j3, k3]] = \text{wo3}[[ i3, j3]]; \\
 {}\quad \quad \text{wo3}[[ k3, i3]] = \text{wo3}[[ i3, j3]]]; \\
 {}\quad , \{ i3, 1, n\}, \{ j3, 1, n\}, \{ k3, 1, n\}]; \\
 {}\quad \text{Return}[\{ \text{zo3}, \text{wo3}\}]) \\ \\
(* \text{ Proposition}\ 4.7( \text{ a})\ *) \\
 \text{Estimate}[\{ \text{zomatrix3}\_, \text{womatrix3}\_\}] := ( 
\text{Clear}[ \text{zo3}, \text{wo3}]; \\
 {}\quad \text{zo3} = \text{zomatrix3}; \\
 {}\quad \text{wo3} = \text{womatrix3}; \\
 {}\quad \text{Do}[ \text{If}[ i3 != j3, \\
 {}\quad \quad \text{If}[ \text{XX1}[[ i3, j3]] > 0, \\
 {}\quad \quad \quad \text{zo3}[[ i3, j3]] = 
\text{Intersection}[ \text{zo3}[[ i3, j3]], \text{Table}[6 - 
\text{wo3}[[ i3, j3]][[ k3]],\\
 {}\quad \quad \quad \{ k3, 1, \text{Length}[ 
\text{wo3}[[ i3, j3]]]\}]]; \\
 {}\quad \quad \quad \text{wo3}[[ i3, j3]] = 
\text{Intersection}[ \text{wo3}[[ i3, j3]], \text{Table}[6 - 
\text{zo3}[[ i3, j3]][[ k3]],\\
 {}\quad \quad \quad \{ k3, 1, \text{Length}[ \text{zo3}[[ i3, 
j3]]]\}]] ]; \\
 {}\quad \quad \text{If}[ \text{XX1}[[ i3, j3]] == 0, \\
 {}\quad \quad \quad \text{zo3}[[ i3, j3]] = \text{Select}[ 
\text{zo3}[[ i3, j3]], \# >= 6 - 2* \text{Floor}[ \text{Max}[ 
\text{wo3}[[ i3, j3]]]/2] \&]; \\
 {}\quad \quad \quad \text{wo3}[[ i3, j3]] = \text{Select}[ 
\text{wo3}[[ i3, j3]], \# >= 6 - 2* \text{Floor}[ \text{Max}[ 
\text{zo3}[[ i3, j3]]]/2] \& ]]],\\
 {}\quad \quad \quad \{ i3, 1, n\}, \{ j3, 1, n\}]; \\
 {}\quad \text{Return}[\{ \text{zo3}, \text{wo3}\}]) \\ \\
 (* \text{ Proposition}\ 4.7( b)\ *) \\
 \text{CenterOfMass}[\{ \text{zomatrix3}\_, \text{womatrix3}\_\}] 
:= ( \text{Clear}[ \text{zo3}, \text{wo3}]; \\
 {}\quad \text{zo3} = \text{zomatrix3}; \\
 {}\quad \text{wo3} = \text{womatrix3}; \\
 {}\quad \text{MZ3} = \text{Max}@ \text{Flatten}[ \text{Table}[ 
 \text{Table}[ \text{zo3}[[ i3, j3]], \{ i3, j3 + 1, n\}], \{ j3, 1, 
 n - 1\}]]; \\
 {}\quad \text{MW3} = \text{Max}@ \text{Flatten}[ \text{Table}[ 
 \text{Table}[ \text{wo3}[[ i3, j3]], \{ i3, j3 + 1, n\}], \{ j3, 1, 
 n - 1\}]]; \\
 {}\quad \text{Do}[ \text{zo3}[[ i3, i3]] = \text{Select}[ \text{zo3}[[ 
i3, i3]], \# <= \text{MZ3} \&]; \\
 {}\quad\quad \text{wo3}[[ i3, i3]] = \text{Select}[ \text{wo3}[[ i3, 
i3]], \# <= \text{MW3} \&], \{ i3, 1, n\}]; \\
 {}\quad \text{Return}[\{ \text{zo3}, \text{wo3}\}]; ) \\ \\
 (* \text{ Determine} \text{ connected} \text{ components} \text{ of} $ 
\text{ z}$- \text{ matrix},$ \text{ w}$- \text{ matrix}\ *) \\
 \text{componentsZ}[ \text{x3}\_] := (\\
 {}\quad \text{MXZ3} = \text{zpart}[ \text{x3}]; \text{ZC3} = \text{MatrixPower}[ \text{MXZ3}, n - 1]; \\
 {}\quad \text{zc3} = \text{Flatten}[ \text{Tuples}[\{\{0\}\}, n - 1], 
1]; \text{Pz3} = \text{allind}; kk3 = 1; \\
 {}\quad \text{While}[ \text{Length}[ \text{Pz3}] > 0, \\
 {}\quad \quad \text{zc3}[[ kk3]] = \text{Union}[\{ 
\text{Pz3}[[1]]\}, \text{Select}[ \text{Pz3}, 
\text{ZC3}[[ \text{Pz3}[[1]], \#]] > 0 \&]]; \\
 {}\quad \quad \text{Pz3} = \text{Complement}[ \text{Pz3}, 
\text{zc3}[[ kk3]]] ;  kk3++]; \\
 {}\quad \text{zc3} = \text{Select}[ \text{zc3}, \text{Length}[\#] > 1 
\&]; \\
 {}\quad \text{Return}[ \text{zc3}]) \\ \\
 {} \text{componentsW}[ \text{x3}\_] := (\\
 {}\quad  \text{MXW3} = \text{wpart}[ \text{x3}];  \text{WC3} = \text{MatrixPower}[ \text{MXW3}, n - 1]; \\
 {}\quad \text{wc3} = \text{Flatten}[ \text{Tuples}[\{\{0\}\}, n - 1], 
1]; \text{Pw3} = \text{allind}; kk3 = 1; \\
 {}\quad \text{While}[ \text{Length}[ \text{Pw3}] > 0, \\
 {}\quad \quad \text{wc3}[[ kk3]] = \text{Union}[\{ 
\text{Pw3}[[1]]\}, \text{Select}[ \text{Pw3}, 
\text{WC3}[[ \text{Pw3}[[1]], \#]] > 0 \&]]; \\
 {}\quad \quad \text{Pw3} = \text{Complement}[ \text{Pw3}, 
\text{wc3}[[ kk3]]] ; kk3++]; \\
 {}\quad \text{wc3} = \text{Select}[ \text{wc3}, \text{Length}[\#] > 1 
\&]; \\
 {}\quad \text{Return}[ \text{wc3}]) \\ \\
(* \text{ Proposition}\ 4.9 *) \\
 \text{Prop1}[\{ \text{zomatrix3}\_, \text{womatrix3}\_\}] := ( 
\text{zo3} = \text{zomatrix3}; \\
 {}\quad \text{wo3} = \text{womatrix3}; \\
 {}\quad \text{Do}[ \text{If}[ i3 != j3, kk3 = 1;  \text{Z3A} = \{\}; \text{W3A} = \{\}; \\
 {}\quad \quad \quad \text{While}[ kk3 <= n, \\
 {}\quad \quad \quad \quad \text{If}[ kk3 != i3 \&\& kk3 != j3, \\
 {}\quad \quad \quad \quad \quad \text{Z3A} = \text{Union}[ \text{Z3A}, 
\text{zo3}[[ i3, kk3]]]; \\
 {}\quad \quad \quad \quad \quad \text{W3A} = \text{Union}[ 
\text{W3A}, \text{wo3}[[ i3, kk3]]]]; \\
 {}\quad \quad \quad \quad kk3++]; \\
 {}\quad \quad \quad \text{Z3B} = \text{Union}[ \text{Z3A}, 
\text{zo3}[[ i3, j3]]]; \\
 {}\quad \quad \quad \text{W3B} = \text{Union}[ \text{W3A}, \text{wo3}[[ 
i3, j3]]]; \\
 {}\quad \quad \quad \text{If}[( \text{Z3B} == \{1\} || \text{Z3B} == 
\{3\} || \text{Z3B} == \{5\}) \&\& \text{Max}[ \text{wo3}[[ i3, j3]]] < 
\text{Min}[ \text{W3A}], \\
 {}\quad \quad \quad \quad \text{wo3}[[ i3, i3]] = \text{Select}[ 
\text{wo3}[[ i3, i3]], \# <= \text{Max}@ \text{wo3}[[ i3, j3]] \& ]; \\
 {}\quad \quad \quad \quad \text{wo3}[[ j3, j3]] = 
\text{Select}[ \text{wo3}[[ j3, j3]], \# <= \text{Max}@ \text{wo3}[[ 
i3, j3]] \& ]; \\
 {}\quad \quad \quad \quad \text{wo3}[[ i3, j3]] = 
\text{Select}[ \text{wo3}[[ i3, j3]], \# >= \text{Min}@( 
\text{wo3}[[ i3, i3]] \\
 {}\quad \quad \quad \quad \sim \text{Join}\sim \text{wo3}[[ j3, j3]]) \& ]; ]; \\
 {}\quad \quad \quad \text{If}[( \text{W3B} == \{1\} || \text{W3B} == 
\{3\} || \text{W3B} == \{5\}) \&\& \text{Max}[ \text{zo3}[[ i3, j3]]] < 
\text{Min}[ \text{Z3A}], \\
 {}\quad \quad \quad \quad \text{zo3}[[ i3, i3]] = \text{Select}[ 
\text{zo3}[[ i3, i3]], \# <= \text{Max}@ \text{zo3}[[ i3, j3]] \& ]; \\
 {}\quad \quad \quad \quad \text{zo3}[[ j3, j3]] = 
\text{Select}[ \text{zo3}[[ j3, j3]], \# <= \text{Max}@ \text{zo3}[[ 
i3, j3]] \& ]; \\
 {}\quad \quad \quad \quad \text{zo3}[[ i3, j3]] = 
\text{Select}[ \text{zo3}[[ i3, j3]], \# >= \text{Min}@( 
\text{zo3}[[ i3, i3]] \\
 {}\quad \quad \quad \quad \sim \text{Join}\sim \text{zo3}[[ j3, j3]]) \& ]; ]], \\
 {}\quad \quad \{ i3, 1, n\}, \{ j3, 1, n\}]; \\
 {}\quad \text{Return}[\{ \text{zo3}, \text{wo3}\}]) \\ \\
 (* \text{ Proposition}\ 4.10\ *) \\
 \text{Prop2}[\{ \text{zomatrix3}\_, \text{womatrix3}\_\}] := ( \\
 {}\quad \text{zo3} = \text{zomatrix3}; \text{wo3} = \text{womatrix3}; \\
 {}\quad \text{Do}[ \text{dis3} = \{\}; \\
 {}\quad \quad jj3 = 1; \\
 {}\quad \quad \text{While}[ jj3 <= n, \\
 {}\quad \quad \quad \text{If}[ jj3 != i3, \\
 {}\quad \quad \quad \quad \text{If}[ \text{XX1}[[ i3, jj3]] == 0, 
\\
 {}\quad \quad \quad \quad \quad \text{dis3} = \text{Union}[ 
\text{dis3}, \text{Dis2}[ \text{zo3}[[ i3, jj3]], \text{wo3}[[ i3, jj3]]]], 
\\
 {}\quad \quad \quad \quad \quad \text{dis3} = \text{Union}[ 
\text{dis3}, \text{Dis2Edge}[ \text{zo3}[[ i3, jj3]], \text{wo3}[[ i3, 
jj3]]]]]]; \\
 {}\quad \quad \quad jj3++]; \\
 {}\quad \quad \text{If}[ \text{Complement}[ \text{dis3}, \{4, 5\}] == 
\{\}, \\
 {}\quad \quad \quad \text{ZO3} = \{\};  \text{WO3} = \{\};  jj3 = 1; \\
 {}\quad \quad \quad \text{While}[ jj3 <= n, \\
 {}\quad \quad \quad \quad \text{If}[ jj3 != i3, \\
 {}\quad \quad \quad \quad \quad \text{ZO3} = \text{ZO3}\sim 
\text{Join}\sim \text{zo3}[[ jj3, jj3]]; \\
 {}\quad \quad \quad \quad \quad \text{WO3} = \text{WO3}\sim 
\text{Join}\sim \text{wo3}[[ jj3, jj3]]]; \\
 {}\quad \quad \quad \quad jj3++]; \\
 {}\quad \quad \quad \text{If}[ \text{Max}[ \text{ZO3}] == 2 || 
\text{Max}[ \text{ZO3}] == 4, \\
 {}\quad \quad \quad \quad \text{zo3}[[ i3, i3]] = \text{Select}[ 
\text{zo3}[[ i3, i3]], \# <= \text{Max}[ \text{ZO3}] \&], \\
 {}\quad \quad \quad \quad \text{zo3}[[ i3, i3]] = \text{Select}[ 
\text{zo3}[[ i3, i3]], \# < \text{Max}[ \text{ZO3}] \&]]; \\
 {}\quad \quad \quad \text{If}[ \text{Max}[ \text{WO3}] == 2 || 
\text{Max}[ \text{WO3}] == 4, \\
 {}\quad \quad \quad \quad \text{wo3}[[ i3, i3]] = \text{Select}[ 
\text{wo3}[[ i3, i3]], \# <= \text{Max}[ \text{WO3}] \& ], \\
 {}\quad \quad \quad \quad \text{wo3}[[ i3, i3]] = \text{Select}[ 
\text{wo3}[[ i3, i3]], \# < \text{Max}[ \text{WO3}] \& ]]], \\
 {}\quad \quad \{ i3, 1, n\}]; \\
 {}\quad \text{Return}[\{ \text{zo3}, \text{wo3}\}]) \\
 \\
 (* \text{ Compute} \text{ the} \text{ number} \text{ of} \text{ 
possible} \text{ levels} \text{ of} \text{ orders} (1) \text{ for} 
\text{ all} \text{ off}-\\ \text{ diagonal} \text{ entries}; {}(2) \text{ for} \text{ all} \text{ entries}\ *) \\
 \text{NumberofType2}[\{ \text{zomatrix3}\_, \text{womatrix3}\_\}] 
:= ( kk3 = 1; \\
 {}\quad \text{Do}[ \text{Do}[ kk3 = kk3* \text{Length}[ 
\text{zomatrix3}[[ i3, j3]]]* \text{Length}[ 
\text{womatrix3}[[ i3, j3]]], \\
 {}\quad \quad \quad \{ i3, j3 + 1, n\}], \{ j3, 1, n - 1\}]; \\
 {}\quad \text{Return}[ kk3]) \\
 \\
 {} \text{NumberofType3}[\{ \text{zomatrix3}\_, 
\text{womatrix3}\_\}] := ( kk3 = 1; \\
 {}\quad \text{Do}[ \text{Do}[ kk3 = kk3* \text{Length}[ 
\text{zomatrix3}[[ i3, j3]]]* \text{Length}[ 
\text{womatrix3}[[ i3, j3]]], \\
 {}\quad \quad \quad \{ i3, j3, n\}], \{ j3, 1, n\}]; \\
 {}\quad \text{Return}[ kk3]) \\ \\
 (* \text{ Repeat} \text{ above} \text{ processes} \text{ until} \text{ 
possible} \text{ levels} \text{ becomes} \text{ invariant}\ *) \\
 \text{UpdateOrder}[\{ \text{zomatrix2}\_, \text{womatrix2}\_\}]:=( \\
 {}\quad \text{zo2}= \text{zomatrix2}; \text{wo2}= \text{womatrix2}; \\
 {}\quad l2= \text{NumberofType3}[\{ \text{zo2}, \text{wo2}\}]; L2= l2+1; \\
 {}\quad \text{While}[ l2!= L2, \\
 {}\quad \quad l2= \text{NumberofType3}[\{ \text{zo2}, \text{wo2}\}]; \\
 {}\quad \quad \{ \text{zo2}, \text{wo2}\}= \text{TriIneq}[\{ \text{zo2}, 
\text{wo2}\}]; \\
 {}\quad \quad \{ \text{zo2}, \text{wo2}\}= \text{Rule2e}[\{ \text{zo2}, 
\text{wo2}\}]; \\
 {}\quad \quad \{ \text{zo2}, \text{wo2}\}= \text{Estimate}[\{ \text{zo2}, 
\text{wo2}\}]; \\
 {}\quad \quad \{ \text{zo2}, \text{wo2}\}= \text{CenterOfMass}[\{ 
\text{zo2}, \text{wo2}\}]; \\
 {}\quad \quad \{ \text{zo2}, \text{wo2}\}= \text{Prop1}[\{ \text{zo2}, 
\text{wo2}\}]; \\
 {}\quad \quad \{ \text{zo2}, \text{wo2}\}= \text{Prop2}[\{ \text{zo2}, 
\text{wo2}\}]; \\
 {}\quad \quad L2= \text{NumberofType3}[\{ \text{zo2}, \text{wo2}\}];  ]; \\
 {}\quad \text{Return}[\{ \text{zo2}, \text{wo2}\}]) \\
 \\
 {} \text{UpdateOrderPart}[\{ \text{zomatrix2}\_, 
\text{womatrix2}\_\}]:=( \\
 {}\quad \text{zo2}= \text{zomatrix2}; \text{wo2}= \text{womatrix2}; \\
 {}\quad l2= \text{NumberofType3}[\{ \text{zo2}, \text{wo2}\}]; L2= l2+1; \\
 {}\quad \text{While}[ l2!= L2, \\
 {}\quad \quad l2= \text{NumberofType3}[\{ \text{zo2}, \text{wo2}\}]; \\
 {}\quad \quad \{ \text{zo2}, \text{wo2}\}= \text{TriIneq}[\{ \text{zo2}, 
\text{wo2}\}]; \\
 {}\quad \quad L2= \text{NumberofType3}[\{ \text{zo2}, \text{wo2}\}]; ]; \\
 {}\quad \text{Return}[\{ \text{zo2}, \text{wo2}\}]) \\
 \\
 (* \text{ PosOffDiag} \text{ is} \text{ to} \text{ find} \text{ a} 
\text{ position} \text{ of} \text{ an} \text{ off}- \text{ diagonal} 
\text{ entry} \text{ that} \text{ contains}\\ \text{ more} \text{ than} 
\text{ one} \text{ element}; \\
 {} \text{ PosDiag} \text{ is} \text{ similar}, \text{ but} \text{ for} 
\text{ diagonal} \text{ entries}\ *) \\
 \text{PosOffDiag}[ \text{x2}\_]:=( \\
 {}\quad ii2=1; \\
 {}\quad \text{While}[ ii2<= n-1, \\
 {}\quad \quad jj2= ii2+1; \\
 {}\quad \quad \text{While}[ jj2<= n, \\
 {}\quad \quad \quad L2= \text{Length}[ \text{x2}[[ ii2, jj2]]]; \\
 {}\quad \quad \quad \text{If}[ L2>1, \text{Return}[\{ ii2, jj2, L2\}]];  jj2++]; ii2++]; \\
 {}\quad \text{Return}[\{1,1,0\}]) \\
 \\
 \quad \text{PosDiag}[ \text{x2}\_]:=( \\
 {}\quad ii2=1; \\
 {}\quad \text{While}[ ii2<= n, \\
 {}\quad \quad L2= \text{Length}[ \text{x2}[[ ii2, ii2]]]; \\
 {}\quad \quad \text{If}[ L2>1,\text{Return}[\{ ii2, L2\}]];  ii2++];\\
 {}\quad \text{Return}[\{1,0\}]) \\
 \\
 \quad (* \text{ CRITERIA} \text{ FOR} \text{ FEASIBLE} \text{ ORDER} 
\text{ MATRICES}\ *) \\
 (* \text{ Proposition} 4.11\ *) \\
 \text{Rule2hQ}[\{ \text{zomatrix2}\_, \text{womatrix2}\_\}]:=( \\
 {}\quad \text{zo2}= \text{zomatrix2}; \text{wo2}= \text{womatrix2}; \text{dx2}= \text{Flatten}[ \text{Dis3}[\{ \text{zo2}, 
\text{wo2}\}],1] ; \\
 {}\quad \text{If}[ \text{Count}[ \text{dx2},\{2\}]==2\&\& \text{Count}[ 
\text{dx2},\{1\}]==0, \text{Return}[ \text{False}],\text{Return}[ \text{True}]]) \\
 \\
 (* \text{ Compare } \text{orders}, \text{ find } \text{candidates } \text{for } 
\text{max } \text{order } \text{terms } \text{among } Z_{ij}\text{'s}, W_{ij}\text{'s}\ *) \\ \\
(* \text{ Compare} z_{ ij}\text{'s} \text{ or } w_{ ij}\text{'s}\ *) \\
\text{CompareOrderzw}[ \text{ordermatrix5}\_,\{ ii5\_, jj5\_\},\{ kk5\_, ll5\_\}]:=(
\\
 {}\quad o5= \text{ordermatrix5}; \\
 {}\quad \text{If}[ \text{Max}@ o5[[ ii5, jj5]]< \text{Min}@ o5[[ kk5, 
ll5]],text{Return}[\{-1\}]]; \\
 {}\quad \text{If}[ \text{Min}@ o5[[ ii5, jj5]]> \text{Max}@ o5[[ kk5, 
ll5]], \text{Return}[\{1\}]]; \\
 {}\quad \text{If}[ \text{Length}[ o5[[ ii5, jj5]]]==1\&\& o5[[ ii5, jj5]]== o5[[ kk5, ll5]]\&\& 
\\
 {}\quad \quad \text{Complement}[ o5[[ ii5, jj5]],\{1,3,5\}]==0,  \text{Return}[\{0\}]]; \\
 {}\quad \text{If}[ ii5!= jj5\&\&  kk5!= ll5, \\
 {}\quad \quad \text{If}[ ii5== kk5\&\&  \text{Min}@ o5[[ ii5, 
jj5]]>= \text{Max}@ o5[[ jj5, ll5]]\&\& \\
 {}\quad \quad \quad \text{Min}@ o5[[ kk5, ll5]]>= \text{Max}@ o5[[ jj5, 
ll5]], \\
 {}\quad \quad \quad \text{If}[ \text{Min}@ o5[[ ii5, jj5]]> \text{Max}@ 
o5[[ jj5, ll5]]\ || \\
 {}\quad \quad \quad \quad \text{Min}@ o5[[ kk5, ll5]]> \text{Max}@ o5[[ jj5,
ll5]]\ || \\
 {}\quad \quad \quad \quad \text{MemberQ}[\{1,3,5\}, \text{Max}@ o5[[ jj5, ll5]]]== 
\text{True}, \text{Return}[\{0\}]]]; \\
 {}\quad \quad \text{If}[ jj5== ll5\&\&  \text{Min}@ o5[[ ii5, 
jj5]]>= \text{Max}@ o5[[ ii5, kk5]]\&\& \\
 {}\quad \quad \quad \text{Min}@ o5[[ kk5, ll5]]>= \text{Max}@ o5[[ ii5, 
kk5]], \\
 {}\quad \quad \quad \text{If}[ \text{Min}@ o5[[ ii5, jj5]]> \text{Max}@ 
o5[[ ii5, kk5]]\ || \\
 {}\quad \quad \quad \quad \text{Min}@ o5[[ kk5, ll5]]> \text{Max}@ o5[[ ii5,
kk5]]\ || \\
{}\quad \quad \quad \quad \text{MemberQ}[\{1,3,5\}, \text{Max}@ o5[[ ii5, kk5]]]== 
\text{True}, \text{Return}[\{0\}]]]; \\
 {}\quad \quad \text{If}[ ii5== ll5\&\& \text{Min}@ o5[[ ii5, 
jj5]]>= \text{Max}@ o5[[ jj5, kk5]]\&\& \\
 {}\quad \quad \quad \text{Min}@ o5[[ kk5, ll5]]>= \text{Max}@ o5[[ jj5, 
kk5]], \\
 {}\quad \quad \quad \text{If}[ \text{Min}@ o5[[ ii5, jj5]]> \text{Max}@ 
o5[[ jj5, kk5]]\ || \\
 {}\quad \quad \quad \quad \text{Min}@ o5[[ kk5, ll5]]> \text{Max}@ o5[[ jj5,
kk5]]\ || \\
 {}\quad \quad \quad \quad \text{MemberQ}[\{1,3,5\}, \text{Max}@ o5[[ jj5, kk5]]]== 
\text{True}, \text{Return}[\{0\}]]]; \\
 {}\quad \quad \text{If}[ kk5== jj5\&\& \text{Min}@ o5[[ ii5, 
jj5]]>= \text{Max}@ o5[[ ll5, ii5]]\&\& \\
 {}\quad \quad \quad \text{Min}@ o5[[ kk5, ll5]]>= \text{Max}@ o5[[ ll5, 
ii5]], \\
 {}\quad \quad \quad \text{If}[ \text{Min}@ o5[[ ii5, jj5]]> \text{Max}@ 
o5[[ ll5, ii5]]\ || \\
 {}\quad \quad \quad \quad \text{Min}@ o5[[ kk5, ll5]]> \text{Max}@ o5[[ ll5,
ii5]]\ || \\
 {}\quad \quad \quad \quad \text{MemberQ}[\{1,3,5\}, \text{Max}@ o5[[ ll5, ii5]]]== 
\text{True}, \text{Return}[\{0\}]]]]; \\
 {}\quad \text{If}[ ii5== jj5\&\&  kk5!= ll5, \\
 {}\quad \quad \text{If}[ ii5== kk5\&\&  \text{Min}@ o5[[ ii5, 
jj5]]>= \text{Max}@ o5[[ ll5, ll5]]\&\& \\
 {}\quad \quad \quad \text{Min}@ o5[[ kk5, ll5]]>= \text{Max}@ o5[[ ll5, 
ll5]], \\
 {}\quad \quad \quad \text{If}[ \text{Min}@ o5[[ ii5, jj5]]> \text{Max}@ 
o5[[ ll5, ll5]]\ || \\
 {}\quad \quad \quad \quad \text{Min}@ o5[[ kk5, ll5]]> \text{Max}@ o5[[ ll5,
ll5]]\ || \\
 {}\quad \quad \quad \quad \text{MemberQ}[\{1,3,5\}, \text{Max}@ o5[[ ll5, ll5]]]== 
\text{True}, \text{Return}[\{0\}]]]; \\
 {}\quad \quad \text{If}[ ii5== ll5\&\&  \text{Min}@ o5[[ ii5, 
jj5]]>= \text{Max}@ o5[[ kk5, kk5]]\&\& \\
 {}\quad \quad \quad \text{Min}@ o5[[ kk5, ll5]]>= \text{Max}@ o5[[ kk5, 
kk5]], \\
 {}\quad \quad \quad \text{If}[ \text{Min}@ o5[[ ii5, jj5]]> \text{Max}@ 
o5[[ kk5, kk5]]\ || \\
 {}\quad \quad \quad \quad \text{Min}@ o5[[ kk5, ll5]]> \text{Max}@ o5[[ kk5,
kk5]]\ || \\
 {}\quad \quad \quad \quad \text{MemberQ}[\{1,3,5\}, \text{Max}@ o5[[ kk5, kk5]]]== 
\text{True}, \text{Return}[\{0\}]]]]; \\
 {}\quad \text{Return}[\{1,0,-1\}]; \\
 {}\quad ) \\
 \\
 (* \text{ Compare} r_{ ij}\text{'s} \text{ or compare }  r_{ ij}\text{ with } 1\ *) \\
 \text{CompareOrderr}[ \text{rordermatrix5}\_, \{ ii5\_, jj5\_\}, \{ kk5\_, 
ll5\_\}] := ( \\
 {}\quad \text{ro5} = \text{rordermatrix5}; \\
 {}\quad \text{If}[ ii5 != jj5 \&\& kk5 != ll5, \\
 {}\quad \quad \text{If}[ \text{Max}@ \text{ro5}[[ ii5, jj5]] < 
\text{Min}@ \text{ro5}[[ kk5, ll5]],  \text{Return}[\{-1\}]]; \\
 {}\quad \quad \text{If}[ \text{Min}@ \text{ro5}[[ ii5, jj5]] > 
\text{Max}@ \text{ro5}[[ kk5, ll5]], \text{Return}[\{1\}]]; \\
 {}\quad \quad \text{If}[ \text{Length}[ \text{ro5}[[ ii5, jj5]]] == 1 
\&\&  \text{ro5}[[ ii5, jj5]] == \text{ro5}[[ kk5, ll5]] 
\&\& \\
 {}\quad \quad \quad \text{Complement}[ \text{ro5}[[ ii5, jj5]], \{1, 3, 
5\}] == 0, \text{Return}[\{0\}]]]; \\
 {}\quad \text{If}[ ii5 == jj5 \&\&  kk5 != ll5, \\
 {}\quad \quad \text{If}[ \text{Min}@ \text{ro5}[[ kk5, ll5]] >= 4, \text{Return}[\{-1\}]]; \\
 {}\quad \quad \text{If}[ \text{Max}@ \text{ro5}[[ kk5, ll5]] <= 2,  \text{Return}[\{1\}]]; \\
 {}\quad \quad \text{If}[ \text{ro5}[[ kk5, ll5]] == \{3\}, \text{Return}[\{0\}] ]]; \\
 {}\quad \text{Return}[\{1, 0, -1\}];  ) \\
 \\ 
 (* \text{ Compare} Z_{ ij}\text{'s} \text{ according} 
\text{ to}\ Z_{ij}= z_{ij}^{-1/2} w{ ij}^{-3/2}= z_{
ij} r_{ ij}^{-3}=w_{
ij}^{-1} r_{ ij}^{-1}= W_{ ij}^
{1/3} {w}_{ij}^{-4/3}\ *) \\
\text{CompareOrderZ1}[\{ \text{zomatrix4}\_, 
\text{womatrix4}\_\},\{ ii4\_, jj4\_\},\{ kk4\_, ll
4\_\}]:=( \\
 {}\quad \text{zo4}= \text{zomatrix4}; \text{wo4}= \text{womatrix4};  \text{dis}= \text{Dis3}[\{ \text{zo4}, \text{wo4}\}]; \\
 {}\quad \text{If}[ \text{ZX1}[[ ii4, jj4]]==1\&\&  \text{ZX1}[[ kk4, ll4]]==0,  \text{Return}[\{1\}]]; \\
 {}\quad \text{If}[ \text{ZX1}[[ ii4, jj4]]==1\&\&  \text{ZX1}[[ kk4, ll4]]==1,  \text{Return}[\{0\}]]; \\
 {}\quad \text{If}[ \text{ZX1}[[ ii4, jj4]]==0\&\& \text{ZX1}[[ kk4, ll4]]==1,  \text{Return}[\{-1\}]]; \\ \\
 {}\quad \text{Comparez} = \text{CompareOrderzw}[ \text{zo4}, \{ ii4, jj4\}, \{ kk4, ll4\}]; \\
 {} \quad \text{Comparew} = \text{CompareOrderzw}[ \text{wo4}, \{ ii4, jj4\}, \{ kk4, ll4\}]; \\
 {}\quad \text{Comparer} = \text{CompareOrderr}[ \text{dis}, \{ ii4, jj4\}, \{ kk4, ll4\}]; \\
\\
 {}\quad \text{If}[ \text{Comparez}==\{1\}\&\& \text{Comparew}==\{1\},  \text{Return}[\{-1\}]]; \\
 {}\quad \text{If}[ \text{Comparez}==\{1\}\&\& \text{Comparew}==\{0\},  \text{Return}[\{-1\}]]; \\
 {}\quad \text{If}[ \text{Comparez}==\{0\}\&\& \text{Comparew}==\{1\},  \text{Return}[\{-1\}]]; \\
 {}\quad \text{If}[ \text{Comparez}==\{-1\}\&\& \text{Comparer}==\{1\},  \text{Return}[\{-1\}]]; \\
 {}\quad \text{If}[ \text{Comparez}==\{0\}\&\& \text{Comparer}==\{1\}, \text{Return}[\{-1\}]]; \\
 {}\quad \text{If}[ \text{Comparez}==\{-1\}\&\& \text{Comparer}==\{0\}, \text{Return}[\{-1\}]]; \\
 {}\quad \text{If}[ \text{Comparew}==\{1\}\&\& \text{WX1}[[ kk4, ll4]]==1, 
\text{Return}[\{-1\}]]; \\
 {}\quad \text{If}[ \text{Comparez}==\{0\}\&\& \text{Comparew}==\{0\},  \text{Return}[\{0\}]]; \\
 {}\quad \text{If}[ \text{Comparez}==\{0\}\&\& \text{Comparer}==\{0\},  \text{Return}[\{0\}]]; \\
 {}\quad \text{If}[ \text{Comparew}==\{0\}\&\& \text{WX1}[[ ii4, 
jj4]]==1\&\& \text{WX1}[[ kk4, ll4]]==1, \text{Return}[\{0\}]]; \\
 {}\quad \text{If}[ \text{Comparez}==\{-1\}\&\& \text{Comparew}==\{-1\},  \text{Return}[\{1\}]]; \\
 {}\quad \text{If}[ \text{Comparez}==\{-1\}\&\& \text{Comparew}==\{0\},  \text{Return}[\{1\}]]; \\
 {}\quad \text{If}[ \text{Comparez}==\{0\}\&\& \text{Comparew}==\{-1\},  \text{Return}[\{1\}]]; \\
 {}\quad \text{If}[ \text{Comparez}==\{1\}\&\& \text{Comparer}==\{-1\}, \text{Return}[\{1\}]]; \\
 {}\quad \text{If}[ \text{Comparez}==\{0\}\&\& \text{Comparer}==\{-1\}, \text{Return}[\{1\}]]; \\
 {}\quad \text{If}[ \text{Comparez}==\{1\}\&\& \text{Comparer}==\{0\}, \text{Return}[\{1\}]]; \\
 {}\quad \text{If}[ \text{Comparew}==\{-1\}\&\& \text{WX1}[[ ii4, jj4]]==1, 
 \text{Return}[\{1\}]]; \\
 {}\quad \text{Return}[\{1,0,-1\}];) \\ 
 \\
 (* \text{ Compare} W_{ ij}\text{'s} \text{ according} 
\text{ to}\ W_{ij}= w_{ij}^{-1/2} z{ ij}^{-3/2}= w_{
ij} r_{ ij}^{-3}=z_{ij}^{-1}r_{ij}^{-1}= Z_{ ij}^
{1/3} {z}_{ij}^{-4/3}\ *) \\
\text{CompareOrderW1}[\{ \text{zomatrix4}\_, 
\text{womatrix4}\_\},\{ ii4\_, jj4\_\},\{ kk4\_, ll
4\_\}]:=( \\
 {}\quad \text{zo4}= \text{zomatrix4}; \text{wo4}= \text{womatrix4};  \text{dis}= \text{Dis3}[\{ \text{zo4}, \text{wo4}\}]; \\
 {}\quad \text{If}[ \text{WX1}[[ ii4, jj4]]==1\&\&  \text{WX1}[[ kk4, ll4]]==0,  \text{Return}[\{1\}]]; \\
 {}\quad \text{If}[ \text{WX1}[[ ii4, jj4]]==1\&\&  \text{WX1}[[ kk4, ll4]]==1,  \text{Return}[\{0\}]]; \\
 {}\quad \text{If}[ \text{WX1}[[ ii4, jj4]]==0\&\& \text{WX1}[[ kk4, ll4]]==1,  \text{Return}[\{-1\}]]; \\ \\
 {}\quad \text{Comparez} = \text{CompareOrderzw}[ \text{zo4}, \{ ii4, jj4\}, \{ kk4, ll4\}]; \\
 {} \quad \text{Comparew} = \text{CompareOrderzw}[ \text{wo4}, \{ ii4, jj4\}, \{ kk4, ll4\}]; \\
 {}\quad \text{Comparer} = \text{CompareOrderr}[ \text{dis}, \{ ii4, jj4\}, \{ kk4, ll4\}]; \\
\\
 {}\quad \text{If}[ \text{Comparew}==\{1\}\&\& \text{Comparez}==\{1\},  \text{Return}[\{-1\}]]; \\
 {}\quad \text{If}[ \text{Comparew}==\{1\}\&\& \text{Comparez}==\{0\},  \text{Return}[\{-1\}]]; \\
 {}\quad \text{If}[ \text{Comparew}==\{0\}\&\& \text{Comparez}==\{1\},  \text{Return}[\{-1\}]]; \\
 {}\quad \text{If}[ \text{Comparew}==\{-1\}\&\& \text{Comparer}==\{1\},  \text{Return}[\{-1\}]]; \\
 {}\quad \text{If}[ \text{Comparew}==\{0\}\&\& \text{Comparer}==\{1\}, \text{Return}[\{-1\}]]; \\
 {}\quad \text{If}[ \text{Comparew}==\{-1\}\&\& \text{Comparer}==\{0\}, \text{Return}[\{-1\}]]; \\
 {}\quad \text{If}[ \text{Comparez}==\{1\}\&\& \text{ZX1}[[ kk4, ll4]]==1, 
\text{Return}[\{-1\}]]; \\
 {}\quad \text{If}[ \text{Comparew}==\{0\}\&\& \text{Comparez}==\{0\},  \text{Return}[\{0\}]]; \\
 {}\quad \text{If}[ \text{Comparew}==\{0\}\&\& \text{Comparer}==\{0\},  \text{Return}[\{0\}]]; \\
 {}\quad \text{If}[ \text{Comparez}==\{0\}\&\& \text{ZX1}[[ ii4, 
jj4]]==1\&\& \text{ZX1}[[ kk4, ll4]]==1, \text{Return}[\{0\}]]; \\
 {}\quad \text{If}[ \text{Comparew}==\{-1\}\&\& \text{Comparez}==\{-1\},  \text{Return}[\{1\}]]; \\
 {}\quad \text{If}[ \text{Comparew}==\{-1\}\&\& \text{Comparez}==\{0\},  \text{Return}[\{1\}]]; \\
 {}\quad \text{If}[ \text{Comparew}==\{0\}\&\& \text{Comparez}==\{-1\},  \text{Return}[\{1\}]]; \\
 {}\quad \text{If}[ \text{Comparew}==\{1\}\&\& \text{Comparer}==\{-1\}, \text{Return}[\{1\}]]; \\
 {}\quad \text{If}[ \text{Comparew}==\{0\}\&\& \text{Comparer}==\{-1\}, \text{Return}[\{1\}]]; \\
 {}\quad \text{If}[ \text{Comparew}==\{1\}\&\& \text{Comparer}==\{0\}, \text{Return}[\{1\}]]; \\
 {}\quad \text{If}[ \text{Comparez}==\{-1\}\&\& \text{ZX1}[[ ii4, jj4]]==1, 
 \text{Return}[\{1\}]]; \\
 {}\quad \text{Return}[\{1,0,-1\}];) \\  \\
(* \text{ Compare}\ Z{ ij}\ \text{ and} z_k \
*) \\
 \text{CompareOrderZ2}[\{ \text{zomatrix4}\_, 
\text{womatrix4}\_\},\{ ii4\_, jj4\_\},\{ kk4\_, ll
4\_\}]:=( \\
 {}\quad \text{zo4}= \text{zomatrix4}; \text{wo4}= \text{womatrix4};  \text{dis}= \text{Dis3}[\{ \text{zo4}, \text{wo4}\}]; \\
 {}\quad \text{If}[ \text{ZX1}[[ ii4, jj4]]==1\&\&  \text{ZX1}[[ kk4, ll4]]==0,  \text{Return}[\{1\}]]; \\
 {}\quad \text{If}[ \text{ZX1}[[ ii4, jj4]]==1\&\&  \text{ZX1}[[ kk4, ll4]]==1,  \text{Return}[\{0\}]]; \\
 {}\quad \text{If}[ \text{ZX1}[[ ii4, jj4]]==0\&\&  \text{ZX1}[[ kk4, ll4]]==1,  \text{Return}[\{-1\}]]; \\ \\
 {}\quad \text{Comparez} = \text{CompareOrderzw}[ \text{zo4}, \{ ii4, jj4\}, \{ kk4, ll4\}]; \\
 {} \quad \text{Comparew} = \text{CompareOrderzw}[ \text{wo4}, \{ ii4, jj4\}, \{ kk4, ll4\}]; \\
 {}\quad \text{Comparer} = \text{CompareOrderr}[ \text{dis}, \{ ii4, jj4\}, \{ kk4, ll4\}]; \\
\\
 {}\quad \text{If}[ \text{Comparez}==\{-1\}\&\& \text{Comparer}==\{1\}, \text{Return}[\{-1\}]]; \\
 {}\quad \text{If}[ \text{Comparez}==\{0\}\&\&  \text{Comparer}==\{1\},  \text{Return}[\{-1\}]]; \\
 {}\quad \text{If}[ \text{Comparez}==\{-1\}\&\& \text{Comparer}==\{0\}, \text{Return}[\{-1\}]]; \\
 {}\quad \text{If}[ \text{Comparez}==\{0\}\&\&  \text{Comparer}==\{0\},  \text{Return}[\{0\}]]; \\
 {}\quad \text{If}[ \text{Comparez}==\{1\}\&\&  \text{Comparer}==\{-1\},  \text{Return}[\{1\}]]; \\
 {}\quad \text{If}[ \text{Comparez}==\{0\}\&\&  \text{Comparer}==\{-1\},  \text{Return}[\{1\}]]; \\
 {}\quad \text{If}[ \text{Comparez}==\{1\}\&\&  \text{Comparer}==\{0\},  \text{Return}[\{1\}]]; \\
 {}\quad \text{Return}[\{-1,0,1\}];  ) \\
 \\
 (* \text{ Compare}\ W{ ij}\ \text{ and} w_k \
*) \\
 \text{CompareOrderW2}[\{ \text{zomatrix4}\_, 
\text{womatrix4}\_\},\{ ii4\_, jj4\_\},\{ kk4\_, ll
4\_\}]:=( \\
 {}\quad \text{zo4}= \text{zomatrix4}; \text{wo4}= \text{womatrix4};  \text{dis}= \text{Dis3}[\{ \text{zo4}, \text{wo4}\}]; \\
 {}\quad \text{If}[ \text{WX1}[[ ii4, jj4]]==1\&\&  \text{WX1}[[ kk4, ll4]]==0,  \text{Return}[\{1\}]]; \\
 {}\quad \text{If}[ \text{WX1}[[ ii4, jj4]]==1\&\&  \text{WX1}[[ kk4, ll4]]==1,  \text{Return}[\{0\}]]; \\
 {}\quad \text{If}[ \text{WX1}[[ ii4, jj4]]==0\&\&  \text{WX1}[[ kk4, ll4]]==1,  \text{Return}[\{-1\}]]; \\ \\
 {}\quad \text{Comparez} = \text{CompareOrderzw}[ \text{zo4}, \{ ii4, jj4\}, \{ kk4, ll4\}]; \\
 {} \quad \text{Comparew} = \text{CompareOrderzw}[ \text{wo4}, \{ ii4, jj4\}, \{ kk4, ll4\}]; \\
 {}\quad \text{Comparer} = \text{CompareOrderr}[ \text{dis}, \{ ii4, jj4\}, \{ kk4, ll4\}]; \\
\\
 {}\quad \text{If}[ \text{Comparew}==\{-1\}\&\& \text{Comparer}==\{1\}, \text{Return}[\{-1\}]]; \\
 {}\quad \text{If}[ \text{Comparew}==\{0\}\&\&  \text{Comparer}==\{1\},  \text{Return}[\{-1\}]]; \\
 {}\quad \text{If}[ \text{Comparew}==\{-1\}\&\& \text{Comparer}==\{0\}, \text{Return}[\{-1\}]]; \\
 {}\quad \text{If}[ \text{Comparew}==\{0\}\&\&  \text{Comparer}==\{0\},  \text{Return}[\{0\}]]; \\
 {}\quad \text{If}[ \text{Comparew}==\{1\}\&\&  \text{Comparer}==\{-1\},  \text{Return}[\{1\}]]; \\
 {}\quad \text{If}[ \text{Comparew}==\{0\}\&\&  \text{Comparer}==\{-1\},  \text{Return}[\{1\}]]; \\
 {}\quad \text{If}[ \text{Comparew}==\{1\}\&\&  \text{Comparer}==\{0\},  \text{Return}[\{1\}]]; \\
 {}\quad \text{Return}[\{-1,0,1\}];  ) \\ 
 \\
 (* \text{ Determine} \text{ possible} \text{ maximal} \text{ terms} 
\text{ in}\ z_i-\text{equation}\ *) \\
\text{MaxOrderZ}[\{ \text{zomatrix3}\_, \text{womatrix3}\_\}, kk3\_] := ( \\
 {}\quad \text{zo3} = \text{zomatrix3}; \text{wo3} = \text{womatrix3};  \text{max3} = \text{allind}; \\
 {}\quad s3 = \text{Complement}[ \text{allind}, \{ kk3\}];  ii3 = 1; \\
 {}\quad \text{While}[ ii3 <= n - 1, jj3 = ii3 + 1; \\
 {}\quad \quad \text{While}[ jj3 <= n - 1, \\
 {}\quad \quad \quad \text{If}[ \text{CompareOrderZ1}[\{ \text{zo3}, 
\text{wo3}\}, \{ kk3, s3[[ ii3]]\}, \{ kk3, s3[[ jj3]]\}] == \{1\}, \\
 {}\quad \quad \quad \quad \text{max3} = \text{Complement}[ \text{max3}, 
\{ s3[[ jj3]]\}]; ]; \\
 {}\quad \quad \quad \text{If}[ \text{CompareOrderZ1}[\{ \text{zo3}, 
\text{wo3}\}, \{ kk3, s3[[ ii3]]\}, \{ kk3, s3[[ jj3]]\}] == \{-1\}, \\
 {}\quad \quad \quad \quad \text{max3} = \text{Complement}[ \text{max3}, 
\{ s3[[ ii3]]\}];  ];  jj3++ ]; \\
 {}\quad \quad \text{If}[ \text{CompareOrderZ2}[\{ \text{zo3}, 
\text{wo3}\}, \{ kk3, kk3\}, \{ kk3, s3[[ ii3]]\}] == \{1\}, \\
 {}\quad \quad \quad \text{max3} = \text{Complement}[ \text{max3}, \{ 
s3[[ ii3]]\}]]; \\
 {}\quad \quad \text{If}[ \text{CompareOrderZ2}[\{ \text{zo3}, 
\text{wo3}\}, \{ kk3, kk3\}, \{ kk3, s3[[ ii3]]\}] == \{-1\}, \\
 {}\quad \quad \quad \text{max3} = \text{Complement}[ \text{max3}, \{ 
kk3\}];  ];  ii3++];  \text{Return}[ \text{max3}]) \\
 \\
(* \text{ Determine} \text{ possible} \text{ maximal} \text{ terms} 
\text{ in}\ w_i-\text{equation}\ *) \\
 \text{MaxOrderW}[\{ \text{zomatrix3}\_, \text{womatrix3}\_\}, kk3\_] := ( \\
 {}\quad \text{zo3} = \text{zomatrix3}; \text{wo3} = \text{womatrix3};  \text{max3} = \text{allind}; \\
 {}\quad s3 = \text{Complement}[ \text{allind}, \{ kk3\}];  ii3 = 1; \\
 {}\quad \text{While}[ ii3 <= n - 1, jj3 = ii3 + 1; \\
 {}\quad \quad \text{While}[ jj3 <= n - 1, \\
 {}\quad \quad \quad \text{If}[ \text{CompareOrderW1}[\{ \text{zo3}, 
\text{wo3}\}, \{ kk3, s3[[ ii3]]\}, \{ kk3, s3[[ jj3]]\}] == \{1\}, \\
 {}\quad \quad \quad \quad \text{max3} = \text{Complement}[ \text{max3}, 
\{ s3[[ jj3]]\}]; ]; \\
 {}\quad \quad \quad \text{If}[ \text{CompareOrderW1}[\{ \text{zo3}, 
\text{wo3}\}, \{ kk3, s3[[ ii3]]\}, \{ kk3, s3[[ jj3]]\}] == \{-1\}, \\
 {}\quad \quad \quad \quad \text{max3} = \text{Complement}[ \text{max3}, 
\{ s3[[ ii3]]\}];  ];  jj3++ ]; \\
 {}\quad \quad \text{If}[ \text{CompareOrderW2}[\{ \text{zo3}, 
\text{wo3}\}, \{ kk3, kk3\}, \{ kk3, s3[[ ii3]]\}] == \{1\}, \\
 {}\quad \quad \quad \text{max3} = \text{Complement}[ \text{max3}, \{ 
s3[[ ii3]]\}]]; \\
 {}\quad \quad \text{If}[ \text{CompareOrderW2}[\{ \text{zo3}, 
\text{wo3}\}, \{ kk3, kk3\}, \{ kk3, s3[[ ii3]]\}] == \{-1\}, \\
 {}\quad \quad \quad \text{max3} = \text{Complement}[ \text{max3}, \{ 
kk3\}];  ];  ii3++];  \text{Return}[ \text{max3}]) \\
 \\
(* \text{ Proposition}\ 4.12\ *) \\
 \text{Rule1cQ}[\{ \text{zomatrix2}\_, \text{womatrix2}\_\}]:=( \\
 {}\quad ii2=1; \\
 {}\quad \text{While}[ ii2<= \text{Length}[ \text{zc1}], \\
 {}\quad \quad \text{zc2}= \text{Intersection}[ \text{zc1}[[ ii2]], 
\text{Flatten}@ \text{Position}[ \text{Diagonal}[ \text{ZX1}],1]]; 
\\
 {}\quad \quad \text{If}[ \text{Length}[ \text{zc2}]>0\&\& \text{Count}[ \text{Flatten}@ \text{zomatrix2}[[ 
\text{zc2}, \text{zc2}]],5]== \text{Length}[ \text{zc2}], \\
 {}\quad \quad \quad \text{Return}[ \text{False}]]; \quad ii2++]; \\
 {}\quad ii2=1; \\
 {}\quad \text{While}[ ii2<= \text{Length}[ \text{wc1}], \\
 {}\quad \quad \text{wc2}= \text{Intersection}[ \text{wc1}[[ ii2]], 
\text{Flatten}@ \text{Position}[ \text{Diagonal}[ \text{WX1}],1]]; 
\\
 {}\quad \quad \text{If}[ \text{Length}[ \text{wc2}]>0\&\& \text{Count}[ \text{Flatten}@ \text{womatrix2}[[ 
\text{wc2}, \text{wc2}]],5]== \text{Length}[ \text{wc2}], \\
 {}\quad \quad \quad \text{Return}[ \text{False}]]; \quad ii2++]; \\
 {}\quad \text{Return}[ \text{True}]) \\
 \\
 (* \text{ Proposition}\ 4.13\ *) \\
 \text{OrderQ1}[\{ \text{zomatrix2}\_, \text{womatrix2}\_\}]:=( \\
 {}\quad kk2=1; \\
 {}\quad \text{While}[ kk2<= n, \\
 {}\quad \quad \text{If}[ \text{Length}[ \text{MaxOrderZ}[\{ 
\text{zomatrix2}, \text{womatrix2}\}, kk2]]==1|| \\
 {}\quad \quad \quad \text{Length}[ \text{MaxOrderW}[\{ 
\text{zomatrix2}, \text{womatrix2}\}, kk2]]==1, \\
 {}\quad \quad \quad \text{Return}[ \text{False}]]; kk2++;  ];  \text{Return}[ \text{True}]) \\
 \\
 \quad (* \text{ Proposition}\ 4.14\ *) \\
 \text{OrderQ2}[\{ \text{zomatrix2}\_, \text{womatrix2}\_\}]:=( \\
 {}\quad \text{dis2}= \text{Dis3}[\{ \text{zomatrix2}, 
\text{womatrix2}\}]; ii2=1; \\
 {}\quad \text{While}[ ii2<= n-1,\ jj2= ii2+1; \\
 {}\quad \quad \text{While}[ jj2<= n, \\
 {}\quad \quad \quad \text{If}[ \text{MaxOrderZ}[\{ \text{zomatrix2}, 
\text{womatrix2}\}, ii2]==\{ ii2, jj2\}\&\& \\
 {}\quad \quad \quad \quad \text{MaxOrderZ}[\{ \text{zomatrix2}, 
\text{womatrix2}\}, jj2]==\{ ii2, jj2\}\&\& \\
 {}\quad \quad \quad \quad \text{Intersection}[ \text{dis2}[[ ii2, 
jj2]],\{3\}]==\{\}, \\
 {}\quad \quad \quad \quad \text{Return}[ \text{False}]]; \\
 {}\quad \quad \quad \text{If}[ \text{MaxOrderW}[\{ \text{zomatrix2}, 
\text{womatrix2}\}, ii2]==\{ ii2, jj2\}\&\& \\
 {}\quad \quad \quad \quad \text{MaxOrderW}[\{ \text{zomatrix2}, 
\text{womatrix2}\}, jj2]==\{ ii2, jj2\}\&\& \\
 {}\quad \quad \quad \quad \text{Intersection}[ \text{dis2}[[ ii2, 
jj2]],\{3\}]==\{\}, \\
 {}\quad \quad \quad \quad \text{Return}[ \text{False}]];  jj2++;  ];  ii2++];  \text{Return}[ \text{True}]) \\
 \\
 \text{(* proposition 4.15 *)}
\text{OrderQ3}[\{ \text{zomatrix2}\_, \text{womatrix2}\_\}]:=( \\
 {}\quad \text{maxZ}= \text{Table}[ \text{MaxOrderZ}[\{ \text{zomatrix2}, 
\text{womatrix2}\}, i2],\{ i2, n\}]; \\
 {}\quad \text{maxW}= \text{Table}[ \text{MaxOrderW}[\{ \text{zomatrix2}, 
\text{womatrix2}\}, i2],\{ i2, n\}]; \\
 {}\quad ii2=1; \\
 {}\quad \text{While}[ ii2<= n-2,  jj2= ii2+1; \\
 {}\quad \quad \text{While}[ jj2<= n-1,  kk2= jj2+1; \\
 {}\quad \quad \quad \text{While}[ kk2<= n, \\
 {}\quad \quad \quad \quad \text{If}[ \text{ZX1}[[ ii2, jj2]] \text{ZX1}[[ 
jj2, kk2]]!=0 \&\&  \text{ZX1}[[ kk2, ii2]]==0, \\
 {}\quad \quad \quad \quad \quad \text{If}[ \text{maxW}[[ jj2]]==\{ ii2, 
kk2\},  \text{Return}[ \text{False}]]; \\
 {}\quad \quad \quad \quad \quad \text{If}[ \text{Complement}[ \text{maxW}[[ 
ii2]]\cup \text{maxW}[[ jj2]]\cup \text{maxW}[[ kk2]],\{ ii2, jj2, 
kk2\}]==\{\}, \\
 {}\quad \quad \quad \quad \quad \quad \text{Return}[ \text{False}]]; ]; \\
 {}\quad \quad \quad \quad \text{If}[ \text{WX1}[[ ii2, jj2]] \text{WX1}[[ 
jj2, kk2]]!=0 \&\&  \text{WX1}[[ kk2, ii2]]==0, \\
 {}\quad \quad \quad \quad \quad \text{If}[ \text{maxZ}[[ jj2]]==\{ ii2, 
kk2\},  \text{Return}[ \text{False}]]; \\
 {}\quad \quad \quad \quad \quad \text{If}[ \text{Complement}[ \text{maxZ}[[ 
ii2]]\cup \text{maxZ}[[ jj2]]\cup \text{maxZ}[[ kk2]],\{ ii2, jj2, 
kk2\}]==\{\}, \\
 {}\quad \quad \quad \quad \quad \quad \text{Return}[ \text{False}]]; ]; \\
 {}\quad \quad \quad \quad \text{If}[ \text{ZX1}[[ jj2, kk2]] \text{ZX1}[[ 
kk2, ii2]]!=0 \&\&  \text{ZX1}[[ ii2, jj2]]==0, \\
 {}\quad \quad \quad \quad \quad \text{If}[ \text{maxW}[[ kk2]]==\{ ii2, 
jj2\}, \text{Return}[ \text{False}]]; \\
 {}\quad \quad \quad \quad \quad \text{If}[ \text{Complement}[ \text{maxW}[[ 
ii2]]\cup \text{maxW}[[ jj2]]\cup \text{maxW}[[ kk2]],\{ ii2, jj2, 
kk2\}]==\{\}, \\
 {}\quad \quad \quad \quad \quad \quad \text{Return}[ \text{False}]];  ]; \\
 {}\quad \quad \quad \quad \text{If}[ \text{WX1}[[ jj2, kk2]] \text{WX1}[[ 
kk2, ii2]]!=0 \&\&  \text{WX1}[[ ii2, jj2]]==0, \\
 {}\quad \quad \quad \quad \quad \text{If}[ \text{maxZ}[[ kk2]]==\{ ii2, 
jj2\},  \text{Return}[ \text{False}]]; \\
 {}\quad \quad \quad \quad \quad \text{If}[ \text{Complement}[ \text{maxZ}[[ 
ii2]]\cup \text{maxZ}[[ jj2]]\cup \text{maxZ}[[ kk2]],\{ ii2, jj2, 
kk2\}]==\{\}, \\
 {}\quad \quad \quad \quad \quad \quad \text{Return}[ \text{False}]]; ]; \\
 {}\quad \quad \quad \quad \text{If}[ \text{ZX1}[[ kk2, ii2]] \text{ZX1}[[ 
ii2, jj2]]!=0 \&\&  \text{ZX1}[[ jj2, kk2]]==0, \\
 {}\quad \quad \quad \quad \quad \text{If}[ \text{maxW}[[ ii2]]==\{ jj2, 
kk2\},  \text{Return}[ \text{False}]]; \\
 {}\quad \quad \quad \quad \quad \text{If}[ \text{Complement}[ \text{maxW}[[ 
ii2]]\cup \text{maxW}[[ jj2]]\cup \text{maxW}[[ kk2]],\{ ii2, jj2, 
kk2\}]==\{\}, \\
 {}\quad \quad \quad \quad \quad \quad \text{Return}[ \text{False}]]; ]; \\
 {}\quad \quad \quad \quad \text{If}[ \text{WX1}[[ kk2, ii2]] \text{WX1}[[ 
ii2, jj2]]!=0 \&\&  \text{WX1}[[ jj2, kk2]]==0, \\
 {}\quad \quad \quad \quad \quad \text{If}[ \text{maxZ}[[ ii2]]==\{ jj2, 
kk2\},  \text{Return}[ \text{False}]]; \\
 {}\quad \quad \quad \quad \quad \text{If}[ \text{Complement}[ \text{maxZ}[[ 
ii2]]\cup \text{maxZ}[[ jj2]]\cup \text{maxZ}[[ kk2]],\{ ii2, jj2, 
kk2\}]==\{\}, \\
 {}\quad \quad \quad \quad \quad \quad \text{Return}[ \text{False}]];  ]; \\
 {}\quad \quad \quad \quad kk2++]; jj2++];  ii2++];  \text{Return}[ \text{True}]) \\
 \\
\text{outputform1}[ \text{ordermatrix2}\_] := ( \\
 {}\quad \text{OMA} = \text{ordermatrix2}; \\
 {}\quad \text{OMB} = \text{Table}[ \text{If}[ \text{Length}@ 
\text{OMA}[[ i2, j2]] > 1, \\
 {}\quad \quad \{ \text{ToString}[ \text{Min}@ \text{OMA}[[ i2, j2]]] <> 
``\sim '' <> \text{ToString}[ \text{Max}@ \text{OMA}[[ i2, j2]]]\}, 
\\
 {}\quad \quad \text{OMA}[[ i2, j2]]], \{ i2, 1, n\}, \{ j2, 1, 
n\}]; \text{Return}[ \text{OMB}]) \\
 \\
\text{(* MAIN PROGRAM FOR FINDING ORDERS *)} \\
\text{indOrder}[ \text{x}\_]:= \text{FindOrder}[ \text{x},60]; \\
 {} \text{FindOrder}[ \text{x}\_, \text{time1}\_]:=( \\
 {}\quad t1= \text{AbsoluteTime}[]; N1= \text{x}; \text{X1}= \text{FMX}[[ \text{x}]]; \\
 {}\quad \text{Clear}[ \text{zwofamily}, \text{zwotype2}, 
\text{zwotype3}]; \\
 {}\quad \text{constrainttime}= \text{time1}; \text{TF1}=0; \\
 {}\quad \text{Print}[ \text{Style}[`` \text{finding } \text{zw}- 
\text{order } \text{matrices } \text{covering } '', \text{Bold},16], \text{Style}[`` \text{FMX}[['' <> \\
 {}\quad \quad \text{ToString}[ N1]<>``]]'', \text{Bold},16, \text{RGBColor}[0.8,0.4,0]], \text{Style}[``...'', \text{Bold},16]]; \\
 {}\quad \text{zo1}= \text{Table}[\{1,2,3,4,5\},\{ i1,1, n\},\{ j1,1, n\}]; \\
 {}\quad \text{wo1}= \text{Table}[\{1,2,3,4,5\},\{ i1,1, n\},\{ j1,1, n\}]; 
\\
 {}\quad \text{ZX1}= \text{zpart}[ \text{X1}];  \text{WX1}= \text{wpart}[ \text{X1}];  \text{XX1}= \text{ZX1}+2* \text{WX1}; \\
 {}\quad \text{zc1}= \text{componentsZ}[ \text{X1}]; \text{wc1}= \text{componentsW}[ \text{X1}]; \\
 {}\quad \text{pmzw}=\{\}; \\
 {}\quad \text{Do}[ \text{If}[ \text{ZX1}[[ \text{pm}[[ i1]], \text{pm}[[ 
i1]]]]== \text{ZX1}\&\& \text{WX1}[[ \text{pm}[[ i1]], \text{pm}[[ i1]]]]== 
\text{WX1}, \\
 {}\quad \quad \quad \text{pmzw}= \text{pmzw}\sim 
\text{Join}\sim \{ \text{pm}[[ i1]]\}], \{ i1,1, \text{Length}[ \text{pm}]\}]; \\ \\
 (* \text{ Proposition}\ 4.4( \text{ a})\ *) \\
 \text{Do}[ \text{If}[ \text{ZX1}[[ i1, i1]]==1, \text{zo1}[[ i1, i1]]=\{5\},  \text{zo1}[[ i1, i1]]=\{0,1,2,3,4\}]; \\
 {}\quad \quad \text{If}[ \text{WX1}[[ i1, i1]]==1, \text{wo1}[[ i1, i1]]=\{5\},  \text{wo1}[[ i1, i1]]=\{0,1,2,3,4\}],  \{ i1,1, n\}]; \\ \\
(* \text{ Proposition}\ 4.8\ *) \\
 \text{Do}[ L1= \text{Length}[ \text{zc1}[[ i1]]]; \\
 {}\quad \quad \text{Pos1}= \text{Position}[ \text{XX1}[[ \text{zc1}[[ i1]], 
\text{zc1}[[ i1]]]],1]; \\
 {}\quad \quad \text{Pos1}= \text{Complement}[ \text{Pos1}, \text{Table}[\{ 
j1, j1\},\{ j1,1, n\}]]; \\
 {}\quad \quad \text{If}[ \text{Length}[ \text{Pos1}]==2\&\& \\
 {}\quad \quad \quad \text{Total}[ \text{Diagonal}[ \text{ZX1}[[ 
\text{zc1}[[ i1]], \text{zc1}[[ i1]]]]]]>0, \\
 {}\quad \quad \quad \text{Do}[ \text{zo1}[[ \text{zc1}[[ i1]][[ 
\text{Pos1}[[ j1]][[1]]]], \text{zc1}[[ i1]][[ \text{Pos1}[[ 
j1]][[2]]]]]]=\{5\}; \\
 {}\quad \quad \quad \quad \text{wo1}[[ \text{zc1}[[ i1]][[ \text{Pos1}[[ 
j1]][[1]]]], \text{zc1}[[ i1]][[ \text{Pos1}[[ j1]][[2]]]]]]=\{1\}, \\
 {}\quad \quad \quad \quad \{ j1,1,2\}]], \{ i1,1, \text{Length}[ \text{zc1}]\}]; \\
 {}\quad \text{Do}[ L1= \text{Length}[ \text{wc1}[[ i1]]]; \\
 {}\quad \quad \text{Pos1}= \text{Position}[ \text{XX1}[[ \text{wc1}[[ i1]], 
\text{wc1}[[ i1]]]],2]; \\
 {}\quad \quad \text{Pos1}= \text{Complement}[ \text{Pos1}, \text{Table}[\{ 
j1, j1\},\{ j1,1, n\}]]; \\
 {}\quad \quad \text{If}[ \text{Length}[ \text{Pos1}]==2\&\& \\
 {}\quad \quad \quad \text{Total}[ \text{Diagonal}[ \text{WX1}[[ 
\text{wc1}[[ i1]], \text{wc1}[[ i1]]]]]]>0, \\
 {}\quad \quad \quad \text{Do}[ \text{wo1}[[ \text{wc1}[[ i1]][[ 
\text{Pos1}[[ j1]][[1]]]], \text{wc1}[[ i1]][[ \text{Pos1}[[ 
j1]][[2]]]]]]=\{5\}; \\
 {}\quad \quad \quad \quad \text{zo1}[[ \text{wc1}[[ i1]][[ \text{Pos1}[[ 
j1]][[1]]]], \text{wc1}[[ i1]][[ \text{Pos1}[[ j1]][[2]]]]]]=\{1\}, \\
 {}\quad \quad \quad \quad \{ j1,1,2\}]], \{ i1,1, \text{Length}[ \text{wc1}]\}]; \\ \\
 (* \text{ Proposition}\ 4.4( b)\ *) \\
 \text{Do}[ \text{If}[ i1!= j1, \\
 {}\quad \quad \quad \text{If}[ \text{XX1}[[ i1, j1]]==1, \\
 {}\quad \quad \quad \quad \text{zo1}[[ i1, j1]]= \text{Intersection}[ 
\text{zo1}[[ i1, j1]],\{4,5\}]; \\
 {}\quad \quad \quad \quad \text{wo1}[[ i1, j1]]= \text{Intersection}[ 
\text{wo1}[[ i1, j1]],\{1,2\}]]], \\
 {}\quad \quad \{ i1,1, n\}, \{ j1,1, n\}]; \\
 {}\quad \text{Do}[ \text{If}[ i1!= j1, \\
 {}\quad \quad \quad \text{If}[ \text{XX1}[[ i1, j1]]==2, \\
 {}\quad \quad \quad \quad \text{wo1}[[ i1, j1]]= \text{Intersection}[ 
\text{wo1}[[ i1, j1]],\{4,5\}]; \\
 {}\quad \quad \quad \quad \text{zo1}[[ i1, j1]]= \text{Intersection}[ 
\text{zo1}[[ i1, j1]],\{1,2\}]]], \\
 {}\quad \quad \{ i1,1, n\},  \{ j1,1, n\}]; \\
 {}\quad \text{Do}[ \text{If}[ i1!= j1, \\
 {}\quad \quad \quad \text{If}[ \text{XX1}[[ i1, j1]]==3, \\
 {}\quad \quad \quad \quad \text{wo1}[[ i1, j1]]= \text{Intersection}[ 
\text{wo1}[[ i1, j1]],\{3\}]; \\
 {}\quad \quad \quad \quad \text{zo1}[[ i1, j1]]= \text{Intersection}[ 
\text{zo1}[[ i1, j1]],\{3\}]]], \\
 {}\quad \quad \{ i1,1, n\}, \{ j1,1, n\}]; \\ \\
 (* \text{ Determined} \text{ the} \text{ order} \text{ of} 
\text{ z}\_\{ \text{ ij}\} \text{ and} \text{ w}\_\{ \text{ ij}\} \text{ if} 
\text{ there} \text{ is} \text{ a}\\
{} \quad \text{ maximal}\ z-\text{ or } w- \text{stroke}\  *) \\
 \text{Do}[ \text{If}[ i1!= j1, \\
 {}\quad \quad \quad \text{If}[ \text{XX1}[[ i1, j1]]==1\&\& \\
 {}\quad \quad \quad \quad ( \text{Total}[ \text{ZX1}[[ i1]]]+ \text{Total}[ 
\text{ZX1}[[ j1]]]- \text{ZX1}[[ i1, i1]]- \text{ZX1}[[ j1, j1]])==2, \\
 {}\quad \quad \quad \quad \text{zo1}[[ i1, j1]]= \text{Intersection}[ 
\text{zo1}[[ i1, j1]],\{5\}]; \\
 {}\quad \quad \quad \quad \text{wo1}[[ i1, j1]]= \text{Intersection}[ 
\text{wo1}[[ i1, j1]],\{1\}]]], \\
 {}\quad \quad \{ i1,1, n\}, \{ j1,1, n\}]; \\
 {}\quad \text{Do}[ \text{If}[ i1!= j1, \\
 {}\quad \quad \quad \text{If}[ \text{XX1}[[ i1, j1]]==2\&\& \\
 {}\quad \quad \quad \quad ( \text{Total}[ \text{WX1}[[ i1]]]+ \text{Total}[ 
\text{WX1}[[ j1]]]- \text{WX1}[[ i1, i1]]- \text{WX1}[[ j1, j1]])==2, \\
 {}\quad \quad \quad \quad \text{zo1}[[ i1, j1]]= \text{Intersection}[ 
\text{zo1}[[ i1, j1]],\{1\}]; \\
 {}\quad \quad \quad \quad \text{wo1}[[ i1, j1]]= \text{Intersection}[ 
\text{wo1}[[ i1, j1]],\{5\}]]], \\
 {}\quad \quad \{ i1,1, n\}, \{ j1,1, n\}]; \\ \\
(* \text{ Proposition}\ 4.4( b)\ *) \\
 \text{Do}[ \text{If}[ i1!= j1, \\
 {}\quad \quad \quad \text{If}[ \text{ZX1}[[ i1, j1]]==0, \\
 {}\quad \quad \quad \quad \text{wo1}[[ i1, j1]]= \text{Intersection}[ 
\text{wo1}[[ i1, j1]],\{2,3,4,5\}]]], \\
 {}\quad \quad \{ i1,1, n\}, \{ j1,1, n\}]; \\
 {}\quad \text{Do}[ \text{If}[ i1!= j1, \\
 {}\quad \quad \quad \text{If}[ \text{WX1}[[ i1, j1]]==0, \\
 {}\quad \quad \quad \quad \text{zo1}[[ i1, j1]]= \text{Intersection}[ 
\text{zo1}[[ i1, j1]],\{2,3,4,5\}]]], \\
 {}\quad \quad \{ i1,1, n\}, \{ j1,1, n\}]; \\ \\
 {}\quad (* \text{ Refine } zw- \text{order} \text{ matrices} 
\text{ by} \text{ principles} \text{ for} \text{ updating} \text{ orders}\ *) 
\\
 \{ \text{zo1}, \text{wo1}\}= \text{UpdateOrder}[\{ \text{zo1}, 
\text{wo1}\}]; \\
 {}\quad \text{Print}[ \text{Style}[`` \text{result} \text{of} \text{setp} 
1:'', \text{Bold},15]]; \\
 {}\quad \text{Print}[`` \text{cumulative} \text{time}: '', 
\text{AbsoluteTime}[]- t1, `` \setminus \text{n}'' ]; \\
 {}\quad \text{Print}[``\{ \text{z}- \text{order}, \text{w}- \text{order}, 
r- \text{order}\} \text{matrices} :'']; \\
 {}\quad \text{Print}[\{ \text{Style}[ \text{MatrixForm}[ 
\text{outputform1}@ \text{zo1}], \text{Blue}],\\
 {}\quad\quad \text{Style}[\text{MatrixForm}[ \text{outputform1}@ \text{wo1}], \text{Blue}], \\
 {}\quad\quad \text{Style}[ \text{MatrixForm}[ \text{outputform1}@ \text{Dis3}[\{
\text{zo1}, \text{wo1}\}]], \text{RGBColor}[0,0.5,0]]\}]; \\
 {}\quad \text{Str}=`` \text{There} \text{are} ''<> \text{ToString}[ 
\text{NumberofType2}[\{ \text{zo1}, \text{wo1}\}]]<> \\
 {}\quad\quad \quad\ `` \text{possible} 
\text{zw}- \text{order} \text{matrices} \text{of} \text{Type} 2.''; \text{Print}[ \text{Str}, `` \setminus \text{n}\setminus \text{n}'']; \\
 {}\quad \text{Str}=`` \text{There} \text{are} ''<> \text{ToString}[ 
\text{NumberofType3}[\{ \text{zo1}, \text{wo1}\}]]<> \\
 {}\quad\quad \quad\ `` \text{possible} 
\text{zw}- \text{order} \text{matrices} \text{of} \text{Type} 3.''; \text{Print}[ \text{Str}, `` \setminus \text{n}\setminus \text{n}'']; \\ \\
\text{(* Step 2: Type 2 order matrices *)} \\
\text{zwotype2} = \{\};  \text{zwotype3} = \{\}; \\
 {} \text{TimeConstrained}[ \text{zwofamily} = \{\{ \text{zo1}, 
\text{wo1}\}\}; \\
 {} \text{While}[ \text{Length}[ \text{zwofamily}] > 0, \\
 {}\quad L1 = \text{Length}[ \text{zwofamily}]; \text{Lz1} = 0; \text{Lw1} = 0; \text{Del} = \{\}; \\
 {}\quad \text{Do}[ \text{If}[ \text{NumberofType2}[ \text{zwofamily}[[ 
i1]]] > 1, \\
 {}\quad \quad \quad \text{Del} = \text{Del}\sim \text{Join}\sim 
\{\{ i1\}\}; \\
 {}\quad \quad \quad \{ \text{Pzi}, \text{Pzj}, \text{Lz1}\} = 
\text{PosOffDiag}[ \text{zwofamily}[[ i1]][[1]]]; \\
 {}\quad \quad \quad \text{If}[ \text{Lz1} == 0, \\
 {}\quad \quad \quad \quad \{ \text{Pwi}, \text{Pwj}, \text{Lw1}\} = 
 \text{PosOffDiag}[ \text{zwofamily}[[ i1]][[2]]]; \\
 {}\quad \quad \quad \quad \text{Do}[ \text{zwofamily} = 
\text{zwofamily}\sim \text{Join}\sim \{ \text{zwofamily}[[ i1]]\}; \\
 {}\quad \quad \quad \quad \quad \text{zwofamily}[[-1]][[2]][[ 
\text{Pwi}, \text{Pwj}]] = \{ \text{zwofamily}[[ i1]][[2]][[ 
\text{Pwi}, \text{Pwj}]][[ j1]]\}; \\
 {}\quad \quad \quad \quad \quad \text{zwofamily}[[-1]][[2]][[ 
\text{Pwj}, \text{Pwi}]] = \{ \text{zwofamily}[[ i1]][[2]][[ 
\text{Pwi}, \text{Pwj}]][[ j1]]\}; \\
 {}\quad \quad \quad \quad \quad \text{zwofamily}[[-1]] = 
\text{UniqueZeroOrder}[ \text{zwofamily}[[-1]]]; \\
 {}\quad \quad \quad \quad \quad \text{zwofamily}[[-1]] = 
\text{UpdateOrder}[ \text{zwofamily}[[-1]]];  \\
 {}\quad \quad \quad \quad \quad \{ j1, 1, \text{Lw1}\}] , \\
 {}\quad \quad \quad \quad \text{Do}[ \text{zwofamily} = 
\text{zwofamily}\sim \text{Join}\sim \{ \text{zwofamily}[[ i1]]\}; \\
 {}\quad \quad \quad \quad \quad \text{zwofamily}[[-1]][[1]][[ 
\text{Pzi}, \text{Pzj}]] = \{ \text{zwofamily}[[ i1]][[1]][[ 
\text{Pzi}, \text{Pzj}]][[ j1]]\}; \\
 {}\quad \quad \quad \quad \quad \text{zwofamily}[[-1]][[1]][[ 
\text{Pzj}, \text{Pzi}]] = \{ \text{zwofamily}[[ i1]][[1]][[ 
\text{Pzi}, \text{Pzj}]][[ j1]]\}; \\
 {}\quad \quad \quad \quad \quad \text{zwofamily}[[-1]] = 
\text{UniqueZeroOrder}[ \text{zwofamily}[[-1]]]; \\
 {}\quad \quad \quad \quad \quad \text{zwofamily}[[-1]] = 
\text{UpdateOrder}[ \text{zwofamily}[[-1]]]; \\
 {}\quad \quad \quad \quad \quad \{ j1, 1, \text{Lz1}\}] ] ],  \{ i1, 1, L1\}]; \\
 {}\quad \text{zwofamily} = \text{Delete}[ \text{zwofamily}, 
\text{Del}]; \\
 {}\quad \text{zwotype2} = \text{zwotype2}\sim 
\text{Join}\sim \text{Select}[ \text{zwofamily}, \text{NumberofType2}[\#] == 
1 \&]; \\
 {}\quad \text{zwofamily} = \text{Select}[ \text{zwofamily}, 
\text{NumberofType2}[\#] > 1 \& ];  ]; \\
 {} (* \text{ End } \text{of } \text{While } \text{loop } *) \\ \\
 \text{zwotype2} = \text{Select}[ \text{zwotype2}, 
\text{Rule2hQ}]; \\
 {} \text{zwotype2} = \text{Select}[ \text{zwotype2}, 
\text{OrderQ1}]; \\
 {} \text{zwotype2} = \text{Select}[ \text{zwotype2}, 
\text{OrderQ2}]; \\
 {} \text{zwotype2} = \text{Select}[ \text{zwotype2}, 
\text{OrderQ3}]; \\
 {} \text{zwotype2} = \text{Select}[ \text{zwotype2}, 
\text{Rule1cQ}]; \\ \\
 {} \text{del} = \{\}; ii1 = 2; \\
 {} \text{While}[ ii1 <= \text{Length}[ \text{zwotype2}], jj1 = 1;  \text{tf1} = 0; \\
 {}\quad \text{While}[ jj1 < ii1 \&\&  \text{tf1} == 0,  kk1 = 1; \\
 {}\quad \quad \text{While}[ kk1 <= \text{Length}[ \text{pmzw}], \\
 {}\quad \quad \quad \text{If}[ \text{zwotype2}[[ ii1]][[1]] == 
 \text{zwotype2}[[ jj1]][[1]][[ \text{pmzw}[[ kk1]], 
\text{pmzw}[[ kk1]]]] \&\& \\
 {}\quad \quad \quad \quad \text{zwotype2}[[ ii1]][[2]] == 
\text{zwotype2}[[ jj1]][[2]][[ \text{pmzw}[[ kk1]], 
\text{pmzw}[[ kk1]]]], \\
 {}\quad \quad \quad \quad \text{del} = \text{del}\sim 
\text{Join}\sim \{\{ ii1\}\}; \\
 {}\quad \quad \quad \quad \text{tf1} = 1];  kk1++]; jj1++]; ii1++]; \\
 {} \text{zwotype2} = \text{Delete}[ \text{zwotype2}, 
\text{del}]; \\
 {} \text{Print}[ ``\setminus \text{n} '' , \text{Style}[`` \text{result} \text{of} \text{setp} 
2:'', \text{Bold}, 15]]; \\
 {} \text{Print}[`` \text{cumulative} \text{time}: '', 
\text{AbsoluteTime}[] - t1, \ ``\setminus \text{n} '' ]; \\
 {} \text{Print}[`` \text{There } \text{are } '', \text{Style}[ 
\text{Length}[ \text{zwotype2}], \text{Bold}, 
\text{RGBColor}[0, 0, 0.7], 14],\\
 {} \quad `` zw-\text{order } \text{matrices } \text{of } \text{Type } 2.'']; \\
 {} \text{Print}[``\{ z-\text{order},\ w-\text{order}, r- 
\text{order }\}\ \text{matrices} :'']; \\
 {} \text{Print}[ \text{Table}[\{ \text{Style}[ \text{MatrixForm}[ 
\text{outputform1}@ \text{zwotype2}[[ j1, 1]]], \text{Blue}], \\
 {} \quad 
\text{Style}[ \text{MatrixForm}[ \text{outputform1}@ 
\text{zwotype2}[[ j1, 2]]], \text{Blue}], \\
 {} \quad \text{Style}[ \text{MatrixForm}[ \text{outputform1}@ \text{Dis3}[ 
\text{zwotype2}[[ j1]]]],\\
 {} \quad  \text{RGBColor}[0, 0.5, 0]]\}, \{ j1, 
1, \text{Length}[ \text{zwotype2}]\}]]; \\
 {} \text{Str} = `` \text{There } \text{are } '' <> \text{ToString}[ 
 \text{Sum}[ \text{NumberofType3}[ \text{zwotype2}[[ i1]]],\\
 {}\quad\quad \{ i1, 1, \text{Length}[ \text{zwotype2}]\}]] <> `` 
\text{ possible } \text{zw }- \text{order } \text{matrices } \text{of } 
 \text{Type } 3.''; \\
 {} \text{Print}[ \text{Str},\ ``\setminus \text{n} '']; \\
 \\ 
 \text{(* Step 3: Type 3 order matrices *)} \\
 {} \text{zwofamily}= \text{zwotype2}; \\
 {} \text{While}[ \text{Length}[ \text{zwofamily}]>0, \\
 {}\quad L1= \text{Length}[ \text{zwofamily}]; \text{Lz1}=0; \text{Lw1}=0;  \text{Del}=\{\}; \\
 {}\quad \text{Do}[ \text{If}[ \text{NumberofType3}[ \text{zwofamily}[[ 
i1]]]>1, \\
 {}\quad \quad \quad \text{Del}= \text{Del}\sim \text{Join}\sim \{\{ 
i1\}\}; \\
 {}\quad \quad \quad \{ \text{Pzd}, \text{Lz1}\}= \text{PosDiag}[ 
\text{zwofamily}[[ i1]][[1]]]; \\
 {}\quad \quad \quad \text{If}[ \text{Lz1}==0, \\
 {}\quad \quad \quad \quad \{ \text{Pwd}, \text{Lw1}\}= \text{PosDiag}[ 
\text{zwofamily}[[ i1]][[2]]]; \\
 {}\quad \quad \quad \quad \text{Do}[ \text{zwofamily}= \text{zwofamily}\sim 
 \text{Join}\sim \{ \text{zwofamily}[[ i1]]\}; \\
 {}\quad \quad \quad \quad \quad \text{zwofamily}[[-1]][[2]][[ \text{Pwd}, 
\text{Pwd}]]=\{ \text{zwofamily}[[ i1]][[2]][[ \text{Pwd}, \text{Pwd}]][[ 
j1]]\}; \\
 {}\quad \quad \quad \quad \quad \text{zwofamily}[[-1]]= 
\text{UpdateOrderPart}[ \text{zwofamily}[[-1]]]; , \\
 {}\quad \quad \quad \quad \quad \{ j1,1, \text{Lw1}\}], \\
 {}\quad \quad \quad \quad \text{Do}[ \text{zwofamily}= \text{zwofamily}\sim 
 \text{Join}\sim \{ \text{zwofamily}[[ i1]]\}; \\
 {}\quad \quad \quad \quad \quad \text{zwofamily}[[-1]][[1]][[ \text{Pzd}, 
\text{Pzd}]]=\{ \text{zwofamily}[[ i1]][[1]][[ \text{Pzd}, \text{Pzd}]][[ 
j1]]\}; \\
 {}\quad \quad \quad \quad \quad \text{zwofamily}[[-1]]= 
\text{UpdateOrderPart}[ \text{zwofamily}[[-1]]]; , \\
 {}\quad \quad \quad \quad \quad \{ j1,1, \text{Lz1}\}]]], \\
 {}\quad \quad \{ i1,1, L1\}]; \\
 {}\quad \text{zwofamily}= \text{Delete}[ \text{zwofamily}, \text{Del}]; \\
 {}\quad \text{zwotype3}= \text{zwotype3}\sim \text{Join}\sim 
\text{Select}[ \text{zwofamily}, \text{NumberofType3}[\#]==1\&]; \\
 {}\quad \text{zwofamily}= \text{Select}[ \text{zwofamily}, 
\text{NumberofType3}[\#]>1\&]; \\
 {}\quad ]; \text{del}=\{\}; ii1=2; \\
 {} \text{While}[ ii1<= \text{Length}[ \text{zwotype3}],  jj1=1;  \text{tf1}=0; \\
 {}\quad \text{While}[ jj1< ii1\&\&  \text{tf1}==0,  kk1=1; \\
 {}\quad \quad \text{While}[ kk1<= \text{Length}[ \text{pmzw}], \\
 {}\quad \quad \quad \text{If}[ \text{zwotype3}[[ ii1]][[1]]== 
\text{zwotype3}[[ jj1]][[1]][[ \text{pmzw}[[ kk1]], \text{pmzw}[[ kk1]]]]\&\&
\\
 {}\quad \quad \quad \quad \text{zwotype3}[[ ii1]][[2]]== \text{zwotype3}[[ 
jj1]][[2]][[ \text{pmzw}[[ kk1]], \text{pmzw}[[ kk1]]]], \\
 {}\quad \quad \quad \quad \text{del}= \text{del}\sim \text{Join}\sim \{\{ 
ii1\}\}; \text{tf1}=1];  kk1++];  jj1++]; ii1++]; \\ 
\\
 {} \text{zwotype3}= \text{Delete}[ \text{zwotype3}, \text{del}]; \\
 {} \text{zwotype3}= \text{Select}[ \text{zwotype3}, \text{OrderQ1}]; \\
 {} \text{zwotype3}= \text{Select}[ \text{zwotype3}, \text{OrderQ2}]; \\
 {} \text{zwotype3}= \text{Select}[ \text{zwotype3}, \text{Rule1cQ}]; \\ \\
 {} l1= \text{Length}[ \text{zwotype3}]; \\
 {} \text{Print}[ \text{Style}[`` \text{result} \text{of} \text{setp} 3:'', 
\text{Bold},15]]; \\
 {} \text{Print}[`` \text{cumulative} \text{time}: '', 
\text{AbsoluteTime}[]- t1 , \ ``\setminus \text{n} '']; \\
 {} \text{Print}[`` \text{There } \text{are } '', \text{Style}[ l1, 
\text{Bold}, \text{RGBColor}[0,0,0.7],14],  `` \text{zw}- \text{order } \text{matrices } \text{of } \text{Type } 3.'']; \\ 
\\
\text{(* Output results of FindOrder *)} \\
{} \quad \text{If}[ l1<=10,  \text{Print}[``\{ \text{z}- \text{order}, \text{w}- \text{order}, r- \text{order}\} \text{ matrices} :'']; \\
 {}\quad \text{Print}[ \text{Table}[\{ \text{Style}[ \text{MatrixForm}[ 
\text{outputform1}@ \text{zwotype3}[[ j1,1]]], \text{Blue}], \\
 {} \quad \quad \text{Style}[ \text{MatrixForm}[ \text{outputform1}@ \text{zwotype3}[[ j1,2]]], \text{Blue}],\\
 {}\quad \quad \text{Style}[ \text{MatrixForm}[ \text{outputform1}@ 
 \text{Dis3}[ \text{zwotype3}[[ j1]]]],\\
 {}\quad \quad\quad\  \text{RGBColor}[0,0.5,0]]\},\{ j1,1, 
\text{Length}[ \text{zwotype3}]\}]]]; \\
 {} \text{If}[ \text{Length}[ \text{zwotype3}]==0,  \text{Str}=`` \text{matrix} \text{FMX}[[''<> \text{ToString}[ N1]<>``]] \text{is} \text{impossible}''; \\
 {}\quad \text{Print}[\ ``\setminus \text{n} '', \text{Str} ,\ ``\setminus \text{n} '']];  \text{TF1}=1; , \text{constrainttime}]; \\
 {} \text{If}[ \text{TF1}==0,  \text{Str}=`` \text{stop} \text{running} \text{FMX}[[''<> \text{ToString}[ N1]<>``]]''; \\
 {}\quad \text{Print}[ \text{Str},\ ``\setminus \text{n} '']; ]) \\
\text{(* End of FindOrder *)}
  \)
}

\normalsize

\section*{Appendix III: Program for Algorithm III by Mathematica} \label{sec:program3}

\small
{\tt
\noindent\(
\text{(* Set variables *)}  \\
\text{Z}= \text{Table}[ \text{ToExpression}[`` \text{Z}''<> 
\text{ToString}[ j1]<> \text{ToString}[ k1]],\{ j1,1, n\},\{ k1,1, n\}]; \\
 {} \text{W}= \text{Table}[ \text{ToExpression}[`` \text{W}''<> 
\text{ToString}[ j1]<> \text{ToString}[ k1]],\{ j1,1, n\},\{ k1,1, n\}]; \\
 {} u= \text{Table}[ \text{ToExpression}[`` u''<> \text{ToString}[ j1]<> 
\text{ToString}[ k1]],\{ j1,1, n\},\{ k1,1, n\}]; \\
 {} v= \text{Table}[ \text{ToExpression}[`` v''<> \text{ToString}[ j1]<> 
\text{ToString}[ k1]],\{ j1,1, n\},\{ k1,1, n\}]; \\
 {} \text{zz}= \text{Table}[ \text{ToExpression}[`` \text{z}''<> 
\text{ToString}[ j1]<> \text{ToString}[ k1]],\{ j1,1, n\},\{ k1,1, n
\}]; \\
 {} \text{ww}= \text{Table}[ \text{ToExpression}[`` \text{w}''<> 
\text{ToString}[ j1]<> \text{ToString}[ k1]],\{ j1,1, n\},\{ k1,1, n
\}]; \\
 {} rr= \text{Table}[ \text{ToExpression}[`` r''<> \text{ToString}[ j1]<> 
\text{ToString}[ k1]],\{ j1,1, n\},\{ k1,1, n
\}]; \\
 {} \text{delta}= \text{Table}[ \text{ToExpression}[``\delta ''<> 
\text{ToString}[ j1]<> \text{ToString}[ k1]],\{ j1,1, n\},\{ k1,1, n\}]; 
\\
 {} \text{z}= \text{Table}[ \text{ToExpression}[`` \text{z}''<> 
\text{ToString}[ j1]],\{ j1,1, n\}]; \\
 {} \text{w}= \text{Table}[ \text{ToExpression}[`` \text{w}''<> 
\text{ToString}[ j1]],\{ j1,1, n\}]; \\
 {} \text{zeta}= \text{Table}[ \text{ToExpression}[``\zeta ''<> 
\text{ToString}[ j1]],\{ j1,1, n\}]; \\
 {} \text{omega}= \text{Table}[ \text{ToExpression}[``\omega ''<> 
\text{ToString}[ j1]],\{ j1,1, n\}]; \\
 {} M= \text{Table}[ \text{ToExpression}[`` m''<> \text{ToString}[ j1]],\{ 
j1,1, n\}]; \\
 {} \text{mu}= \text{Table}[ \text{ToExpression}[``\mu ''<> \text{ToString}[ 
j1]],\{ j1,1, n\}]; \\
 {} \text{Vardis}= \text{Flatten}[ \text{Join}[ \text{zz}, \text{ww}, 
\text{delta}]]; \\
 {} \text{Var}= \text{Flatten}[ \text{Join}[ \text{zz}, \text{ww}, 
\text{delta}, \text{z}, \text{w}, M]]; \\
 {} \text{zpart}[\{ \text{zmx4}\_, \text{wmx4}\_\}]:= \text{zmx4}; \\
 {} \text{wpart}[\{ \text{zmx4}\_, \text{wmx4}\_\}]:= \text{wmx4}; \\ 
 \\
 \text{(* Determine whether the possible max order terms in a $z_i$- or $w_i$-equation }\\
\text{are indeed maximal *)}  \\
\text{CertainMaxOrderZQ}[\{ \text{zomatrix3}\_,  
\text{womatrix3}\_\}, k3\_]:=( \\
 {}\quad \text{zo3}= \text{zomatrix3}; \text{wo3}= \text{womatrix3};  \text{max3}= \text{MaxOrderZ}[\{ \text{zo3}, \text{wo3}\}, k3]; \\
 {}\quad \text{If}[ \text{Length}[ \text{max3}]==2, \text{Return}[ \text{True}]]; \\
 {}\quad \text{If}[ \text{Length}[ \text{max3}]>2,  \text{Max3}=\{\}; \\
 {}\quad \quad \text{Do}[ \text{Do}[ \text{If}[ \text{max3}[[ i3]]== k3\&\&  \text{max3}[[ i3]]!= \text{max3}[[ j3]], \\
 {}\quad \quad \quad \quad \quad \text{If}[ \text{CompareOrderZ2}[\{ 
\text{zo3}, \text{wo3}\},\{ k3, \text{max3}[[ i3]]\},\{ k3, \text{max3}[[ 
j3]]\}]==\{0\}, \\
 {}\quad \quad \quad \quad \quad \quad \text{Max3}= \text{Max3}\sim 
\text{Join}\sim \{\{ \text{max3}[[ i3]], 
\text{max3}[[ j3]]\}\}]; ]; \\
 {}\quad \quad \quad \quad \text{If}[ \text{max3}[[ j3]]== k3\&\& \text{max3}[[ i3]]!= \text{max3}[[ j3]], \\
 {}\quad \quad \quad \quad \quad \text{If}[ \text{CompareOrderZ2}[\{ 
\text{zo3}, \text{wo3}\},\{ k3, \text{max3}[[ j3]]\},\{ k3, \text{max3}[[ 
i3]]\}]==\{0\}, \\
 {}\quad \quad \quad \quad \quad \quad \text{Max3}= \text{Max3}\sim 
\text{Join}\sim \{\{ \text{max3}[[ i3]], 
\text{max3}[[ j3]]\}\}];  ]; \\
 {}\quad \quad \quad \quad \text{If}[ \text{max3}[[ i3]]!= k3\&\& \text{max3}[[ j3]]!= k3\&\& \text{max3}[[ i3]]!= \text{max3}[[ j3]], \\
 {}\quad \quad \quad \quad \quad \text{If}[ \text{CompareOrderZ1}[\{ 
\text{zo3}, \text{wo3}\},\{ k3, \text{max3}[[ j3]]\},\{ k3, \text{max3}[[ 
i3]]\}]==\{0\}, \\
 {}\quad \quad \quad \quad \quad \quad \text{Max3}= \text{Max3}\sim 
\text{Join}\sim \{\{ \text{max3}[[ i3]], 
\text{max3}[[ j3]]\}\}];  ];  , \\
 {}\quad \quad \quad \quad \{ j3, i3+1, \text{Length}[ \text{max3}]\}]; , \{ i3,1, \text{Length}[ \text{max3}]-1\}]; \\
 {}\quad \quad \text{If}[ \text{Length}[ \text{Max3}]>1, \\
 {}\quad \quad \quad \text{Do}[ \text{Do}[ \text{Do}[ \text{If}[ 
\text{Intersection}[ \text{Max3}[[ i3]], \text{Max3}[[ j3]]]!=\{\}, \\
 {}\quad \quad \quad \quad \quad \quad \quad \text{Max3}[[ i3]]= 
\text{Max3}[[ i3]]\cup \text{Max3}[[ j3]];  \text{Max3}[[ j3]]=\{\}; ];  , \\
 {}\quad \quad \quad \quad \quad \quad \{ j3, i3+1, \text{Length}[ 
\text{Max3}]\}], \{ i3,1, \text{Length}[ \text{Max3}]\}],  \{ ii3,1, \text{Length}[ \text{Max3}]\}]; ]; \\
 {}\quad \quad \text{Max3}= \text{Complement}[ \text{Max3},\{\{\}\}]; \\
 {}\quad \quad \text{If}[ \text{Length}[ \text{Max3}]==1\&\& \text{Flatten}[ \text{Max3}]== \text{max3},  \text{Return}[ \text{True}],  \text{Return}[ \text{False}]];  ]) \\ 
 \\
 \text{CertainMaxOrderWQ}[\{ \text{zomatrix3}\_,  
\text{womatrix3}\_\}, k3\_]:=( \\
 {}\quad \text{zo3}= \text{zomatrix3}; \text{wo3}= \text{womatrix3};  \text{max3}= \text{MaxOrderW}[\{ \text{zo3}, \text{wo3}\}, k3]; \\
 {}\quad \text{If}[ \text{Length}[ \text{max3}]==2, \text{Return}[ \text{True}]]; \\
 {}\quad \text{If}[ \text{Length}[ \text{max3}]>2,  \text{Max3}=\{\}; \\
 {}\quad \quad \text{Do}[ \text{Do}[ \text{If}[ \text{max3}[[ i3]]== k3\&\&  \text{max3}[[ i3]]!= \text{max3}[[ j3]], \\
 {}\quad \quad \quad \quad \quad \text{If}[ \text{CompareOrderW2}[\{ 
\text{zo3}, \text{wo3}\},\{ k3, \text{max3}[[ i3]]\},\{ k3, \text{max3}[[ 
j3]]\}]==\{0\}, \\
 {}\quad \quad \quad \quad \quad \quad \text{Max3}= \text{Max3}\sim 
\text{Join}\sim \{\{ \text{max3}[[ i3]], 
\text{max3}[[ j3]]\}\}]; ]; \\
 {}\quad \quad \quad \quad \text{If}[ \text{max3}[[ j3]]== k3\&\& \text{max3}[[ i3]]!= \text{max3}[[ j3]], \\
 {}\quad \quad \quad \quad \quad \text{If}[ \text{CompareOrderW2}[\{ 
\text{zo3}, \text{wo3}\},\{ k3, \text{max3}[[ j3]]\},\{ k3, \text{max3}[[ 
i3]]\}]==\{0\}, \\
 {}\quad \quad \quad \quad \quad \quad \text{Max3}= \text{Max3}\sim 
\text{Join}\sim \{\{ \text{max3}[[ i3]], 
\text{max3}[[ j3]]\}\}];  ]; \\
 {}\quad \quad \quad \quad \text{If}[ \text{max3}[[ i3]]!= k3\&\& \text{max3}[[ j3]]!= k3\&\& \text{max3}[[ i3]]!= \text{max3}[[ j3]], \\
 {}\quad \quad \quad \quad \quad \text{If}[ \text{CompareOrderW1}[\{ 
\text{zo3}, \text{wo3}\},\{ k3, \text{max3}[[ j3]]\},\{ k3, \text{max3}[[ 
i3]]\}]==\{0\}, \\
 {}\quad \quad \quad \quad \quad \quad \text{Max3}= \text{Max3}\sim 
\text{Join}\sim \{\{ \text{max3}[[ i3]], 
\text{max3}[[ j3]]\}\}];  ];  , \\
 {}\quad \quad \quad \quad \{ j3, i3+1, \text{Length}[ \text{max3}]\}]; , \{ i3,1, \text{Length}[ \text{max3}]-1\}]; \\
 {}\quad \quad \text{If}[ \text{Length}[ \text{Max3}]>1, \\
 {}\quad \quad \quad \text{Do}[ \text{Do}[ \text{Do}[ \text{If}[ 
\text{Intersection}[ \text{Max3}[[ i3]], \text{Max3}[[ j3]]]!=\{\}, \\
 {}\quad \quad \quad \quad \quad \quad \quad \text{Max3}[[ i3]]= 
\text{Max3}[[ i3]]\cup \text{Max3}[[ j3]];\text{Max3}[[ j3]]=\{\}; ];  , \\
 {}\quad \quad \quad \quad \quad \quad \{ j3, i3+1, \text{Length}[ 
\text{Max3}]\}], \{ i3,1, \text{Length}[ \text{Max3}]\}],  \{ ii3,1, \text{Length}[ \text{Max3}]\}]; ]; \\
 {}\quad \quad \text{Max3}= \text{Complement}[ \text{Max3},\{\{\}\}]; \\
 {}\quad \quad \text{If}[ \text{Length}[ \text{Max3}]==1\&\& \text{Flatten}[ \text{Max3}]== \text{max3},  \text{Return}[ \text{True}],  \text{Return}[ \text{False}]];  ]) \\ \\
\text{(* Determine the $z_i$-equations corresponding to $zw$-order matrices for each $i$} \\
\text{in the given set of bodies *)}  \\
\text{EqnMZO}[\{ \text{zomatrix2}\_, \text{womatrix2}\_\}, 
\text{VertexSet2}\_] := ( \\
 {}\quad \text{zo2} = \text{zomatrix2}; \text{wo2} = \text{womatrix2}; \text{eqnset2} = \{\}; \\
 {}\quad \text{Do}[ \text{If}[ \text{CertainMaxOrderZQ}[\{ \text{zo2}, 
\text{wo2}\}, i2] == \text{True}, \\
 {}\quad \quad \quad \text{eqn2} = 0; \text{max2} = \text{MaxOrderZ}[\{ \text{zo2}, 
\text{wo2}\}, i2]; \\
 {}\quad \quad \quad \text{Do}[ \text{If}[ \text{MemberQ}[ \text{max2}, 
 j2] == \text{True}, \text{eqn2} = \text{eqn2} + M[[ 
j2]]* \text{Z}[[ i2, j2]]], \{ j2, 1, n\}]; \\
 {}\quad \quad \quad \text{eqnset2} = \text{eqnset2} \cup \{ 
\text{eqn2}\}; ] , \{ i2, \text{VertexSet2}\}]; \\
 {}\quad \text{Return}[ \text{eqnset2}]; ) \\
 \\
 \text{(* Determine the $w_i$-equations corresponding to $zw$-order matrices for each $i$} \\
\text{in the given set of bodies *)}  \\
 \text{EqnMWO}[\{ \text{zomatrix2}\_, \text{womatrix2}\_\}, 
\text{VertexSet2}\_] := ( \\
 {}\quad \text{zo2} = \text{zomatrix2}; \text{wo2} = \text{womatrix2}; \text{eqnset2} = \{\}; \\
 {}\quad \text{Do}[ \text{If}[ \text{CertainMaxOrderWQ}[\{ \text{zo2}, 
\text{wo2}\}, i2] == \text{True}, \\
 {}\quad \quad \quad \text{eqn2} = 0; \text{max2} = \text{MaxOrderW}[\{ \text{zo2}, 
\text{wo2}\}, i2]; \\
 {}\quad \quad \quad \text{Do}[ \text{If}[ \text{MemberQ}[ \text{max2}, 
 j2] == \text{True}, \text{eqn2} = \text{eqn2} + M[[ 
j2]]* \text{W}[[ i2, j2]]], \{ j2, 1, n\}]; \\
 {}\quad \quad \quad \text{eqnset2} = \text{eqnset2} \cup \{ 
\text{eqn2}\}; ] , \{ i2, \text{VertexSet2}\}]; \\
 {}\quad \text{Return}[ \text{eqnset2}]; ) \\
 \\
 \text{(* define functions for searching mass relations *)}  \\
  \\
\text{(* Replace } m_iZ_{{ii}}\text{'s} ,\ m_iW_{{ii}}\text{'s} \text{ by } z_i, w_i \text{ resp}.\text{  }\text{*)}  \\
\text{Relation1}[ \text{x2}\_] := ( \\
 {}\quad \text{y2} = \text{x2}; \\
 {}\quad \text{Do}[ \text{Do}[ \text{y2}[[ j2]] = \text{y2}[[ j2]] 
/. \{ M[[ i2]]* \text{Z}[[ i2, i2]] \to \text{z}[[ 
i2]], M[[ i2]]* \text{W}[[ i2, i2]] \to \text{w}[[ i2]]\}, \\
 {}\quad \quad \quad \{ i2, 1, n\}], \{ j2, 1, \text{Length}[ \text{y2}]\}]; \text{Return}[ \text{y2}]) \\
 \\
 \text{(* For each pair of } (i,j) \text{ with } i>j, \text{ replace } Z_{ij} , W_{ij} \text{ by } -Z_{{ji}} ,\ -W_{{ji}} \text{ resp}.\text{  }\text{*)}  \\
 \text{Relation2}[ \text{x2}\_]:=( \\
 {}\quad \text{y2}= \text{x2}; \\
 {}\quad \text{Do}[ \text{Do}[ \text{Do}[ \text{y2}[[ k2]]= \text{y2}[[ 
k2]]/.\{ \text{Z}[[ i2, j2]]\to- \text{Z}[[ j2, i2]], \text{W}[[ i2, j2]]\to-
\text{W}[[ j2, i2]]\}, \\
 {}\quad \quad \quad \quad \{ i2, j2+1, n\}], \{ j2,1, n-1\}],  \{ k2,1, \text{Length}[ \text{y2}]\}]; \text{Return}[ \text{y2}]) \\
 \\
  \text{(* For each pair of } (i,j) , \text{ replace } Z_{ij} , W_{ij} \text{ by } -z_{{ji}}^{-1/2}w_{ij}^{-3/2} ,\ w_{{ji}}^{-1/2}z_{ij}^{-3/2} \text{ resp}.\text{  }\text{*)}  \\
 \text{Relation3ver1}[ \text{x2}\_]:=( \\
 {}\quad \text{y2}= \text{x2}; \\
 {}\quad \text{Do}[ \text{Do}[ \text{Do}[ \text{y2}[[ k2]]= \text{y2}[[ 
k2]]/.\{ \text{Z}[[ i2, j2]]\to \text{zz}[[ i2, j2]]\wedge (-1/2)* 
\text{ww}[[ i2, j2]]\wedge (-3/2),\\
 {}\quad\quad\quad\quad \text{W}[[ i2, j2]]\to \text{zz}[[ i2, j2]]\wedge (-1/2)* \text{ww}[[ i2, j2]]\wedge (-3/2)\}, \\
 {}\quad \quad \quad \quad \{ j2, i2+1, n\}],  \{ i2,1, n-1\}],  \{ k2,1,\text{Length}[ \text{y2}]\}]; \text{Return}[ \text{y2}]) \\
 \\
 \text{(* For every  $i,j$, replace  $Z_{ij}$, $W_{ij}$ by $u_{ij}v_{ij}^3$, $v_{ij}u_{ij}^3$ resp. *)}  \\
 \text{Relation3ver2}[ \text{x2}\_] := ( \\
 {}\quad \text{y2} = \text{x2}; \\
 {}\quad \text{Do}[ \text{Do}[ \text{Do}[ \text{y2}[[ k2]] = 
\text{y2}[[ k2]] /. \{ \text{Z}[[ i2, j2]] \to u[[ i2, j2]]\wedge 
(1)* v[[ i2, j2]]\wedge (3),\\
 {}\quad\quad\quad\quad \text{W}[[ i2, j2]] \to v[[ i2, j2]]\wedge (1)* u[[ i2, j2]]\wedge (3)\}, \\
 {}\quad \quad \quad \quad \{ j2, i2 + 1, n\}],  \{ i2, 1, n - 1\}], \{ k2, 1, \text{Length}[ \text{y2}]\}];  \text{Return}[ \text{y2}]) \\
 \\
\text{(* change variable according to the setting  } z_i=\zeta _i^2, w_i=\omega _i^2\text{  }\text{*)}  \\
\text{ChangeVar1}[ \text{x2}\_]:=( \\
 {}\quad \text{y2}= \text{x2}; \\
 {}\quad \text{Do}[ \text{Do}[ \text{y2}[[ j2]]= \text{y2}[[ j2]]/.\{ 
\text{z}[[ i2]]\to \text{zeta}[[ i2]]\wedge (2),\\
 {}\quad\quad\quad\quad \text{w}[[ i2]]\to \text{omega}[[ i2]]\wedge (2)\}, \\
 {}\quad \quad \quad \{ i2,1, n\}], \{ j2,1, \text{Length}[ \text{y2}]\}];  \text{Return}[ \text{y2}]) \\
 \\
 \quad \text{ChangeVar1Inv}[ \text{x2}\_]:=( \\
 {}\quad \text{y2}= \text{x2}; \\
 {}\quad \text{Do}[ \text{Do}[ \text{y2}[[ j2]]= \text{y2}[[ j2]]/.\{ 
\text{zeta}[[ i2]]\to \text{z}[[ i2]]\wedge (1/2),\\
 {}\quad\quad\quad\quad \text{omega}[[ i2]]\to \text{w}[[ i2]]\wedge (1/2)\}, \\
 {}\quad \quad \quad \{ i2,1, n\}], \{ j2,1, \text{Length}[ \text{y2}]\}];  \text{Return}[ \text{y2}]) \\
 \\
 \text{(* change variable according to the setting } m_i=\mu _i^4\text{  }\text{*)}  \\
 \text{ChangeVar2}[ \text{x2}\_] := ( \\
 {}\quad \text{y2} = \text{x2}; \\
 {}\quad \text{Do}[ \text{Do}[ \text{y2}[[ j2]] = \text{y2}[[ j2]] /. \{ 
M[[ i2]] \to \text{mu}[[ i2]]\wedge 4\}, \\
 {}\quad \quad \quad \{ i2, 1, n\}],  \{ j2, 1, \text{Length}[ \text{y2}]\}];  \text{Return}[ \text{y2}]) \\
 \\
\text{ChangeVar2Inv}[ \text{x2}\_] := ( \\
 {}\quad \text{y2} = \text{x2}; \\
 {}\quad \text{Do}[ \text{Do}[ \text{y2}[[ j2]] = \text{y2}[[ j2]] /. \{ 
\text{mu}[[ i2]] \to M[[ i2]]\wedge (1/4)\}, \\
 {}\quad \quad \quad \{ i2, 1, n\}], \{ j2, 1, \text{Length}[ \text{y2}]\}];  \text{Return}[ \text{y2}]) \\
 \\
\text{(* change variable according to the setting } u_{ij}=z_{ij}^{-1/2},\ v_{ij}=w_{ij}^{-1/2}
\text{*)}  \\
\text{ChangeVar3}[ \text{x2}\_]:=( \\
 {}\quad \text{y2}= \text{x2}; \\
 {}\quad \text{Do}[ \text{Do}[ \text{Do}[ \text{y2}[[ k2]]= \text{Simplify}@ 
\text{Expand}[ \text{y2}[[ k2]]/.\{ \text{zz}[[ i2, j2]]\to u[[ i2, 
j2]]\wedge (-2), \\
 {}\quad\quad\quad\quad \text{ww}[[ i2, j2]]\to v[[ i2, j2]]\wedge (-2)\}], \\
 {}\quad \quad \quad \quad \{ j2, i2+1, n\}], \{ i2,1, n-1\}],  \{ k2,1, \text{Length}[ \text{y2}]\}]; \text{Return}[ \text{y2}]) \\
 \\
 \quad \text{ChangeVar3Inv}[ \text{x2}\_]:=( \\
 {}\quad \text{y2}= \text{x2}; \\
 {}\quad \text{Do}[ \text{Do}[ \text{Do}[ \text{y2}[[ k2]]= \text{Simplify}@ 
\text{Expand}[ \text{y2}[[ k2]]/.\{ u[[ i2, j2]]\to \text{zz}[[ i2, 
j2]]\wedge (-1/2), \\
 {}\quad\quad\quad\quad v[[ i2, j2]]\to \text{ww}[[ i2, j2]]\wedge (-1/2)\}], \\
 {}\quad \quad \quad \quad \{ j2, i2+1, n\}],  \{ i2,1, n-1\}],  \{ k2,1, \text{Length}[ \text{y2}]\}];  \text{Return}[ \text{y2}]) \\
 \\
\text{(* If } z_i\succ z_{ij}\text{  }\text{and}\text{  }z_j\succ z_{ij}, \text{ then } z_i\sim z_j, 
\text{ so replace } z_j \text{ by } z_i, \text{ similar for } w_i\text{'s } \text{*)}  \\
\text{ClusterReplace}[\{ \text{zomatrix2}\_, \text{womatrix2}\_\}, 
\text{x2}\_]:=( \\
 {}\quad \text{y2}= \text{x2}; \\
 {}\quad \text{Do}[ \text{Do}[ \text{Do}[ \text{If}[ \text{Min}[ 
\text{zomatrix2}[[ i2, i2]]\sim \text{Join}\sim 
\text{zomatrix2}[[ j2, j2]]] \\ 
 {}\quad \quad \quad \quad \quad > \text{Max}[ \text{zomatrix2}[[ i2
, j2]]], \\
 {}\quad \quad \quad \quad \quad \text{y2}[[ k2]]= \text{Simplify}[ 
\text{y2}[[ k2]]/.\{ \text{z}[[ i2]]\to \text{z}[[ j2]]\}]]; \\
 {}\quad \quad \quad \quad \text{If}[ \text{Min}[ \text{womatrix2}[[ i2, 
i2]]\sim \text{Join}\sim \text{womatrix2}[[ j2, j2]]]\\
  {}\quad \quad \quad \quad \quad > \text{Max}[ 
\text{womatrix2}[[ i2, j2]]], \\
 {}\quad \quad \quad \quad \quad \text{y2}[[ k2]]= \text{y2}[[ k2]]/.\{ 
\text{w}[[ i2]]\to \text{w}[[ j2]]\}], \\
 {}\quad \quad \quad \quad \{ i2, j2+1, n\}], \{ j2,1, n-1\}], \{ k2,1, \text{Length}[ \text{y2}]\}];  \text{Return}[ \text{y2}]) \\
 \\
 \text{(* If a polynomial } p \text{ has factor var}_i^k \text{ for some variable} \text{var}_i ,
\text{then we change}  \\
{}\quad p \text{ to } p\left/\text{var}_i^k\right. \text{*)}  \\
\text{CancelFactor}[ \text{x4}\_] := ( \\
 {}\quad \text{y4} = \text{x4}; \\
 {}\quad \text{Do}[ \text{varset4} = \text{Variables}[ \text{y4}[[ k4]]]; 
\\
 {}\quad \quad \text{Do}[ \text{Ex4} = \text{Exponent}[ \text{y4}[[ 
k4]], \text{varset4}[[ j4]]\wedge (-1)]; \\
 {}\quad \quad \quad \text{If}[ \text{Ex4} < 0,  \text{y4}[[ k4]] = \text{Simplify}[ 
\text{Expand}[ \text{y4}[[ k4]]* \text{varset4}[[ j4]]\wedge ( \text{Ex4})]] 
 ], \\
 {}\quad \quad \quad \{ j4, 1, \text{Length}[ \text{varset4}]\}],  \{ k4, 1, \text{Length}[ \text{y4}]\}]; \text{Return}[ \text{y4}]) \\
 \\
\text{(* If a rational function $f=q/p$ and $p$ has a factor var}_i^k \text{ for some } \text{variable var}_i ,\\ 
{}\quad \text{then we change } g \text{ to} g*\text{var}_i^k \text{*)}  \\
\text{CancelDenominator}[ \text{x3}\_] := ( \\
 {}\quad \text{y3} = \text{x3}; \\
 {}\quad \text{Do}[ \text{varset3} = \text{Variables}[ \text{y3}[[ k3]]]; 
\\
 {}\quad \quad \text{Do}[ \text{Ex3} = \text{Exponent}[ \text{y3}[[ 
k3]], \text{varset3}[[ j3]]\wedge (-1)]; \\
 {}\quad \quad \quad \text{If}[ \text{Ex3} > 0,  {}\quad \quad \quad \quad \text{y3}[[ k3]] = \text{Simplify}[ \text{Expand}[ \text{y3}[[ k3]]* \text{varset3}[[ j3]]\wedge ( \text{Ex3})]] ], \\
 {}\quad \quad \quad \{ j3, 1, \text{Length}[ \text{varset3}]\}],  \{ k3, 1, \text{Length}[ \text{y3}]\}];  \text{Return}[ \text{y3}]) \\
 \\
 \text{(* Determine whether the polynomial gives a mass relation. If it does, then} \\
{}\quad \text{then we collect this mass relation}. \text{*)}  \\ 
\text{MassRelationQ}[ \text{x2}\_]:=( \\
 {}\quad \text{If}[ \text{Count}[ \text{Table}[ \text{Complement}[ 
\text{Variables}[ \text{x2}[[ k2]]], M\sim \text{Join}\sim \text{mu}],\{
k2,1, \text{Length}[ \text{x2}]\}],\{\}] \\
 {}\quad \quad >0, \text{Return}[ \text{True}], \text{Return}[ \text{False}]]) \\
 \\
 \quad \text{TakeMassRela}[ \text{x2}\_]:=(  \text{y2}= \text{x2}; \\
 {}\quad \text{y2}= \text{Select}[ \text{y2}, \text{Complement}[ 
\text{Variables}[\#], M\sim \text{Join}\sim \text{mu}]==\{\}\&];  \text{Return}[ \text{y2}]) \\
 \\
 \text{(* Determine whether the $zw$-order matrix gives any wedge equations}. \\
{} \quad \text{If it does then we collect  these wedge equations *)}  \\
\\
\text{FindWedgeV}[\{ I2\_, J2\_\}]:=( \\
 {}\quad SS2= \text{Select}[ \text{Table}[ k,\{ k,1, n\}], \text{Max}[ 
\text{Diagonal}[ I2][[\#]], \text{Diagonal}[ J2][[\#]]]==5\&]; \\
 {}\quad S2= \text{Select}[ \text{Subsets}[ SS2], \text{Length}[\#]>0\&\&  \text{Length}[\#]< n\&];  jj2=1; \\
 {}\quad \text{While}[ jj2<= \text{Length}[ S2],  \text{V2}= S2[[ jj2]];  \text{\text{V2}C}= \text{Complement}[ SS2, \text{V2}];\\
 {} \quad U2= \text{Complement}[ \text{allind}, \text{V2}\sim \text{Join}\sim 
 \text{\text{V2}C}]; \\
 {}\quad \quad \text{If}[ \text{Complement}[ \text{Flatten}[ I2[[ \text{V2}, 
\text{V2}]]],\{5\}]==\{\}\&\& \\
 {}\quad \quad \quad \text{Max}[ \text{Flatten}[ I2[[ \text{\text{V2}C}, 
\text{\text{V2}C}]]]]<5\&\& \\
 {}\quad \quad \quad \text{Complement}[ \text{Flatten}[ I2[[ \text{V2}, 
\text{\text{V2}C}]]],\{5\}]==\{\}\&\& \\
 {}\quad \quad \quad \text{Max}[ \text{Flatten}[ J2[[ \text{V2}, \text{V2}]]]]<5\&\& \\
 {}\quad \quad \quad \text{Complement}[ \text{Flatten}[ J2[[ \text{\text{V2}C}, 
\text{\text{V2}C}]]],\{5\}]==\{\}\&\& \\
 {}\quad \quad \quad \text{Complement}[ \text{Flatten}[ J2[[ \text{V2}, 
\text{\text{V2}C}]]],\{5\}]==\{\}\&\& \\
 {}\quad \quad \quad \text{Length}@ U2<=1\&\& \\
 {}\quad \quad \quad \text{Max}@( \text{Flatten}@( I2[[ U2, U2]])\sim 
\text{Join}\sim \text{Flatten}@( J2[[ U2, U2]]))<=1, \\
 {}\quad \quad \quad \text{Return}[\{ \text{V2}, \text{\text{V2}C}\}]];  jj2++]; \\
 {}\quad \text{Return}[\{\{\},\{\}\}]) \\
 \\
 \text{(* Determine how many terms are there in the polynomial}.\ \text{*)}  \\
  \text{NumberofInv}[ S4\_, \text{var4}\_]:=( \\
 {}\quad ii4=0; \\
 {}\quad \text{Do}[ \text{If}[ \text{Intersection}[ \text{Variables}[ S4[[ 
j4]]],\{ \text{var4}\}]!=\{\}, \\
 {}\quad \quad \quad ii4++],  \{ j4,1, \text{Length}[ S4]\}]; \text{Return}[ ii4])\\
 \\
 \text{(* For collection } \left\{f_i\right\} \text{ of polynomials, if there is some } f_i=g*\text{var}, \text{where } \\
{}\quad \text{var is variable and } g \text{ is polynomial independent of var},   
\text{then we make} \\
{}\quad \text{substitution }  f/g \text{ for var to each } f_j \text {in the collection}.\text{  }\text{*)}  \\
\text{simpleSubstitution}[ S3\_, \text{varset3a}\_]:=( \\
 {}\quad s3= S3; \text{varset3b}= \text{Select}[ \text{varset3a}, \text{NumberofInv}[ s3,\#]>1\& ]; \\
 {}\quad ii3=1; \text{tf2}=0; \\
 {}\quad \text{While}[ ii3<= \text{Length}[ s3]\&\& \text{tf2}==0, \\
 {}\quad \quad \text{If}[ \text{Length}[ \text{ExpandAll}@( s3[[ ii3]])]==2, 
\\
 {}\quad \quad \quad \text{varset3c}= \text{Intersection}[ \text{Variables}[ 
s3[[ ii3]]], \text{varset3b}];  jj3=1; \\
 {}\quad \quad \quad \text{While}[ jj3<= \text{Length}[ \text{varset3c}]\&\& 
\text{tf2}==0, \\
 {}\quad \quad \quad \quad \text{If}[ \text{Exponent}[ s3[[ ii3]], 
\text{varset3c}[[ jj3]]]==1, \\
 {}\quad \quad \quad \quad \quad \text{substi3}= \text{Solve}[ s3[[ 
ii3]]==0, \text{varset3c}[[ jj3]]][[1]]; \\
 {}\quad \quad \quad \quad \quad \text{Do}[ s3[[ k3]]= \text{Simplify}@ 
\text{Expand}[ s3[[ k3]]/. \text{substi3}], \\
 {}\quad \quad \quad \quad \quad \quad \{ k3,1, \text{Length}[ s3]\}]; \\
 {}\quad \quad \quad \quad \quad \text{tf2}=1];  jj3++]];  ii3++];  \text{Return}[ s3];  ) \\
 \\
 \text{(* Delete duplicates from a collection of polynomials of some
 variables *)}  \\
 \text{DelDupPoly}[ \text{polyset3a}\_, \text{varset3}\_] := ( \\
 {}\quad \text{polyset3b} = \text{polyset3a}; \text{del} = \{\}; \\
 {}\quad \text{Do}[ \text{If}[ \text{Intersection}[ \text{Variables}[ 
\text{polyset3b}[[ i3]]], \text{varset3}] != \{\}, \\
 {}\quad \quad \quad \text{var3} = \text{Intersection}[ 
\text{Variables}[ \text{polyset3b}[[ i3]]], \text{varset3}][[1]]; \\
 {}\quad \quad \quad \text{exp3} = \text{Exponent}[ \text{polyset3b}[[ 
i3]], \text{var3}]; \\
 {}\quad \quad \quad \text{Do}[ \text{If}[ \text{Simplify}[ 
 \text{polyset3b}[[ i3]]* \text{Coefficient}[ \text{polyset3b}[[ 
j3]], \text{var3}\wedge \text{exp3}] \\
 {}\quad \quad \quad \quad \quad - \text{polyset3b}[[ j3]]* \text{Coefficient}[ \text{polyset3b}[[ i3]], \text{var3}\wedge \text{exp3}]] 
== 0, \\
 {}\quad \quad \quad \quad \quad \text{del} = \text{del}\sim 
\text{Join}\sim \{\{ j3\}\}], \\
 {}\quad \quad \quad \quad \{ j3, i3 + 1, \text{Length}[ 
\text{polyset3b}]\}]],  \{ i3, 1, \text{Length}[ \text{polyset3b}] - 1\}]; \\
 {}\quad \text{polyset3b} = \text{Delete}[ \text{polyset3b}, 
\text{del}];  \text{Return}[ \text{polyset3b}];  ) \\
 \\
 \text{(* For collection $\left\{f_i\right\}$ of polynomials and set $V$ of variables, }  \\ \text{if } f_i=f_j*g,\ i<j, 
\text{where } g \text{ is independent of variables in } V,\\ \text{ then delete } f_j \text{ from the }
\text{collection}\ \text{*)}  \\
\text{DelDupPoly2}[ \text{polyset3a}\_, \text{varset3}\_]:=( \\
 {}\quad \text{polyset3b}= \text{polyset3a}; \text{del}=\{\}; \\
 {}\quad \text{Do}[ \text{Do}[ \text{If}[ \text{PolynomialQ}[ 
\text{Simplify}@( \text{Expand}@( \text{polyset3b}[[ i3]]/ \text{polyset3b}[[
j3]])), \\ 
 {}\quad \quad \quad \quad \text{varset3}]== \text{True}\&\&  \text{PolynomialQ}[ \text{Simplify}@( \text{Expand}@( \text{polyset3b}[[ j3]]/ \text{polyset3b}[[ i3]])), \\
 {}\quad \quad \quad \quad \text{varset3}]== \text{True} , \\
 {}\quad \quad \quad \quad \text{del}= \text{del}\sim \text{Join}\sim \{\{ 
j3\}\}], \\
 {}\quad \quad \quad \{ j3, i3+1, \text{Length}[ \text{polyset3b}]\}],  \{ i3,1, \text{Length}[ \text{polyset3b}]-1\}]; \\
 {}\quad \text{polyset3b}= \text{Delete}[ \text{polyset3b}, \text{del}];  \text{Return}[ \text{polyset3b}]; ) \\
 \\
 \text{(* If $g$ is a multiple of $f$, then we delete $g$. *)} \\
 \text{DelDupPoly3}[ \text{polyset3a}\_, \text{varset3}\_] := ( \\
 {}\quad \text{polyset3b} = \text{polyset3a};  \text{del} = \{\}; \\
 {}\quad \text{Do}[ \text{Do}[ \text{tf3} = 0; \\
 {}\quad \quad \quad \text{If}[ \text{PolynomialQ}[ 
\text{Simplify}@( \text{Expand}@( \text{polyset3b}[[ i3]]/ 
\text{polyset3b}[[ j3]])),\\
 {}\quad \quad \quad \quad \text{varset3}] == \text{True}, \\
 {}\quad \quad \quad \quad \text{del} = \text{del}\sim 
\text{Join}\sim \{\{ i3\}\};  \text{tf3} = 1]; \\
 {}\quad \quad \quad \text{If}[ \text{tf3} == 0 \&\& \text{PolynomialQ}[ \text{Simplify}@( \text{Expand}@( \text{polyset3b}[[ j3]]/ \text{polyset3b}[[ i3]])),\\
 {}\quad \quad \quad \quad \text{varset3}] == \text{True}, \\
 {}\quad \quad \quad \quad \text{del} = \text{del}\sim \text{Join}\sim 
\{\{ j3\}\}], \\ 
 {}\quad \quad \quad  \{ j3, i3 + 1, \text{Length}[\text{polyset3b}]\}], \{ i3, 1, \text{Length}[ \text{polyset3b}] - 1\}]; \\
 {}\quad \text{polyset3b} = \text{Delete}[ \text{polyset3b},
\text{del}];  \text{Return}[ \text{polyset3b}]; ) \\
 \\
 \text{(*} \text{ apply above ``simplesubstitution'', change each member} \text{ to polynomial},
  \\
{}\quad \text{ and delete constant polynomials and duplicates}.\text{*)}  \\
\text{SimpleSubstitution}[ S2\_, \text{varset2}\_] := ( \\
 {}\quad s2 = S2; \text{tf2} = 1; \\
 {}\quad \text{While}[ \text{tf2} == 1, \\
 {}\quad \quad s2 = \text{simpleSubstitution}[ s2, \text{varset2}]; \\
 {}\quad \quad s2 = \text{CancelDenominator}[ s2]; \\
 {}\quad \quad s2 = \text{CancelFactor}[ s2]; \\
 {}\quad \quad s2 = \text{Complement}[ s2, \{1, 0, -1\}]; \\
 {}\quad \quad s2 = \text{DelDupPoly}[ s2, \text{varset2}];  ]; \\
 {}\quad \text{Return}[ s2]) \\
 \\
 \text{(*} \text{if } f=g_1^{k_1}g_2^{k_2}\cdot \cdot \cdot g_n^{k_n} , \text{ then replace } f \text{ by } g_1g_2\cdot \cdot \cdot g_n\text{*)}  \\
 \text{MinimizeExpofFac}[ S4\_]:=( \\
 {}\quad s4= S4;  L4= \text{Length}[ s4]; \\
 {}\quad \text{If}[ L4>0, \\
 {}\quad \quad \text{Do}[ \text{list4}= \text{FactorList}[ s4[[ k4]]]; \\
 {}\quad \quad \quad \text{If}[ \text{Length}[ \text{list4}]>1,  s4[[ k4]]=1; \\
 {}\quad \quad \quad \quad \text{Do}[ \text{poly4}= \text{list4}[[ j4,1]]; 
\\
 {}\quad \quad \quad \quad \quad \text{If}[ \text{Length}@ \text{poly4}>1, 
\\
 {}\quad \quad \quad \quad \quad \quad s4[[ k4]]= \text{ExpandAll}@( s4[[ 
k4]]* \text{poly4})], \\
 {}\quad \quad \quad \quad \quad \{ j4,2, \text{Length}[ \text{list4}]\}]], 
 \{ k4,1, L4\}]];  \text{Return}[ s4]; ) \\
 \\
 \text{(* if } f=g_1^{k_1}g_2^{k_2}\cdot \cdot \cdot g_n^{k_n}h_1^{l_1}h_2^{l_2}\cdot \cdot \cdot h_m^{l_m} \text{ and all }
\text{terms } \text{in } h_i's \text{ are } \text{of} \text{ the } \text{same} \text{ sign}, \text{then } \\
{}\quad \text{replace } f \text{ by } 
g_1g_2\cdots g_n\text{*)} \\
 \text{MinimizeExpofFacP}[ S3\_]:=( \\
 {}\quad s3= S3;  L3= \text{Length}[ s3]; \\
 {}\quad \text{If}[ L3>0,  \text{Do}[ \text{list3}= \text{FactorList}[ s3[[ k3]]]; \\
 {}\quad \quad \quad \text{If}[ \text{Length}[ \text{list3}]>1,  s3[[ k3]]=1; \\
 {}\quad \quad \quad \quad \text{Do}[ \text{If}[ \text{Length}@ 
\text{StringPosition}[ \text{ToString}@( 
\text{list3}[[ j3,1]]),``-'']>0\&\& \\
 {}\quad \quad \quad \quad \quad \quad \text{Length}@ \text{StringPosition}[ 
\text{ToString}@(- \text{list3}[[ j3,1]]),``-'']>0, \\
 {}\quad \quad \quad \quad \quad \quad s3[[ k3]]= s3[[ k3]]* \text{list3}[[ 
j3,1]]], \\
 {}\quad \quad \quad \quad \quad \{ j3,2, \text{Length}[ \text{list3}]\}]], 
 \{ k3,1, L3\}]]; \text{Return}[ s3]; ) \\
 \\
 \text{(*} \text{Elimination process: apply some} \text{ algorithms} \text{ on collection} \text{ of} \text{ polynomials} \\
{}\quad \text{ to} \text{ eliminate given} \text{ variable}.\text{*)}  \\
\text{Elimi}[ S2\_, \text{var2}\_]:=( \\
 {}\quad s2= \text{Select}[ S2, \text{Intersection}[ \text{Variables}[\#],\{ 
\text{var2}\}]!=\{\}\& ]; \\
 {}\quad \text{s2b}= \text{Complement}[ S2, s2]; \\
 {}\quad \text{exp2c}=2* \text{Max}[ \text{Table}[ \text{Exponent}[ s2[[ 
k2]], \text{var2}],\{ k2,1, \text{Length}[ s2]\}]]; \\
 {}\quad kk2=0; \text{possfac2}=1; \\
 {}\quad \text{While}[( \text{Length}@ s2>1)\&\& ( kk2< \text{exp2c}+2), \\
 {}\quad \quad s2= \text{SortBy}[ s2, \text{Length}[\#]\& ]; \\
 {}\quad \quad s2= \text{SortBy}[ s2, \text{Exponent}[\#, \text{var2}]\& ]; 
\\
 {}\quad \quad \text{exp2a}= \text{Exponent}[ s2[[1]], \text{var2}]; \text{TF2}=1; \\
 {}\quad \quad \text{CoeList2}= \text{CoefficientList}[ s2[[1]], 
\text{var2}];  \\
 {}\quad \quad \text{CoeFactor2}= \text{Replace}[ \text{CoeList2}, 
\text{List}\to \text{PolynomialGCD},1, \text{Heads}\to \text{True}];  \\
 {}\quad \quad \text{Do}[ \text{If}[ \text{Exponent}[ \text{CoeFactor2}, 
\text{var2b}]>0, \text{TF2}=0];  ,  \{ \text{var2b}, \text{VarSet1b}\}]; \\
 {}\quad \quad \text{If}[ \text{TF2}==1, \\
 {}\quad \quad \quad s2[[1]]= \text{ExpandAll}@ \text{Simplify}[ s2[[1]]/ 
\text{CoeFactor2}]; \\
 {}\quad \quad \quad \text{possfac2}= \text{possfac2}* \text{CoeFactor2};]; \\
 {}\quad \quad \text{Do}[ \text{exp2b}= \text{Exponent}[ s2[[ j2]], 
\text{var2}]; \\
 {}\quad \quad \quad \text{If}[ \text{exp2b}>= \text{exp2a}, \\
 {}\quad \quad \quad \quad \text{poly2}= \text{Expand}[ s2[[1]]* 
\text{var2}\wedge ( \text{exp2b}- \text{exp2a})]; \\
 {}\quad \quad \quad \quad \text{Coe2a}= \text{Coefficient}[ \text{poly2}, 
\text{var2}\wedge \text{exp2b}]; \\
 {}\quad \quad \quad \quad \text{Coe2b}= \text{Coefficient}[ s2[[ j2]], 
\text{var2}\wedge \text{exp2b}]; \\
 {}\quad \quad \quad \quad s2[[ j2]]= \text{ExpandAll}@( s2[[ j2]]* 
\text{Coe2a}- \text{poly2}* \text{Coe2b});  ], \{ j2,2, \text{Length}[ s2]\}]; \\
 {}\quad \quad s2= \text{CancelFactor}[ s2]; \\
 {}\quad \quad s2= \text{MinimizeExpofFac}[ s2]; \\
 {}\quad \quad \text{s2b}= \text{s2b}\sim \text{Join}\sim \text{Select}[ 
s2, \text{Intersection}[ \text{Variables}[\#],\{ \text{var2}\}]==\{\}\&]; \\
 {}\quad \quad s2= \text{Select}[ s2, \text{Intersection}[ 
\text{Variables}[\#],\{ \text{var2}\}]!=\{\}\&];  kk2++;  ]; \\
 {}\quad s2= \text{SortBy}[ s2, \text{Exponent}[\#, \text{var2}]\& ]; \text{Return}[\{ \text{s2b}, \text{possfac2}\}]; ) \\
 \\
\text{(* for each variable $v$ in the varset, delete the polynomial (if it exists)} \\ {} \quad \text{ in $S5$ that is the unique polynomial involving variable $v$ *)}\\
\text{DeleteVar}[ S2\_, \text{varset2}\_] := ( \\
 {}\quad s2 = S2; \\
 {}\quad \text{If}[ \text{Length}[ s2] > 1, \\
 {}\quad \quad \text{Do}[ \text{If}[ \text{Length}[ \text{Select}[ 
s2, ( \text{Variables}[\#] \cap \{ \text{varset2}[[ k2]]\}) != 
\{\} \& ]] == 1, \text{del} = \{\}; \\
 {}\quad \quad \quad \quad \text{Do}[ \text{If}[( \text{Variables}[ 
s2[[ j2]]] \cap \{ \text{varset2}[[ k2]]\}) != \{\}, \\
 {}\quad \quad \quad \quad \quad \quad \text{del} = \text{del}\sim 
\text{Join}\sim \{\{ j2\}\}], \\
 {}\quad \quad \quad \quad \quad \{ j2, 1, \text{Length}[ s2]\}]; s2 = \text{Delete}[ s2, \text{del}]],  \{ k2, 1, \text{Length}[ \text{varset2}]\}]]; \\
 {}\quad \text{Return}[ s2];  ) \\
 \\
 \text{(* First algorithm for finding mass relations *)}  \\
 \text{FindMassRel1}[]:=( \\
 {}\quad t1= \text{AbsoluteTime}[]; \\
 {}\quad \text{Str}=`` \text{finding} \text{the} \text{mass} 
\text{relation} \text{of} \text{FMX}[[''<> \text{ToString}[ N1]<>``]]''; \\
 {}\quad \text{Print}[ \text{Str}];  \text{Print}[];  \text{TF1}=0; \\
 {}\quad \text{EqnZ}=( \text{ZX1}* \text{Z}). M;  \text{EqnZ}= \text{Relation1}@ \text{Complement}[ 
\text{EqnZ},\{0\}]; \\
 {}\quad \text{EqnW}=( \text{WX1}* \text{W}). M;  \text{EqnW}= \text{Relation1}@ \text{Complement}[ 
\text{EqnW},\{0\}]; \\
 {}\quad \text{Polyset1}= \text{Union}[ \text{EqnW}, \text{EqnZ}]; \\
 {}\quad \text{Print}[ \text{Style}[`` \text{Collected} \text{equations}: 
'', \text{Bold},15, \text{Blue}]]; \\
 {}\quad \text{If}[ \text{Length}@ \text{EqnZ}+ \text{Length}@ 
\text{EqnW}>0, \\
 {}\quad \quad \text{Print}[ \text{Style}[`` \text{Main} \text{equations}: 
'', \text{Bold},13, \text{RGBColor}[0.8,0.4,0]]]; \\
 {}\quad \quad \text{Do}[ \text{Print}[0,`` = '', \text{EqnZ}[[ k1]]], \{ k1,1, \text{Length}@ \text{EqnZ}\}]; \\
 {}\quad \quad \text{Do}[ \text{Print}[0,`` = '', \text{EqnW}[[ k1]]], \{ k1,1, \text{Length}@ \text{EqnW}\}];  ]; \\
 {}\quad \text{If}[ \text{Length}@ \text{ClusterZ}[\{ \text{zo1}, 
\text{wo1}\}]+ \text{Length}@ \text{ClusterW}[\{ \text{zo1}, \text{wo1}\}]>0,
\\
 {}\quad \quad \text{Print}[ \text{Style}[`` \text{Cluster} 
\text{equations}: '', \text{Bold},13, \text{RGBColor}[0.8,0.4,0]]]; \\
 {}\quad \quad \text{Do}[ \text{Print}[0,`` = '', \text{ClusterZ}[\{ 
\text{zo1}, \text{wo1}\}][[ k1]]], \{ k1,1, \text{Length}@ \text{ClusterZ}[\{ \text{zo1}, 
\text{wo1}\}]\}]; \\
 {}\quad \quad \text{Do}[ \text{Print}[0,`` = '', \text{ClusterW}[\{ 
\text{zo1}, \text{wo1}\}][[ k1]]], \{ k1,1, \text{Length}@ \text{ClusterW}[\{ \text{zo1}, 
\text{wo1}\}]\}]; ]; \\
 {}\quad \text{Polyset1}= \text{Relation1}[ \text{Polyset1}]; \\
 {}\quad \text{Polyset1}= \text{Relation2}[ \text{Polyset1}]; \\
 {}\quad \text{Polyset1}= \text{ClusterReplace}[\{ \text{zo1}, \text{wo1}\}, 
\text{Polyset1}]; \\
 {}\quad \text{Polyset1}= \text{Complement}[ \text{Polyset1},\{1,0,-1\}]; \\
 {}\quad \text{Polyset1b}= \text{GroebnerBasis}[ \text{Polyset1}, M\sim 
\text{Join}\sim \text{mu}]; \\
 {}\quad \text{Polyset1}= \text{Polyset1b}; \\
 {}\quad \text{Polyset1}= \text{CancelFactor}[ \text{Polyset1}]; \\
 {}\quad \text{Polyset1}= \text{Complement}[ \text{Polyset1},\{1,0,-1\}]; \\
 {}\quad \text{If}[ \text{MassRelationQ}[ \text{Polyset1}]== \text{True}, \\
 {}\quad \quad \text{TF1}=1;  \text{Polyset1}= \text{TakeMassRela}[ \text{Polyset1}]; \\
 {}\quad \quad \text{Print}[`` \text{Finished }, \text{we } \text{have } 
\text{mass } \text{relation } '', \text{Style}[ \text{Polyset1}, 
\text{Blue}],\\
 {}\quad `` \text{for } \text{this } \text{diagram}.'']]; \\
 {}\quad \text{If}[ \text{TF1}==0,  \text{Print}[`` \text{We } \text{don't } \text{find } \text{mass } \text{relation. } \text{We } \text{have } \text{the } \text{relations}:'']; \\
 {}\quad \quad \text{Print}[ \text{Polyset1},\ ``\setminus \text{n} '' ]]; \\
 {}\quad \text{Print}[`` \text{spending} \text{time}: '',
\text{AbsoluteTime}[]- t1]; ) \\
 \\
\text{(* Second algorithm for finding mass relations *)}  \\ 
 \text{indMassRel2}[]:=( \\
 {}\quad \text{TF1}=0; t1= \text{AbsoluteTime}[]; \\
 {}\quad \text{Str}=`` \text{finding } \text{mass } 
\text{relation } \text{of } \text{FMX}[[''<> \text{ToString}[ N1]<>``]]...''; 
\\
 {}\quad \text{Print}[ \text{Str}, ``\setminus \text{n} '']; \\
 {}\quad \text{If}[ \text{Length}[ \text{zwotype3}]>0, \text{TF1}=1;  jj1=1; \\
 {}\quad \quad \text{While}[ jj1<= \text{Length}[ \text{zwotype3}]\&\& \text{TF1}==1, \\
 {}\quad \quad \quad \{ \text{S1a}, \text{S1b}\}= \text{FindWedgeV}[ 
\text{zwotype3}[[ jj1]]]; \\
 {}\quad \quad \quad \text{If}[ \text{S1a}==\{\}, \text{TF1}=0;  ]; jj1++]; \\
 {}\quad \quad \text{If}[ \text{TF1}==1, \\
 {}\quad \quad \quad \text{massrellcm}=1; \text{PossFaclcm}=1; jj1=1; \\
 {}\quad \quad \quad \text{While}[ jj1<= \text{Length}[ \text{zwotype3}],  \{ \text{S1a}, \text{S1b}\}= \text{FindWedgeV}[ 
\text{zwotype3}[[ jj1]]]; \\
 {}\quad \quad \quad \quad \text{WedgeV1a}= \text{Flatten}@ \text{Table}[ 
M[[ j1]] M[[ i1]]( \text{zeta}[[ i1]] \text{omega}[[ j1]])\wedge (-1),\\
 {}\quad \quad \quad \quad \quad \quad \quad \quad \quad \{ i1, 
\text{S1a}\},\{ j1, \text{S1b}\}]; \\
 {}\quad \quad \quad \quad \text{sign1}= \text{Tuples}[\{1,-1\},\{ 
\text{Length}[ \text{S1a}], \text{Length}[ \text{S1b}]\}]; 
\\
 {}\quad \quad \quad \quad \text{WedgeV1b}= \text{Table}[ \text{Dot}[ 
\text{WedgeV1a}, \text{Flatten}[ \text{sign1}[[ j1]],1]],\\ 
 {}\quad \quad \quad \quad \quad \quad \quad \quad \quad \{ j1,1, \text{Length}[ \text{sign1}]\}]; \\
 {}\quad \quad \quad \quad \text{VarSet1a}= \text{Variables}[ 
\text{WedgeV1a}]; \\
 {}\quad \quad \quad \quad \text{WedgeV1b}= \text{WedgeV1b}[[1; ;  2\wedge ( \text{Length}[ \text{S1b}]* 
\text{Length}[ \text{S1a}]-1)]]; \\
 {}\quad \quad \quad \quad \text{Do}[ \text{WedgeV1b}[[ j1]]= 
\text{ExpandAll}@ \text{WedgeV1b}[[ j1]], 
\\
 {}\quad \quad \quad \quad \quad \{ j1,1, \text{Length}@ \text{WedgeV1b}\}]; 
\\
 {}\quad \quad \quad \quad \text{index1}=1; \\
 {}\quad \quad \quad \quad \text{Circledzc}= \text{Table}[ 
\text{Intersection}[ \text{zc1}[[ i1]], \text{Flatten}[ \text{Position}[ \\
  {}\quad \quad \quad \quad \quad \quad \quad \quad \quad \quad
\text{Diagonal}[ \text{ZX1}],1]]],\{ i1,1, 
\text{Length}[ \text{zc1}]\}]; \\
 {}\quad \quad \quad \quad \text{Circledwc}= \text{Table}[ 
\text{Intersection}[ \text{wc1}[[ i1]], \text{Flatten}[ \text{Position}[ \\
  {}\quad \quad \quad \quad \quad \quad \quad \quad \quad \quad
\text{Diagonal}[ \text{WX1}],1]]],\{ i1,1, 
\text{Length}[ \text{wc1}]\}]; \\
 {}\quad \quad \quad \quad \text{MassCenRela}= \text{Flatten}[ \text{Table}[ 
 \text{Sum}[ M[[ \text{Circledzc}[[ i1, j1]] ]]*\\
  {}\quad \quad \quad \quad \quad \quad \quad \quad \quad \quad  \text{z}[[ 
\text{Circledzc}[[ i1, j1]] ]], \{ j1,1, \text{Length}[ \text{Circledzc}[[
i1]]]\} ] , \\
  {}\quad \quad \quad \quad \quad \quad \quad \quad \quad \quad \{ i1,1, \text{Length}[ \text{Circledzc}]\}]]\sim \text{Join}\sim 
\text{Flatten}[ \text{Table}[ \text{Sum}[\\
  {}\quad \quad \quad \quad \quad \quad \quad \quad \quad \quad  M[[ \text{Circledwc}[[ i1, j1]] 
 ]]*  \text{w}[[ \text{Circledwc}[[ i1, j1]] ]],\\
  {}\quad \quad \quad \quad \quad \quad \quad \quad \quad \quad  \{ j1,1, \text{Length}[ \text{Circledwc}[[ i1]]]\} ] ,\{ i1,1, \text{Length}[ \text{Circledwc}]\}]]; \\
 {}\quad \quad \quad \quad \text{Str}=`` \text{Finding} \text{mass} 
\text{relation} \text{with} \text{zwotype3}[[''<> \text{ToString}[ jj1]<>``]]
''; \\
 {}\quad \quad \quad \quad \text{Print}[ \text{Str}, ``\setminus \text{n} '']; \\
 {}\quad \quad \quad \quad \text{Print}[`` \text{The } \text{order } 
\text{matrices } \text{are } '']; \\
 {}\quad \quad \quad \quad \text{Print}[\{ \text{MatrixForm}[ 
\text{zwotype3}[[ jj1,1]]], \text{MatrixForm}[ \text{zwotype3}[[ jj1,2]]]\}];
\\
 {}\quad \quad \quad \quad \text{Print}[ \text{Style}[`` \text{Wedge } 
\text{equations }: '', \text{Bold},13, \text{RGBColor}[0.8,0.4,0]]]; \\
 {}\quad \quad \quad \quad \text{Replace}[ \text{Prepend}[ \text{Insert}[ 
\text{WedgeV1a},`` \pm '', \text{Table}[\{ k1\},\{ k1,2, \text{Length}@ 
\text{WedgeV1a}\}]],\\ 
 {}\quad \quad \quad \quad \quad \quad \quad `` \text{ one } \text{of } 0 = ''], \text{List}\to \text{Print},1, \text{Heads}\to \text{True}]; \\
 {}\quad \quad \quad \quad \text{While}[ \text{index1}<= \text{Length}[ 
\text{WedgeV1b}], \\
 {}\quad \quad \quad \quad \quad \text{Polyset1}= \text{MassCenRela}; \\
 {}\quad \quad \quad \quad \quad \text{Polyset1}= \text{Complement}[ 
\text{Polyset1},\{1,0,-1\}]; \\
 {}\quad \quad \quad \quad \quad \text{Polyset1}= \text{Polyset1}\sim 
\text{Join}\sim \{ \text{WedgeV1b}[[ \text{index1}]]\}; \\
 {}\quad \quad \quad \quad \quad \text{Str}=`` \text{running } \text{wedge } 
\text{equation } ''<> \text{ToString}[ \text{index1}]<>``...''; \\
 {}\quad \quad \quad \quad \quad \text{Print}[ \text{Str}, ``\setminus \text{n} '',  ``\setminus \text{n} '',`` \text{polynomials}: '',  ``\setminus \text{n} '',\text{Polyset1}]; \\
 {}\quad \quad \quad \quad \quad \text{PossFac}=1; \\
 {}\quad \quad \quad \quad \quad \text{Polyset1}= \text{Relation1}[ 
\text{Polyset1}]; \\
 {}\quad \quad \quad \quad \quad \text{Polyset1}= \text{Relation2}[ 
\text{Polyset1}]; \\
 {}\quad \quad \quad \quad \quad \text{Polyset1}= \text{ChangeVar1}[ 
\text{Polyset1}]; \\
 {}\quad \quad \quad \quad \quad \text{Polyset1}= \text{ChangeVar2}[ 
\text{Polyset1}]; \\
 {}\quad \quad \quad \quad \quad \text{Polyset1}= \text{CancelDenominator}[ 
\text{Polyset1}]; \\
 {}\quad \quad \quad \quad \quad \text{Polyset1}= \text{ClusterReplace}[ 
\text{zwotype3}[[ jj1]], \text{Polyset1}]; \\
 {}\quad \quad \quad \quad \quad \text{Polyset1}= \text{Table}[ 
\text{Simplify}[ \text{Polyset1}[[ j1]]],\{ 
j1,1, \text{Length}[ \text{Polyset1}]\}]; \\
 {}\quad \quad \quad \quad \quad \text{Polyset1}= \text{Complement}[ 
\text{Polyset1},\{1,0,-1\}]; \\
 {}\quad \quad \quad \quad \quad \text{VarSet1a}=\{\}; \\
 {}\quad \quad \quad \quad \quad \text{VarSet1a}= \text{Variables}[ 
\text{Polyset1}]; \\
 {}\quad \quad \quad \quad \quad \text{VarSet1b}= \text{Complement}[ 
\text{VarSet1a}, \text{Union}[ M, \text{mu}]]; \\
 {}\quad \quad \quad \quad \quad \text{VarSet1b}= \text{SortBy}[ 
\text{VarSet1b}, \text{Max}[ \text{Table}[ 
\text{Exponent}[ \text{Polyset1}[[ k1]],\#],\\ 
 {}\quad \quad \quad \quad \quad \quad \quad \quad \quad \quad \{ k1,1, \text{Length}[ P
 \text{olyset1}]\}]]\&]; \\
 {}\quad \quad \quad \quad \quad \text{Polyset1}= \text{MinimizeExpofFac}[ 
\text{Polyset1}]; \\
 {}\quad \quad \quad \quad \quad \text{While}[ \text{Length}[ 
\text{VarSet1b}]>0\&\& \text{Length}[ \text{Polyset1}]>1\&\& 
\\
 {}\quad \quad \quad \quad \quad \quad \text{MassRelationQ}[ 
\text{Polyset1}]== \text{False}, \\
 {}\quad \quad \quad \quad \quad \quad \text{VA1}= \text{Elimi}[ 
\text{Polyset1}, \text{VarSet1b}[[1]]]; \\
 {}\quad \quad \quad \quad \quad \quad \text{Polyset1}= \text{VA1}[[1]]; \\
 {}\quad \quad \quad \quad \quad \quad \text{PossFac}= \text{PossFac}* 
\text{VA1}[[2]]; \\
 {}\quad \quad \quad \quad \quad \quad \text{Polyset1}= \text{Complement}[ 
\text{Polyset1},\{0\}]; \\
 {}\quad \quad \quad \quad \quad \quad \text{Polyset1}= \text{CancelFactor}[ 
\text{Polyset1}]; \\
 {}\quad \quad \quad \quad \quad \quad \text{If}[ \text{Length}[ 
\text{Polyset1}]>1, \\
 {}\quad \quad \quad \quad \quad \quad \quad \text{Polyset1}= 
\text{DeleteVar}[ \text{Polyset1}, \text{VarSet1b}]; \\
 {}\quad \quad \quad \quad \quad \quad \quad \text{Polyset1}= 
\text{DelDupPoly2}[ \text{Polyset1}, \text{VarSet1b}];  ]; \\
 {}\quad \quad \quad \quad \quad \quad \text{VarSet1a}=\{\}; \\
 {}\quad \quad \quad \quad \quad \quad \text{Do}[ \text{VarSet1a}= 
\text{VarSet1a}\text{cup} \text{Variables}[ \text{Polyset1}[[ k1]]], \\
 {}\quad \quad \quad \quad \quad \quad \quad \{ k1,1, \text{Length}[ 
\text{Polyset1}]\}]; \\
 {}\quad \quad \quad \quad \quad \quad \text{VarSet1b}= \text{Complement}[ 
\text{VarSet1a}, \text{Union}[ M, \text{mu}]]; \\
 {}\quad \quad \quad \quad \quad \quad \text{VarSet1b}= \text{SortBy}[ 
\text{VarSet1b}, \text{Max}[ \text{Table}[ 
\text{Exponent}[ \text{Polyset1}[[ k1]],\#], \\ 
{}\quad \quad \quad \quad \quad \quad \quad \quad \quad \quad \quad \{ k1,1, \text{Length}[  \text{Polyset1}]\}]]\&]; ]; \\
 {}\quad \quad \quad \quad \quad \text{Polyset1}= \text{CancelFactor}[ 
\text{Polyset1}]; \\
 {}\quad \quad \quad \quad \quad \text{Polyset1}= \text{Complement}[ 
\text{Polyset1},\{1,0,-1\}]; \\
 {}\quad \quad \quad \quad \quad \text{If}[ \text{MassRelationQ}[ 
\text{Polyset1}]== \text{True}, \\
 {}\quad \quad \quad \quad \quad \quad \text{Polyset1}= \text{Simplify}[ 
\text{TakeMassRela}[ \text{Polyset1}]]; \\
 {}\quad \quad \quad \quad \quad \quad \text{Polyset1}= 
\text{MinimizeExpofFacP}[ \text{Polyset1}]; \\
 {}\quad \quad \quad \quad \quad \quad \text{Polyset1}= 
\text{ChangeVar2Inv}[ \text{Polyset1}]; \\
 {}\quad \quad \quad \quad \quad \quad \text{PossFac}= 
\text{ChangeVar2Inv}[\{ \text{PossFac}\}][[1]]; \\
 {}\quad \quad \quad \quad \quad \quad \text{PossFac}= 
\text{MinimizeExpofFacP}[\{ \text{PossFac}\}][[1]]; \\
 {}\quad \quad \quad \quad \quad \quad \text{Str}=`` \text{for } 
\text{zwotype3 }[['' <> \text{ToString}[ jj1]<>``]] \text{and} \text{WedgeV}[[''  <> \\ 
 {}\quad \quad \quad \quad \quad \quad \quad \quad \quad  \text{ToString}[ \text{index1}]<>``]], \text{we } \text{have } \text{mass } \text{relations }''; \\
 {}\quad \quad \quad \quad \quad \quad \text{Print}[ \text{Str}]; \text{Print}[ \text{Style}[ 
\text{Polyset1}, \text{Blue}]]; \\
 {}\quad \quad \quad \quad \quad \quad \text{Do}[ \text{massrellcm}= 
\text{PolynomialLCM}[ \text{massrellcm}, 
\text{poly1}]; , \\
 {}\quad \quad \quad \quad \quad \quad \quad \{ \text{poly1}, 
\text{Polyset1}\}]; \\
 {}\quad \quad \quad \quad \quad \quad \text{PossFaclcm}= 
\text{PolynomialLCM}[ \text{PossFaclcm}, \text{PossFac}]; 
\\
 {}\quad \quad \quad \quad \quad \quad \text{Print}[`` \text{or } '', 
\text{Style}[ \text{PossFac}, \text{RGBColor}[0,0.5,0]],`` ( \text{the } 
\text{part } \text{obtained }\\
 {}\quad \quad \quad \quad \quad \quad \quad\quad\quad \text{by } \text{factors } \text{of }
\text{coefficients } )'', ``\setminus \text{n}\setminus \text{n} ''];  ]; \\
 {}\quad \quad \quad \quad \quad \text{If}[ \text{MassRelationQ}[ 
\text{Polyset1}]== \text{False}, \\
 {}\quad \quad \quad \quad \quad \quad \text{TF1}=0; \\
 {}\quad \quad \quad \quad \quad \quad \text{Str}=`` \text{We } \text{don't } \text{obtain } \text{any } \text{mass } \text{relation } \text{for } 
\text{zwotype3 }[['' \\
 {}\quad \quad \quad \quad \quad \quad \quad \quad <> \text{ToString}[ jj1] <>``]] \text{ and } 
\text{WedgeV}[[''<> \text{ToString}[ \text{index1}] \\
 {}\quad \quad \quad \quad \quad \quad \quad \quad <>``]], \text{we } \text{have } \text{relations}''; \\
 {}\quad \quad \quad \quad \quad \quad \text{Print}[ \text{Str}]; \text{Print}[ \text{Polyset1}];  ]; \\
 {}\quad \quad \quad \quad \quad \text{index1}++;  ];  jj1++];  ]; \\
 {}\quad \quad \text{If}[ \text{TF1}==1, \\
 {}\quad \quad \quad \text{Print}[`` \text{Finished}, \text{for } 
\text{this } \text{diagram }, \text{we } \text{have } \text{mass } 
\text{relations }'', \\
 {}\quad \quad \quad \quad \quad \quad \text{Style}[ \text{massrellcm}, \text{Blue}], `` 
\text{ or } '', \text{Style}[ \text{PossFaclcm}, \text{RGBColor}[0,0.5,0]], \\
 {}\quad \quad \quad \quad \quad \quad`` (\text{the } \text{part } \text{obtained } \text{by } \text{factors } \text{of } \text{coefficients} )'' ], \\
 {}\quad \quad \quad \text{Print}[`` \text{Finished}, \text{ we } 
\text{don't } \text{obtain } \text{any } \text{mass } \text{relations } 
\text{for } \text{this } \text{diagram}.'']]; ]; \\
 {}\quad \text{Print}[`` \text{spending} \text{time}: '', 
\text{AbsoluteTime}[]- t1,\ ``\setminus \text{n}\setminus \text{n}\setminus \text{n} '']; \\
\\
\text{(* Third algorithm for finding mass relations *)}  \\
 {} \text{FindMassRel3}[\{ \text{zomatrix1}\_, \text{womatrix1}\_\}, 
 \text{VertexSet1a}\_, \text{VertexSet1b}\_] := \\
\text{FindMassRel3}[\{ \text{zomatrix1}, \text{womatrix1}\}, 
\text{VertexSet1a}, \text{VertexSet1b}, \text{False}, \{\}]; \\
 {} \text{FindMassRel3}[\{ \text{zomatrix1}\_, \text{womatrix1}\_\}, 
 \text{VertexSet1a}\_, \text{VertexSet1b}\_,\\
 {}\quad\quad\quad\quad\quad\quad\quad \text{RandomTF}\_, 
\text{additionaleqnset}\_] := ( \\
 {}\quad t1 = \text{AbsoluteTime}[]; \\
 {}\quad \text{ODind1} = \text{Flatten}[ \text{Table}[ 
\text{Table}[\{ j1, i1\}, \{ i1, j1 + 1, 
n\}], \{ j1, 1, n - 1\}], 1]; \\
 {}\quad \text{EdgePos1} = \text{Select}[ \text{ODind1}, 
\text{XX1}[[\#[[1]], \#[[2]]]] > 0 \&]; \\
 {}\quad \text{TriangleSet} = \{\}; \\
 {}\quad \text{Do}[ \text{Do}[ \text{Do}[ \text{If}[ 
\text{Complement}[\{ \{ i1, j1\}, \{ j1, k1\}, \{ i1, 
k1\}\}, \text{EdgePos1}] == \{\}, \\
 {}\quad \quad \quad \quad \quad \text{AppendTo}[ 
\text{TriangleSet}, \{ i1, j1, k1\}];  ]; , \\
 {}\quad \quad \quad \quad \{ i1, 1, j1 - 1\}]; , \{ j1, 2, k1 - 1\}];  ,  \{ k1, 3, n\}]; \\
{}\quad  \text{Idenu} = \text{Table}[ u[[\ s1[[1]], s1[[2]]\ ]]\wedge (-2) + 
u[[\ s1[[2]], s1[[3]]\ ]]\wedge (-2)\\
 {}\quad\quad\quad\quad - u[[\ s1[[1]], s1[[3]]\ ]]\wedge (-2), \{ s1, \text{TriangleSet}\}]; \\
{}\quad  {} \text{Idenv} = \text{Table}[ v[[\ s1[[1]], s1[[2]]\ ]]\wedge 
(-2) + v[[\ s1[[2]], s1[[3]]\ ]]\wedge (-2) \\
 {}\quad\quad\quad\quad - v[[\ s1[[1]], s1[[3]]\ ]]\wedge (-2), \{ s1, \text{TriangleSet}\}]; \\
 {}\quad \text{Print}[ `` \text{finding }  \text{mass } 
\text{relation } \text{of } \text{FMX}[['' <> \text{ToString}[ N1] <> ``]] 
\text{with} '']; \\
 {}\quad \text{Print}[\{ \text{MatrixForm}[ \text{outputform1}@ 
\text{zomatrix1}], \text{MatrixForm}[ \text{outputform1}@ 
\text{womatrix1}]\}]; \\
 {}\quad \text{TF1} = 0; \text{PossFac} = 1; \\
 {}\quad \text{EqnZ} = \text{Relation1}@ \text{EqnMZO}[\{ 
\text{zomatrix1}, \text{womatrix1}\}, 
\text{VertexSet1a}]; \\
 {}\quad \text{EqnW} = \text{Relation1}@ \text{EqnMWO}[\{ 
\text{zomatrix1}, \text{womatrix1}\}, 
\text{VertexSet1b}]; \\
 {}\quad \text{Polyset1} = \text{EqnZ} \cup \text{EqnW} \cup 
\text{additionaleqnset}; \text{Polyset1b} = \{\}; \\
 {}\quad \text{Polyset1} = \text{Relation1}[ \text{Polyset1}]; \\
 {}\quad \text{Polyset1} = \text{Relation2}[ \text{Polyset1}]; \\
 {}\quad \text{Polyset1} = \text{Relation3ver2}[ \text{Polyset1}]; \\
 {}\quad \text{VarSet1a} = \text{Variables}[ \text{Polyset1}]; \\
 {}\quad \text{VarSet1b} = \text{Complement}[ \text{VarSet1a}, 
\text{Union}[ M, \text{mu}]]; \\
 {}\quad \text{Do}[ \text{If}[ \text{Complement}[ \text{Variables}[ 
\text{Idenu}[[ k1]]], \text{VarSet1a}] == \{\}, \\
 {}\quad \quad \quad \text{Polyset1} = \text{Polyset1}\sim 
\text{Join}\sim \{ \text{Idenu}[[ k1]]\}],  \{ k1, 1, \text{Length}[ \text{Idenu}]\}]; \\
 {}\quad \text{Do}[ \text{If}[ \text{Complement}[ \text{Variables}[ 
\text{Idenv}[[ k1]]], \text{VarSet1a}] == \{\}, \\
 {}\quad \quad \quad \text{Polyset1} = \text{Polyset1}\sim 
\text{Join}\sim \{ \text{Idenv}[[ k1]]\}], \{ k1, 1, \text{Length}[ \text{Idenv}]\}]; \\
 {}\quad \text{Polyset1} = \text{ClusterReplace}[\{ \text{zomatrix1}, 
 \text{womatrix1}\}, \text{Polyset1}]; \\
 {}\quad \text{Polyset1} = \text{CancelDenominator}[ \text{Polyset1}]; \\
 {}\quad \text{Polyset1} = \text{CancelFactor}[ \text{Polyset1}]; \\
 {}\quad \text{Polyset1} = \text{Complement}[ \text{Polyset1}, \{1, 0, 
-1\}]; \\
 {}\quad \text{Polyset1} = \text{ChangeVar3}[ \text{Polyset1}]; \\
 {}\quad \text{Polyset1} = \text{CancelDenominator}[ \text{Polyset1}]; \\
 {}\quad \text{Polyset1} = \text{CancelFactor}[ \text{Polyset1}]; \\
 {}\quad \text{Polyset1} = \text{Complement}[ \text{Polyset1}, \{1, 0, 
-1\}]; \\
 {}\quad \text{VarSet1a} = \{\}; \\
 {}\quad \text{Do}[ \text{VarSet1a} = \text{VarSet1a} \cup 
\text{Variables}[ \text{Polyset1}[[ k1]]], \{ k1, 1, \text{Length}[ \text{Polyset1}]\}]; \\
 {}\quad \text{VarSet1b} = \text{Complement}[ \text{VarSet1a}, 
\text{Union}[ M, \text{mu}]]; \\
 {}\quad \text{VarSet1b} = \text{SortBy}[ \text{VarSet1b}, 
\text{Max}[ \text{Table}[ \text{Exponent}[ \text{Polyset1}[[ k1]], 
\#],\\
 {}\quad \quad\quad\quad\quad\quad \{ k1, 1, \text{Length}[ \text{Polyset1}]\}]] \&]; \\
 {}\quad \text{Print}[ `` \text{Change } \text{of } \text{variables } '' , ``\setminus \text{n} '' , \text{Polyset1}]; \\
 {}\quad \text{Print}[ \text{ToString}[ \text{Length}[ \text{Polyset1}]] 
<> `` \text{polynomials}, '' <> \text{ToString}[ \text{Length}[ 
\text{VarSet1b}]] <> \\
 {}\quad \quad \quad \quad `` \text{variables} '' <> \text{ToString}[ 
\text{VarSet1b}], `` \setminus \text{n}  \text{Simple} \text{Substitutions} '']; \\
 {}\quad \text{Polyset1} = \text{SimpleSubstitution}[ \text{Polyset1}, 
\text{VarSet1b}]; \\
 {}\quad \text{Polyset1} = \text{CancelFactor}[ \text{Polyset1}]; \\
 {}\quad \text{Polyset1} = \text{Complement}[ \text{Polyset1}, \{1, 0, 
-1\}]; \\
 {}\quad \text{Polyset1} = \text{DelDupPoly2}[ \text{Polyset1}, 
\text{VarSet1b}]; \\
{}\quad \text{Polyset1b}=\text{Polyset1b}\cup \text{TakeMassRela}(\text{Polyset1});\\
{}\quad \text{Polyset1}=\text{Complement}[\text{Polyset1},\text{Polyset1b}];
 {}\quad \text{VarSet1a} = \{\}; \\
 {}\quad \text{Do}[ \text{VarSet1a} = \text{VarSet1a} \cup 
\text{Variables}[ \text{Polyset1}[[ k1]]], \{ k1, 1, \text{Length}[ \text{Polyset1}]\}]; \\
 {}\quad \text{VarSet1b} = \text{Complement}[ \text{VarSet1a}, 
\text{Union}[ M, \text{mu}]]; \\
 {}\quad \text{VarSet1b} = \text{SortBy}[ \text{VarSet1b}, 
\text{Max}[ \text{Table}[ \text{Exponent}[ \text{Polyset1}[[ k1]], 
\#],\\
 {}\quad \quad \quad \quad \{ k1, 1, \text{Length}[ \text{Polyset1}]\}]] \&]; \\
 {}\quad \text{Print}[ \text{Polyset1}, ``\setminus \text{n} '', \text{ToString}[ \text{Length}[ \text{Polyset1}]] <> `` \text{polynomials}, '' <>  \\
 {}\quad \quad \quad \quad \text{ToString}[ \text{Length}[ 
\text{VarSet1b}]] <> `` \text{variables} '' <> \text{ToString}[ \text{VarSet1b}]];  \\
 {}\quad \text{Polyset1} = \text{MinimizeExpofFac}[ \text{Polyset1}]; \\
 {}\quad \text{While}[ \text{Length}[ \text{VarSet1b}] > 0 \&\&  \text{Length}[ \text{Polyset1}] > 1 \&\& \\
 {}\quad \quad \text{MassRelationQ}[ \text{Polyset1}] == \text{False}, 
\\
 {}\quad \quad \text{If}[ \text{RandomTF} == \text{True}, \\
 {}\quad \quad \quad \text{Var1} = \text{RandomChoice}[ \text{VarSet1b}], 
\\
 {}\quad \quad \quad \text{Var1} = \text{VarSet1b}[[1]]]; \\
 {}\quad \quad \text{str} = `` \text{Eliminating} \text{variable} '' <> 
\text{ToString}[ \text{Var1}] <> ``...''; \\
 {}\quad \quad \text{Print}[ \text{str}, ``\setminus \text{n} ''];
 {}\quad \quad \text{VA1} = \text{Elimi}[ \text{Polyset1}, 
\text{Var1}]; \\
 {}\quad \quad \text{Polyset1} = \text{VA1}[[1]]; \\
 {}\quad \quad \text{PossFac} = \text{PossFac}* \text{VA1}[[2]]; \\
 {}\quad \quad \text{Polyset1} = \text{Complement}[ \text{Polyset1}, 
\{0\}]; \\
 {}\quad \quad \text{Polyset1} = \text{CancelFactor}[ \text{Polyset1}]; 
\\
 {}\quad \quad \text{Polyset1b} = \text{Polyset1b} \cup 
\text{TakeMassRela}[ \text{Polyset1}]; \\
 {}\quad \quad \text{Polyset1} = \text{Complement}[ \text{Polyset1}, 
\text{Polyset1b}]; \\
 {}\quad \quad \text{If}[ \text{Length}[ \text{Polyset1}] > 1, \\
 {}\quad \quad \quad \text{Polyset1} = \text{DeleteVar}[ 
\text{Polyset1}, \text{VarSet1b}]; \\
 {}\quad \quad \quad \text{Polyset1} = \text{DelDupPoly2}[ 
\text{Polyset1}, \text{VarSet1b}]; ]; \\
 {}\quad \quad \text{VarSet1a} = \{\}; \\
 {}\quad \quad \text{Do}[ \text{VarSet1a} = \text{VarSet1a} \cup 
\text{Variables}[ \text{Polyset1}[[ k1]]], \{ k1, 1, \text{Length}[ \text{Polyset1}]\}]; \\
 {}\quad \quad \text{VarSet1b} = \text{Complement}[ \text{VarSet1a}, 
\text{Union}[ M, \text{mu}]]; \\
 {}\quad \quad \text{VarSet1b} = \text{SortBy}[ \text{VarSet1b}, 
 \text{Max}[ \text{Table}[ \text{Exponent}[ \text{Polyset1}[[ k1]], 
\#], \\
 {}\quad \quad \quad \quad \quad \quad \quad \{ k1, 1, \text{Length}[ \text{Polyset1}]\}]] \&]; \\
 {}\quad \quad \text{str} = `` \text{after } \text{eliminating } '' <> 
\text{ToString}[ \text{Var1}]; \text{Print}[ \text{str}]; \\
 {}\quad \quad \text{NV1} = \text{Table}[ \text{Length}[ f1], \{ f1, 
\text{Polyset1b}\}]\sim \text{Join}\sim 
\text{Table}[ \text{Length}[ f1], \{ f1, \text{Polyset1}\}]; \\
 {}\quad \quad \text{If}[ \text{Max}[ \text{NV1}] <= 100, \text{Print}[ \text{Polyset1b}\sim \text{Join}\sim 
\text{Polyset1}]]; \\
 {}\quad \quad \text{Print}[ \text{ToString}[ \text{Length}[ 
\text{Polyset1b}\sim \text{Join}\sim \text{Polyset1}]] <> `` 
\text{polynomials},\\
 {}\quad \quad\quad\quad \quad '' <> \text{ToString}[ \text{Length}[ \text{VarSet1b}]] <> `` \text{variables} '' <> \text{ToString}[ \text{VarSet1b}]]; \\
 {}\quad \quad \text{Print}[`` \text{numbers} \text{of} \text{terms} : 
'', \text{NV1},``\setminus \text{n} \setminus \text{n} '',]; \\
 {}\quad \text{If}[ \text{Polyset1b} != \{\}, \\
 {}\quad \quad \text{Polyset1b} = \text{DelDupPoly3}[ \text{Polyset1b}, 
\text{VarSet1b}]; \\
 {}\quad \quad \text{Polyset1b} = \text{Complement}[ \text{Polyset1b}, 
\{1, -1\}]; \\
 {}\quad \quad \text{Polyset1b} = \text{Simplify}[ \text{Polyset1b}]; \\
 {}\quad \quad \text{Polyset1b} = \text{MinimizeExpofFacP}[ 
\text{Polyset1b}]; \\
 {}\quad \quad \text{Str} = `` \text{we } \text{have } \text{mass } 
\text{relations }''; \\
 {}\quad \quad \text{Print}[ \text{Str}, ``\setminus \text{n} '', \text{Style}[ \text{Polyset1b}, \text{Blue}]]; \\
 {}\quad \quad \text{If}[ \text{PossFac} =!= 1,  \text{Print}[`` \text{or } '', \text{Style}[ \text{PossFac}, \text{RGBColor}[0, 0.5, 0]], `` ( \text{ the } \text{part } \\
 {}\quad \quad \quad \quad \quad \quad \text{obtained } \text{by } \text{factors } 
 \text{of } \text{coefficients} )'', ``\setminus \text{n} \setminus \text{n} '']]; \\
 {}\quad \quad \text{Polyset1b} = \text{ChangeVar2}[ 
\text{Polyset1b}]; \\
 {}\quad \quad \text{Polyset1b} = \text{MinimizeExpofFacP}[ 
\text{Polyset1b}]; \\
 {}\quad \quad \text{Polyset1b} = \text{ChangeVar2Inv}[ 
\text{Polyset1b}]; \\
 {}\quad \quad \text{PossFac} = \text{ChangeVar2}[\{ 
\text{PossFac}\}][[1]]; \\
 {}\quad \quad \text{PossFac} = \text{MinimizeExpofFacP}[\{ 
\text{PossFac}\}][[1]]; \\
 {}\quad \quad \text{PossFac} = \text{ChangeVar2Inv}[\{ 
\text{PossFac}\}][[1]]; \\
 {}\quad \quad \text{Print}[`` \text{another } \text{expression }:'', ``\setminus \text{n} '',  \text{Style}[ \text{Polyset1b}, \text{Blue}]]; \\
 {}\quad \quad \text{If}[ \text{PossFac} =!= 1, \text{Print}[`` \text{or } '', \text{Style}[ 
\text{PossFac}, \text{RGBColor}[0, 0.5, 0]], ``\ ( \text{the } \text{part } 
\\
 {}\quad \quad \quad \quad \quad \quad \text{obtained } \text{by } \text{factors } 
 \text{of } \text{coefficients} )'']];  , \\
 {}\quad \quad \text{If}[ \text{Length}[ \text{Polyset1}] > 1, \\
 {}\quad \quad \quad \text{FacFamilyV1} = \text{Table}[\{\}, \{ k1, 1, 
\text{Length}[ \text{Polyset1}]\}]; \\
 {}\quad \quad \quad \text{Print}[ \text{FacFamilyV1}, ``\setminus \text{n} '',]; \\
 {}\quad \quad \quad \text{Do}[ \text{FacFamilyV1}[[ k1]] = 
\text{Table}[ \text{FactorList}[ \text{Polyset1}[[ k1]]][[ j1, 1]],\\
 {}\quad \quad \quad \quad \{ j1, 1, \text{Length}@ \text{FactorList}[ \text{Polyset1}[[ k1]]]\}], 
\\
 {}\quad \quad \quad \quad \{ k1, 1, \text{Length}@ 
\text{FacFamilyV1}\}]; \\
 {}\quad \quad \quad \text{Print}[ \text{FacFamilyV1}]; \\
 {}\quad \quad \quad \text{Print}[]; \\
 {}\quad \quad \quad \text{Do}[ \text{Polyset1}[[ k1]] = 1; \\
 {}\quad \quad \quad \quad \text{Do}[ \text{Do}[ \text{TF1b} = 1; \\
 {}\quad \quad \quad \quad \quad \quad \text{If}[ 
\text{PolynomialQ}[ \text{Simplify}@( \text{Expand}[ 
\text{Polyset1}[[ j1]]/ \text{FacFamilyV1}[[ k1]][[ i1]]]), \\
 {}\quad \quad \quad \quad \quad \quad \quad \text{Variables}[ \text{Polyset1}[[ j1]]]] == \text{True}, \\
 {}\quad \quad \quad \quad \quad \quad \quad \text{TF1b} = \text{TF1b}, \text{TF1b} = \text{TF1b}*0]; , \{ j1, 1, \text{Length}[ 
\text{Polyset1}]\}]; \\
 {}\quad \quad \quad \quad \quad \text{If}[ \text{TF1b} == 1, \\
 {}\quad \quad \quad \quad \quad \quad \text{Polyset1}[[ k1]] = 
\text{Polyset1}[[ k1]]* \text{FacFamilyV1}[[ k1, i1]]];  , \\
 {}\quad \quad \quad \quad \quad \{ i1, 1, \text{Length}[ 
\text{FacFamilyV1}[[ k1]]]\}]; , \\
 {}\quad \quad \quad \quad \{ k1, 1, \text{Length}[ \text{Polyset1}]\}]; \\
 {}\quad \quad \quad \text{Polyset1} = \text{CancelDenominator}[ 
\text{Polyset1}]; \\
 {}\quad \quad \quad \text{Polyset1} = \text{CancelFactor}[ 
\text{Polyset1}]; \\
 {}\quad \quad \quad \text{Polyset1} = \text{DeleteVar}[ 
\text{Polyset1}, \text{VarSet1b}]; \\
 {}\quad \quad \quad \text{Polyset1} = \text{DelDupPoly2}[ 
\text{Polyset1}, \text{VarSet1b}];  ]; \\
 {}\quad \quad \text{If}[ \text{Length}[ \text{Polyset1}] == 1, \\
 {}\quad \quad \quad \text{Polyset1b} = \text{FactorList}[ 
\text{Polyset1}[[1]]]; \\
 {}\quad \quad \quad \text{Polyset1}[[1]] = 1; \\
 {}\quad \quad \quad \text{Do}[ \text{Polyset1}[[1]] = 
\text{Polyset1}[[1]]* \text{Polyset1b}[[ i1]][[1]], \{ i1, 1, \text{Length}[ 
\text{Polyset1b}]\}]]; \\
 {}\quad \quad \text{Print}[ \text{Polyset1}," ``\setminus \text{n} '' ", \text{ToString}[ \text{Length}[ 
\text{Polyset1}]] <> `` \text{polynomials}'', ``\setminus \text{n} '' ]; ]; \\
 {}\quad \text{Print}[`` \text{time} \text{spent}: '', 
\text{AbsoluteTime}[] - t1]; ) \\
 \\
\)
}

\normalsize

\newpage

\section*{Appendix IV: Some frequently appeared factors in mass relations} 


Here we show details of some frequently appeared 
long factors in mass relations, ordered according to their lengths. 
The subscript of $\mu$ is the number of terms it contains.  

\footnotesize{
\beq
&& \mu_{15}(m_i,m_j,m_k) \\
&\!=\!& 
m_i^6 m_j^6 -2 m_i^7 m_j^3 m_k^2 +m_i^6 m_j^4 m_k^2 +6 m_i^5 m_j^5 m_k^2 +m_i^4 m_j^6 m_k^2 
-2 m_i^3 m_j^7 m_k^2 +m_i^8 m_k^4   - m_i^7 m_j m_k^4 \\
&&   +m_i^6 m_j^2 m_k^4 +2 m_i^5 m_j^3 m_k^4 -2 m_i^4 m_j^4 m_k^4+2 m_i^3 m_j^5 m_k^4 
+m_i^2 m_j^6 m_k^4 
    -m_i m_j^7 m_k^4  +m_j^8 m_k^4, 
\eeq  } \vspace{-4mm}
\footnotesize{
\beq
&& \mu_{21}(m_i,m_j,m_k) \\
&\!\!=\!\!& -m_i^{11} m_j^2+m_i^{11} m_k^2+4 m_i^9 m_j^2 m_k^2+3 m_i^8 m_j^3 m_k^2+31 m_i^9 m_k^4 
+31 m_i^8 m_j m_k^4   - 6 m_i^7 m_j^2 m_k^4   \\
&&   -9 m_i^6 m_j^3 m_k^4 -3 m_i^5 m_j^4 m_k^4 +31 m_i^7 m_k^6 +62 m_i^6 m_j m_k^6 
+35 m_i^5 m_j^2 m_k^6   + 9 m_i^4 m_j^3 m_k^6 +6 m_i^3 m_j^4 m_k^6  \\  
&&   +m_i^2 m_j^5 m_k^6 +m_i^5 m_k^8 +3 m_i^4 m_j m_k^8 
+2 m_i^3 m_j^2 m_k^8      -2 m_i^2 m_j^3 m_k^8  -3 m_i m_j^4 m_k^8 -m_j^5 m_k^8,  
\eeq }  \vspace{-4mm}
\footnotesize{
\beq
&& \mu_{33}(m_i,m_j,m_k)\\
&\!\!=\!\!&  3 m_i^{12} m_j^4+9 m_i^{11} m_j^5-24 m_i^{12} m_j^2 m_k^2-30 m_i^{11} m_j^3  m_k^2 
-12 m_i^{10} m_j^4 m_k^2-30 m_i^9 m_j^5 m_k^2  -28 m_i^8 m_j^6 m_k^2 \\
&&  +21 m_i^{12} m_k^4  +21 m_i^{11} m_j m_k^4-360 m_i^{10} m_j^2 m_k^4
   -670 m_i^9 m_j^3 m_k^4  -292 m_i^8 m_j^4 m_k^4  +36 m_i^7 m_j^5 m_k^4  \\
&&   +47 m_i^6 m_j^6 m_k^4  +29 m_i^5 m_j^7 m_k^4-154 m_i^{10} m_k^6 -308 m_i^9 m_j m_k^6 
-514 m_i^8 m_j^2 m_k^6-770 m_i^7 m_j^3 m_k^6 \\
&& -472  m_i^6 m_j^4 m_k^6-68 m_i^5 m_j^5 m_k^6 
  -10 m_i^4 m_j^6 m_k^6-14 m_i^3 m_j^7 m_k^6- 10m_i^2 m_j^8 m_k^6+21 m_i^8 m_k^8 +63 m_i^7 m_j m_k^8  \\
&&   +39 m_i^6 m_j^2 m_k^8
  -45 m_i^5 m_j^3 m_k^8-51 m_i^4 m_j^4 m_k^8-3 m_i^3 m_j^5 m_k^8 -3 m_i^2 m_j^6 m_k^8  
   -15 m_i m_j^7 m_k^8-6 m_j^8 m_k^8.  
\eeq   }  \vspace{-4mm}
\footnotesize{
\beq
&& \mu_{45}(m_i,m_j,m_k,m_l)\\ 
&\!\!=\!\!&
-m_j^{11} m_k^2+m_j^{11} m_l^2+3 m_i m_j^8 m_k^2 m_l^2+4 m_j^9 m_k^2 m_l^2+3 m_j^8 m_k^3 m_l^2+31 m_i m_j^8 m_l^4 
+31 m_j^9 m_l^4  \\
&& +31 m_j^8 m_k m_l^4-3 m_i^2 m_j^5 m_k^2 m_l^4-9 m_i m_j^6 m_k^2 m_l^4-6 m_j^7 m_k^2 m_l^4-6 m_i m_j^5 m_k^3 m_l^4
-9 m_j^6 m_k^3 m_l^4 \\
&& -3 m_j^5 m_k^4 m_l^4+31 m_i^2 m_j^5 m_l^6+62 m_i m_j^6 m_l^6+31 m_j^7 m_l^6+62 m_i m_j^5 m_k m_l^6+62 m_j^6 m_k m_l^6 \\
&& +m_i^3 m_j^2 m_k^2 m_l^6+6 m_i^2 m_j^3 m_k^2 m_l^6+9 m_i m_j^4 m_k^2 m_l^6+35 m_j^5 m_k^2 m_l^6
+3 m_i^2 m_j^2 m_k^3 m_l^6  +12 m_i m_j^3 m_k^3 m_l^6 \\
&& +9 m_j^4 m_k^3 m_l^6+3 m_i m_j^2 m_k^4 m_l^6+6 m_j^3 m_k^4 m_l^6+m_j^2 m_k^5 m_l^6
+m_i^3 m_j^2 m_l^8   +3 m_i^2 m_j^3 m_l^8+3 m_i m_j^4 m_l^8 \\
&&+m_j^5 m_l^8+3 m_i^2 m_j^2 m_k m_l^8+6 m_i m_j^3 m_k m_l^8  
+3 m_j^4 m_k m_l^8 -m_i^3 m_k^2 m_l^8-3 m_i^2 m_j m_k^2 m_l^8  +2 m_j^3 m_k^2 m_l^8  \\
&& -3 m_i^2 m_k^3 m_l^8-6 m_i m_j m_k^3 m_l^8
-2 m_j^2 m_k^3 m_l^8 
-3 m_i m_k^4 m_l^8-3 m_j m_k^4 m_l^8-m_k^5 m_l^8.
\eeq   } \vspace{-4mm}
\footnotesize{
\beq
&& \mu_{84}(m_i,m_j,m_k,m_l,m_n)\\ 
&\!\!=\!\!&
m_i^3 m_k^{11} m_l^2 + 3 m_i^2 m_j m_k^{11} m_l^2 + 3 m_i m_j^2 m_k^{11} m_l^2 + 
 m_j^3 m_k^{11} m_l^2 - m_i^3 m_k^{11} m_n^2 - 3 m_i^2 m_j m_k^{11} m_n^2  \\
&&  -3 m_i m_j^2 m_k^{11} m_n^2 - m_j^3 m_k^{11} m_n^2 - 4 m_i^3 m_k^9 m_l^2 m_n^2 - 
 9 m_i^2 m_j m_k^9 m_l^2 m_n^2 - 6 m_i m_j^2 m_k^9 m_l^2 m_n^2   \\
&& - m_j^3 m_k^9 m_l^2 m_n^2 - 3 m_i^3 m_k^8 m_l^3 m_n^2 - 
 6 m_i^2 m_j m_k^8 m_l^3 m_n^2 - 3 m_i m_j^2 m_k^8 m_l^3 m_n^2 +  3 m_i^2 m_j m_k^8 m_l^2 m_n^3  \\
&& + 6 m_i m_j^2 m_k^8 m_l^2 m_n^3 + 
 3 m_j^3 m_k^8 m_l^2 m_n^3 - 31 m_i^3 m_k^9 m_n^4 - 62 m_i^2 m_j m_k^9 m_n^4  -  31 m_i m_j^2 m_k^9 m_n^4 \\
&& - 31 m_i^3 m_k^8 m_l m_n^4 -  62 m_i^2 m_j m_k^8 m_l m_n^4 - 31 m_i m_j^2 m_k^8 m_l m_n^4 +  6 m_i^3 m_k^7 m_l^2 m_n^4 
 + 9 m_i^2 m_j m_k^7 m_l^2 m_n^4  \\
&&  +  3 m_i m_j^2 m_k^7 m_l^2 m_n^4 + 9 m_i^3 m_k^6 m_l^3 m_n^4 + 
 12 m_i^2 m_j m_k^6 m_l^3 m_n^4 
 + 3 m_i m_j^2 m_k^6 m_l^3 m_n^4 +  3 m_i^3 m_k^5 m_l^4 m_n^4 \\
 && + 3 m_i^2 m_j m_k^5 m_l^4 m_n^4 + 
 31 m_i^2 m_j m_k^8 m_n^5 + 62 m_i m_j^2 m_k^8 m_n^5 
  + 31 m_j^3 m_k^8 m_n^5 -  9 m_i^2 m_j m_k^6 m_l^2 m_n^5 \\
  &&- 12 m_i m_j^2 m_k^6 m_l^2 m_n^5 - 
 3 m_j^3 m_k^6 m_l^2 m_n^5 - 6 m_i^2 m_j m_k^5 m_l^3 m_n^5 
 -  6 m_i m_j^2 m_k^5 m_l^3 m_n^5 - 31 m_i^3 m_k^7 m_n^6 \\
 && -  31 m_i^2 m_j m_k^7 m_n^6 - 62 m_i^3 m_k^6 m_l m_n^6 - 
 62 m_i^2 m_j m_k^6 m_l m_n^6  
  - 35 m_i^3 m_k^5 m_l^2 m_n^6 -  34 m_i^2 m_j m_k^5 m_l^2 m_n^6  \\
&&  + 3 m_i m_j^2 m_k^5 m_l^2 m_n^6 + 
 3 m_j^3 m_k^5 m_l^2 m_n^6 - 9 m_i^3 m_k^4 m_l^3 m_n^6  
 - 6 m_i^2 m_j m_k^4 m_l^3 m_n^6 - 6 m_i^3 m_k^3 m_l^4 m_n^6 \\
&& - 3 m_i^2 m_j m_k^3 m_l^4 m_n^6 - m_i^3 m_k^2 m_l^5 m_n^6 +  62 m_i^2 m_j m_k^6 m_n^7 
 + 62 m_i m_j^2 m_k^6 m_n^7 +  62 m_i^2 m_j m_k^5 m_l m_n^7 
   \eeq
 \beq
 && + 62 m_i m_j^2 m_k^5 m_l m_n^7 +  9 m_i^2 m_j m_k^4 m_l^2 m_n^7    
 + 6 m_i m_j^2 m_k^4 m_l^2 m_n^7 +  12 m_i^2 m_j m_k^3 m_l^3 m_n^7 \\
 && + 6 m_i m_j^2 m_k^3 m_l^3 m_n^7 +  3 m_i^2 m_j m_k^2 m_l^4 m_n^7 - m_i^3 m_k^5 m_n^8 
 - 31 m_i m_j^2 m_k^5 m_n^8 -  31 m_j^3 m_k^5 m_n^8 \\
 && - 3 m_i^3 m_k^4 m_l m_n^8 - 2 m_i^3 m_k^3 m_l^2 m_n^8 -  6 m_i m_j^2 m_k^3 m_l^2 m_n^8 
  - 3 m_j^3 m_k^3 m_l^2 m_n^8 +  2 m_i^3 m_k^2 m_l^3 m_n^8 \\
  && - 3 m_i m_j^2 m_k^2 m_l^3 m_n^8 + 
 3 m_i^3 m_k m_l^4 m_n^8 + m_i^3 m_l^5 m_n^8  
  + 3 m_i^2 m_j m_k^4 m_n^9 +  6 m_i^2 m_j m_k^3 m_l m_n^9 \\
  &&  + m_j^3 m_k^2 m_l^2 m_n^9 - 
 6 m_i^2 m_j m_k m_l^3 m_n^9 - 3 m_i^2 m_j m_l^4 m_n^9 
  -   3 m_i m_j^2 m_k^3 m_n^{10} - 3 m_i m_j^2 m_k^2 m_l m_n^{10}  \\
 && + 3 m_i m_j^2 m_k m_l^2 m_n^{10} + 3 m_i m_j^2 m_l^3 m_n^{10} + m_j^3 m_k^2 m_n^{11} 
  - m_j^3 m_l^2 m_n^{11}.
\eeq }  \vspace{-4mm}
\footnotesize{
\beq
&& \mu_{88}(m_i,m_j,m_k,m_l)\\ 
&\!=\!&
9 m_i m_j^{11} m_k^4+3 m_j^{12} m_k^4+9 m_j^{11} m_k^5-30 m_i m_j^{11} m_k^2 m_l^2-24 m_j^{12} m_k^2 m_l^2-30 m_j^{11} m_k^3 m_l^2
\\&&-28 m_i^2 m_j^8 m_k^4 m_l^2-30 m_i m_j^9 m_k^4 m_l^2-12 m_j^{10} m_k^4 m_l^2-56 m_i m_j^8 m_k^5 m_l^2-30 m_j^9 m_k^5 m_l^2
\\&&-28 m_j^8 m_k^6 m_l^2+21 m_i m_j^{11} m_l^4+21 m_j^{12} m_l^4+21 m_j^{11} m_k m_l^4-310 m_i^2 m_j^8 m_k^2 m_l^4
-670 m_i m_j^9 m_k^2 m_l^4
\\&&-360 m_j^{10} m_k^2 m_l^4-620 m_i m_j^8 m_k^3 m_l^4-670 m_j^9 m_k^3 m_l^4+29 m_i^3 m_j^5 m_k^4 m_l^4+47 m_i^2 m_j^6 m_k^4 m_l^4
\\&&+36 m_i m_j^7 m_k^4 m_l^4-292 m_j^8 m_k^4 m_l^4+87 m_i^2 m_j^5 m_k^5 m_l^4+94 m_i m_j^6 m_k^5 m_l^4+36 m_j^7 m_k^5 m_l^4
\\&&+87 m_i m_j^5 m_k^6 m_l^4+47 m_j^6 m_k^6 m_l^4+29 m_j^5 m_k^7 m_l^4-154 m_i^2 m_j^8 m_l^6-308 m_i m_j^9 m_l^6-154 m_j^{10} m_l^6
\\&&-308 m_i m_j^8 m_k m_l^6-308 m_j^9 m_k m_l^6-50 m_i^3 m_j^5 m_k^2 m_l^6-460 m_i^2 m_j^6 m_k^2 m_l^6-770 m_i m_j^7 m_k^2 m_l^6
\\&&-514 m_j^8 m_k^2 m_l^6-150 m_i^2 m_j^5 m_k^3 m_l^6-920 m_i m_j^6 m_k^3 m_l^6-770 m_j^7 m_k^3 m_l^6-10 m_i^4 m_j^2 m_k^4 m_l^6
\\&&-14 m_i^3 m_j^3 m_k^4 m_l^6-10 m_i^2 m_j^4 m_k^4 m_l^6-168 m_i m_j^5 m_k^4 m_l^6-472 m_j^6 m_k^4 m_l^6-40 m_i^3 m_j^2 m_k^5 m_l^6
\\&&-42 m_i^2 m_j^3 m_k^5 m_l^6-20 m_i m_j^4 m_k^5 m_l^6-68 m_j^5 m_k^5 m_l^6-60 m_i^2 m_j^2 m_k^6 m_l^6-42 m_i m_j^3 m_k^6 m_l^6
\\&&-10 m_j^4 m_k^6 m_l^6-40 m_i m_j^2 m_k^7 m_l^6-14 m_j^3 m_k^7 m_l^6-10 m_j^2 m_k^8 m_l^6+21 m_i^3 m_j^5 m_l^8+63 m_i^2 m_j^6 m_l^8
\\&&+63 m_i m_j^7 m_l^8+21 m_j^8 m_l^8+63 m_i^2 m_j^5 m_k m_l^8+126 m_i m_j^6 m_k m_l^8+63 m_j^7 m_k m_l^8+6 m_i^4 m_j^2 m_k^2 m_l^8
\\&&-6 m_i^3 m_j^3 m_k^2 m_l^8-54 m_i^2 m_j^4 m_k^2 m_l^8-3 m_i m_j^5 m_k^2 m_l^8+39 m_j^6 m_k^2 m_l^8+24 m_i^3 m_j^2 m_k^3 m_l^8
\\&&-18 m_i^2 m_j^3 m_k^3 m_l^8-108 m_i m_j^4 m_k^3 m_l^8-45 m_j^5 m_k^3 m_l^8-6 m_i^4 m_k^4 m_l^8-15 m_i^3 m_j m_k^4 m_l^8
\\&&+27 m_i^2 m_j^2 m_k^4 m_l^8-15 m_i m_j^3 m_k^4 m_l^8-51 m_j^4 m_k^4 m_l^8-24 m_i^3 m_k^5 m_l^8-45 m_i^2 m_j m_k^5 m_l^8
\\&&+6 m_i m_j^2 m_k^5 m_l^8-3 m_j^3 m_k^5 m_l^8-36 m_i^2 m_k^6 m_l^8-45 m_i m_j m_k^6 m_l^8-3 m_j^2 m_k^6 m_l^8-24 m_i m_k^7 m_l^8
\\&&-15 m_j m_k^7 m_l^8-6 m_k^8 m_l^8.
\eeq }  \vspace{-2mm}
\footnotesize{
\beq
&& \mu_{111}(m_i,m_j,m_k,m_l)\\ 
&\!=\!&
-m_i^3 m_j^9 m_k^6-3 m_i^2 m_j^9 m_k^7-3 m_i m_j^9 m_k^8-m_j^9 m_k^9+3 m_i^3 m_j^9 m_k^4 m_l^2+3 m_i^2 m_j^{10} m_k^4 m_l^2
\\&&+9 m_i^2 m_j^9 m_k^5 m_l^2+6 m_i m_j^{10} m_k^5 m_l^2+3 m_i^4 m_j^6 m_k^6 m_l^2+3 m_i^3 m_j^7 m_k^6 m_l^2+9 m_i m_j^9 m_k^6 m_l^2
\\&&+3 m_j^{10} m_k^6 m_l^2+12 m_i^3 m_j^6 m_k^7 m_l^2+9 m_i^2 m_j^7 m_k^7 m_l^2+3 m_j^9 m_k^7 m_l^2+18 m_i^2 m_j^6 m_k^8 m_l^2
\\&&+9 m_i m_j^7 m_k^8 m_l^2+12 m_i m_j^6 m_k^9 m_l^2+3 m_j^7 m_k^9 m_l^2+3 m_j^6 m_k^{10} m_l^2-3 m_i^3 m_j^9 m_k^2 m_l^4
\\&&-6 m_i^2 m_j^{10} m_k^2 m_l^4-3 m_i m_j^{11} m_k^2 m_l^4-9 m_i^2 m_j^9 m_k^3 m_l^4-12 m_i m_j^{10} m_k^3 m_l^4-3 m_j^{11} m_k^3 m_l^4
\\&&+21 m_i^4 m_j^6 m_k^4 m_l^4+42 m_i^3 m_j^7 m_k^4 m_l^4+21 m_i^2 m_j^8 m_k^4 m_l^4-9 m_i m_j^9 m_k^4 m_l^4-6 m_j^{10} m_k^4 m_l^4
\\&&+84 m_i^3 m_j^6 m_k^5 m_l^4+126 m_i^2 m_j^7 m_k^5 m_l^4+42 m_i m_j^8 m_k^5 m_l^4-3 m_j^9 m_k^5 m_l^4-3 m_i^5 m_j^3 m_k^6 m_l^4
\\&&-6 m_i^4 m_j^4 m_k^6 m_l^4-3 m_i^3 m_j^5 m_k^6 m_l^4+126 m_i^2 m_j^6 m_k^6 m_l^4+126 m_i m_j^7 m_k^6 m_l^4+21 m_j^8 m_k^6 m_l^4
\\&&-15 m_i^4 m_j^3 m_k^7 m_l^4-24 m_i^3 m_j^4 m_k^7 m_l^4-9 m_i^2 m_j^5 m_k^7 m_l^4+84 m_i m_j^6 m_k^7 m_l^4+42 m_j^7 m_k^7 m_l^4
\\&&-30 m_i^3 m_j^3 m_k^8 m_l^4-36 m_i^2 m_j^4 m_k^8 m_l^4-9 m_i m_j^5 m_k^8 m_l^4+21 m_j^6 m_k^8 m_l^4-30 m_i^2 m_j^3 m_k^9 m_l^4 \\
&&-24 m_i m_j^4 m_k^9 m_l^4-3 m_j^5 m_k^9 m_l^4-15 m_i m_j^3 m_k^{10} m_l^4-6 m_j^4 m_k^{10} m_l^4-3 m_j^3 m_k^{11} m_l^4
+m_i^3 m_j^9 m_l^6
\\&&+3 m_i^2 m_j^{10} m_l^6+3 m_i m_j^{11} m_l^6+m_j^{12} m_l^6+3 m_i^2 m_j^9 m_k m_l^6+6 m_i m_j^{10} m_k m_l^6+3 m_j^{11} m_k m_l^6
\eeq
\beq
&&+3 m_i^4 m_j^6 m_k^2 m_l^6+9 m_i^3 m_j^7 m_k^2 m_l^6+9 m_i^2 m_j^8 m_k^2 m_l^6+6 m_i m_j^9 m_k^2 m_l^6
+3 m_j^{10} m_k^2 m_l^6
\\&&+12 m_i^3 m_j^6 m_k^3 m_l^6+27 m_i^2 m_j^7 m_k^3 m_l^6+18 m_i m_j^8 m_k^3 m_l^6+4 m_j^9 m_k^3 m_l^6
+3 m_i^5 m_j^3 m_k^4 m_l^6 \\
&&+9 m_i^4 m_j^4 m_k^4 m_l^6+9 m_i^3 m_j^5 m_k^4 m_l^6+21 m_i^2 m_j^6 m_k^4 m_l^6+27 m_i m_j^7 m_k^4 m_l^6
+9 m_j^8 m_k^4 m_l^6
\\&&+15 m_i^4 m_j^3 m_k^5 m_l^6+36 m_i^3 m_j^4 m_k^5 m_l^6+27 m_i^2 m_j^5 m_k^5 m_l^6
+18 m_i m_j^6 m_k^5 m_l^6+9 m_j^7 m_k^5 m_l^6
\\&&+m_i^6 m_k^6 m_l^6+3 m_i^5 m_j m_k^6 m_l^6+3 m_i^4 m_j^2 m_k^6 m_l^6+31 m_i^3 m_j^3 m_k^6 m_l^6 
+54 m_i^2 m_j^4 m_k^6 m_l^6  
\\&&+27 m_i m_j^5 m_k^6 m_l^6 +6 m_j^6 m_k^6 m_l^6+6 m_i^5 m_k^7 m_l^6+15 m_i^4 m_j m_k^7 m_l^6+12 m_i^3 m_j^2 m_k^7 m_l^6
\\&&+33 m_i^2 m_j^3 m_k^7 m_l^6+36 m_i m_j^4 m_k^7 m_l^6+9 m_j^5 m_k^7 m_l^6+15 m_i^4 m_k^8 m_l^6+30 m_i^3 m_j m_k^8 m_l^6
\\&&+18 m_i^2 m_j^2 m_k^8 m_l^6+18 m_i m_j^3 m_k^8 m_l^6+9 m_j^4 m_k^8 m_l^6+20 m_i^3 m_k^9 m_l^6+30 m_i^2 m_j m_k^9 m_l^6
\\&&+12 m_i m_j^2 m_k^9 m_l^6+4 m_j^3 m_k^9 m_l^6+15 m_i^2 m_k^{10} m_l^6+15 m_i m_j m_k^{10} m_l^6+3 m_j^2 m_k^{10} m_l^6
\\&&+6 m_i m_k^{11} m_l^6+3 m_j m_k^{11} m_l^6+m_k^{12} m_l^6.
\eeq  } \vspace{-2mm}
\footnotesize{
\beq
&& \mu_{221}(m_i,m_j,m_k,m_l,m_n)\\ 
&\!=\!&
3 m_i^4 m_k^{12} m_l^4 + 3 m_i^3 m_j m_k^{12} m_l^4 - 9 m_i^2 m_j^2 m_k^{12} m_l^4 -  15 m_i m_j^3 m_k^{12} m_l^4 - 6 m_j^4 m_k^{12} m_l^4 
+ 9 m_i^4 m_k^{11} m_l^5 
\\&&+  27 m_i^3 m_j m_k^{11} m_l^5 + 27 m_i^2 m_j^2 m_k^{11} m_l^5 +  9 m_i m_j^3 m_k^{11} m_l^5 - 9 m_i^3 m_j m_k^{11} m_l^4 m_n 
\! -  27 m_i^2 m_j^2 m_k^{11} m_l^4 m_n 
 \\&& - 27 m_i m_j^3 m_k^{11} m_l^4 m_n  -  9 m_j^4 m_k^{11} m_l^4 m_n - 24 m_i^4 m_k^{12} m_l^2 m_n^2 -  66 m_i^3 m_j m_k^{12} m_l^2 m_n^2 
\\&& - 54 m_i^2 m_j^2 m_k^{12} m_l^2 m_n^2  -  6 m_i m_j^3 m_k^{12} m_l^2 m_n^2 + 6 m_j^4 m_k^{12} m_l^2 m_n^2 -  30 m_i^4 m_k^{11} m_l^3 m_n^2
\\&& - 90 m_i^3 m_j m_k^{11} m_l^3 m_n^2 -  90 m_i^2 m_j^2 m_k^{11} m_l^3 m_n^2 - 30 m_i m_j^3 m_k^{11} m_l^3 m_n^2 
-  12 m_i^4 m_k^{10} m_l^4 m_n^2 
\\&&- 18 m_i^3 m_j m_k^{10} m_l^4 m_n^2 -  10 m_i^2 m_j^2 m_k^{10} m_l^4 m_n^2 - 14 m_i m_j^3 m_k^{10} m_l^4 m_n^2 
-  10 m_j^4 m_k^{10} m_l^4 m_n^2 
\\&& - 30 m_i^4 m_k^9 m_l^5 m_n^2 - 34 m_i^3 m_j m_k^9 m_l^5 m_n^2 + 22 m_i^2 m_j^2 m_k^9 m_l^5 m_n^2 +  26 m_i m_j^3 m_k^9 m_l^5 m_n^2 
\\&&- 28 m_i^4 m_k^8 m_l^6 m_n^2 -  56 m_i^3 m_j m_k^8 m_l^6 m_n^2 - 28 m_i^2 m_j^2 m_k^8 m_l^6 m_n^2 +  30 m_i^3 m_j m_k^{11} m_l^2 m_n^3 
\\&&+ 90 m_i^2 m_j^2 m_k^{11} m_l^2 m_n^3 +  90 m_i m_j^3 m_k^{11} m_l^2 m_n^3 + 30 m_j^4 m_k^{11} m_l^2 m_n^3 
+  30 m_i^3 m_j m_k^9 m_l^4 m_n^3 
\\ 
&&+ 34 m_i^2 m_j^2 m_k^9 m_l^4 m_n^3 -  22 m_i m_j^3 m_k^9 m_l^4 m_n^3 - 26 m_j^4 m_k^9 m_l^4 m_n^3 +  56 m_i^3 m_j m_k^8 m_l^5 m_n^3 
\\&&+ 112 m_i^2 m_j^2 m_k^8 m_l^5 m_n^3 +  56 m_i m_j^3 m_k^8 m_l^5 m_n^3 + 21 m_i^4 m_k^{12} m_n^4 +  63 m_i^3 m_j m_k^{12} m_n^4 
+ 63 m_i^2 m_j^2 m_k^{12} m_n^4 
\\&&+  21 m_i m_j^3 m_k^{12} m_n^4 + 21 m_i^4 m_k^{11} m_l m_n^4 +  63 m_i^3 m_j m_k^{11} m_l m_n^4 + 63 m_i^2 m_j^2 m_k^{11} m_l m_n^4 
\\&&+  21 m_i m_j^3 m_k^{11} m_l m_n^4 - 360 m_i^4 m_k^{10} m_l^2 m_n^4 -  770 m_i^3 m_j m_k^{10} m_l^2 m_n^4 
- 460 m_i^2 m_j^2 m_k^{10} m_l^2 m_n^4 
\\&&-  50 m_i m_j^3 m_k^{10} m_l^2 m_n^4 - 670 m_i^4 m_k^9 m_l^3 m_n^4 -  1390 m_i^3 m_j m_k^9 m_l^3 m_n^4 
- 770 m_i^2 m_j^2 m_k^9 m_l^3 m_n^4 
\\&&-  50 m_i m_j^3 m_k^9 m_l^3 m_n^4 - 292 m_i^4 m_k^8 m_l^4 m_n^4 -  584 m_i^3 m_j m_k^8 m_l^4 m_n^4 - 291 m_i^2 m_j^2 m_k^8 m_l^4 m_n^4 
\\&&-  27 m_i m_j^3 m_k^8 m_l^4 m_n^4 \! - 28 m_j^4 m_k^8 m_l^4 m_n^4 +  36 m_i^4 m_k^7 m_l^5 m_n^4 +\! 14 m_i^3 m_j m_k^7 m_l^5 m_n^4 
+\!  7 m_i^2 m_j^2 m_k^7 m_l^5 m_n^4 
\\&&+ 29 m_i m_j^3 m_k^7 m_l^5 m_n^4 +  47 m_i^4 m_k^6 m_l^6 m_n^4 + 7 m_i^3 m_j m_k^6 m_l^6 m_n^4 -  40 m_i^2 m_j^2 m_k^6 m_l^6 m_n^4 
\\&&+ 29 m_i^4 m_k^5 m_l^7 m_n^4 +  29 m_i^3 m_j m_k^5 m_l^7 m_n^4 - 21 m_i^3 m_j m_k^{11} m_n^5 -  63 m_i^2 m_j^2 m_k^{11} m_n^5 
- 63 m_i m_j^3 m_k^{11} m_n^5 
\\&&-  21 m_j^4 m_k^{11} m_n^5 + 670 m_i^3 m_j m_k^9 m_l^2 m_n^5 +  1390 m_i^2 m_j^2 m_k^9 m_l^2 m_n^5 + 770 m_i m_j^3 m_k^9 m_l^2 m_n^5 
\\&&+  50 m_j^4 m_k^9 m_l^2 m_n^5 + 620 m_i^3 m_j m_k^8 m_l^3 m_n^5 +  1240 m_i^2 m_j^2 m_k^8 m_l^3 m_n^5 
+ 620 m_i m_j^3 m_k^8 m_l^3 m_n^5 
\\&& -  36 m_i^3 m_j m_k^7 m_l^4 m_n^5 - 14 m_i^2 m_j^2 m_k^7 m_l^4 m_n^5 -  7 m_i m_j^3 m_k^7 m_l^4 m_n^5 - 29 m_j^4 m_k^7 m_l^4 m_n^5 
\\&&-  94 m_i^3 m_j m_k^6 m_l^5 m_n^5 - 14 m_i^2 m_j^2 m_k^6 m_l^5 m_n^5 +  80 m_i m_j^3 m_k^6 m_l^5 m_n^5 - 87 m_i^3 m_j m_k^5 m_l^6 m_n^5 \\
&& -  87 m_i^2 m_j^2 m_k^5 m_l^6 m_n^5 - 154 m_i^4 m_k^{10} m_n^6 -  308 m_i^3 m_j m_k^{10} m_n^6 - 154 m_i^2 m_j^2 m_k^{10} m_n^6 
-  308 m_i^4 m_k^9 m_l m_n^6 \\
&&- 616 m_i^3 m_j m_k^9 m_l m_n^6 -  308 m_i^2 m_j^2 m_k^9 m_l m_n^6 - 514 m_i^4 m_k^8 m_l^2 m_n^6 -  978 m_i^3 m_j m_k^8 m_l^2 m_n^6
\\&& - 774 m_i^2 m_j^2 m_k^8 m_l^2 m_n^6 -  620 m_i m_j^3 m_k^8 m_l^2 m_n^6 - 310 m_j^4 m_k^8 m_l^2 m_n^6 -  770 m_i^4 m_k^7 m_l^3 m_n^6 \\
&& - 1390 m_i^3 m_j m_k^7 m_l^3 m_n^6 -  620 m_i^2 m_j^2 m_k^7 m_l^3 m_n^6 - 472 m_i^4 m_k^6 m_l^4 m_n^6 
-  800 m_i^3 m_j m_k^6 m_l^4 m_n^6 
\eeq
\beq
&&- 291 m_i^2 m_j^2 m_k^6 m_l^4 m_n^6 +  7 m_i m_j^3 m_k^6 m_l^4 m_n^6 - 40 m_j^4 m_k^6 m_l^4 m_n^6 -  68 m_i^4 m_k^5 m_l^5 m_n^6 \\
&&- 84 m_i^3 m_j m_k^5 m_l^5 m_n^6 +  31 m_i^2 m_j^2 m_k^5 m_l^5 m_n^6 + 87 m_i m_j^3 m_k^5 m_l^5 m_n^6 -  10 m_i^4 m_k^4 m_l^6 m_n^6 
\\&&+ 22 m_i^3 m_j m_k^4 m_l^6 m_n^6 -  28 m_i^2 m_j^2 m_k^4 m_l^6 m_n^6 - 14 m_i^4 m_k^3 m_l^7 m_n^6 +  26 m_i^3 m_j m_k^3 m_l^7 m_n^6\\
&&- 10 m_i^4 m_k^2 m_l^8 m_n^6 +  308 m_i^3 m_j m_k^9 m_n^7 + 616 m_i^2 m_j^2 m_k^9 m_n^7 +  308 m_i m_j^3 m_k^9 m_n^7 
+ 308 m_i^3 m_j m_k^8 m_l m_n^7 
\\&&+  616 m_i^2 m_j^2 m_k^8 m_l m_n^7 + 308 m_i m_j^3 m_k^8 m_l m_n^7 +  770 m_i^3 m_j m_k^7 m_l^2 m_n^7 
+ 1390 m_i^2 m_j^2 m_k^7 m_l^2 m_n^7 
\\&&+  620 m_i m_j^3 m_k^7 m_l^2 m_n^7 + 920 m_i^3 m_j m_k^6 m_l^3 m_n^7 +  1540 m_i^2 m_j^2 m_k^6 m_l^3 m_n^7 
+ 620 m_i m_j^3 m_k^6 m_l^3 m_n^7 
\\&& +  168 m_i^3 m_j m_k^5 m_l^4 m_n^7 + 184 m_i^2 m_j^2 m_k^5 m_l^4 m_n^7 +  27 m_i m_j^3 m_k^5 m_l^4 m_n^7
 - 29 m_j^4 m_k^5 m_l^4 m_n^7 
\\&& + 20 m_i^3 m_j m_k^4 m_l^5 m_n^7 - 44 m_i^2 m_j^2 m_k^4 m_l^5 m_n^7 +  56 m_i m_j^3 m_k^4 m_l^5 m_n^7 
+ 42 m_i^3 m_j m_k^3 m_l^6 m_n^7 
\\&& -  78 m_i^2 m_j^2 m_k^3 m_l^6 m_n^7 + 40 m_i^3 m_j m_k^2 m_l^7 m_n^7 +  21 m_i^4 m_k^8 m_n^8 + 21 m_i^3 m_j m_k^8 m_n^8 
- 154 m_i^2 m_j^2 m_k^8 m_n^8 
\\&& -  308 m_i m_j^3 m_k^8 m_n^8 - 154 m_j^4 m_k^8 m_n^8 + 63 m_i^4 m_k^7 m_l m_n^8 +  63 m_i^3 m_j m_k^7 m_l m_n^8 
+ 39 m_i^4 m_k^6 m_l^2 m_n^8 
\\&& +  33 m_i^3 m_j m_k^6 m_l^2 m_n^8 - 460 m_i^2 m_j^2 m_k^6 m_l^2 m_n^8 -  770 m_i m_j^3 m_k^6 m_l^2 m_n^8 - 310 m_j^4 m_k^6 m_l^2 m_n^8 
\\&&-  45 m_i^4 m_k^5 m_l^3 m_n^8 - 69 m_i^3 m_j m_k^5 m_l^3 m_n^8 -  150 m_i^2 m_j^2 m_k^5 m_l^3 m_n^8 - 150 m_i m_j^3 m_k^5 m_l^3 m_n^8 
\\&& -  51 m_i^4 m_k^4 m_l^4 m_n^8 - 81 m_i^3 m_j m_k^4 m_l^4 m_n^8 -  10 m_i^2 m_j^2 m_k^4 m_l^4 m_n^8 + 22 m_i m_j^3 m_k^4 m_l^4 m_n^8 
\\&&-  28 m_j^4 m_k^4 m_l^4 m_n^8 - 3 m_i^4 m_k^3 m_l^5 m_n^8 -  3 m_i^3 m_j m_k^3 m_l^5 m_n^8 - 42 m_i^2 m_j^2 m_k^3 m_l^5 m_n^8 + 
 78 m_i m_j^3 m_k^3 m_l^5 m_n^8
 \\&& - 3 m_i^4 m_k^2 m_l^6 m_n^8 +  27 m_i^3 m_j m_k^2 m_l^6 m_n^8 - 60 m_i^2 m_j^2 m_k^2 m_l^6 m_n^8 -  15 m_i^4 m_k m_l^7 m_n^8 
 + 9 m_i^3 m_j m_k m_l^7 m_n^8 
 \\&& - 6 m_i^4 m_l^8 m_n^8 -  63 m_i^3 m_j m_k^7 m_n^9 - 63 m_i^2 m_j^2 m_k^7 m_n^9 -  126 m_i^3 m_j m_k^6 m_l m_n^9
  - 126 m_i^2 m_j^2 m_k^6 m_l m_n^9 
  \\&& + 3 m_i^3 m_j m_k^5 m_l^2 m_n^9 + 27 m_i^2 m_j^2 m_k^5 m_l^2 m_n^9 +  50 m_i m_j^3 m_k^5 m_l^2 m_n^9
   + 50 m_j^4 m_k^5 m_l^2 m_n^9 
 \\&& +  108 m_i^3 m_j m_k^4 m_l^3 m_n^9 + 180 m_i^2 m_j^2 m_k^4 m_l^3 m_n^9 +  15 m_i^3 m_j m_k^3 m_l^4 m_n^9 
  + 63 m_i^2 m_j^2 m_k^3 m_l^4 m_n^9
\\ && +  14 m_i m_j^3 m_k^3 m_l^4 m_n^9 - 26 m_j^4 m_k^3 m_l^4 m_n^9 -  6 m_i^3 m_j m_k^2 m_l^5 m_n^9 
  - 54 m_i^2 m_j^2 m_k^2 m_l^5 m_n^9 
\\&& + 40 m_i m_j^3 m_k^2 m_l^5 m_n^9 + 45 m_i^3 m_j m_k m_l^6 m_n^9 \! -  27 m_i^2 m_j^2 m_k m_l^6 m_n^9 + 24 m_i^3 m_j m_l^7 m_n^9 + 
 63 m_i^2 m_j^2 m_k^6 m_n^{10} 
\\&& + 63 m_i m_j^3 m_k^6 m_n^{10} +  63 m_i^2 m_j^2 m_k^5 m_l m_n^{10} + 63 m_i m_j^3 m_k^5 m_l m_n^{10}
  -  54 m_i^2 m_j^2 m_k^4 m_l^2 m_n^{10} 
\\&&- 90 m_i m_j^3 m_k^4 m_l^2 m_n^{10} -  18 m_i^2 m_j^2 m_k^3 m_l^3 m_n^{10} - 90 m_i m_j^3 m_k^3 m_l^3 m_n^{10} 
  +  27 m_i^2 m_j^2 m_k^2 m_l^4 m_n^{10}  
\\&&+ 27 m_i m_j^3 m_k^2 m_l^4 m_n^{10} -  10 m_j^4 m_k^2 m_l^4 m_n^{10} - 45 m_i^2 m_j^2 m_k m_l^5 m_n^{10} 
  +  27 m_i m_j^3 m_k m_l^5 m_n^{10}    - 36 m_i^2 m_j^2 m_l^6 m_n^{10} 
\\&&  -  21 m_i m_j^3 m_k^5 m_n^{11} - 21 m_j^4 m_k^5 m_n^{11} +  6 m_i m_j^3 m_k^3 m_l^2 m_n^{11} 
  + 30 m_j^4 m_k^3 m_l^2 m_n^{11}    - 24 m_i m_j^3 m_k^2 m_l^3 m_n^{11} 
\\&& + 15 m_i m_j^3 m_k m_l^4 m_n^{11} -  9 m_j^4 m_k m_l^4 m_n^{11} + 24 m_i m_j^3 m_l^5 m_n^{11} 
 +  6 m_j^4 m_k^2 m_l^2 m_n^{12} - 6 m_j^4 m_l^4 m_l^{12}
\eeq  } \vspace{-2mm}
\footnotesize{
\beq
&& \mu_{374}(m_i,m_j,m_k,m_l,m_n)\\ 
&\!=\!&
m_i^3 m_j^3 m_k^{12} m_l^6 + 3 m_i^2 m_j^4 m_k^{12} m_l^6 + 3 m_i m_j^5 m_k^{12} m_l^6 +
  m_j^6 m_k^{12} m_l^6 - 3 m_i^4 m_j^2 m_k^{11} m_l^7 - 9 m_i^3 m_j^3 m_k^{11} m_l^7\\&& - 
 9 m_i^2 m_j^4 m_k^{11} m_l^7 - 3 m_i m_j^5 m_k^{11} m_l^7 + 
 3 m_i^5 m_j m_k^{10} m_l^8 + 9 m_i^4 m_j^2 m_k^{10} m_l^8 + 
 9 m_i^3 m_j^3 m_k^{10} m_l^8 \\
 &&+ 3 m_i^2 m_j^4 m_k^{10} m_l^8 \! - m_i^6 m_k^9 m_l^9 \! - 
 3 m_i^5 m_j m_k^9 m_l^9 \! - 3 m_i^4 m_j^2 m_k^9 m_l^9 \! - m_i^3 m_j^3 m_k^9 m_l^9 + 
 3 m_i^3 m_j^3 m_k^{11} m_l^6 m_n \\
 &&+ 9 m_i^2 m_j^4 m_k^{11} m_l^6 m_n + 
 9 m_i m_j^5 m_k^{11} m_l^6 m_n + 3 m_j^6 m_k^{11} m_l^6 m_n - 
 6 m_i^4 m_j^2 m_k^{10} m_l^7 m_n  \\
 &&- 18 m_i^3 m_j^3 m_k^{10} m_l^7 m_n- 
 18 m_i^2 m_j^4 m_k^{10} m_l^7 m_n - 6 m_i m_j^5 m_k^{10} m_l^7 m_n + 
 3 m_i^5 m_j m_k^9 m_l^8 m_n  \\
 &&+ 9 m_i^4 m_j^2 m_k^9 m_l^8 m_n + 
 9 m_i^3 m_j^3 m_k^9 m_l^8 m_n + 3 m_i^2 m_j^4 m_k^9 m_l^8 m_n + 
 3 m_i^4 m_j^2 m_k^{12} m_l^4 m_n^2  \\
 &&+ 9 m_i^3 m_j^3 m_k^{12} m_l^4 m_n^2 + 
 9 m_i^2 m_j^4 m_k^{12} m_l^4 m_n^2 + 3 m_i m_j^5 m_k^{12} m_l^4 m_n^2 - 
 6 m_i^5 m_j m_k^{11} m_l^5 m_n^2  \\
 &&- 15 m_i^4 m_j^2 m_k^{11} m_l^5 m_n^2 - 
 9 m_i^3 m_j^3 m_k^{11} m_l^5 m_n^2 + 3 m_i^2 m_j^4 m_k^{11} m_l^5 m_n^2 + 
 3 m_i m_j^5 m_k^{11} m_l^5 m_n^2  \\
 &&+ 3 m_i^6 m_k^{10} m_l^6 m_n^2 + 
 3 m_i^5 m_j m_k^{10} m_l^6 m_n^2 - 9 m_i^4 m_j^2 m_k^{10} m_l^6 m_n^2 - 
 15 m_i^3 m_j^3 m_k^{10} m_l^6 m_n^2  \\
 &&- 3 m_i^2 m_j^4 m_k^{10} m_l^6 m_n^2 + 
 6 m_i m_j^5 m_k^{10} m_l^6 m_n^2 + 3 m_j^6 m_k^{10} m_l^6 m_n^2 + 
 3 m_i^6 m_k^9 m_l^7 m_n^2 + 9 m_i^5 m_j m_k^9 m_l^7 m_n^2 \\
 &&+ 
 15 m_i^4 m_j^2 m_k^9 m_l^7 m_n^2 + 9 m_i^3 m_j^3 m_k^9 m_l^7 m_n^2 - 
 6 m_i^2 m_j^4 m_k^9 m_l^7 m_n^2 - 6 m_i m_j^5 m_k^9 m_l^7 m_n^2  
 \eeq
 \beq
&&- 
 9 m_i^5 m_j m_k^8 m_l^8 m_n^2 - 9 m_i^4 m_j^2 m_k^8 m_l^8 m_n^2 + 
 9 m_i^3 m_j^3 m_k^8 m_l^8 m_n^2 + 9 m_i^2 m_j^4 m_k^8 m_l^8 m_n^2  + 
 3 m_i^6 m_k^7 m_l^9 m_n^2 \\&&- 3 m_i^5 m_j m_k^7 m_l^9 m_n^2 - 
 15 m_i^4 m_j^2 m_k^7 m_l^9 m_n^2 - 9 m_i^3 m_j^3 m_k^7 m_l^9 m_n^2 + 
 3 m_i^6 m_k^6 m_l^{10} m_n^2 
\\ &&+ 6 m_i^5 m_j m_k^6 m_l^{10} m_n^2 + 
 3 m_i^4 m_j^2 m_k^6 m_l^{10} m_n^2 + 6 m_i^4 m_j^2 m_k^{11} m_l^4 m_n^3 + 
 15 m_i^3 m_j^3 m_k^{11} m_l^4 m_n^3 \\&&+ 9 m_i^2 m_j^4 m_k^{11} m_l^4 m_n^3 - 
 3 m_i m_j^5 m_k^{11} m_l^4 m_n^3 - 3 m_j^6 m_k^{11} m_l^4 m_n^3 - 
 6 m_i^5 m_j m_k^{10} m_l^5 m_n^3 \\&&- 6 m_i^4 m_j^2 m_k^{10} m_l^5 m_n^3 + 
 18 m_i^3 m_j^3 m_k^{10} m_l^5 m_n^3 + 30 m_i^2 m_j^4 m_k^{10} m_l^5 m_n^3 + 
 12 m_i m_j^5 m_k^{10} m_l^5 m_n^3 \\&&- 9 m_i^5 m_j m_k^9 m_l^6 m_n^3 - 
 27 m_i^4 m_j^2 m_k^9 m_l^6 m_n^3 - 35 m_i^3 m_j^3 m_k^9 m_l^6 m_n^3 - 
 21 m_i^2 m_j^4 m_k^9 m_l^6 m_n^3 \\&&+ 4 m_j^6 m_k^9 m_l^6 m_n^3 + 
 18 m_i^4 m_j^2 m_k^8 m_l^7 m_n^3 + 18 m_i^3 m_j^3 m_k^8 m_l^7 m_n^3 - 
 18 m_i^2 m_j^4 m_k^8 m_l^7 m_n^3 \\&&- 18 m_i m_j^5 m_k^8 m_l^7 m_n^3 - 
 9 m_i^5 m_j m_k^7 m_l^8 m_n^3 + 9 m_i^4 m_j^2 m_k^7 m_l^8 m_n^3 + 
 45 m_i^3 m_j^3 m_k^7 m_l^8 m_n^3 
\\ &&+ 27 m_i^2 m_j^4 m_k^7 m_l^8 m_n^3 - 
 12 m_i^5 m_j m_k^6 m_l^9 m_n^3 - 24 m_i^4 m_j^2 m_k^6 m_l^9 m_n^3 - 
 12 m_i^3 m_j^3 m_k^6 m_l^9 m_n^3 \\&&+ 3 m_i^5 m_j m_k^{12} m_l^2 m_n^4 + 
 9 m_i^4 m_j^2 m_k^{12} m_l^2 m_n^4 + 9 m_i^3 m_j^3 m_k^{12} m_l^2 m_n^4 + 
 3 m_i^2 m_j^4 m_k^{12} m_l^2 m_n^4 \\&&- 3 m_i^6 m_k^{11} m_l^3 m_n^4 - 
 3 m_i^5 m_j m_k^{11} m_l^3 m_n^4 + 9 m_i^4 m_j^2 m_k^{11} m_l^3 m_n^4 + 
 15 m_i^3 m_j^3 m_k^{11} m_l^3 m_n^4 \\&&+ 6 m_i^2 m_j^4 m_k^{11} m_l^3 m_n^4 - 
 6 m_i^6 m_k^{10} m_l^4 m_n^4 - 15 m_i^5 m_j m_k^{10} m_l^4 m_n^4 + 
 15 m_i^4 m_j^2 m_k^{10} m_l^4 m_n^4 \\&&+ 48 m_i^3 m_j^3 m_k^{10} m_l^4 m_n^4 + 
 15 m_i^2 m_j^4 m_k^{10} m_l^4 m_n^4 - 15 m_i m_j^5 m_k^{10} m_l^4 m_n^4 - 
 6 m_j^6 m_k^{10} m_l^4 m_n^4 \\
 &&- 3 m_i^6 m_k^9 m_l^5 m_n^4 - 
 51 m_i^5 m_j m_k^9 m_l^5 m_n^4 - 42 m_i^4 m_j^2 m_k^9 m_l^5 m_n^4 + 
 66 m_i^3 m_j^3 m_k^9 m_l^5 m_n^4 
\\ &&+ 69 m_i^2 m_j^4 m_k^9 m_l^5 m_n^4 + 
 9 m_i m_j^5 m_k^9 m_l^5 m_n^4 + 21 m_i^6 m_k^8 m_l^6 m_n^4 - 
 42 m_i^5 m_j m_k^8 m_l^6 m_n^4 
\\ &&- 126 m_i^4 m_j^2 m_k^8 m_l^6 m_n^4 - 
 48 m_i^3 m_j^3 m_k^8 m_l^6 m_n^4 + 15 m_i^2 m_j^4 m_k^8 m_l^6 m_n^4 + 
 9 m_i m_j^5 m_k^8 m_l^6 m_n^4 \\&&+ 9 m_j^6 m_k^8 m_l^6 m_n^4 + 
 42 m_i^6 m_k^7 m_l^7 m_n^4 + 42 m_i^5 m_j m_k^7 m_l^7 m_n^4 - 
 42 m_i^4 m_j^2 m_k^7 m_l^7 m_n^4 \\&&- 54 m_i^3 m_j^3 m_k^7 m_l^7 m_n^4 - 
 39 m_i^2 m_j^4 m_k^7 m_l^7 m_n^4 - 27 m_i m_j^5 m_k^7 m_l^7 m_n^4 + 
 21 m_i^6 m_k^6 m_l^8 m_n^4 \\&&+ 51 m_i^5 m_j m_k^6 m_l^8 m_n^4 + 
 30 m_i^4 m_j^2 m_k^6 m_l^8 m_n^4 + 21 m_i^3 m_j^3 m_k^6 m_l^8 m_n^4 + 
 21 m_i^2 m_j^4 m_k^6 m_l^8 m_n^4 \\&&- 3 m_i^6 m_k^5 m_l^9 m_n^4 + 
 15 m_i^5 m_j m_k^5 m_l^9 m_n^4 + 9 m_i^4 m_j^2 m_k^5 m_l^9 m_n^4 - 
 9 m_i^3 m_j^3 m_k^5 m_l^9 m_n^4 \\&&- 6 m_i^6 m_k^4 m_l^{10} m_n^4 + 
 3 m_i^5 m_j m_k^4 m_l^{10} m_n^4 + 9 m_i^4 m_j^2 m_k^4 m_l^{10} m_n^4 - 
 3 m_i^6 m_k^3 m_l^{11} m_n^4\\&& - 3 m_i^5 m_j m_k^3 m_l^{11} m_n^4 + 
 3 m_i^5 m_j m_k^{11} m_l^2 m_n^5 + 3 m_i^4 m_j^2 m_k^{11} m_l^2 m_n^5 - 
 9 m_i^3 m_j^3 m_k^{11} m_l^2 m_n^5 \\&&- 15 m_i^2 m_j^4 m_k^{11} m_l^2 m_n^5 - 
 6 m_i m_j^5 m_k^{11} m_l^2 m_n^5 + 12 m_i^5 m_j m_k^{10} m_l^3 m_n^5 + 
 30 m_i^4 m_j^2 m_k^{10} m_l^3 m_n^5 \\&&+ 18 m_i^3 m_j^3 m_k^{10} m_l^3 m_n^5 - 
 6 m_i^2 m_j^4 m_k^{10} m_l^3 m_n^5 - 6 m_i m_j^5 m_k^{10} m_l^3 m_n^5 + 
 9 m_i^5 m_j m_k^9 m_l^4 m_n^5 \\&&+ 69 m_i^4 m_j^2 m_k^9 m_l^4 m_n^5 + 
 66 m_i^3 m_j^3 m_k^9 m_l^4 m_n^5 - 42 m_i^2 m_j^4 m_k^9 m_l^4 m_n^5 - 
 51 m_i m_j^5 m_k^9 m_l^4 m_n^5 \\&&- 3 m_j^6 m_k^9 m_l^4 m_n^5 - 
 42 m_i^5 m_j m_k^8 m_l^5 m_n^5 + 84 m_i^4 m_j^2 m_k^8 m_l^5 m_n^5 + 
 252 m_i^3 m_j^3 m_k^8 m_l^5 m_n^5 \\&&+ 84 m_i^2 m_j^4 m_k^8 m_l^5 m_n^5 - 
 42 m_i m_j^5 m_k^8 m_l^5 m_n^5 - 126 m_i^5 m_j m_k^7 m_l^6 m_n^5 - 
 126 m_i^4 m_j^2 m_k^7 m_l^6 m_n^5 \\&&+ 132 m_i^3 m_j^3 m_k^7 m_l^6 m_n^5 + 
 132 m_i^2 m_j^4 m_k^7 m_l^6 m_n^5 + 9 m_i m_j^5 m_k^7 m_l^6 m_n^5 + 
 9 m_j^6 m_k^7 m_l^6 m_n^5 \\&&- 84 m_i^5 m_j m_k^6 m_l^7 m_n^5 - 
 186 m_i^4 m_j^2 m_k^6 m_l^7 m_n^5 - 78 m_i^3 m_j^3 m_k^6 m_l^7 m_n^5 + 
 6 m_i^2 m_j^4 m_k^6 m_l^7 m_n^5 \\&&- 18 m_i m_j^5 m_k^6 m_l^7 m_n^5 + 
 9 m_i^5 m_j m_k^5 m_l^8 m_n^5 - 45 m_i^4 m_j^2 m_k^5 m_l^8 m_n^5 - 
 27 m_i^3 m_j^3 m_k^5 m_l^8 m_n^5 \\&&+ 27 m_i^2 m_j^4 m_k^5 m_l^8 m_n^5 + 
 24 m_i^5 m_j m_k^4 m_l^9 m_n^5 - 12 m_i^4 m_j^2 m_k^4 m_l^9 m_n^5 - 
 36 m_i^3 m_j^3 m_k^4 m_l^9 m_n^5 \\&&+ 15 m_i^5 m_j m_k^3 m_l^{10} m_n^5 + 
 15 m_i^4 m_j^2 m_k^3 m_l^{10} m_n^5 + m_i^6 m_k^{12} m_n^6 + 
 3 m_i^5 m_j m_k^{12} m_n^6 \\&&+ 3 m_i^4 m_j^2 m_k^{12} m_n^6 + 
 m_i^3 m_j^3 m_k^{12} m_n^6 + 3 m_i^6 m_k^{11} m_l m_n^6 + 
 9 m_i^5 m_j m_k^{11} m_l m_n^6 \\&&+ 9 m_i^4 m_j^2 m_k^{11} m_l m_n^6 + 
 3 m_i^3 m_j^3 m_k^{11} m_l m_n^6 + 3 m_i^6 m_k^{10} m_l^2 m_n^6 + 
 6 m_i^5 m_j m_k^{10} m_l^2 m_n^6 \\&&- 3 m_i^4 m_j^2 m_k^{10} m_l^2 m_n^6 - 
 15 m_i^3 m_j^3 m_k^{10} m_l^2 m_n^6 - 9 m_i^2 m_j^4 m_k^{10} m_l^2 m_n^6 + 
 3 m_i m_j^5 m_k^{10} m_l^2 m_n^6 \\&&+ 3 m_j^6 m_k^{10} m_l^2 m_n^6 + 
 4 m_i^6 m_k^9 m_l^3 m_n^6 - 21 m_i^4 m_j^2 m_k^9 m_l^3 m_n^6 - 
 35 m_i^3 m_j^3 m_k^9 m_l^3 m_n^6 \\
 &&- 27 m_i^2 m_j^4 m_k^9 m_l^3 m_n^6 - 
 9 m_i m_j^5 m_k^9 m_l^3 m_n^6 + 9 m_i^6 m_k^8 m_l^4 m_n^6 + 
 9 m_i^5 m_j m_k^8 m_l^4 m_n^6 \\
 &&+ 15 m_i^4 m_j^2 m_k^8 m_l^4 m_n^6 - 
 48 m_i^3 m_j^3 m_k^8 m_l^4 m_n^6 - 126 m_i^2 m_j^4 m_k^8 m_l^4 m_n^6 - 
 42 m_i m_j^5 m_k^8 m_l^4 m_n^6 \\
 &&+ 21 m_j^6 m_k^8 m_l^4 m_n^6 + 
 9 m_i^6 m_k^7 m_l^5 m_n^6 + 9 m_i^5 m_j m_k^7 m_l^5 m_n^6 + 
 132 m_i^4 m_j^2 m_k^7 m_l^5 m_n^6 \\
 &&+ 132 m_i^3 m_j^3 m_k^7 m_l^5 m_n^6 - 
 126 m_i^2 m_j^4 m_k^7 m_l^5 m_n^6 - 126 m_i m_j^5 m_k^7 m_l^5 m_n^6 + 
 6 m_i^6 m_k^6 m_l^6 m_n^6 
  \eeq
 \beq
&&- 9 m_i^5 m_j m_k^6 m_l^6 m_n^6 + 
 120 m_i^4 m_j^2 m_k^6 m_l^6 m_n^6 + 269 m_i^3 m_j^3 m_k^6 m_l^6 m_n^6 + 
 120 m_i^2 m_j^4 m_k^6 m_l^6 m_n^6 \\&&- 9 m_i m_j^5 m_k^6 m_l^6 m_n^6 + 
 6 m_j^6 m_k^6 m_l^6 m_n^6 + 9 m_i^6 m_k^5 m_l^7 m_n^6 - 
 9 m_i^5 m_j m_k^5 m_l^7 m_n^6 \\&&- 21 m_i^4 m_j^2 m_k^5 m_l^7 m_n^6 + 
 45 m_i^3 m_j^3 m_k^5 m_l^7 m_n^6 + 27 m_i^2 m_j^4 m_k^5 m_l^7 m_n^6 - 
 27 m_i m_j^5 m_k^5 m_l^7 m_n^6 \\&&+ 9 m_i^6 m_k^4 m_l^8 m_n^6 - 
 33 m_i^4 m_j^2 m_k^4 m_l^8 m_n^6 + 15 m_i^3 m_j^3 m_k^4 m_l^8 m_n^6 + 
 54 m_i^2 m_j^4 m_k^4 m_l^8 m_n^6 \\&&+ 4 m_i^6 m_k^3 m_l^9 m_n^6 - 
 6 m_i^5 m_j m_k^3 m_l^9 m_n^6 - 21 m_i^4 m_j^2 m_k^3 m_l^9 m_n^6 - 
 31 m_i^3 m_j^3 m_k^3 m_l^9 m_n^6 \\&&+ 3 m_i^6 m_k^2 m_l^{10} m_n^6 - 
 9 m_i^5 m_j m_k^2 m_l^{10} m_n^6 + 3 m_i^4 m_j^2 m_k^2 m_l^{10} m_n^6 + 
 3 m_i^6 m_k m_l^{11} m_n^6 \\&&- 3 m_i^5 m_j m_k m_l^{11} m_n^6 + m_i^6 m_l^{12} m_n^6 - 
 3 m_i^5 m_j m_k^{11} m_n^7 - 9 m_i^4 m_j^2 m_k^{11} m_n^7 - 
 9 m_i^3 m_j^3 m_k^{11} m_n^7 \\&&- 3 m_i^2 m_j^4 m_k^{11} m_n^7 - 
 6 m_i^5 m_j m_k^{10} m_l m_n^7 - 18 m_i^4 m_j^2 m_k^{10} m_l m_n^7 - 
 18 m_i^3 m_j^3 m_k^{10} m_l m_n^7 \\
 &&- 6 m_i^2 m_j^4 m_k^{10} m_l m_n^7 - 
 6 m_i^5 m_j m_k^9 m_l^2 m_n^7 - 6 m_i^4 m_j^2 m_k^9 m_l^2 m_n^7 + 
 9 m_i^3 m_j^3 m_k^9 m_l^2 m_n^7 \\&&+ 15 m_i^2 m_j^4 m_k^9 m_l^2 m_n^7 + 
 9 m_i m_j^5 m_k^9 m_l^2 m_n^7 + 3 m_j^6 m_k^9 m_l^2 m_n^7 - 
 18 m_i^5 m_j m_k^8 m_l^3 m_n^7 \\&&- 18 m_i^4 m_j^2 m_k^8 m_l^3 m_n^7 + 
 18 m_i^3 m_j^3 m_k^8 m_l^3 m_n^7 + 18 m_i^2 m_j^4 m_k^8 m_l^3 m_n^7 - 
 27 m_i^5 m_j m_k^7 m_l^4 m_n^7 
\\ &&- 39 m_i^4 m_j^2 m_k^7 m_l^4 m_n^7 - 
 54 m_i^3 m_j^3 m_k^7 m_l^4 m_n^7 - 42 m_i^2 m_j^4 m_k^7 m_l^4 m_n^7 + 
 42 m_i m_j^5 m_k^7 m_l^4 m_n^7 \\&&+ 42 m_j^6 m_k^7 m_l^4 m_n^7 - 
 18 m_i^5 m_j m_k^6 m_l^5 m_n^7 + 6 m_i^4 m_j^2 m_k^6 m_l^5 m_n^7 - 
 78 m_i^3 m_j^3 m_k^6 m_l^5 m_n^7 \\&&- 186 m_i^2 m_j^4 m_k^6 m_l^5 m_n^7 - 
 84 m_i m_j^5 m_k^6 m_l^5 m_n^7 - 27 m_i^5 m_j m_k^5 m_l^6 m_n^7 + 
 27 m_i^4 m_j^2 m_k^5 m_l^6 m_n^7 \\&&+ 45 m_i^3 m_j^3 m_k^5 m_l^6 m_n^7 - 
 21 m_i^2 m_j^4 m_k^5 m_l^6 m_n^7 - 9 m_i m_j^5 m_k^5 m_l^6 m_n^7 + 
 9 m_j^6 m_k^5 m_l^6 m_n^7 \\&&- 36 m_i^5 m_j m_k^4 m_l^7 m_n^7 - 
 6 m_i^4 m_j^2 m_k^4 m_l^7 m_n^7 + 30 m_i^3 m_j^3 m_k^4 m_l^7 m_n^7 - 
 6 m_i^2 m_j^4 m_k^4 m_l^7 m_n^7 \\&&- 36 m_i m_j^5 m_k^4 m_l^7 m_n^7 - 
 18 m_i^5 m_j m_k^3 m_l^8 m_n^7 + 12 m_i^4 m_j^2 m_k^3 m_l^8 m_n^7 + 
 3 m_i^3 m_j^3 m_k^3 m_l^8 m_n^7 \\&&+ 33 m_i^2 m_j^4 m_k^3 m_l^8 m_n^7 - 
 12 m_i^5 m_j m_k^2 m_l^9 m_n^7 + 36 m_i^4 m_j^2 m_k^2 m_l^9 m_n^7 - 
 12 m_i^3 m_j^3 m_k^2 m_l^9 m_n^7 \\&&- 15 m_i^5 m_j m_k m_l^{10} m_n^7 + 
 15 m_i^4 m_j^2 m_k m_l^{10} m_n^7 - 6 m_i^5 m_j m_l^{11} m_n^7 + 
 3 m_i^4 m_j^2 m_k^{10} m_n^8 \\&&+ 9 m_i^3 m_j^3 m_k^{10} m_n^8 + 
 9 m_i^2 m_j^4 m_k^{10} m_n^8 + 3 m_i m_j^5 m_k^{10} m_n^8 + 
 3 m_i^4 m_j^2 m_k^9 m_l m_n^8 \\&&+ 9 m_i^3 m_j^3 m_k^9 m_l m_n^8 + 
 9 m_i^2 m_j^4 m_k^9 m_l m_n^8 + 3 m_i m_j^5 m_k^9 m_l m_n^8 + 
 9 m_i^4 m_j^2 m_k^8 m_l^2 m_n^8 \\&&+ 9 m_i^3 m_j^3 m_k^8 m_l^2 m_n^8 - 
 9 m_i^2 m_j^4 m_k^8 m_l^2 m_n^8 - 9 m_i m_j^5 m_k^8 m_l^2 m_n^8 + 
 27 m_i^4 m_j^2 m_k^7 m_l^3 m_n^8 \\&&+ 45 m_i^3 m_j^3 m_k^7 m_l^3 m_n^8 + 
 9 m_i^2 m_j^4 m_k^7 m_l^3 m_n^8 - 9 m_i m_j^5 m_k^7 m_l^3 m_n^8 + 
 21 m_i^4 m_j^2 m_k^6 m_l^4 m_n^8 \\&&+ 21 m_i^3 m_j^3 m_k^6 m_l^4 m_n^8 + 
 30 m_i^2 m_j^4 m_k^6 m_l^4 m_n^8 + 51 m_i m_j^5 m_k^6 m_l^4 m_n^8 + 
 21 m_j^6 m_k^6 m_l^4 m_n^8 \\&&+ 27 m_i^4 m_j^2 m_k^5 m_l^5 m_n^8 - 
 27 m_i^3 m_j^3 m_k^5 m_l^5 m_n^8 - 45 m_i^2 m_j^4 m_k^5 m_l^5 m_n^8 + 
 9 m_i m_j^5 m_k^5 m_l^5 m_n^8 \\&&+ 54 m_i^4 m_j^2 m_k^4 m_l^6 m_n^8 + 
 15 m_i^3 m_j^3 m_k^4 m_l^6 m_n^8 - 33 m_i^2 m_j^4 m_k^4 m_l^6 m_n^8 + 
 9 m_j^6 m_k^4 m_l^6 m_n^8 \\
 &&+ 33 m_i^4 m_j^2 m_k^3 m_l^7 m_n^8 + 
 3 m_i^3 m_j^3 m_k^3 m_l^7 m_n^8 + 12 m_i^2 m_j^4 m_k^3 m_l^7 m_n^8 - 
 18 m_i m_j^5 m_k^3 m_l^7 m_n^8 \\&&+ 18 m_i^4 m_j^2 m_k^2 m_l^8 m_n^8 - 
 54 m_i^3 m_j^3 m_k^2 m_l^8 m_n^8 + 18 m_i^2 m_j^4 m_k^2 m_l^8 m_n^8 + 
 30 m_i^4 m_j^2 m_k m_l^9 m_n^8 \\&&- 30 m_i^3 m_j^3 m_k m_l^9 m_n^8 + 
 15 m_i^4 m_j^2 m_l^{10} m_n^8 - m_i^3 m_j^3 m_k^9 m_n^9 - 
 3 m_i^2 m_j^4 m_k^9 m_n^9 - 3 m_i m_j^5 m_k^9 m_n^9 \\
 && - m_j^6 m_k^9 m_n^9 - 
 9 m_i^3 m_j^3 m_k^7 m_l^2 m_n^9 - 15 m_i^2 m_j^4 m_k^7 m_l^2 m_n^9 - 
 3 m_i m_j^5 m_k^7 m_l^2 m_n^9 + 3 m_j^6 m_k^7 m_l^2 m_n^9 \\&& -
 12 m_i^3 m_j^3 m_k^6 m_l^3 m_n^9 - 24 m_i^2 m_j^4 m_k^6 m_l^3 m_n^9 - 
 12 m_i m_j^5 m_k^6 m_l^3 m_n^9 - 9 m_i^3 m_j^3 m_k^5 m_l^4 m_n^9 \\&&+ 
 9 m_i^2 m_j^4 m_k^5 m_l^4 m_n^9 + 15 m_i m_j^5 m_k^5 m_l^4 m_n^9 - 
 3 m_j^6 m_k^5 m_l^4 m_n^9 - 36 m_i^3 m_j^3 m_k^4 m_l^5 m_n^9 \\&&- 
 12 m_i^2 m_j^4 m_k^4 m_l^5 m_n^9 + 24 m_i m_j^5 m_k^4 m_l^5 m_n^9 - 
 31 m_i^3 m_j^3 m_k^3 m_l^6 m_n^9 - 21 m_i^2 m_j^4 m_k^3 m_l^6 m_n^9 \\
 &&- 
 6 m_i m_j^5 m_k^3 m_l^6 m_n^9 + 4 m_j^6 m_k^3 m_l^6 m_n^9 - 
 12 m_i^3 m_j^3 m_k^2 m_l^7 m_n^9 + 36 m_i^2 m_j^4 m_k^2 m_l^7 m_n^9 \\&&- 
 12 m_i m_j^5 m_k^2 m_l^7 m_n^9 - 30 m_i^3 m_j^3 m_k m_l^8 m_n^9 + 
 30 m_i^2 m_j^4 m_k m_l^8 m_n^9 - 20 m_i^3 m_j^3 m_l^9 m_n^9 \\
 &&+  3 m_i^2 m_j^4 m_k^6 m_l^2 m_n^{10} + 6 m_i m_j^5 m_k^6 m_l^2 m_n^{10} + 
 3 m_j^6 m_k^6 m_l^2 m_n^{10} + 9 m_i^2 m_j^4 m_k^4 m_l^4 m_n^{10} \\&& + 
 3 m_i m_j^5 m_k^4 m_l^4 m_n^{10} - 6 m_j^6 m_k^4 m_l^4 m_n^{10} + 
 15 m_i^2 m_j^4 m_k^3 m_l^5 m_n^{10} + 15 m_i m_j^5 m_k^3 m_l^5 m_n^{10} \\&& + 
 3 m_i^2 m_j^4 m_k^2 m_l^6 m_n^{10} - 9 m_i m_j^5 m_k^2 m_l^6 m_n^{10} + 
 3 m_j^6 m_k^2 m_l^6 m_n^{10} + 15 m_i^2 m_j^4 m_k m_l^7 m_n^{10} \\&& - 
 15 m_i m_j^5 m_k m_l^7 m_n^{10} + 15 m_i^2 m_j^4 m_l^8 m_n^{10} - 
 3 m_i m_j^5 m_k^3 m_l^4 m_n^{11} - 3 m_j^6 m_k^3 m_l^4 m_n^{11} \\&& - 
 3 m_i m_j^5 m_k m_l^6 m_n^{11} + 3 m_j^6 m_k m_l^6 m_n^{11} - 
 6 m_i m_j^5 m_l^7 m_n^{11} + m_j^6 m_l^6 m_n^{12}
\eeq 
}

\normalsize






\bibliographystyle{amsplain}

\end{document}